\documentclass[a4paper,11pt,reqno]{amsart}
\usepackage[T1]{fontenc}
\usepackage[utf8]{inputenc}
\usepackage[margin=1.1in]{geometry}
\usepackage{lmodern}
\usepackage[english]{babel}
\usepackage{graphicx,mathrsfs}
\usepackage{cases}
\usepackage{upgreek}
\usepackage{amsmath}
\usepackage{amsfonts}
\usepackage{amssymb} 
\usepackage{amsthm}
\usepackage{bbm}
\usepackage{bm} 
\usepackage{lipsum}
\usepackage{ulem}
\usepackage{subcaption}
\usepackage{xcolor}
\usepackage{url}
\usepackage{hyperref}
\usepackage{csquotes}
\usepackage{stmaryrd}
\usepackage{pgf,tikz}
\usepackage{pgfplots,pgfplotstable}
\usetikzlibrary{calc,trees,positioning,arrows,chains,shapes.geometric,%
    decorations.pathreplacing,decorations.pathmorphing,shapes,%
    matrix,shapes.symbols}
\usepackage{multirow}
\usepackage{tikz}
\usepackage{wrapfig}
\usepackage{graphicx}
\graphicspath{ {./images/} }
\usepackage{subcaption}
\usepackage{blindtext}
\usepackage[ruled,vlined]{algorithm2e}
\usepackage{xcolor}

\tikzset{
>=stealth',
  punktchain/.style={
    rectangle, 
    rounded corners, 
    draw=black, very thick,
    text width=10em, 
    minimum height=3em, 
    text centered, 
    on chain},
  line/.style={draw, thick, <-},
  element/.style={
    tape,
    top color=white,
    bottom color=blue!50!black!60!,
    minimum width=8em,
    draw=blue!40!black!90, very thick,
    text width=10em, 
    minimum height=3.5em, 
    text centered, 
    on chain},
  every join/.style={->, thick,shorten >=1pt},
  decoration={brace},
  tuborg/.style={decorate},
  tubnode/.style={midway, right=2pt},
}

\pgfplotsset{contents/.style={axis x line=none, axis y line=none, scale = 2,area style,ymin=-2,ymax=22,enlargelimits=true}}

\allowdisplaybreaks
\numberwithin{equation}{section}

\theoremstyle{plain}
\newtheorem{thm}{Theorem}[section]
\newtheorem{lem}[thm]{Lemma}%[section]
\newtheorem{prop}[thm]{Proposition}%[section]

\newtheorem{rem}[thm]{Remark}%[section]

\def \be {\begin{equation}}
\def \ee {\end{equation}}

\def\balpha{\boldsymbol \alpha}
\def\ud{\mathrm{d}}

\renewcommand{\phi}{\varphi}
\renewcommand{\epsilon}{\varepsilon}
\renewcommand{\tilde}{\widetilde}
\renewcommand{\hat}{\widehat}

\newtheorem{innercustomgeneric}{\customgenericname}
\providecommand{\customgenericname}{}
\newcommand{\newcustomtheorem}[2]{%
  \newenvironment{#1}[1]
  {%
   \renewcommand\customgenericname{#2}%
   \renewcommand\theinnercustomgeneric{##1}%
   \innercustomgeneric
  }
  {\endinnercustomgeneric}
}

\newcustomtheorem{customthm}{Theorem}

\begin{document}

\title[Exploration for MFG]{Exploration noise for learning linear-quadratic mean field games}

\author{Fran\c cois Delarue}
\author{Athanasios Vasileiadis}

\address[F. Delarue and A. Vasileiadis]
{\newline \indent Universit\'e C\^ote d'Azur, CNRS, 
Laboratoire de Math\'ematiques  J.A. Dieudonn\'e,
\newline 
\indent 28 Avenue Valrose, 06108 Nice Cedex 2, France
\newline }
\email{francois.delarue@univ-cotedazur.fr, athanasios.vasileiadis@univ-cotedazur.fr}
\thanks{F. Delarue and A. Vasileiadis acknowledge the financial support of French ANR project ANR-19-P3IA-0002 -- 3IA C\^ote d'Azur -- Nice -- Interdisciplinary Institute for Artificial Intelligence. 
F. Delarue is also supported by French ANR project ANR-16-CE40-0015-01 on ``Mean Field Games''.
}

\date{\today}



\begin{abstract}
The goal of this paper is to demonstrate that common noise may serve as an exploration noise 
for learning the solution of a mean field game. This concept is here exemplified 
through a toy linear-quadratic model, for which a suitable form of common noise
has already been proven to restore existence and uniqueness. We here go one step further and prove that the same form of common noise may force the convergence of the learning algorithm called `fictitious play', and this without any further potential or monotone structure. 
Several numerical examples are provided in order to support our theoretical analysis.  
\end{abstract}
%\tableofcontents
\maketitle

\section{Introduction}

\subsection{General context.}
Since their inception fifteen years ago in the seminal works of Lasry and Lions
\cite{Lasry2006,LasryLions2,LasryLions}, Lions \cite{Lionscollege}
 and Huang et al. \cite{HuangCainesMalhame1,Huang2006,Huang2007}, mean field games (MFGs) have met a tremendous success, inspiring mathematical works from different areas like partial differential equations (PDEs), control theory, stochastic analysis, calculus of variation, toward a consistent and expanded theory for games with many weakly interacting rational agents. Meanwhile, the increasing number of applications has stimulated a long series of works on discretization and numerical methods for approximating the underlying equilibria (which we also call solutions); see for instance 
 Achdou et al.
 \cite{AchdouCamilliCapuzzo1,AchdouCamilliCapuzzo2}, 
 Achdou and Capuzzo-Dolcetta
\cite{AchdouCapuzzo} and Achdou and Laurière \cite{AchdouLauriere} for discretization with finite differences; 
Carlini and Silva
 \cite{CarliniSilva1,CarliniSilva2} for semi-Lagrangian schemes; 
 Benamou and Carlier
 \cite{BenamouCarlier} 
 and
 Bricen\~{o} Arias et al.
 \cite{AriasKaliseSilva,Kobeissi} for variational discretization; and 
 Achdou and Laurière 
\cite{achdou-cetraro} for an overview. 
These contributions often include numerical methods for solving the discrete schemes
such as the 
Picard method, Newton method, fictitious play. We review the latter one in detail in the sequel. 
Recently, other works have also demonstrated the possible efficiency of tools from machine learning within this  complex framework: standard equations for characterizing the equilibria may be approximately solved by means of a neural network;
% and the weights of the latter may be tuned numerically by means of 
%a descent method; 
see for instance
Carmona and Laurière \cite{CarmonaLauriereI,CarmonaLauriereII}. Last (but not least), motivated by the desire to develop methods for models with partially unknown data,  
several authors have integrated important concepts from reinforcement learning in their studies; we refer to 
Carmona et al. \cite{CarmonaLauriereTanI,CarmonaLauriereTanII}, Elie et al. \cite{EliePerolat} and Guo et al. \cite{GuoHuXuZhang}. 

The aim of our work is to make one new step forward in the latter direction with a study at the frontier between the theoretical analysis of MFGs and the development of appropriate forms of learning. In particular, our main objective is to provide a proof of concept for the notion of exploration noise, which is certainly a key ingredient in machine learning. For the reader who is aware of the MFG theory, our basic idea is to prove that common noise may indeed serve for the exploration of the space of solutions and, in the end, for the improvement of the existing learning algorithms.  For sure, this looks a very ambitious program since there has not been, so far, any complete understanding of the impact of the common noise onto the shape of the solutions. However, several recent works clearly indicate that noise might be helpful for 
numerical purposes. Indeed, in a series of works, 
Bayraktar et al.
 \cite{mfggenetic}, Delarue \cite{DelarueSPDE},
 and Foguen Tchuendom 
\cite{fog2018}, it was shown that common noise could help to force uniqueness of equilibria in 
a certain number of MFGs.
This is a striking fact because nonuniqueness is the typical rule for MFGs, except in some particular classes with some specific structure; see for instance the famous uniqueness result of Lasry and Lions \cite{Lasry2006}  for games with monotonous coefficients and Carmona and Delarue \cite[chapter 3]{CarmonaDelarue_book_I}.  Conceptually, 
the key condition for forcing uniqueness is that, under the action of the (hence common) noise, 
the equilibria 
%which manifests as a randomisation of the equilibria, is sufficiently strong to 
%smooth out the so-called MFG forward-backward system that is used to describe the solutions.
%From a purely probabilistic point of view, it is well-known that a similar regularisation effect occurs for diffusion processes when the noise pushes the 
can visit sufficiently well the state space in which they live. This makes the whole rather subtle because, in full theory, the state space is the space of probability measures, which is typically infinite-dimensional. In this respect, the main examples for which such a smoothing effect has been rigorously established so far are $(i)$ a linear quadratic model with possibly nonlinear mean field interaction terms in the cost functional (Foguen Tchuendom \cite{fog2018}), $(ii)$ a general model with an infinite dimensional noise combined with a suitable form of local interaction in the dynamics (Delarue \cite{DelarueSPDE}); and $(iii)$ models defined on finite state spaces and forced by a variant of the Wright-Fisher noise used in population genetics literature (Bayraktar et al. \cite{mfggenetic}).

\color{black}
In order to prove the possible efficiency of our approach, we here restrict the whole discussion to the first of the three aforementioned instances. 
We provide below a long informative introduction in which 
we present the model in detail together with the related literature and our own results. The guideline of this introduction is the following. 
We specify the model in 
Subsection \ref{subse:model}, both without and with common noise. In particular, 
we recall therein the various known characterizations 
of the equilibria in both cases. 
In 
Subsection 
\ref{subse:learning}, we provide a brief review of the existing learning procedures
and  
we exemplify the form of the so-called fictitious play
for the linear-quadratic model studied in the paper. 
The thrust of our paper is to introduce a variant of the fictitious play, which we refer to as the \textit{tilted fictitious play}
and whose convergence can be proven in the presence of common noise under more general conditions 
than those of the standard fictitious play (within the class of linear-quadratic MFGs 
introduced in Subsection \ref{subse:model}). This tilted fictitious play is defined in 
Subsection 
\ref{subse:tilted}. 
Although the 
 tilted fictitious play
 can be regarded as a theoretical 
 learning scheme, 
 we explain in 
 Subsection 
 \ref{subse:exploration}
 how it can be adapted 
 to 
 statistical learning.
 Notably,  
in this adaptation, the common noise serves as an exploration noise for learning the equilibria of the underlying MFG. This interpretation of the common noise in terms of exploration is in fact one key point of the paper.   
Exploitation is then briefly addressed in Subsection 
\ref{subse:exploitation}. Therein, we present the main bounds that we are able to prove for the mean \textcolor{black}{exploitability} of the policy that is returned by the tilted fictitious play. 
Finally, 
we provide an overview of our numerical results 
in Subsection 
\ref{subse:num:ex}, and we conclude the introduction by providing a comparison with the existing literature in 
 Subsection 
 \ref{subse:comparison}.
 The main assumption and notation are 
 given in
 Subsection
\ref{subse:1:9}. The organization of the paper is also presented in Subsection 
\ref{subse:1:9}.
\color{black}

\subsection{Our model.}
\label{subse:model}
\textcolor{black}{In this subsection, we expose the linear-quadratic class of MFGs that is addressed in the paper, without and with common noise. We insist 
in both cases
on the form of the 
equilibrium feedbacks. As made clear in 
\eqref{eq:intro:optimal:alpha}, those feedbacks are necessarily affine. 
Moreover, 
for any of them,  
the intercept in the affine structure is characterized by a backward stochastic differential equation (BSDE), which is recalled in 
 \eqref{eq:intro:BSDE}
 and
which plays a key role in the subsequent analysis.}

\textcolor{black}{To make it clear, the MFG that we consider below comprises the state equation}
\begin{equation}
\label{eq:state:intro}
\ud X_{t} = \alpha_{t} \ud t + \sigma \ud B_{t} + \varepsilon \ud W_{t}, \quad t \in [0,T],
\end{equation}
together with the cost functional
\begin{equation}
\label{eq:cost:intro}
J\bigl({\boldsymbol \alpha} ; {\boldsymbol m} \bigr) = \frac12  {\mathbb E}
\biggl[ \bigl\vert R X_{T} + g(m_{T}) \bigr\vert^2 + 
\int_{0}^T \Bigl\{ \bigl\vert Q X_{t} + f(m_{t}) \bigr\vert^2 + 
\vert \alpha_{t} \vert^2
\Bigr\} \ud t \biggr].
\end{equation}
Above, $X_{t}$ is the state at time $t$ of a representative agent evolving within a continuum of 
other agents. It takes values in ${\mathbb R}^d$, for some integer $d \geq 1$, 
and evolves from the initial time $0$ to the terminal 
time $T$ according to the equation \eqref{eq:state:intro}. Therein, 
 ${\boldsymbol B}=(B_{t})_{0 \le t \le T}$ and ${\boldsymbol W}=(W_{t})_{0 \le t \le T}$ are 
two independent $d$-dimensional Brownian motions
on a complete probability space $(\Omega,{\mathcal A},{\mathbb P})$,
and $\sigma \geq 0$ and $\varepsilon \geq 0$ account for their respective intensity
in the dynamics. Whereas ${\boldsymbol B}$ is thought as a private (or idiosyncratic) noise
felt by the representative player only (and not by the others), the process 
${\boldsymbol W}$ is said to be a \textit{common} or \textit{systemic} noise as 
all the others in the continuum also feel it.  
The initial condition (also called initial private state) $X_{0}$ may be random and, in any case, $X_{0}$
is independent of the pair $({\boldsymbol B},{\boldsymbol W})$. 
The process ${\boldsymbol \alpha}=(\alpha_{t})_{0 \le t \le T}$ is a so-called control process, which is usually assumed to be progressively-measurable with respect to the augmented filtration 
${\mathbb F}$ generated by $(X_{0},\sigma {\boldsymbol B},\varepsilon {\boldsymbol W})$ 
(when $\sigma$ or $\varepsilon$ are zero, the corresponding process is no longer used to generate the filtration). 
The key fact in MFG theory is that the representative player aims at choosing 
the best possible ${\boldsymbol \alpha}$ in order to minimize the 
cost functional $J$ in 
\eqref{eq:cost:intro}. 
The leading symbol ${\mathbb E}$ 
in the definition of $J$ is understood as 
an expectation with respect to all the inputs $(X_{0},{\boldsymbol B},{\boldsymbol W})$. 
If there were no dependence on the extra term ${\boldsymbol m}=(m_{t})_{0 \le t \le T}$
in $f$ and $g$, the minimization  of $J$ would reduce to a mere linear-quadratic stochastic control problem driven by {$Q$ and $R$, with the latter being two $d$-square matrices}\footnote{We could take $Q$ and $R$ as $e \times d$ matrices, for a general $e \geq 1$ and then 
$f$ and $g$ as being $e$-dimensional. For simplicity, we feel easier to work with $e=d$.}. The essence of MFGs is that 
$(m_{t})_{0 \le t \le T}$ accounts for the flow of marginal statistical states of the continuum of players surrounding the representative agent. In full generality, each $m_{t}$ should be regarded as a probability measure hence describing the statistical distribution of the other agents at time $t$. For simplicity, we here just assume 
$m_{t}$ to be the $d$-dimensional mean of the other agents at time $t$; consistently, 
$f=(f^1,\cdots,f^d)$ and $g=(g^1,\cdots,g^d)$ are ${\mathbb R}^d$-valued functions defined on ${\mathbb R}^d$. 
However, the notion of \textit{mean} should be clarified, because of the distinction between the two \textit{private} and \textit{common} noises. From a modelling point of view, this \textit{mean} should result from a law of large numbers taken over players that would be subjected to independent
and identically distributed (i.i.d.) initial and private noises (consistently with our former description of ${\boldsymbol B}$) but to the same 
common noise ${\boldsymbol W}$. Because of this, $(m_{t})_{0 \leq t \leq T}$ is itself 
required to be a stochastic process, progressively-measurable with 
respect to the filtration generated by $\varepsilon {\boldsymbol W}$. 
When $\varepsilon=0$, the filtration becomes trivial and, accordingly, $(m_{t})_{0 \leq t \leq T}$ is assumed to be deterministic.
The notion of Nash equilibrium or MFG solution then comes through a fixed point argument. In short, 
$(m_{t})_{0 \leq t \leq T}$ is said to be an equilibrium if the minimizer $(X_{t}^\star)_{0 \leq t \le T}$ of \eqref{eq:state:intro}--\eqref{eq:cost:intro} satisfies:
\begin{equation}
\label{eq:intro:fixed}
m_{t} = {\mathbb E} \bigl[ X_{t}^\star \, \vert \, \varepsilon {\boldsymbol W} \bigr], \quad t \in [0,T],
\end{equation}
with the conditional expectation becoming an expectation if $\varepsilon=0$
\textcolor{black}{(whence our choice to use 
$\varepsilon {\boldsymbol W} $
as random variable in the conditioning, as this notation allows us to distinguish easily  between the two cases
without and with common noise)}.
Throughout, $f$ and $g$ are typically assumed to be 
bounded and Lipschitz continuous functions. When $\varepsilon=0$, this is enough for ensuring the existence of solutions to the fixed point \eqref{eq:intro:fixed}, but uniqueness is known to fail in general, except under some additional conditions. For instance, the Lasry-Lions monotonicity condition, when adapted to this setting, says that uniqueness indeed holds true 
if $Q^\dagger f$ and $R^\dagger g$ (with $\dagger$ denoting the transpose) are non-decreasing in the sense that (see \cite{delfog2019})
\begin{equation}
\label{eq:monotonicity}
\forall x,x' \in {\mathbb R}^d, \quad 
(x-x') \cdot \bigl(Q^\dagger f(x)- Q^\dagger f(x')\bigr) \geq 0, \quad
(x-x') \cdot \bigl( R^\dagger g(x)-R^\dagger g(x') \bigr) \geq 0,
\end{equation} 
with $\cdot$ denoting the standard inner product in ${\mathbb R}^d$. 
When $\epsilon>0$, existence and uniqueness hold true, even though the coefficients are not monotone (see \cite{fog2018}). This is a very clear instance of the effective impact of the noise onto the search of equilibria. 

The reason why the MFG \eqref{eq:state:intro}--\eqref{eq:cost:intro}--\eqref{eq:intro:fixed}
becomes uniquely solvable under the action of the common noise may be explained as follows. 
In short (and this is the rationale for working in this set-up), 
the linear-quadratic structure of 
\eqref{eq:state:intro}
forces uniqueness of the minimizer to 
\eqref{eq:cost:intro} when ${\boldsymbol m}$ is fixed. Even more the optimal control is given in the Markovian form 
\begin{equation}
\label{eq:intro:optimal:alpha}
\alpha_{t}^\star = - \eta_{t} X_{t}^\star - h_{t}, \quad t \in [0,T], 
\end{equation}
 where ${\boldsymbol \eta}=(\eta_{t})_{0 \leq t \leq T}$ 
 is the $d \times d$-matrix valued solution of an \textit{autonomous} Riccati equation that
 {only depends on $Q$ and $R$ and that is in particular} independent of the 
 input ${\boldsymbol m}$. In other words, only the intercept 
 $({\boldsymbol h}_{t})_{0 \leq t \leq T}$ 
in the above formula  
 depends on the inputs $f$, $g$, $R$, $Q$ and ${\boldsymbol m}$. 
 The characterization  of ${\boldsymbol h}$ is usually 
 obtained by the (stochastic if $\varepsilon >0$) Pontryagin principle, namely 
 ${\boldsymbol h}$ solves the Backward Stochastic Differential Equation (BSDE)
 \begin{equation}
 \label{eq:intro:BSDE}
 h_{t} = R^\dagger g(m_{T}) + \int_{t}^T \bigl\{ Q^\dagger f(m_{s}) - \eta_{s} h_{s} \bigr\} \ud s - \varepsilon \int_{t}^T k_{s} \ud W_{s}, \quad t \in [0,T]. 
 \end{equation}
When $\varepsilon=0$, the above stochastic integral disappears and the  
equation \eqref{eq:intro:BSDE} becomes a mere Ordinary Differential Equation (ODE)
but set backwards in time.
When $\varepsilon >0$, the solution is the pair $({\boldsymbol h},{\boldsymbol k})$, 
which is required to be progressively measurable with respect to the filtration generated by 
${\boldsymbol W}$. 
Existence and uniqueness are well-known facts in BSDE theory. 
 By taking the conditional mean given ${\boldsymbol W}$ in \eqref{eq:state:intro}, we deduce that solving the MFG problem thus amounts to find 
 a pair $({\boldsymbol m},{\boldsymbol h})$ satisfying the forward-backward stochastic 
 differential equation (FBSDE): 
 \begin{equation}
 \label{eq:intro:FBSDE}
 \begin{split}
 &\ud m_{t} = - \bigl( \eta_{t} m_{t} + h_{t} \bigr) \ud t + \varepsilon \ud W_{t}, \quad m_{0}= {\mathbb E}(X_{0}), 
 \\
  &\ud h_{t} = - \bigl( Q^\dagger f(m_{t}) - \eta_{t} h_{t} \bigr) \ud t + \varepsilon k_{t} \ud W_{t}, \quad 
  h_{T}= R^\dagger g(m_{T}). 
 \end{split}
 \end{equation}
Unique solvability was proven in \cite{Delarue02,MaProtterYong}. The smoothing effect of the noise manifests at the level of the related system of Partial Differential Equations (PDEs), which is sometimes 
 called the \textit{master equation} of the game:
\begin{equation}
\label{eq:intro:master:1}
\partial_{t} \theta_\varepsilon(t,x) + \tfrac{\varepsilon^2}2
\Delta^2_{xx} \theta_\varepsilon(t,x) - \bigl( \eta_{t} x + \theta_\varepsilon(t,x) \bigr) \cdot \nabla_{x} \theta_\varepsilon(t,x) +Q^\dagger f(x) - \eta_{t} \theta_\varepsilon(t,x) = 0,
\end{equation}
with $\theta_\varepsilon(T,\cdot)=R^\dagger g(\cdot)$ as  a boundary condition at the terminal time, 
$\theta_\varepsilon$ being a function from $[0,T] \times {\mathbb R}^d \rightarrow {\mathbb R}^d$. 
When $\epsilon >0$, the PDE is uniformly parabolic and hence has a unique classical solution (with bounded derivatives). 
When $\epsilon=0$, it becomes an hyperbolic equation and singularities may emerge, precisely when the solutions to \eqref{eq:intro:FBSDE} (which are nothing but the characteristics of 
\eqref{eq:intro:master:1}) cease to be unique.

\subsection{Learning procedures}
\label{subse:learning}
\textcolor{black}{We now provide an overview of the existing methods that can be used for solving 
the standing MFG numerically, or at least for decoupling 
the two forward and backward equations in 
\eqref{eq:intro:FBSDE}. In particular, we spend some time here recalling the definition of the 
so-called \textit{fictitious play}, see 
\eqref{eq:backward:step:n} 
and
\eqref{eq:forward:step:n}.}

\textcolor{black}{In our setting, numerical solutions to the MFG may}
be found by solving
either the FBSDE  
\eqref{eq:intro:FBSDE}
or the
nonlinear equation \eqref{eq:intro:master:1}. 
For sure, independently of any applications to MFGs, 
there have been well-known numerical methods for the two objects, see for instance 
(and for a tiny example) 
Beck et al.
\cite{BeckEJentzen},
Bender and Zhang
\cite{BenderZhang},
Cvitani\'c and Zhang
\cite{CvitanicZhang},
Delarue and Menozzi
\cite{DelarueMenozzi},
E et al. 
\cite{EHanJentzen},
Huijskens et al.
\cite{Huijskens},
Ma et al. 
\cite{MaShenZhao},
and 
Milstein and Tretyakov 
\cite{MilsteinTretyakov}. 
In Delarue and Menozzi \cite{DelarueMenozzi},
Huijskens et al. 
\cite{Huijskens},
Ma et al. 
\cite{MaShenZhao},
and 
Milstein and Tretyakov 
\cite{MilsteinTretyakov}, the problem is solved by constructing an approximation of 
$\theta_\varepsilon$ by means of a backward induction. This is a typical strategy in the field, which is fully consistent with the dynamic programming principle that holds true for uniquely solvable mean field games (see \cite[chapter 4]{CarmonaDelarue_book_II}). 
Although formulated differently, 
Bender and Zhang
\cite{BenderZhang} also relies on a backward induction. In comparison with 
Bender and Zhang
\cite{BenderZhang},
Delarue and Menozzi 
\cite{DelarueMenozzi},
Huijskens et al.
\cite{Huijskens},
Ma et al. 
\cite{MaShenZhao},
and 
Milstein and Tretyakov 
\cite{MilsteinTretyakov},
the works
Beck et al.
\cite{BeckEJentzen}, 
Cvitani\'c and Zhang
\cite{CvitanicZhang}, and E et al. \cite{EHanJentzen}
proceed in a completely different manner because the backward equation therein is reformulated into a forward equation with an unknown initial condition; the goal is then to tune both the initial condition and the martingale representation term of the backward equation in order to minimize, at terminal time, the distance to the prescribed boundary condition. Those methods have the following main limitation within a learning prospect for MFGs: they require the 
coefficients $f$, $g$, $Q$ and $R$ entering the model to be explicitly known. 
Even more, they make no real use of the control structure 
\eqref{eq:cost:intro}
that underpins the game. 

In fact, 
Delarue and Menozzi
\cite{DelarueMenozzi}, 
Huijskens et al. \cite{Huijskens},
Ma et al. \cite{MaShenZhao},
and Milstein and Tretyakov \cite{MilsteinTretyakov}
suffer from another drawback because all these works 
involve a space discretization  that consists in approximating 
the function $\theta_\varepsilon$ at the nodes of a spatial grid. Accordingly, 
the complexity increases with the physical dimension $d$ of the state variable. In particular, 
similar strategies would become even more costly for a more general mean field dependence than the one addressed in 
\eqref{eq:cost:intro}. Indeed, in the general case, the spatial variable is no longer the mean but the whole statistical distribution (which is an infinite dimensional object). For sure, the latter raises challenging questions that go beyond the scope of this paper because we cannot guess of a numerical method that would directly allow to handle the infinite dimensional statistical distribution of the solution. However, this is an objective that should be kept in mind. Say for instance that particle or quantization methods would be natural candidates to overcome such an issue, see for instance 
Chaudru de Raynal and Garcia Trillos
\cite{ChaudruGarcia} and Crisan and McMurray \cite{CrisanMcMurray}. 

Whatever the method, the true difficulty  
is to decouple \eqref{eq:intro:FBSDE}, because the two forward and backward time directions are conflicting. 
%Despite the aforementioned drawbacks that it has, 
%this is the main advantage of 
%solving \eqref{eq:intro:master:1} first, since the approximation $\theta$
%is then computed along the sole backward direction. 
%
One of our basic concern here is thus to take 
benefit of the noise in order to define 
 an iterative scheme that decouples
\eqref{eq:intro:FBSDE} efficiently. 
At the same time, because the very structure of a MFG corresponds very naturally to the precepts of reinforcement learning,
we want to have a method that may work with unknown 
coefficients $f$, $g$, $Q$ and $R$, and that  
may benefit from the observation of the cost if it is available. In this regard, we ask our scheme to 
have a learning structure that should manifest in a sequence of steps of the form 
\begin{enumerate}
 \item computation of a best action,
 \item update of the state variable,  
\end{enumerate}
and hence that would be adapted to real data. 
Surprisingly, this is not an easy question, even though the equation \eqref{eq:intro:FBSDE} is well-posed.
%, and this even if we do not ask the scheme to have the above learning structure. 
As demonstrated in the recent work \cite{ChassagneuxCrisanDelarue}, naive Picard iterations may indeed fail. They would consist in solving inductively the 
backward equation
 \begin{equation}
\label{eq:backward:step:n}  \ud h_{t}^{n+1} = - \bigl(Q^\dagger f(\overline{m}_{t}^{n}) - \eta_{t} h_{t}^{n+1} \bigr) \ud t + \varepsilon k_{t}^{n+1} \ud W_{t}, \quad 
  h_{T}^{n+1}= R^\dagger g(\overline m_{T}^{n}). 
 \end{equation}
for a given proxy $\overline{\boldsymbol m}^n:=(\overline m_{t}^{n})_{0 \le t \le T}$ of the forward equation and then in plugging the solution 
${\boldsymbol h}^{n+1}:=(h_{t}^{n+1})_{0 \le t \le T}$ into the forward equation  
  \begin{equation}
\label{eq:forward:step:n}  
\ud \overline{m}_{t}^{n+1} = - \bigl( \eta_{t} \overline{m}_{t}^{n+1} + h_{t}^{n+1} \bigr) \ud t + \varepsilon \ud W_{t}, \quad m_{0}^{n+1}= {\mathbb E}(X_{0}). 
 \end{equation}
Intuitively, the reason why it may fail is that the updating rule $\overline{\boldsymbol m}^{n} \mapsto 
\overline{\boldsymbol m}^{n+1}$ is too ambitious. In other words, 
the increment may be too high and smaller steps are needed to guarantee the convergence of the procedure. 

A more successful strategy is known in MFG theory (and more generally in game theory) under the name of \textit{fictitious play}, see \cite{CardaliaguetHadikhanloo,EliePerolat,Perrin-3,Hadikhanloo,HadikhanlooSilva,Perrin-2,perrin2020fictitious}. It is a learning procedure with a decreasing harmonic step of size $1/n$ at rank 
$n$ of the iteration. 
Having as before a proxy 
$\overline{\boldsymbol m}^{n}$ for the state of the population at rank $n$ of the learning procedure, 
${\boldsymbol h}^{n+1}$ is computed as above. Next, the same forward equation 
as before is also solved, but the solution is denoted by 
${\boldsymbol m}^{n+1}$, namely
  \begin{equation}
  \label{eq:forward:step:fictitious:n}  
\ud {m}_{t}^{n+1} = - \bigl( \eta_{t} m_{t}^{n+1} + h_{t}^{n} \bigr) \ud t + \varepsilon \ud W_{t}, \quad m_{0}^{n+1}= {\mathbb E}(X_{0}),
 \end{equation}
 and then the updating rule is given by 
 \begin{equation}
 \label{eq:update:step:n}
 \overline{m}_{t}^{n+1} = \frac1{n+1} m_{t}^{n+1} + \frac{n}{n+1} \overline m_{t}^{n}, \quad t \in [0,T]. 
 \end{equation}
Very importantly, both the Picard iteration and the fictitious play have a learning interpretation. 
In both cases, the backward equation 
\eqref{eq:backward:step:n}
provides the best response ${\boldsymbol \alpha}^{n+1,\star}$ to the minimization  of the functional 
${\boldsymbol \alpha} \mapsto J({\boldsymbol \alpha};\overline{\boldsymbol m}^n)$, 
which has indeed the same form as in 
\eqref{eq:intro:optimal:alpha}:
\begin{equation*}
\alpha_{t}^{n+1} = - \eta_{t} X_{t}^{n+1} - h_{t}^{n+1}, \quad t \in [0,T], 
\end{equation*}
where $(X_{t}^{n+1,\star})_{0 \leq t \leq T}$ is obtained implicitly by solving 
the corresponding state equation 
\eqref{eq:state:intro}. Equivalently, the above formula gives the form of 
the optimal feedback function to the minimization  of the functional 
${\boldsymbol \alpha} \mapsto J({\boldsymbol \alpha};\overline{\boldsymbol m}^n)$ (bearing in mind that the feedback function may be here random because of the common noise).

In fact, the fictitious play has been addressed 	so far  in the following two main cases: 
potential MFGs (\cite{CardaliaguetHadikhanloo})
and MFGs with monotone coefficients (here, $Q^\dagger f$ and $R^\dagger g$ are non-decreasing; see \cite{EliePerolat,Perrin-3,Hadikhanloo,HadikhanlooSilva,perrin2020fictitious} within a general setting).
Also, except 
in the recent work \cite{perrin2020fictitious} (whose framework is a bit different because the fictitious play is in continuous time), the analysis has just been carried out in the case $\varepsilon=0$, i.e., 
when there is no common noise. Here, it is certainly useful to recall that 
potential MFGs are a class of MFGs for which there exists a potential, namely a functional 
${\mathcal J}({\boldsymbol \alpha})$
associated with the same state dynamics as in 
\eqref{eq:state:intro}--\eqref{eq:cost:intro},
such that any minimizer of 
${\mathcal J}$ is a solution of the MFG. 
In our setting, the shape of ${\mathcal J}({\boldsymbol \alpha})$
is 
\begin{equation}
\label{eq:potential:structure:1}
{\mathcal J}\bigl({\boldsymbol \alpha} \bigr) =  
{\mathbb E} \biggl[ {\mathcal G}\bigl({\mathcal L}(X_{T} \vert \varepsilon {\boldsymbol W}) \bigr)
+ \int_{0}^T \frac12 \Bigl[ {\mathcal F} \bigl( {\mathcal L}(X_{t} \vert \varepsilon {\boldsymbol W}) \bigr) + \vert \alpha_{t} \vert^2 \Bigr] \ud t \biggr],
\end{equation}
where
${\mathcal L}(X_{t} \vert \varepsilon {\boldsymbol W})$
denotes the conditional law of $X_{t}$ given 
the common noise and 
\begin{equation}
\label{eq:potential:structure:2}
{\mathcal G}(\mu) = \frac12  \int_{{\mathbb R}^d} \vert R x \vert^2 \ud\mu(x) + G\bigl(\bar \mu\bigr), 
\quad 
{\mathcal F}(\mu) = \frac12   \int_{{\mathbb R}^d} \vert Q x \vert^2 \ud\mu(x) + F \bigl(\bar \mu \bigr),
\end{equation}
with $G$ and $F$ denoting primitives of $R^\dagger g$ and $Q^\dagger f$ (if any). 
In particular, the model is always potential 
when $d=1$, but there is no potential structure when $d \geq 2$, unless $Q^\dagger f$
and $R^\dagger g$ both derive from (Euclidean) potentials, meaning that $(Q^\dagger f)^i = \partial F/\partial x_{i}$
and $(R^\dagger g)^i = \partial G/\partial x_{i}$. 

\subsection{Tilted harmonic and geometric fictitious plays}
\label{subse:tilted}
\textcolor{black}{This subsection contains one of the main novelties of the paper. 
For reasons that are explained below, we introduce two variants of the fictitious play, which we 
call `tilted harmonic' and `tilted geometric'.
The definition
of the geometric version 
(which is the version that is effectively addressed in the paper)
 consists of the four equations
\eqref{eq:mn+1:cost:functional}--\eqref{eq:mn+1:conditional:state}--\eqref{eq:mn+1:conditional:intro}--\eqref{eq:mn+1:cost:update:mean} below and differs significantly from 
  \eqref{eq:forward:step:fictitious:n}--\eqref{eq:update:step:n}.} 

Consistently with our agenda, our goal is \textcolor{black}{indeed} to prove that the fictitious play may converge thanks  
to 
the presence of the common noise (i.e., $\varepsilon >0$). Seemingly, 
the above discussion about the potential structure of our model in dimension $d=1$ demonstrates that this question becomes especially relevant in dimension greater than or equal to 2 (the results from 
\cite{CardaliaguetHadikhanloo} could be easily adapted to this setting, even in the presence of the common noise). In fact, as shown in Subsection 
\ref{subse:3:3}, the question is also interesting in dimension $d=1$ in cases when equilibria are not unique. 

However, this program \textcolor{black}{is more challenging than it seems} because we are not able, even in the presence of the common noise, to prove the convergence of the fictitious play stated in \eqref{eq:backward:step:n}
and \eqref{eq:forward:step:fictitious:n}. Instead, we take benefit of the noise in order to reformulate the two equations \eqref{eq:backward:step:n}
and \eqref{eq:forward:step:fictitious:n} into a new system
obtained by a mere shift of the common noise. By Girsanov theorem, the new shifted common noise has the same law as the original one but under a tilted probability measure. 
Using the harmonic updating rule 
 \eqref{eq:update:step:n}, 
 this so-called \textit{tilted harmonic} scheme
is then shown to converge (in the case $\varepsilon >0$).  

In order to state the new fictitious play properly, we thus need to allow for another form of common noise 
in 
\eqref{eq:state:intro}. 
For a
process ${\boldsymbol h}=(h_{t})_{0 \leq t \leq T}$, 
progressively-measurable with respect to the filtration generated by ${\boldsymbol W}$, we thus introduce the \textit{shifted} Brownian motion 
\begin{equation*}
{\boldsymbol W}^{ \boldsymbol h/{\varepsilon}} = \biggl(W_{t}^{ \boldsymbol h/{\varepsilon}} :=
W_{t} + \frac1{\varepsilon} \int_{0}^t h_{s} \ud s \biggr)_{0 \le t \le T},  
\end{equation*}
together with the tilted probability measure ${\mathbb P}^{\boldsymbol h}$ 
whose density with respect to ${\mathbb P}$ is
\begin{equation}
\label{eq:Girsanov}
{\mathcal E}\bigl(
 {\tfrac{1}{\varepsilon}} {\boldsymbol h}\bigr)
:=
\exp \biggl(  - \frac1{\varepsilon} \int_{0}^T h_{t} \cdot  \ud W_{t} - \frac1{2 \varepsilon^2} \int_{0}^T \vert h_{t}  \vert^2 \ud t \biggr).
\end{equation}
Accordingly, 
for any two 
frozen 
continuous paths ${\boldsymbol n}:=(n_{t})_{0 \le t \le T}$
and ${\boldsymbol w}:=(w_{t})_{0 \le t \le T}$, 
we define the \textit{non-averaged} cost functional
\begin{equation}
\label{eq:R:intro:w}
{\mathcal R}^{x_0}\bigl({\boldsymbol \alpha};{\boldsymbol n};\textcolor{black}{\varepsilon}{\boldsymbol w}\bigr)
:=  \frac12  \biggl[ \bigl\vert  {R} x_{T} + g(n_{T}) \bigr\vert^2 + 
\int_{0}^T \Bigl\{ \bigl\vert  {Q} x_{t} + f(n_{t}) \bigr\vert^2 + 
\vert \alpha_{t} \vert^2
\Bigr\} \ud t \biggr],
\end{equation}
where
\begin{equation}
\label{eq:state:intro:w}
x_{t} = x_{0} + \int_{0}^t \alpha_{s} \ud s + \sigma  B_{t} + \varepsilon w_{t}, \quad t \in [0,T],
\end{equation}
the latter being nothing but the integral version of 
\textcolor{black}{\eqref{eq:state:intro}}, when ${\boldsymbol W}$ is replaced by the frozen trajectory 
${\boldsymbol w}$. Now, when 
${\boldsymbol m}$
and 
${\boldsymbol h}$
are two  
progressively-measurable processes with respect to the filtration generated by $\varepsilon {\boldsymbol W}$,
we have a look at 
the new cost functional\footnote{\color{black} The reader should observe that, when there is no common noise
and 
when 
${\boldsymbol \alpha}$ is progressively-measurable with respect to 
${\mathbb F}^{X_0,{\boldsymbol B}}$ 
and ${\boldsymbol m}$ is deterministic, 
${\mathbb E}[{\mathcal R}({\boldsymbol \alpha};{\boldsymbol m};0)]$
coincides with the original cost $J({\boldsymbol \alpha};{\boldsymbol m})$ in 
\eqref{eq:cost:intro}.}:
\begin{equation}
J^{{\varepsilon}}({\boldsymbol \alpha};{\boldsymbol m};{\boldsymbol h})
:= {\mathbb E}^{\boldsymbol h/{\varepsilon}} \Bigl[ {\mathcal R}^{X_0} \bigl({\boldsymbol \alpha};{\boldsymbol m}; {\varepsilon} {\boldsymbol W}^{\boldsymbol h/ {\varepsilon}}\bigr)\Bigr].
\end{equation}
We now define the first version of our titled fictitious play according to a two step iterative learning procedure, whose description at rank $n$ goes as follows:
\begin{description}
\item[Harmonic best action] 
For a proxy $\overline{\boldsymbol m}^n:=(\overline m_{t}^{n})_{0 \le t \le T}$
of the conditional mean 
${\boldsymbol m}=(m_{t})_{0 \le t \le T}$
of the in-equilibrium population 
(as given by the 
forward component 
of 
\eqref{eq:intro:FBSDE})
and a proxy ${\boldsymbol h}^n=(h_{t}^n)_{0 \le t \le T}$
of the opposite\footnote{The opposite comes from the sign $-$ in the formula \eqref{eq:intro:optimal:alpha} of the optimal feedback.} of the ${\mathbb F}^{{\boldsymbol W}}$-adapted intercept 
of the equilibrium feedback in 
\eqref{eq:intro:optimal:alpha} (as given by the backward component 
of \eqref{eq:intro:FBSDE}), solve
\begin{equation}
\label{eq:mn+1:cost:functional:0}
{\boldsymbol \alpha}^{n+1} = \textrm{\rm argmin}_{\boldsymbol \alpha}
{\mathbb E}^{{\boldsymbol h}^n \hspace{-2pt}  {/\varepsilon}}
\bigl[ {\mathcal R}^{X_0} \bigl( {\boldsymbol \alpha}; \overline{\boldsymbol m}^n ; \textcolor{black}{\varepsilon} {\boldsymbol W}^{{\boldsymbol h}^n \hspace{-2pt} /\varepsilon}\bigr) \bigr],
\end{equation}
the infimum being taken over all ${\mathbb F}$-progressively measurable 
(${\mathbb R}^d$-valued) controls ${\boldsymbol \alpha}$. 

The optimal feedback being of the same linear form as in 
\eqref{eq:intro:optimal:alpha} (the proof is given below), 
we may call ${\boldsymbol h}^{n+1}=(h^{n+1}_{t})_{0 \le t \le T}$
the opposite of the resulting intercept. 
\item[Harmonic update]
Given ${\boldsymbol h}^{n+1}$, the optimal trajectory of 
the above minimization problem is
\begin{equation}
\label{eq:mn+1:conditional:state}
X_{t}^{n+1} = X_{0} - \int_{0}^t \bigl( \eta_{s} X_{s}^{n+1} + 
h_{s}^{n+1} \bigr) \ud s + \sigma  B_{t} + \varepsilon W_{t}^{{\boldsymbol h}^n/\varepsilon}, \quad t \in [0,T].
\end{equation}
We then let\footnote{\color{black} 
The reader 
will find in 
Remark \ref{rem:2.2:measurability}
a useful comment about the definition of 
$(m^{n+1}_{t})_{0 \leq \leq T}$: the conditional expectation 
can be equivalently taken under 
${\mathbb P}^{{\boldsymbol h}^n}$. This comes from the fact the Girsanov density is
measurable with respect to the $\sigma$-field $\sigma({\boldsymbol W})$ generated by 
${\boldsymbol W}$.}
\begin{equation}
\label{eq:mn+1:conditional:intro}
m^{n+1}_{t} = {\mathbb E} \bigl[ X_{t}^{n+1}  \, \vert \,  {\boldsymbol W} \bigr], \quad 
t \in [0,T], 
\end{equation}
together with
%\footnote{\color{black} The reader will find in 
%Remark 
%\ref{rem:fictitious:play:rate} 
%additional comments about the choice of the harmonic sequence as learning rate.}
\begin{equation}
\label{eq:mn+1:cost:update:mean:0}
\overline{m}^{n+1}_{t} = \frac1{n+1} \sum_{k=1}^{n+1} m_{t}^{k} = \tfrac1{n+1} m_{t}^{n+1} + \tfrac{n}{n+1}
\overline{m}^n_{t}, \quad 
t \in [0,T]. 
\end{equation}
\end{description}
For sure the rationale behind this strategy relies on Girsanov's transformation. 
Under the tilted probability measure ${\mathbb P}^{{\boldsymbol h}/\varepsilon}$, 
the process ${\boldsymbol W}^{{\boldsymbol h}/\varepsilon}$ is a new Brownian motion, 
with the same law as the original (or \textit{historical}) common noise under ${\mathbb P}$. 
However, the main trick here is to 
%force the conditioning in the McKean-Vlasov constraint 
%\eqref{eq:mn+1:conditional:intro}
%to be taken with respect to the original common noise. \textcolor{red}{This is not correct since the conditional expectation under the tilted measure would be the same here: this comes from the fact the Girsanov density is adapted with respect to the common noise. 
dynamically change the form of the common noise (dynamically with respect to the 
rank of the iteration in the fictitious play). Precisely, this permits to decouple the two forward and backward equations, as clearly shown if we 
 {write 
the 
equation for 
$\overline{\boldsymbol m}^{n+1}$} as an equation with respect to 
the historical common noise (the proof is given in 
\eqref{eq:def:overline:m:n} below): 
\begin{equation}
\label{eq:n:n+1:0}
\overline m_{t}^{n+1} = m_{0} - \int_{0}^t \Bigl( \eta_{s} m_{s}^{n+1} + 
\tfrac1{n+1} \bigl[ h_{s}^{n+1} 
- h_{s}^0 \bigr]   \Bigr) \ud s   + \varepsilon W_{t}, \quad t \in [0,T]. 
\end{equation}
The forward equation then becomes asymptotically autonomous provided that 
${\boldsymbol h}^{n+1}$ 
can be bounded independently of $n$, which we succeed to prove in Section \ref{Sec:2}. 
In turn, the scheme can be easily shown to converge
(notice however that we are not able to prove the convergence of the standard non-tilted fictitious
play  
outside any further potential or monotonicity assumption, even in the presence of the common noise). Moreover, the rate can be proved to decay (at least) like 
${\mathcal O}(1/n)$, with ${\mathcal O}(\cdot)$ standing for the `big ${\mathcal O}$ Landau notation'. Unfortunately, the 
best estimate we have for the 
constant driving the term ${\mathcal O}(1/n)$ blows up exponentially fast with 
$1/\varepsilon^2$.  
Although the numerical experiments that are reported below indicate that the rate may 
decrease faster than $1/n$ and grow slower than $\exp({\mathcal O}(1/\varepsilon^2))$, 
the theoretical guarantee that is hence available for the 
version 
of the fictitious play
comprising the four equations  
\eqref{eq:mn+1:cost:functional:0}--\eqref{eq:mn+1:conditional:state}--\eqref{eq:mn+1:conditional:intro}--\eqref{eq:mn+1:cost:update:mean:0}
is thus rather poor when the viscosity parameter $\varepsilon$ tends to $0$. 
This prompts us to provide a variant of the scheme, with an updating step 
(in equation \eqref{eq:mn+1:conditional:intro})
that is different from $1/(n+1)$ 
and 
that 
yields
 a more favourable trade-off between $n$ 
and $\varepsilon^{-2}$
(or equivalently that allows for a cheaper choice of $n$ when $\varepsilon$ is small). 
In order to clarify things, we refer below to the scheme 
\eqref{eq:mn+1:cost:functional:0}--\eqref{eq:mn+1:conditional:state}--\eqref{eq:mn+1:conditional:intro}--\eqref{eq:mn+1:cost:update:mean:0}
as the `tilted harmonic fictitious play'.

In a nutshell, the key point is to perform the
following modifications in each of the two steps of the fictitious play, with the new version being called `geometric' for reasons that become obvious in the next few lines:
\begin{description}
\item[Geometric best action] 
With the same notation as before, use
$\varpi X_0$ as initial private state, 
$\varpi {\boldsymbol \alpha}$ as control,  and 
${\mathbb P}^{\varpi {\boldsymbol h}^n/\varepsilon}$ and 
${\boldsymbol W}^{\varpi {\boldsymbol h}^n/\varepsilon}$
as tilted probability measure and tilted noise, where the rate $\varpi$ is a fixed real that is typically chosen in 
the interval
$(1,\sqrt{2}]$. In words, replace 
\eqref{eq:mn+1:cost:functional:0}
by
\begin{equation}
\label{eq:mn+1:cost:functional}
{\boldsymbol \alpha}^{n+1} = \textrm{\rm argmin}_{\boldsymbol \alpha}
{\mathbb E}^{{\varpi \boldsymbol h}^n \hspace{-2pt}  {/\varepsilon}}
\bigl[ {\mathcal R}^{\varpi X_0} \bigl( \varpi {\boldsymbol \alpha}; \overline{\boldsymbol m}^n ; \textcolor{black}\varpi {\varepsilon} {\boldsymbol W}^{\varpi {\boldsymbol h}^n \hspace{-2pt} /\varepsilon}\bigr) \bigr].
\end{equation}

\item[Geometric update]
Accordingly, in 
\eqref{eq:mn+1:conditional:state}, replace 
$\varepsilon W_{t}^{{\boldsymbol h}^n/\varepsilon}$
by 
$\varepsilon W_{t}^{\varpi {\boldsymbol h}^n/\varepsilon}$, 
and next use the updating formula:
\begin{equation}
\label{eq:mn+1:cost:update:mean}
	\overline{m}^{n+1}_{t} = \frac{\varpi(1-\varpi^{-1})}{1-\varpi^{-(n+1)}} \sum_{k=1}^{n+1} \varpi^{-k} m_{t}^{k} = 
	\tfrac{\varpi^{-n}(1-\varpi^{-1})}{1- \varpi^{-(n+1)}}
	m^{n+1}_t +  \Bigl( 1 - 
	\tfrac{\varpi^{-n}(1-\varpi^{-1})}{1- \varpi^{-(n+1)}}
	 \Bigr) 
	\overline{m}^n_{t}, 
\end{equation}
for $t \in [0,T]$.
\end{description}
Here is the main idea.
In comparison with \eqref{eq:mn+1:conditional:state}, the optimal trajectory of  
\eqref{eq:mn+1:cost:functional} is 
\begin{equation}
\label{eq:mn+1:conditional:state:varpi}
X_{t}^{n+1} = X_{0} - \int_{0}^t \bigl( \eta_{s} X_{s}^{n+1} + 
h_{s}^{n+1} \bigr) \ud s + \tfrac1{\varpi} \sigma  B_{t} + \varepsilon W_{t}^{\varpi {\boldsymbol h}^n 
\hspace{-2pt} /\varepsilon}, \quad t \in [0,T].
\end{equation}
The above differs from 
\eqref{eq:mn+1:conditional:state}
because the 
intensity of the idiosyncratic noise is $\sigma/\varpi$. Fortunately, 
this term disappears when taking conditional expectations given the common noise, 
as done in the formula  
\eqref{eq:mn+1:conditional:intro}. (In fact, one can also recover $\sigma$ as intensity by considering 
$\varpi {\boldsymbol X}^{n+1}$ as optimal trajectory, but with $\varpi X_0$ as initial condition.) 
Moreover, it must be stressed that the common noise right above is 
$\varepsilon {\boldsymbol W}^{\varpi {\boldsymbol h}^n/\varepsilon}$, 
whereas it is 
$\varepsilon {\boldsymbol W}^{{\boldsymbol h}^n/\varepsilon}$
in 
\eqref{eq:mn+1:conditional:state}.
This says that, under the historical probability measure ${\mathbb P}$, 
\eqref{eq:mn+1:conditional:state:varpi} rewrites
\begin{equation}
\label{eq:mn+1:conditional:state:varpi:historical}
X_{t}^{n+1} = X_{0} - \int_{0}^t \bigl( \eta_{s} X_{s}^{n+1} + 
h_{s}^{n+1} - \varpi h_s^n \bigr) \ud s + \tfrac1{\varpi} \sigma  B_{t} + \varepsilon W_{t}, \quad t \in [0,T].
\end{equation}
The presence of the factor $\varpi$ in the last term of the drift makes it possible to  use the 
geometric updating formula
\eqref{eq:mn+1:cost:update:mean}. In particular, this is our result to show (see 
\S \ref{subse:analysis:fictitious:play}
below) that, under the initialization ${\boldsymbol h}^0=0$, 
\eqref{eq:n:n+1:0}
becomes
\begin{equation}
\label{eq:n:n+1}
\overline m_{t}^{n+1} = m_{0} - \int_{0}^t \Bigl( \eta_{s} m_{s}^{n+1} + 
 \tfrac{(1-\varpi) \varpi^{-n}}{1- \varpi^{-(n+1)}}  h^{n+1}_s   \Bigr) \ud s   + \varepsilon W_{t}, \quad t \in [0,T]. 
\end{equation}
The reader may easily compare with 
\eqref{eq:n:n+1:0}. 
Whereas
the 
difference
(with a standard decoupled Ornstein-Uhlenbeck process) 
decreases like
${\mathcal O}(1/n)$ in 
\eqref{eq:n:n+1:0}, it decreases like 
${\mathcal O}(\varpi^{-n})$ in 
\eqref{eq:n:n+1}. In the end, this gives a geometric rate of convergence (in the parameter $n$). Although this does not change the presence of a 
leading constant of size $\exp({\mathcal O}(1/\varepsilon^{2}))$, this makes it possible, in order to reach 
a given 
theoretical guarantee, to choose a value of $n$ (much) lower than in the 
tilted harmonic fictitious play. 
Our main statement in this regard is Theorem
\ref{main:thm:1}.
Also, it is worth adding that, under the obvious convention that 
\begin{equation}
\label{eq:convention:harmonic:geometric}
 \frac{(1-\varpi^{-1}) \varpi^{-n}}{1- \varpi^{-(n+1)}}  = \frac1{n+1}
 \end{equation} 
 when $\varpi =1$, the 
 updating rule 
 \eqref{eq:mn+1:cost:update:mean:0}
coincides with 
 \eqref{eq:mn+1:cost:update:mean}. For this reason, we sometimes speak about 
 the harmonic scheme 
 \eqref{eq:mn+1:cost:functional:0}--\eqref{eq:mn+1:conditional:state}--\eqref{eq:mn+1:conditional:intro}--\eqref{eq:mn+1:cost:update:mean:0}
 as a particular case of the geometric scheme 
 when $\varpi=1$. (Notice however that, although it could be easily adapted to the case 
 $\varpi=1$, 
 see Remark 
 \ref{rem:fictitious:play:rate}, 
 the statement of 
 Theorem
\ref{main:thm:1} does not formally apply to $\varpi =1$ and even if it did, the statement would be in fact trivial.)
Last but not least, we stress that, in the definition of the geometric best action, the tilted noise is biased, with 
$\varpi-1$ as bias. This is an important feature of the scheme and this is the reason why we feel important to distinguish 
the best action 
\eqref{eq:mn+1:cost:functional}
 from
 the best action 
 \eqref{eq:mn+1:cost:functional:0}
 and to call the former `geometric'
 and the second `harmonic'.

\textcolor{black}{For sure, the reader may want to reformulate the tilted fictitious play for a more general MFG,
with a structure that would no longer be linear-quadratic. Whereas the theory of MFG with an additive finite-dimensional common noise
(such as ${\boldsymbol W}=(W_t)_{0 \leq t \leq T}$) is by now 
well-established (see for instance the book Carmona and Delarue \cite{CarmonaDelarue_book_II}), 
the real interest for addressing the same tilted form of the fictitious play but in a wider setting is however not clear to us. 
Indeed, aforementioned known theoretical results 
on the smoothing effect of the common noise (see Delarue \cite{DelarueSPDE}) 
require in general an infinite dimensional noise of a much more complicated structure than the additive finite-dimensional noise $(W_t)_{0 \leq t \leq T}$ 
used in \eqref{eq:n:n+1}. For this reason, we have decided to restrict the whole exposition to the linear-quadratic setting, even though
\eqref{eq:mn+1:cost:functional},
\eqref{eq:mn+1:conditional:state},
\eqref{eq:mn+1:conditional:intro} and
\eqref{eq:mn+1:cost:update:mean}
could be recast, for the same additive finite-dimensional noise, 
within a larger framework. 
Outside the class of linear-quadratic games, 
the main changes are the following ones. First, 
the optimal feedback in 
\eqref{eq:mn+1:conditional:state}
is no longer affine. Second, the mean field fixed point can no longer be formulated in terms 
of the sole conditional expectation, as it is done in 
\eqref{eq:mn+1:conditional:state},
but involves the full statistical law of $X^{n+1}_t$ given ${\boldsymbol W}$. Third (and subsequently), 
the rule
\eqref{eq:mn+1:cost:update:mean}
must be reformulated in terms of the
full statistical distributions and not only in terms of their means. We leave the details to the reader.
Needless to say, obtaining a relevant extension of the tilted fictitious play
for general MFGs remains a very interesting but highly challenging objective.}

\subsection{Exploration}
\label{subse:exploration}
\textcolor{black}{We now explain how the common noise in the 
tilted fictitious play can be regarded as an exploration noise 
for the original MFG without common noise.}

For sure, changing the common noise 
as we have done in the cost functional 
\eqref{eq:mn+1:cost:functional}
(see also 
\eqref{eq:mn+1:cost:functional:0})
raises \textcolor{black}{indeed} many practical questions\footnote{The reader may find it reminiscent of the weak formulation of MFGs introduced in
Carmona and Lacker  \cite{CarmonaLacker}, but this is substantially different because the Girsanov transformation is here applied to the common noise (and not to the idiosyncratic one).}. 
In order to understand this properly, we should 
follow the presentation given in 
Carmona et al. 
\cite{CarmonaLauriereTanI,CarmonaLauriereTanII}
and
think of ${\mathcal R}$ as being a black-box representing a decentralized unit. For instance, the box may regulate the consumption/production/storage of energy of a single individual connected to a smart grid, see for instance Alasseur \cite{Alasseur} and the references therein; the box may also be an autonomous car moving in a flock of vehicles, see for instance Huang et al. \cite{HuangDiDuChen}. 

The black-box operation is described in Figure 
\ref{fig:1}.
In this picture, the decentralized black-box receives four inputs: 
$(i)$ the control, as tuned by the individual operating the black-box; $(ii)$ 
the initial private state $X_0$; $(iii)$
the two idiosyncratic and independent noises; $(iv)$ the state of the population.
In this representation, the only input that can be tuned by 
the individual operating the black-box is the control itself.

\begin{figure}[htb]
\begin{tikzpicture}
  [node distance=1cm,
  start chain=going below,]
    %% No. 2

   \node[fill=yellow,punktchain, join] (C1) {Decentralized control};
   
     \node[fill=green,punktchain, on chain= going below] (C4)      {Common and idiosyncratic noises};
    
     \node[fill=green,punktchain, on chain= going below] (C40)      {State of the population};

   \node[fill=red,punktchain, on chain=going right,yshift=2.2cm] (C10) {Decentralized black-box};

   \node[fill=red,punktchain, on chain=going right] (C20) {Output};
       
            \node[fill=green,punktchain, on chain= going below] (C50)      {Initial private state};

   \draw[->, thick,] 
  let \p1=(C1.east), \p2=(C10.north) in
($(\x1,\y1)$) -| ($(\x2,\y2)$)
node[right, midway]  {};

   \draw[->, thick,dashed] 
  let \p1=(C4.east), \p2=(C10.west) in
($(\x1,\y1)$) -- ($(\x2,\y2)$)
node[right, midway]  {};

   \draw[->, thick,dashed] 
  let \p1=(C40.east), \p2=(C10.south) in
($(\x1,\y1)$) -| ($(\x2,\y2)$)
node[right, midway]  {};

   \draw[->, thick,] 
  let \p1=(C10.east), \p2=(C20.west) in
($(\x1,\y1)$) -- ($(\x2,\y2)$)
node[right, midway]  {};

   \draw[->, thick,dashed] 
  let \p1=(C50.west), \p2=(C10.south) in
($(\x1,\y1)$) -|($(\x2,\y2)$)
node[right, midway]  {};

% node[right, midway]  {optimization};
%
%   \draw[->, thick,]
%  let \p1=(C40.east), \p2=(C10.west) in
%($(\x1,\y1)$) -- ($(\x2,\y2)$)
%
% node[below, midway]  {optimization};
%
%  \draw[->, thick,]
%  let \p1=(C1.south), \p2=(C40.north) in
%($(\x1,\y1)$) -- ($(\x2,\y2)$)
%
% node[left, midway]  {mean field limit};
% 
%   \draw[->, thick,]
%  let \p1=(C4.south), \p2=(C10.north) in
%($(\x1,\y1)$) -- ($(\x2,\y2)$)
%
% node[right, midway]  {mean field limit};

  \end{tikzpicture}
  \vskip 3pt
  
  \caption{\small Black-box operation for an MFG with a common noise. In the four input arrows, only the plain line can be tuned. Given the input state of the population, the private initial state $X_0$ and the realizations of the two noises, the decentralized black-box returns an output, in the form of a cost depending on the input state and on the realizations of the noises.}
\label{fig:1}
\end{figure}

However, this picture makes sense only if the original model itself is subjected to a common noise. 
In the absence of common noise, 
we 
thus need to restore a form of common noise 
in order to conciliate our tilted fictitious play with the above picture. 
This comes through the notion of exploration. Below, we thus regard the common noise as a way to explore the space of possible solutions. 
This amounts to say that the original mean field game is no longer the mean field game 
with common noise, but the mean field game without common noise, i.e. $\varepsilon=0$. 
Consistently, the individual operating
the black-box can 
restore the presence of a 
common noise by modifying her/his control accordingly.  
Formally, any control ${\boldsymbol \alpha}$, as chosen above by 
a tagged individual, is then subjected to an additional randomization of the form
\begin{equation*}
{\boldsymbol \alpha}= (\alpha_{t})_{0 \le t \le T}
\mapsto   \biggl(\alpha_{t}
+ \varepsilon  \dot{W}_{t}^{\boldsymbol h}
\biggr)_{0 \le t \le T},
\end{equation*}
where ${\boldsymbol h}$ is given as an information on the whole state of the population, in addition to ${\boldsymbol m}$.  Figure \ref{fig:1} becomes
the new Figure \ref{fig:2} below. 
\textcolor{black}{In this new figure, the decentralized black-box is exactly the same as the decentralized black-box appearing in 
Figure 
\ref{fig:1} in the absence of the common noise therein. In clear, the black-box takes as inputs the following three features: 
a time-dependent flow of actions impacting the dynamics of a single individual, 
the initial private state $X_0$
 and the time-dependent flow of 
probability measures characterizing the state of the environment. 
As for the output, the black-box returns the realization ${\mathcal R}^{X_0}$ of the cost to a single individual, for the given action, 
the given realization of the initial condition and the given environment (see \eqref{eq:R:intro:w}). In particular, it is worth mentioning that 
the common randomization does not impact the operation performed by the black-box, which is an important fact from the practical point of view. What changes in the presence of the common randomization is the form of the input 
that is inserted in the black-box: the input at time $t$ is `corrupted' by 
$ \varepsilon  \dot{W}_{t}^{\boldsymbol h}$. 
In particular, the reader should agree that Figure \ref{fig:2} is consistent with the definitions
\eqref{eq:mn+1:cost:functional:0} 
and
\eqref{eq:mn+1:cost:functional} of the best actions in our
two tilted fictitious plays, up to a slight but subtle difference between 
\eqref{eq:mn+1:cost:functional:0} 
and
\eqref{eq:mn+1:cost:functional}:
in order to compute 
\eqref{eq:mn+1:cost:functional}, 
 the initial private state $X_0$ in the black-box must be free (as made clear in the caption), in the sense that it can 
be changed for the purpose of the experiment. This is an important feature because the non-averaged cost functional 
${\mathcal R}$ in the geometric fictitious play is initialized from $\varpi X_0$ (and not $X_0$), see
\eqref{eq:mn+1:cost:functional}. We feel that this additional assumption on the model is affordable in practice.}

Now, if we had to think of a (possibly infinite) cloud of players, the actions of all of them would be corrupted by the same realization of the noise. 
Assume indeed that, 
given the same two proxies $\overline{\boldsymbol m}^n$ and ${\boldsymbol h}^n$ as in 
\eqref{eq:mn+1:cost:functional:0}
and
\eqref{eq:mn+1:cost:functional}, the players 
choose some common feedback 
function (which is exactly what happens for an MFG equilibrium). The resulting action of each of them is then randomized by the same realization of the common noise
and, subsequently, 
the black-box returns the (non-averaged) cost to each player. 
If the players are driven by independent 
copies of $(B_t)_{0 \leq t \leq T}$ (which fits the fact that $(B_t)_{0 \leq t \leq T}$ is an idiosyncratic noise)
and by independent copies of 
$\varpi X_0$ (which fits the fact that $\varpi X_0$ is the new initial private condition in the $\varpi$-geometric scheme), 
then the empirical mean cost to all the players should be regarded as the conditional expectation of 
the cost ${\mathcal R}^{\varpi X_0}$ given ${\boldsymbol W}$. 
Assuming that the common noise is observable (which makes perfect sense if it is sampled by some 
experimenter), we can compute ${\mathcal E}(\varpi {{\boldsymbol h}^n} \hspace{-2pt} /\varepsilon)$ and then multiply 
it with the conditional expectation of the cost. 
Sampling the common noise as many times as desired, we get an empirical approximation of the mean cost under 
${\mathbb P}^{\varpi {\boldsymbol h}^n \hspace{-2pt} / \varepsilon}$ in 
\eqref{eq:mn+1:cost:functional}.
Although this picture may look rather naive, it is in fact the basis of a numerical method 
that is detailed next, see
Figure 
\ref{fig:4}
for a primer.
\begin{figure}[htb]
\begin{tikzpicture}
  [node distance=1cm,
  start chain=going below,]
    %% No. 2

   \node[fill=yellow,punktchain, join] (C1) {Decentralized control};
   
     \node[fill=green,punktchain, on chain= going below] (C4)      {Idiosyncratic noise};
    
     \node[fill=green,punktchain, on chain= going below] (C40)      {State of the population};       
       
   \node[fill=red,punktchain, on chain=going right,yshift=2.2cm] (C10) {Decentralized black-box};

   \node[fill=red,punktchain, on chain=going right] (C20) {Output};

 \node[fill=yellow,punktchain, on chain=going above,yshift=1cm] (C15) {Common randomization};

     \node[fill=yellow,punktchain, on chain= going below,yshift=.6cm] (C50)      {Initial private state};

   \draw[->, thick,] 
  let \p1=(C1.east), \p2=(C10.north) in
($(\x1,\y1)$) -| ($(\x2,\y2)$)
node[right, midway]  {};

   \draw[->, thick,dashed] 
  let \p1=(C4.east), \p2=(C10.west) in
($(\x1,\y1)$) -- ($(\x2,\y2)$)
node[right, midway]  {};

   \draw[->, thick,dashed] 
  let \p1=(C40.east), \p2=(C10.south) in
($(\x1,\y1)$) -| ($(\x2,\y2)$)
node[right, midway]  {};

   \draw[->, thick,] 
  let \p1=(C10.east), \p2=(C20.west) in
($(\x1,\y1)$) -- ($(\x2,\y2)$)
node[right, midway]  {};

   \draw[->, thick,] 
  let \p1=(C15.west), \p2=(C10.north) in
($(\x1,\y1)$) -| ($(\x2,\y2)$)
node[right, midway]  {};

   \draw[->, thick,] 
  let \p1=(C50.west), \p2=(C10.north) in
($(\x1,\y1)$) -|($(\x2,\y2)$)
node[right, midway]  {};

% node[right, midway]  {optimization};
%
%   \draw[->, thick,]
%  let \p1=(C40.east), \p2=(C10.west) in
%($(\x1,\y1)$) -- ($(\x2,\y2)$)
%
% node[below, midway]  {optimization};
%
%  \draw[->, thick,]
%  let \p1=(C1.south), \p2=(C40.north) in
%($(\x1,\y1)$) -- ($(\x2,\y2)$)
%
% node[left, midway]  {mean field limit};
% 
%   \draw[->, thick,]
%  let \p1=(C4.south), \p2=(C10.north) in
%($(\x1,\y1)$) -- ($(\x2,\y2)$)
%
% node[right, midway]  {mean field limit};

  \end{tikzpicture}
  \vskip 3pt
  
  \caption{\small Black-box operation for an MFG without a common noise, but subjected to a common randomization  of the control. In the five input arrows, only the plain lines (corresponding to yellow boxes) can be tuned. Given the input state of the population and the realizations of the two noises, the decentralized black-box returns an output, in the form of a (random) cost depending on the input state and on the realizations of the noises.}
\label{fig:2}
\end{figure}

%Even under the additional assumption that $X_0$ can be freely changed, 
This concept faces however obvious mathematical difficulties, because the new control after randomization  is no longer of finite energy. We solve this issue 
\textcolor{black}{by replacing the time derivative of 
${\boldsymbol W}$ by finite differences or, equivalently,} by
replacing ${\boldsymbol W}=(W_{t})_{0 \le t \le T}$
by its piecewise linear interpolation along a mesh of $[0,T]$, say of uniform step $T/p$, for 
a given integer $p \geq 1$. We denote this interpolation by 
${\boldsymbol W}^p=(W_{t}^p)_{0 \le t \le T}$. 
Accordingly, the return of the black-box should be \textit{renormalized}, letting
\begin{equation}
\label{eq:mathcalrp}
{\mathcal R}^{p,x_0}({\boldsymbol \alpha};{\boldsymbol m};\textcolor{black}{\varepsilon} {\boldsymbol w}) := 
{\mathcal R}^{x_0}({\boldsymbol \alpha}+\varepsilon \dot{\boldsymbol w};{\boldsymbol m};0)
 - \tfrac12 \varepsilon^2 p,
\end{equation}
for a piecewise-affine path ${\boldsymbol w}$, affine on each $[\ell T/p,(\ell+1)T/p)$ for $\ell \in \{0,\cdots,p-1\}$.
Our second main statement, see Theorem
\ref{main:thm:2}, is to prove that the geometric fictitious play that is hence obtained by 
replacing ${\boldsymbol W}$ by ${\boldsymbol W}^p$ and 
$ {\mathcal R}^{x_0}( {\boldsymbol \alpha}; \overline{\boldsymbol m}^n ; \textcolor{black}{\varepsilon}{\boldsymbol W}^{{\boldsymbol h}^n})$
by 
$ {\mathcal R}^{p,x_0}( {\boldsymbol \alpha}; \overline{\boldsymbol m}^n ; \textcolor{black}{\varepsilon} {\boldsymbol W}^{p,{\boldsymbol h}^n})$, with 
\begin{equation}
\label{eq:tilted:discrete:p}
W_{t}^{p,{\boldsymbol h}^n}= W_{t} + \frac1{\varepsilon} \int_{0}^t h^n_{s} \ud s, \quad t \in [0,T], 
\end{equation}
converges, when the number $n$ of iterations tends to $\infty$, to the solution of the discrete-time version, with $T/p$ as step size, of the mean field game with $\varepsilon {\boldsymbol W}^p$ as common noise. Implicitly, this requires to 
force ${\boldsymbol \alpha}$ and 
${\boldsymbol h}^n$ to be constant on any subdivision of the mesh, but we feel more appropriate not to detail all the ingredients here. We refer the reader to 
Subsection \ref{subse:2.2} for a complete description. Importantly, the state dynamics over which the return 
${\mathcal R}^p$ is computed write 
\begin{equation}
\label{eq:state:dynamcis:Rp}
\begin{split}
X_{t} &= \varpi X_{0} + \int_{0}^t \varpi \alpha_{s} \ud s + \sigma  B_{t} + \varepsilon W_{t}^{p,\varpi {\boldsymbol h}^n/\varepsilon}
\\
&= \varpi X_{0} + \int_{0}^t \Bigl( \varpi \alpha_{s} + \varepsilon \dot{W}_{s}^{p,\varpi {\boldsymbol h}^n/\varepsilon}
\Bigr) 
\ud s + \sigma  B_{t}, \quad t \in [0,T]. 
\end{split}
\end{equation}
In this approach, the control after randomization  is thus given by 
$\varpi {\boldsymbol \alpha} + \varepsilon \dot{\boldsymbol W}^{p,\varpi {\boldsymbol h}^n/\varepsilon}$, which clarifies the meaning of the common randomization  in Figure \ref{fig:2}. 
In the end, this fits well the concept of exploration, as stated by Sutton and Barto \cite[chapter 1, p. 1]{SuttonBarto}: `\textit{The learner is not told which actions to take, but instead must discover which actions yield the most reward by trying them}'. 
\textcolor{black}{Noticeably,
our form of randomization \eqref{eq:state:dynamcis:Rp} can be restricted to controls in semi-feedback form, meaning that 
the instantaneous control 
in \eqref{eq:mn+1:conditional:state} can be chosen as the image of the current state 
of ${\boldsymbol X}^{n+1}$
by a time-space random function 
adapted to the filtration generated by 
${\boldsymbol W}$.
% $\sigma({\boldsymbol W})$-measurable \textcolor{blue}{(with $\sigma({\boldsymbol W})$ denoting the $\sigma$-field generated by 
%${\boldsymbol W}$, when the latter is seen as a random variable with values in the space of continuous functions)}.
 In particular, 
thinking of the mean field game as a game with infinitely many agents, 
all of them are then understood to play (at each iteration of the fictitious play) 
the same random semi-feedback function. 
This makes a subtle difference with standard exploration 
methods for multi-agent reinforcement learning in which the random control of each agent carries its own noise.}
Last but not least, it must be noticed that the algorithm runs from the sole observations of the returns of 
the black-box, and in particular without any further detailed 
of the cost coefficients $f$, $g$, $Q$ and $R$, provided we use a reinforcement learning method to solve the black-box. 

Importantly, it must be emphasized  again that in the two Figures \ref{fig:1} and \ref{fig:2} the decentralized  black-box does not return the expected cost, but only the realization  of the cost for the given realizations  of the noises. At each step of the fictitious play, the optimization  of the expected cost is
formally performed over all the possible trajectories of the independent and common noises, bearing in mind that the control is adapted to both noises and that the state of the population is adapted to the common noise. This is stylized in the form of Figure \ref{fig:4}.
Therein, this is our choice to represent the possible trajectories of the two noises in the form of 
an infinite sequence of realizations  from an i.i.d. sample 
of the idiosyncratic and common noises (and similarly for the initial condition), this formal representation being very convenient for introducing next the numerical implementation. 
\color{black}
In clear, assuming that we have two independent families 
$({\boldsymbol B}^i=(B_t^i)_{0 \leq t \leq T})_{i \geq 1}$
and
$({\boldsymbol W}^j=(W_t^j)_{0 \leq t \leq T})_{j \geq 1}$
of ($d$-dimensional) independent Brownian motions
and, independently, a family $(X_0^i)_{i \geq 1}$ of i.i.d. initial conditions, 
we are given at rank $n$ of the iterative process represented 
in Figure \ref{fig:4} two collections of proxies
 $(\overline{\boldsymbol m}^{n,j})_{j \geq 1}$
 and 
  $({\boldsymbol h}^{n,j})_{j \geq 1}$
  of i.i.d.
continuous paths with values in ${\mathbb R}^d$, with each 
$\overline{\boldsymbol m}^{n,j}$
and
${\boldsymbol h}^{n,j}$
 being assumed to be adapted with respect 
to ${\boldsymbol W}^j$.
We then consider a collection of i.i.d. ${\mathbb R}^d$-valued 
control processes ${\boldsymbol \alpha}^{i,j}=
(\alpha_t^{i,j})_{0 \le t \le T}$, 
with each 
${\boldsymbol \alpha}^{i,j}$
 being assumed to be adapted with respect 
to $(X_0^i,{\boldsymbol B}^i,{\boldsymbol W}^j)$.
Each 
${\boldsymbol \alpha}^{i,j}$
and
each ${\boldsymbol h}^{n,j}$
are assumed to be 
constant on 
each $[\ell T/p,(\ell+1)T/p)$ for $\ell \in \{0,\cdots,p-1\}$, where 
$p$ stands as before for the discretization parameter.
For a given outcome $\omega$ in the probability space carrying all the processes,
we then choose as inputs of the black-box (corresponding to 
${\mathcal R}^{p,\varpi X_0}$ in 
\eqref{eq:mathcalrp}) 
the noise 
${\boldsymbol B}^{i}(\omega)$, 
the environment 
$\overline{\boldsymbol m}^{n,j}(\omega)$
and
the control 
$\varpi {\boldsymbol \alpha}^{i,j}(\omega)$ perturbed by the additional 
randomization $\varepsilon \dot{\boldsymbol W}^{p,\varpi {\boldsymbol h}^{n,j} \hspace{-2pt} / \varepsilon,j}(\omega)$. 
All these inputs are represented by the green boxes in Figure \ref{fig:4}. 
As made clear by the red boxes on Figure \ref{fig:4}, the black-box returns a cost that depends on 
$i,j$ (in Figure \ref{fig:4}, we use the notation $\sharp i$ and $\sharp j$ to refer to 
the various inputs and outputs associated with ${\boldsymbol B}^i(\omega)$ and 
${\boldsymbol W}^j(\omega)$). Multiplying by ${\mathcal E}(\varpi {\boldsymbol h}^{n,j} \hspace{-2pt} /\varepsilon)(\omega)$ and averaging over $i,j$, 
we hence get the averaged cost that appears in the right-hand side
of 
\eqref{eq:mn+1:cost:functional}. This makes it possible to optimize
with respect to the control ${\boldsymbol \alpha}$ (or equivalently with respect to 
$({\boldsymbol \alpha}^{i,j})_{i,j\geq 1}$)
and thus to compute the optimal control  
${\boldsymbol \alpha}^{n+1}$ that appears in the left-hand side of 
\eqref{eq:mn+1:cost:functional} (up to the time-discretization procedure that we feel better not to address at this early stage of the paper): this is the yellow box in 
Figure \ref{fig:4}. Once the optimal control has been computed, we may follow 
the arrows connecting the three blue boxes in 
Figure \ref{fig:4}: for each pair $(i,j)$, we can compute the realization 
${\boldsymbol X}^{n+1,i,j}(\omega)$ of the optimal state. 
Averaging with respect to 
$i$, we get 
${\boldsymbol m}^{n+1,j}(\omega)$
and then, following the updating rule 
\eqref{eq:mn+1:cost:update:mean}, we can update the state of the environment:
the new state is 
$\overline{\boldsymbol m}^{n+1,j}(\omega)$.
The new value of ${\boldsymbol h}^{n+1,j}(\omega)$ is directly taken from the affine form 
of ${\boldsymbol \alpha}^{n+1,j}(\omega)$, see 
\eqref{eq:intro:optimal:alpha}.
\color{black}

\begin{figure}[htb]
\begin{tikzpicture}
  [node distance=1cm,
  start chain=going below,]
    %% No. 2

  \node[fill=green,punktchain] (C2) {Realization of common noise \#$j$};

   \node[fill=green,punktchain] (C1) {Realization of idiosyncratic noise \#$i$};
    %  \node[punktchain, on chain=going right] (C10) {Decentralized Black Box \#1};

%   \node at (0,-4) {$\vdots$};
   %
     \node[fill=yellow,punktchain, xshift=3cm] (C3) {Control adapted to noises \#$i,j$};

    %\node at (0,-8) {$\vdots$};
      
     %\node[punktchain, on chain= going below] (C4)      {Idiosyncratic Noise};
    
     %\node[punktchain, on chain= going below] (C40)      {State of the population};

       { [start chain=numbers going below]
       \chainin (C2);

   \node[fill=red,punktchain, on chain=going right, xshift=.8cm,yshift=-1.2cm] (C11) {Decentralized black-box};

   \node[fill=red,punktchain, on chain=going right, xshift=0.5cm] (C21) {Cost \#$i,j$};

   \node[fill=red,punktchain, on chain=going below, yshift=-1cm] (C31) {Averaging over $i,j$ of costs \#$i,j$};
   
    \node[fill=cyan,punktchain, on chain=going below, xshift=-4cm] (C41) {New private state \#$i,j$};
    
    \node[fill=cyan,punktchain, on chain=going below] (C51) {Averaging over $i$ of states \#$i,j$};
    
    \node[fill=cyan,punktchain, on chain=going below] (C61) {New population state \#$j$};}

      { [start chain=numbers going below]
       \chainin (C2);

       \node[fill=green, punktchain, on chain=going above] (C71) {Initial private state};       
       
              \node[fill=green,punktchain, on chain=going right, xshift=.8cm] (C4) {State of the population adapted to common noise \#$j$};

   }

\draw[->, thick,] 
 let \p1=(C1.north), \p2=(C11.west) in
($(\x1,\y1)$) |- ($(\x2,\y2)$)
node[right, midway]  {};

\draw[->, thick,] 
 let \p1=(C2.south), \p2=(C11.west) in
($(\x1,\y1)$) |- ($(\x2,\y2)$)
node[above, pos=.8]  {inputs};

\draw[->, thick,] 
 let \p1=(C3.north), \p2=(C11.west) in
($(\x1,\y1)$) |- ($(\x2,\y2)$)
node[right, midway]  {};

\draw[-, thick,] 
 let \p1=(C4.west), \p2=(C3.north) in
($(\x1,\y1)$) -| ($(\x2,\y2)$)
node[left, midway]  {};

\draw[->, thick,] 
 let \p1=(C11.east), \p2=(C21.west) in
($(\x1,\y1)$) -- ($(\x2,\y2)$)
node[above, midway]  {return};

\draw[->, thick,] 
 let \p1=(C21.south), \p2=(C31.north) in
($(\x1,\y1)$) -- ($(\x2,\y2)$)
node[left, midway]  {expectation};

\draw[->, thick,] 
 let \p1=(C31.west), \p2=(C41.north) in
($(\x1,\y1)$) -| ($(\x2,\y2)$)
node[right, pos=.8]  {optimal state};

\draw[->, thick,] 
 let \p1=(C31.west), \p2=(C3.east) in
($(\x1,\y1)$) -- ($(\x2,\y2)$)
node[above, midway]  {optimization};

\draw[->, thick,] 
 let \p1=(C41.south), \p2=(C51.north) in
($(\x1,\y1)$) -| ($(\x2,\y2)$)
node[right, pos=.8]  {};

\draw[->, thick,] 
 let \p1=(C51.south), \p2=(C61.north) in
($(\x1,\y1)$) -| ($(\x2,\y2)$)
node[right, pos=.8]  {};

\draw[-, thick,] 
 let \p1=(C71.east), \p2=(C3.north) in
($(\x1,\y1)$) -| ($(\x2,\y2)$)
node[right, midway]  {};

%
%   \draw[->, thick,dashed] 
%  let \p1=(C4.east), \p2=(C10.west) in
%($(\x1,\y1)$) -- ($(\x2,\y2)$)
%node[right, midway]  {};
%
%   \draw[->, thick,dashed] 
%  let \p1=(C40.east), \p2=(C10.south) in
%($(\x1,\y1)$) -| ($(\x2,\y2)$)
%node[right, midway]  {};
%
%   \draw[->, thick,] 
%  let \p1=(C10.east), \p2=(C20.west) in
%($(\x1,\y1)$) -- ($(\x2,\y2)$)
%node[right, midway]  {};
%
%   \draw[->, thick,] 
%  let \p1=(C15.west), \p2=(C10.north) in
%($(\x1,\y1)$) -| ($(\x2,\y2)$)
%node[right, midway]  {};

% node[right, midway]  {optimization};
%
%   \draw[->, thick,]
%  let \p1=(C40.east), \p2=(C10.west) in
%($(\x1,\y1)$) -- ($(\x2,\y2)$)
%
% node[below, midway]  {optimization};
%
%  \draw[->, thick,]
%  let \p1=(C1.south), \p2=(C40.north) in
%($(\x1,\y1)$) -- ($(\x2,\y2)$)
%
% node[left, midway]  {mean field limit};
% 
%   \draw[->, thick,]
%  let \p1=(C4.south), \p2=(C10.north) in
%($(\x1,\y1)$) -- ($(\x2,\y2)$)
%
% node[right, midway]  {mean field limit};

  \end{tikzpicture}
    \vskip 3pt
  
  \caption{\small Black-box (in red) inserted at the core of a learning step.
  Expectations are formally written as means over a sequence of realizations from an i.i.d. sample of 
  initial conditions and of
 idiosyncratic and common noises.  
   For any respective realizations \#$i$ 
   \textcolor{black}{(\#$i$ reading `number $i$')}
   and 
  \#$j$ \textcolor{black}{(\#$j$ reading `number $j$')} of the idiosyncratic and common noises plugged into the black-box, we compute the corresponding state of the population (that depends on the realization  \#$j$ of the common noise) 
and we choose the control (that is adapted to the realizations  \#$i,j$ of the two noises and of the initial condition). 
The inputs are thus in green, except the control which is subjected to optimization (hence in yellow, as a mixture of green and red). 
The return (in red) of the black-box depends on the realizations  \#$i,j$. 
The expectation is formally obtained by 
averaging with respect to \#$i,j$. Once the optimizer has been found, we compute (in blue) the optimal states, depending on the realizations  \#$i,j$. The output state of the population after one learning step is obtained by averaging over $i$.}
\label{fig:4}

\end{figure}

\subsection{Exploitation}
\label{subse:exploitation}
\textcolor{black}{We now summarize the outline of the exploitation analysis that is achieved in the paper. 
In particular, 
we present the main bound that we can prove next for 
the so-called \textcolor{black}{exploitability} of the (geometric) tilted fictitious play (the meaning of which is explained below).}

In reinforcement learning, this is \textcolor{black}{indeed} a common practice to distinguish exploration from exploitation. 
Whereas exploration is intended as a way to visit the space of actions, exploitation is related to the error that is achieved by the learning method.
When the learning addresses a stochastic control problem, 
the analysis goes through the loss, which is the (absolute value of the)  difference between the best possible cost and the cost to the strategy returned by the algorithm.
Because we are dealing with a game, we use here the concept of approximated Nash equilibrium to define the \textcolor{black}{exploitability}. In brief, the point is to prove that the output of the algorithm is a $\varrho$-Nash equilibrium, for $\varrho>0$ and to define the \textcolor{black}{exploitability} as the infimum of those $\varrho$. In our case,
the situation is a bit more subtle because the learning procedure returns a random approximated equilibrium; ideally, we should associate with it a random \textcolor{black}{exploitability}. However, the analysis would be too difficult. Instead, we 
follow
\eqref{eq:mn+1:cost:functional}
and 
average the cost with respect to the common noise. We then define the notion of approximated equilibrium 
with respect to this averaged cost. The resulting \textcolor{black}{exploitability} can then be decomposed as the sum of two terms:
\begin{itemize}
\item The `error' resulting from the approximation, learnt by our fictitious play, 
of
the discrete-time mean field game 
with step size $T/p$ and with  
common noise of intensity $\varepsilon$. The implementation of the algorithm involves $n$ iterations and 
a piecewise linear approximation 
of the Brownian motion ${\boldsymbol W}$ with $T/p$ steps. For fixed\footnote{\color{black} For convenience, we assume $\varepsilon$ to be less than 1 in the analysis. Obviously, the results would remain true for 
$\varepsilon >1$, but the various constants in our analysis could then depend on any \textit{a priori} upper bound on $\varepsilon$.} $\varepsilon \in (0,1]$ and 
$p \in {\mathbb N}$, this error tends to $0$ as $n$ tends to $\infty$, with an explicit rate $\varrho_{\varepsilon,p}(n)$.
\item The `error' resulting from the common noise and discrete-time approximation of the original time-continuous mean field game without common noise. 
Intuitively, the (unique) solution of the time-discrete mean field game with common noise produces a
random approximated Nash equilibrium of the original mean field game
whose accuracy gets finer and finer as $p$ increases and $\varepsilon$ decreases. 
\end{itemize}
In fact, this principle can be put in the form of a more general stability result 
that permits to evaluate in the end the trade-off between exploration and exploitation. 
\textcolor{black}{When $X_0$ has sub-Gaussian tails, the exploitability is bounded by}
\begin{equation}
\label{eq:intro:regret:bound}
O \biggl(
 \exp(C \varepsilon^{-2})  \varpi^{-n}+
    \varepsilon + \tfrac1p      \biggr),
\end{equation}
see Theorem \ref{thl:main:regret}. 
For fixed values of $n$ and $p$ (the latter two parametrizing the 
complexity of the algorithm and the required memory\footnote{We feel better not to give any order of complexity and memory in terms of $n$ and $p$. This would have no sense because the integration and  optimization steps are not discretized here.}), we are hence able to tune the intensity of the noise. 

To our mind, this result demonstrates the interest of our concept, 
even though it says nothing about the equilibria that are 
hence selected in this way when $\varepsilon$ tends to $0$. In fact, the latter is a difficult question, which has been addressed for instance in 
\cite{delfog2019} (for the same model as in 
\eqref{eq:state:intro}--\eqref{eq:cost:intro}, but with $d=1$ and for some specific choices of $f$ and $g$)
 and \cite{cecdel2019} (for finite state potential games); generally speaking, this problem raises many theoretical questions that are out of the scope of this paper and we just address it here through numerical examples (see the subsection below).   
%\textcolor{red}{We should see how to include $p$ in some way in the discussion.}
The diagram \eqref{fig:3} right below hence summarizes the balance between exploitation and exploration in our case.
\color{black}
In words, 
the geometric tilted fictitious play allows us to learn a solution of the discrete-time MFG
 with step size $T/p$ and 
 with common noise
(or, equivalently, \textit{exploration} noise)  
of intensity $\varepsilon >0$ by choosing $n$ large enough. 
The equilibrium that is hence learnt is an approximate equilibrium of the original 
MFG 
(in continuous time and without common noise). 
For a small intensity $\varepsilon$, the \textcolor{black}{exploitability} is small if $n$ and $p$ are large enough, hence the trade-off between $\varepsilon$
and $(n,p)$.
\color{black}

\begin{figure}[htb]
\begin{tikzpicture}
  [node distance=1cm,
  start chain=going below,]
    %% No. 2

   \node[punktchain, join] (C1) {MFG without common noise};
   
     \node[punktchain, on chain= going above, xshift=-3.5cm] (C2)      {MFG with common noise};
    
     \node[punktchain, on chain= going right, xshift=2cm] (C3)      {Approximate solution};

   \draw[->, thick,] 
  let \p1=(C1.west), \p2=(C2.south) in
($(\x1,\y1)$) -| ($(\x2,\y2)$)
node[left, pos=.8]  {exploration};

   \draw[->, thick,] 
  let \p1=(C1.west), \p2=(C2.south) in
($(\x1,\y1)$) -| ($(\x2,\y2)$)
node[left, pos=.65]  {$\varepsilon >0$};

   \draw[->, thick,dashed] 
  let \p1=(C2.east), \p2=(C3.west) in
($(\x1,\y1)$) -- ($(\x2,\y2)$)
node[above, midway]  {fictitious play};
%($(\x1,\y1)$) -- ($(\x2,\y2)$)
   \draw[->, thick,dashed] 
  let \p1=(C2.east), \p2=(C3.west) in
($(\x1,\y1)$) -- ($(\x2,\y2)$)
node[below, midway] {$n,p \uparrow \infty$};

   \draw[->, thick,] 
  let \p1=(C3.south), \p2=(C1.east) in
($(\x1,\y1)$) |- ($(\x2,\y2)$)
node[right, pos=.2]  {exploitation};

   \draw[->, thick,] 
  let \p1=(C3.south), \p2=(C1.east) in
($(\x1,\y1)$) |- ($(\x2,\y2)$)
node[right, pos=.35]  {$\varepsilon \downarrow 0$ vs. $n,p \uparrow \infty$};

% node[right, midway]  {optimization};
%
%   \draw[->, thick,]
%  let \p1=(C40.east), \p2=(C10.west) in
%($(\x1,\y1)$) -- ($(\x2,\y2)$)
%
% node[below, midway]  {optimization};
%
%  \draw[->, thick,]
%  let \p1=(C1.south), \p2=(C40.north) in
%($(\x1,\y1)$) -- ($(\x2,\y2)$)
%
% node[left, midway]  {mean field limit};
% 
%   \draw[->, thick,]
%  let \p1=(C4.south), \p2=(C10.north) in
%($(\x1,\y1)$) -- ($(\x2,\y2)$)
%
% node[right, midway]  {mean field limit};

  \end{tikzpicture}
  \vskip 3pt
  
  \caption{\small Exploration vs. exploitation.}
\label{fig:3}
\end{figure}
For sure, our analysis of exploitation is carried out in the ideal case when 
the underlying expectations in 
\eqref{eq:mn+1:cost:functional}
are understood in the theoretical sense
and when the optimal control problem at each step can be 
computed perfectly. 
We do not address here the 
approximation of those theoretical expectations by 
empirical means nor the 
numerical approximation of the optimizers.

\subsection{Numerical examples}
\label{subse:num:ex}
We complete the paper with some numerical examples that demonstrate the relevance of our concept. 
The numerical implementation requires additional ingredients that are explained in detail in Section \ref{Se:3}. Obviously, the main difficulty is the encoding of the decentralized black-box, as represented in Figures
\ref{fig:2} and \ref{fig:4}. 
As clearly suggested by the latter figure, 
expectations are then approximated by averaging the costs over realizations from a finite i.i.d. sample of 
 idiosyncratic and common noises.  
 In this regard, the principle 
highlighted by
Figure 
\ref{fig:4} is the same, but the 
optimization step needs to be clarified. 
Although we do not provide any further theoretical justification of the accuracy of the numerical optimization that is hence performed, we feel useful to stress that
controls are chosen in a semi-feedback form, namely 
of the same linear form 
as in 
\eqref{eq:intro:optimal:alpha}. Numerically, the coefficient $\eta_{t}$ is hence parameterized in the form of a 
coefficient that is only allowed to depend on time; and the intercept $h_{t}$ is sought as 
a function of the current mean state of the population. This latter function is parameterized in the form of a finite expansion along an Hermite polynomial basis, our choice for   
Hermite polynomials being dictated by the Gaussian nature of the trajectories in 
\eqref{eq:n:n+1}, \textcolor{black}{or in the form of a neural network}. 
In our numerical experiments, both the linear coefficient and the coefficients in the regression of the intercept along the \textcolor{black}{given class of functions} are found by ADAM optimization method. 
The results exposed in Section \ref{Se:3} demonstrate the following features:
\begin{enumerate}
\item For a given value of the intensity of the common noise (say $\varepsilon=1$),
we run examples in dimension $d=2$. Both the (tilted) harmonic and geometric fictitious play converge well, without any significant differences between the two of them on the examples under study. 
Solutions are compared to 
numerical solutions of
 \eqref{eq:intro:FBSDE} found by a BSDE solver that uses explicitly the shape of the coefficients 
 $f$, $g$, $Q$ and $R$ and that does not use the observations of the cost. 
 \item \textcolor{black}{We also run examples in higher dimension, namely $d=12$ and $d=20$. 
 Obviously, this raises challenging questions in terms of complexity, 
in particular for 
regressing
${\boldsymbol h}^{n+1}$ in  
\eqref{eq:mn+1:conditional:state}
by means of a suitable basis. 
 Although we show numerically that neural networks may behave well in these examples, 
 the main difficulty 
 in the higher dimensional setting comes from the various Monte-Carlo estimations 
 on which Figure 
  \ref{fig:4} implicitly relies. Indeed, 
  the variance of the cost 
  in 
  \eqref{eq:mn+1:cost:functional}
 may increase fast with the dimension.
 This phenomenon has an impact on the behavior of the algorithm, 
 which may fail to converge 
 as 
  clearly illustrated in some of the examples below.
Next, we
address one simple reduction variance method, which 
returns much better results in dimension $d=12$ and $d=20$ and which 
demonstrates  
that the algorithm may remain relevant in this more challenging framework 
at the price of some marginal modifications.
In light of these results, we believe that it would be highly valuable to provide 
 a more exhaustive analysis of the possible strategies
for reducing the variance underpinning the various Monte-Carlo computations. 
We hope to address this point in future contributions.  
 }
\item \textcolor{black}{The standard fictitious play, with idiosyncratic noise but without common noise, may fail to converge, meaning that,
not only there may not be any known mathematical guarantee supporting the convergence, but even more, 
for one of the examples we treat below in dimension $d=2$, the cost returned by the algorithm has an oscillatory behavior
that is not observed in presence of the common noise.
Even though we do not have a mathematical explanation for these observations,  
this is a crucial point of the paper as it demonstrates numerically that the common noise helps the algorithm to converge}. From a conceptual point of view, this is a key observation. 
\textcolor{black}{Of course, it would be very much desirable to have a table 
of comparison, with a mathematical description of the behavior of the algorithm 
without and with common noise and for various types coefficients. 
This looks however out of reach of the existing literature. 
Still, it is worth observing that, 
in dimension 1, the usual algorithm is known to converge because the model is potential, see \cite{CardaliaguetHadikhanloo}, and that, in this 
setting, we have not observed any 
numerical oscillatory behavior similar to the two dimensional one (for the same types of coefficients). 
We elaborate on this point in \S 
\ref{subsubse:M=1N=1}.
We also feel useful to recall from the previous Subsection 
\ref{subse:exploitation}
 that, from a theoretical point of view, we are able to 
provide explicit bounds for the rate of convergence, which is another substantial contribution of our work.  Indeed, 
as explained in the forthcoming \S \ref{subsubse:rate}, we know a few results only in which the rate of convergence of the fictitious play 
without common noise is addressed explicitly.  }
%In dimension 1, the usual algorithm is known to converge since the model is potential, see \cite{CardaliaguetHadikhanloo}. \textcolor{red}{TO BE UPDATED: RATE OF CV}
\item In order to study the behavior  of our fictitious play when the intensity $\varepsilon$ of the common noise becomes small, we focus on a one-dimensional MFG that has multiple equilibria when there is no common noise.
This case is highly challenging. 
Not only theoretical bounds like 
\eqref{eq:intro:regret:bound}
are especially bad when $\varepsilon$ is small, but also additional numerical issues 
arise in the small noise regime. In particular, the variance of the various estimators suffer from the same drawback as in the high-dimensional setting
and may be very large. 
However, we here show that, numerically, the
geometric tilted fictitious play run with a high rate $\varpi$ and a small intensity $\varepsilon$ is able to select quite quickly the same equilibrium as the one predicted in \cite{delfog2019,cecdel2019}. In contrast, 
the standard fictitious play (without common noise) also converges but may not select the right equilibrium (for the same choice of parameters). 
Also, it is worth observing that, in this experiment, 
the geometric variant of the titled fictitious play performs better than the harmonic one, which is consistent with our theoretical analysis. By the way, in the first arXiv version 
\cite{DelarueVasileiadis}
of this work (in which we just studied the harmonic variant), 
we complemented the numerical analysis with a preferential sampling 
method in order to reduce the underlying variance and hence obtain a better accuracy in the selection of an equilibrium. 
This would be an interesting question to address the possible interest of such a preferential sampling method when combined with the geometric variant of the algorithm. 
We leave this for future works. 
\end{enumerate}
In the end, the numerical experiments carried out here confirm the relevance of the tilted scheme. However, it is fair to say that it is more subtle to demonstrate the superiority of the geometric variant over the harmonic variant, even if the experiment (4) reported above clearly points in that direction. In our opinion, the geometric variant has the great merit of offering theoretical convergence guarantees that are affordable numerically. However, the experiments described in Section \ref{Se:3} show that the harmonic variant is also numerically relevant. Numerical results could probably be optimized by combining the two approaches, so as to benefit from the advantages of both.  
 
It should be stressed that our numerical experiments are run under \texttt{Tensorflow}, using a pre-implemented version of 
ADAM optimization method in order to compute an approximation of the best response in 
\eqref{eq:mn+1:cost:functional}. 
Accordingly, the optimization algorithm itself relies explicitly on the linear-quadratic structure 
of the mean field game through the internal automatic differentiation procedure (used for computing gradients in descents). In this sense, our numerical experiments use in fact more than the sole observations of the costs. Anyhow, this does not change the conclusion: descent methods, based on accurate approximations of the gradients, do benefit from the 
presence of the common noise. For instance, the construction of 
accurate approximations of the gradient is addressed in Carmona and Laurière \cite{CarmonaLauriereTanI,CarmonaLauriereTanII} (within a slightly different setting), in which a model-free reinforcement learning method is fully implemented\footnote{The reader may find in Munos \cite{munos} a nice explanation about the distinction between model-free and model-based reinforcement learning. In any case, our algorithm is not model-based: It would be model-based if we tried to learn first $f$, $g$, $Q$ or $R$. We refer to \S \ref{para:error} for a discussion about the possible numerical interest to learn $Q$ and $R$ first.}.

\subsection{Comparison with existing works}
\label{subse:comparison}
Exploration and exploitation are important concepts in reinforcement 
learning and related optimal control.

In comparison with the discrete-time literature, there have been less papers on the analysis 
of exploration/exploitation in the time continuous setting. In both  Murray and Paladino \cite{murray2018model} and Wang et al. \cite{JMLR:v21:19-144}, the randomization of the actions goes through 
a formulation of the corresponding control problem in terms of relaxed controls. 
In Murray and Paladino  \cite{murray2018model}, the authors address questions that are seemingly different from ours, as  the objective is to allow for a model with some uncertainty on the state 
dynamics. Accordingly, the cost functional is averaged out with respect to some prior
probability measure on the vector field driving the dynamics. Under suitable assumptions on this prior probability measure, a dynamic programming principle and a then a Hamilton-Jacobi-Bellman are derived. 
In fact, the  
paper
Wang et al.
\cite{JMLR:v21:19-144}, which addresses stochastic optimal controls, 
is closer to the spirit of our work. Therein, relaxed controls are combined with an additional entropic regularization  that forces exploration. In case when the control problem has a linear-quadratic structure, quite similar to the one we use here (except that there is no mean field interaction), the entropic regularization is shown to work as a Gaussian exploration. Although our choice for working with a Gaussian exploration looks consistent with the result of 
Wang et al.
\cite{JMLR:v21:19-144}, there remain however some conceptual differences between the two approaches: In the theory of relaxed controls, the drifts in the dynamics are averaged out with respect to the distribution of the controls; In our paper, the dynamics are directly subjected to the randomized action. In this respect, our work is closer to
the earlier contribution of Doya \cite{DBLP:journals/neco/Doya00}.

Recently, the  approach initiated in 
Wang et al. 
\cite{JMLR:v21:19-144}
has been extended to mean field games. 
In 
Guo et al. 
\cite{Guo-Xu-Zariphopoulou-arxiv}
and
Firoozi and Jaimungal
\cite{Jaimungal-Firoozi-arxiv},
the authors study
the impact of an entropic regularization onto 
the shape of the equilibria.
In both papers, the models under study
are linear-quadratic
and subjected to a sole idiosyncratic noise (i.e., there is no common noise). 
However, they differ on the following important point:
in 
Firoozi and Jaimungal
\cite{Jaimungal-Firoozi-arxiv}, the intensity of the idiosyncratic noise is constant, whilst it depends on the standard deviation of the control in 
Guo et al. 
\cite{Guo-Xu-Zariphopoulou-arxiv}. In this sense, 
Firoozi and Jaimungal \cite{Jaimungal-Firoozi-arxiv} is closer to the set-up that we investigate here. 
Accordingly, the presence of the entropic regularization  leads to different consequences: 
In Guo et al. \cite{Guo-Xu-Zariphopoulou-arxiv}, the effective intensity of the idiosyncratic noise grows up under the action of the entropy and this is shown to help numerically in some learning method (of a quite different spirit than ours). 
In Firoozi and Jaimungal \cite{Jaimungal-Firoozi-arxiv}, the entropy plays no role on the structure of the equilibria, which demonstrates, if needed, that our approach here is substantially different.

Within the mean field framework, there have been several recent contributions on reinforcement learning for models featuring a common noise. In
Carmona et al. \cite{CarmonaLauriereTanI}, the authors investigate the convergence of a policy gradient method for a discrete time linear quadratic mean field control problem (and not an MFG) with a common noise. In comparison with \eqref{eq:cost:intro}, the cost functional itself is quadratic with respect to the mean field interaction. The linear quadratic structure then allows us to simplify the search for the optimal feedbacks, in the form of two linear functions, one linear function of the mean state of the population and one linear function of the deviation to the mean state.  Accordingly, the 
 problem is rewritten in terms of two separate (finite-dimensional) linear-quadratic control problems, one with each of the
 two linear factors. Convergence of the descent for finding the optimizers is studied for a model free method using a black-box simulating the evolution of the population. 
 This black-box is comparable to ours. Importantly, non-degeneracy of the very first inputs of the common noise is used in the convergence analysis.
 In another work (Carmona et al. \cite{CarmonaLauriereTanII}), the same three authors have developed a model free $Q$-learning method for a mean field control problem in discrete time and finite space. 
 The model may feature a common noise, but the latter has no explicit impact onto the convergence analysis carried out in the paper.  Last, in
 Elie et al. \cite{EliePerolat} and Perrin et al. \cite{perrin2020fictitious}, the authors deal with discrete and continuous time learning for MFGs using fictitious play and introducing a form of common noise. Their analysis is supported by various applications and numerical examples (including a discussion on the tools from deep learning to compute the best responses at any step of the fictitious play). The analysis also relies on the notion of exploitability. As the common noise therein is not used for exploratory reasons, the coefficients are assumed to satisfy the monotonicity Lasry-Lions condition 
 in order to guarantee the convergence of the fictitious play. 
 
 \textcolor{black}{After the publication of the first arXiv version \cite{DelarueVasileiadis} of our paper,   
 several related works by other authors were released. 
For instance,
the authors 
of
Muller et al. 
\cite{mueller}
address 
Policy Space Response Oracles (PSRO)
within the mean field framework in order to 
approximate Nash equilibria but also relaxed 
equilibria that are said to be correlated. 
In particular,
the authors explain how to 
implement PSRO 
via regret minimization
in order approximate correlated equilibria.
In Hu and Laurière \cite{hu:hal-03656245},
the authors provide a very nice survey of recent developments in machine learning for stochastic control and games. 
Last,
it is also fair to quote 
Hambly et al. 
 \cite{hambly}, in which the authors address the global convergence of the natural policy gradient method to the Nash
equilibrium
in a general $N$-player linear-quadratic game. Noticeably, the proof of convergence therein requires a certain amount of noise.}

\subsection{Main assumption, useful notation and organization.}
\label{subse:1:9}
\subsubsection{Assumption}
Throughout the analysis, $\sigma$ in 
\eqref{eq:state:intro} is a fixed non-negative real. 
The intensity $\varepsilon$ of the common noise is taken in $[0,1]$. Most of the time, it is implicitly 
required to be strictly positive, but we sometimes refer to the case $\varepsilon=0$ in order to 
compare with the situation without common noise.
\textcolor{black}{As we are mainly interested with the case when $\varepsilon$ is small, 
we assume $\varepsilon$ to be in $[0,1]$ in any case.}
As for the coefficients $f$ and $g$, they are assumed to be bounded and Lipschitz continuous. We write 
\begin{equation*}
\|f \|_{1,\infty} = \sup_{x \in {\mathbb R}^d} \vert f(x) \vert 
+ \sup_{x,x' \in {\mathbb R}^d : x \not = x'} 
\frac{\vert f(x) - f(x') \vert}{\vert x-x'\vert} < \infty,
\end{equation*}
and similarly for $g$.

\subsubsection{Notation} 
The notation $I_{d}$ stands for the $d$-dimensional identity matrix.
\textcolor{black}{For a random variable $Z$ with values in a Polish space ${\mathcal S}$,
we denote by $\sigma(Z)$ the $\sigma$-field generated by $Z$.}
For a process ${\boldsymbol Z}=(Z_{t})_{0 \leq t \leq T}$ with values in a Polish space ${\mathcal S}$, we denote 
by ${\mathbb F}^{{\boldsymbol Z}}$ the augmented filtration generated by 
${{\boldsymbol Z}}$. In particular, the notation ${\mathbb F}^{{\boldsymbol W}}$
is frequently used to denote the augmented filtration generated by the common noise when $\varepsilon \in (0,1]$. 
We recall  that 
${\mathbb F}={\mathbb F}^{(X_{0},\sigma {\boldsymbol B},\textcolor{black}{\varepsilon}{\boldsymbol W})}$.% \textcolor{blue}{(the filtration ${\mathbb F}$ always comprises the common noise 
%${\boldsymbol W}$)}. 

Also, for a filtration ${\mathbb G}$, 
we write 
${\mathbb S}^{d}({\mathbb G})$ 
for the space of continuous 
and ${\mathbb G}$-adapted processes ${\boldsymbol Z}=(Z_{t})_{0 \leq t \leq T}$
with values in ${\mathbb R}^d$ that satisfy 
\begin{equation*}
 {\mathbb E} \Bigl[ \sup_{0 \leq t \leq T} \vert Z_{t} \vert^2 \Bigr] < \infty.
\end{equation*}

\subsubsection{Organization of the paper}
The mathematical analysis is carried out in Section \ref{Sec:2}. 
Subsection \ref{subse:2.1} addresses 
the error associated with the scheme 
\eqref{eq:mn+1:cost:functional}--\eqref{eq:mn+1:conditional:state}--\eqref{eq:mn+1:conditional:intro}--\eqref{eq:mn+1:cost:update:mean} without 
any time discretization, see Theorem \ref{main:thm:1}.
Similar results, but 
with the additional time discretization that makes the exploration possible, 
are established in Subsection 
\ref{subse:2.2}, see in particular Theorem \ref{main:thm:2}.
In Subsection \ref{subse:exploitation:exploration}, 
we make the connection with the original mean field game without common noise. In particular,  
we prove that our learning procedure permits to construct approximate equilibria to 
the original problem. 
The bound 
\eqref{eq:intro:regret:bound}
for the \textcolor{black}{exploitability} is established in Theorem  \ref{thl:main:regret}. 
Following our agenda, we provide the results of some numerical experiments in 
Section \ref{Se:3}. The method is tested on some benchmark examples that are presented in 
Subsection 
\ref{Subse:3:1}. The implemented version of the algorithm is explained in 
Subsection 
\ref{Subse:3:2}. 
The results, for a fixed intensity of the common noise, are exposed
in Subsection \ref{Subse:3:3:a}.
In the final Subsection 
\ref{subse:3:3}, we provide an example that illustrates the behavior  of the algorithm for a decreasing intensity of the common noise.

\section{Theoretical results }
\label{Sec:2}
The theoretical results are presented in three main steps. 
The general philosophy, as exposed in Figure \ref{fig:1}, is addressed in Subsection 
\ref{subse:2.1}. 
The analysis of the algorithm under a time-discrete randomization of the actions is 
addressed in Subsection \ref{subse:2.2}.
Lastly, the dilemma between exploration and exploitation is investigated in 
Subsection \ref{subse:exploitation:exploration}.

\subsection{Tilted fictitious play with common noise}
\label{subse:2.1}
Throughout the subsection, the intensity $\varepsilon \in (0,1]$ of the common noise
{and the learning parameter $\varpi$ are fixed. We take $\varpi$ in $(1,\sqrt{2}]$ 
(we refer to Remark \ref{rem:varpi} for the need to have $\varpi$ close to $1$).}

\color{black}

\subsubsection{Construction of the learning sequence}
\label{subse:5:2:fictitious:play}
We here formalize the scheme introduced in 
\eqref{eq:mn+1:conditional:state}--\eqref{eq:mn+1:conditional:intro}--\eqref{eq:mn+1:cost:functional}--\eqref{eq:mn+1:cost:update:mean}. 
We hence construct a sequence of proxies $({\boldsymbol m}^n)_{n \geq 0}$
and $({\boldsymbol h}^n)_{n \geq 0}$. The two initial processes 
${\boldsymbol m}^0$ and ${\boldsymbol h}^0$ are two ${\mathbb F}^{{\boldsymbol W}}$-adapted continuous processes with values in ${\mathbb R}^d$. 
Typically, we choose 
${\boldsymbol m}^0=(m^0_{t}={\mathbb E}(X_{0}))_{0 \le t \le T}$
and 
${\boldsymbol h}^0=(h^0_{t}=0)_{0 \le t \le T}$.
Assuming that, at some rank $n$, we have already defined $({\boldsymbol m}^1,\cdots,{\boldsymbol m}^n)$ 
and $({\boldsymbol h}^1,\cdots,{\boldsymbol h}^n)$,
each in ${\mathbb S}^{d}({\mathbb F}^{{\boldsymbol W}})$, we call 
%\textcolor{red}{put $\varepsilon  {\boldsymbol W}^{{\boldsymbol h}^n}$ in ${\mathcal R}$?}
\begin{equation}
\label{eq:opti:n+1:n}
	{\boldsymbol \alpha}^{n+1,\varpi} := \textrm{\rm argmin}_{\boldsymbol \alpha}
	{\mathbb E}^{ \varpi {\boldsymbol h}^n \hspace{-2pt}/\varepsilon}
	\Bigl[ {\mathcal R}^{ {\varpi X_0}} \Bigl( \varpi {\boldsymbol \alpha}; \overline{\boldsymbol m}^n ; \varpi \textcolor{black}{\varepsilon} {\boldsymbol W}^{ \varpi {\boldsymbol h}^n
	\hspace{-2pt}/\varepsilon} \Bigr) \Bigr],
\end{equation}
the infimum being taken over controlled processes 
${\boldsymbol \alpha}$ that are ${\mathbb F}$-progressively 
and that satisfy ${\mathbb E}^{\varpi {\boldsymbol h}^n \hspace{-2pt} / \varepsilon} \int_{0}^T \vert \alpha_{t}^{n+1,\varpi} \vert^2 \ud t < \infty$. 
Above, 
the process $\overline{\boldsymbol m}^n =(\overline{m}^n_{t})_{0 \le t \le T}$
is defined in terms of 
$({\boldsymbol m}^0,\cdots,{\boldsymbol m}^n)$
through the formulas
$\overline{\boldsymbol m}^0 := {\boldsymbol m}^0$ (if $n=0$) 
and 
\begin{equation}
\label{eq:overline:mnt}
\overline{m}^n_{t} := \frac{\varpi(1-\varpi^{-1})}{1-\varpi^{-n}} \sum_{k=1}^n \varpi^{-k} m_{t}^k, \quad t \in [0,T], 
\end{equation}
if $n \geq 1$, with the latter being consistent with 
the geometric updating rule 
\eqref{eq:mn+1:cost:update:mean}.
The following lemma, 
the proof of which is deferred to Subsection
\ref{subsubse:auxiliary:statements}, explains how the next proxy ${\boldsymbol m}^{n+1}$ can be computed through the best
response of the control problem 
\eqref{eq:opti:n+1:n}:
\begin{lem}
\label{lem:2.1}
Under the above assumptions, the process
${\boldsymbol \alpha}^{n+1,\varpi}$ writes 
\begin{equation*}
\alpha_{t}^{n+1,\varpi} = - \Bigl( \eta_{t} X_{t}^{n+1,\varpi} + h_{t}^{n+1}\Bigr), \quad t \in [0,T], 
\end{equation*}
where 

$(i)$ ${\boldsymbol \eta}=(\eta_{t})_{0 \leq t \leq T}$ solves the Riccati equation:
\begin{equation}
\label{eq:riccati}
\dot{\eta}_{t} - \eta_{t}^2 + Q^\dagger Q=0, \quad t \in [0,T] \, ; \quad \eta_{T}=R^\dagger R; 
\end{equation}

$(ii)$ ${\boldsymbol h}^{n+1}=(h_{t}^{n+1})_{0 \leq t \leq T} \in {\mathbb S}^{d}({\mathbb F}^{{\boldsymbol W}})$ solves the backward SDE:
\begin{equation}
\label{eq:h:n+1:new:one}
	\begin{split}
		&\ud h_{t}^{n+1} =   \bigl( - \tfrac1{\varpi} Q^\dagger f( \overline m^n_{t}) + \eta_{t} h_{t}^{n+1} \bigr) \ud t 
		+ \varepsilon k_{t}^{n+1} \ud W_{t}^{{\varpi {\boldsymbol h}^n \hspace{-2pt} /\varepsilon} }, \quad t \in [0,T],
		\\
		&h_{T}^{n+1} = \tfrac1{\varpi}  R^\dagger g \bigl( \overline m^{n}_{T} \bigr); 
	\end{split}
\end{equation}

$(iii)$ ${\boldsymbol X}^{n+1,\varpi}=(X_{t}^{n+1,\varpi})_{0 \le t \le T}$ solves the forward SDE:
\begin{equation*}
\ud X_{t}^{n+1,\varpi} = - \Bigl( \eta_{ t} X_{t}^{n+1,\varpi} + h_{t}^{n+1} \Bigr) \ud t + \tfrac1{\varpi} \sigma \ud B_{t} + \varepsilon \ud W_{t}^{ {\varpi {\boldsymbol h}^n\hspace{-2pt} /\varepsilon}}, \quad t \in [0,T]; \quad  X_0^{n+1,\varpi} = X_0. 
\end{equation*}
\end{lem}
The statement makes it possible to let
\begin{equation}
\label{eq:def:m:n}
	m^{n+1}_{t} := {\mathbb E} \bigl[ X_{t}^{n+1,\varpi} \, \vert \, \sigma({\boldsymbol W}) \bigr], \quad t \in [0,T]
	 ; \quad m_0^{n+1} = {\mathbb E}(X_0), 
\end{equation}
and then, consistently with 
\eqref{eq:overline:mnt}, 
\begin{equation}
\label{eq:def:overline:m:n}
	\overline{m}^{n+1}_{t} := \tfrac{\varpi(1-\varpi^{-1})}{1-\varpi^{-(n+1)}} \sum_{k=1}^{n+1} \varpi^{-k} m_{t}^{k} = 
	\tfrac{\varpi^{-n}(1-\varpi^{-1})}{1- \varpi^{-(n+1)}}
	m^{n+1}_t +  \Bigl( 1 - 
	\tfrac{\varpi^{-n}(1-\varpi^{-1})}{1- \varpi^{-(n+1)}}
	 \Bigr) 
	\overline{m}^n_{t}. 
\end{equation}

\begin{rem}
\label{rem:2.2:measurability}
Notice that because the density ${\mathcal E}(\varpi {\boldsymbol h}^n \hspace{-2pt} / \varepsilon)$ is $\sigma({\boldsymbol W})$-measurable, we also have
\begin{equation*}
	m^{n+1}_{t} = {\mathbb E}^{\varpi {\boldsymbol h}^n \hspace{-2pt} / \varepsilon} \bigl[ X_{t}^{n+1,\varpi} \, \vert \, \sigma({\boldsymbol W}) \bigr], \quad t \in [0,T], 
\end{equation*}
\textcolor{black}{which is a direct consequence of Bayes' rule, see \cite[Exercise 23.14]{JacodProtter}}.
\end{rem}

\begin{rem} 
%The reader will notice that the cost functional that is used 
%in 
%\eqref{eq:opti:n+1:n}
%is not the same as the cost functional underpinning the MFG because the controlled dynamics are initialized from 
%$\varpi X_0$ and not from $X_0$. 
%The reason for this choice is clarified in the next subsection. In showrt, 
%the presence of the additional $\varpi$ forces the rate of convergence
%to be geometric.
%Notably, the limit of the scheme is independent of $\varpi$
%because $(\overline {\boldsymbol m}^n)_{n \geq 1}$ is shown 
%to converge to the solution of the mean field game with common noise  
%\eqref{eq:state:intro}--\eqref{eq:cost:intro}--\eqref{eq:intro:fixed}.
Although each 
$\overline{\boldsymbol m}^{n}$ depends in an obvious manner on 
$\varpi$, this is our choice not to add $\varpi$ as a label in 
the notation. The reason is that the limit is independent of $\varpi$, 
as clarified in \S \ref{subse:analysis:fictitious:play}
below. 
Similarly, 
we do not put $\varpi$ in 
the notation ${\boldsymbol h}^n$: as shown in Theorem 
\ref{main:thm:1}, the limit is independent of $\varpi$ up to a rescaling by $\varpi$. 
This is in contrast with 
the quantities ${\boldsymbol X}^{n,\varpi}$ 
and ${\boldsymbol \alpha}^{n,\varpi}$: 
the limits depend on $\varpi$ in a non-trivial manner, 
which we also make clear in 
the next paragraph. 
 
On another matter, it must be noticed 
that Lemma 
\ref{lem:2.1}
remains valid when 
$\varpi=1$. As explained in 
Introduction, see
\eqref{eq:convention:harmonic:geometric}, the updating rate 
in 
\eqref{eq:def:overline:m:n}
must  then be understood 
as $1/(n+1)$ and, accordingly, 
the 
formula 
\eqref{eq:overline:mnt}
coincides with
the harmonic updating rule 
\eqref{eq:mn+1:cost:update:mean:0}.
This remark is important in order to compare the two harmonic and 
geometric schemes. 
\end{rem}

\subsubsection{Main statement}
\label{subse:analysis:fictitious:play}
As noticed in the introduction, 
it is especially convenient to reformulate the dynamics of 
${\boldsymbol X}^{n+1,\varpi}$ under the historical probability. Quite clearly, we can write:
\begin{equation}
\label{eq:X:n:n+1}
\ud X_{t}^{n+1,\varpi} = - \Bigl( \eta_{ t} X_{t}^{n+1,\varpi} + h_{t}^{n+1} -  {\varpi} h_{t}^{n} \Bigr) \ud t + \tfrac1{\varpi} \sigma \ud B_{t} + \varepsilon \ud W_{t}, \quad t \in [0,T]; 
\quad X_0^{n+1,\varpi} = X_0,
\end{equation}
and then 
\begin{equation}
\label{eq:mmm:n:n+1}
\ud m_{t}^{n+1} = - \Bigl( \eta_{ t} m_{t}^{n+1} + h_{t}^{n+1} -  {\varpi} h_{t}^{n} \Bigr) \ud t + \varepsilon \ud W_{t}, \quad t \in [0,T]; 
\quad m_0^{n+1} = {\mathbb E}(X_0). 
\end{equation}
%The analysis of the fictitious play then goes through the additional processes 
%$(\overline {\boldsymbol h}^n)_{n \geq 1}$, 
%whose definition is similar to 
%\eqref{eq:overline:mnt}:
%\begin{equation*}
%	\overline{h}^n_{t} = \tfrac{\varpi(1-\varpi)}{1-\varpi^{-n}} \sum_{k=1}^n \varpi^{-k} h_{t}^k, \quad t \in [0,T]. 
%\end{equation*}
By {dividing \eqref{eq:mmm:n:n+1} by $\varpi^{n+1}$, then summing over the index $n$ and eventually multiplying by $\varpi(1-\varpi^{-1})/(1-\varpi^{-(n+1)})$}, we get:
\begin{equation}
\label{eq:m:n:n+1}
	\ud \overline m_{t}^{n+1} = - \Bigl( \eta_{t} \overline m_{t}^{n+1} + \tfrac{(1-\varpi^{-1}) \varpi^{-n}}{1- \varpi^{-(n+1)}} h_t^{n+1}   \Bigr) \ud t + \varepsilon  \ud W_{t}, \quad t \in [0,T].
\end{equation}
By coupling 
with the backward equation
in the statement of Lemma 
\ref{lem:2.1} (when reformulated under the historical probability), 
we obtain the forward-backward system:
\begin{equation}
	\label{eq:system:n}
	\begin{split}
		&\ud \overline m_{t}^{n+1} = - \Bigl( \eta_{t} \overline m_{t}^{n+1} +  \tfrac{(1-\varpi^{-1}) \varpi^{-n}}{1- \varpi^{-(n+1)}}  h_t^{n+1} \Bigr) \ud t + \varepsilon \ud W_{t},
		\\
		&\ud h_{t}^{n+1} = \Bigl( - \tfrac1{\varpi} Q^\dagger f\bigl(\overline m_{t}^n\bigr) + \eta_{t} h_{t}^{n+1} \Bigr) \ud t +  {\varpi} k_{t}^{n+1} h^n_{t} \ud t +  {\varepsilon} k_{t}^{n+1} \ud W_{t}, \quad t \in [0,T],
		\\
		&h_{T}^{n+1} = \tfrac1{\varpi} R^\dagger g \bigl( \overline m^n_{T} \bigr). 
	\end{split}
\end{equation}
Our main result in this regard is the following statement: 
\begin{thm}%{\textrm{\rm [Convergence of the tilted fictitious play.]}}
\label{main:thm:1}
	The scheme \eqref{eq:system:n} converges 
	to the decoupled version of the FBSDE system: 
	\begin{equation}
		\label{eq:decoupled:limit}
		\begin{split}
			&\ud  m_{t} = -  \eta_{t}  m_{t} \ud t + \varepsilon \ud W_{t},
			\\
			&\ud h_{t}^{\varpi}  =  \Bigl( - \tfrac1{\varpi} Q^\dagger f(m_{t}) + \eta_{t} h_{t}^{\varpi} + {\varpi} k_{t}^{\varpi}  h_{t}^{\varpi}  \Bigr) \ud t + \varepsilon k_{t}^\varpi \ud W_{t}, \quad t \in [0,T],
			\\
			&m_0 = {\mathbb E}(X_0), \quad h_{T}=\tfrac1{\varpi}  R^\dagger g(m_{T}),
		\end{split}
	\end{equation}
	with an explicit bound on the rate of convergence, namely
	\begin{equation}
		\label{eq:main:thm:1:bound:infinity}
	\textrm{\rm essup}_{\omega \in \Omega} \Bigl[ \sup_{0 \leq t \leq T} \Bigl( \vert m_{t} - \overline m_{t}^n \vert^2 + \vert h_{t}^\varpi- h_{t}^n \vert^2 \Bigr) \Bigr] \leq    \varpi^{-2n} \exp\bigl(\tfrac{C}{ \varepsilon^{2}}\bigr),
	\end{equation}
for a constant $C$ that depends on $d$,  $T$ and the norms
$\|Q^\dagger f\|_{1,\infty}$ and $\|R^\dagger g \|_{1,\infty}$. 

Moreover, up to a modification of the constant $C$, the weak error of the scheme for the
	Fortet-Mourier distance satisfies: 
	\begin{equation}
	\label{eq:main:thm:1:bound}
\begin{split}
	  &\sup_{F} \Bigl\vert {\mathbb E}^{\varpi {\boldsymbol h}^n \hspace{-2pt}/\varepsilon}
	  \Bigl[ F\bigl(\overline{\boldsymbol m}^n,{\boldsymbol h}^n\bigr) \Bigr] - 
	  {\mathbb E}^{\varpi {\boldsymbol h}^\varpi \hspace{-2pt} /\varepsilon}
	  \Bigl[ F\bigl({\boldsymbol m},{\boldsymbol h}\bigr) \Bigr]
\Bigr\vert 
 \leq   \varpi^{- n} \exp\bigl(\tfrac{C}{\varepsilon^{2}}\bigr),	
\end{split}
\end{equation}
the supremum in the left-hand side being taken over all the functions 
$F$ on ${\mathcal C}([0,T];{\mathbb R}^d \times {\mathbb R}^d)$ that are 
bounded by $1$ and $1$-Lipschitz continuous. 
\end{thm}

%Consider now 
%\begin{equation}
%		\label{eq:decoupled:limit}
%		\begin{split}
%			&\ud  m_{t} = -  \eta_{t}  m_{t} \ud t + \varepsilon \ud W_{t},\\
%			&\ud h_{t}   =  \bigl( - Q^\dagger f(m_{t}) + \eta_{t} h_{t} +  h_{t}  k_{t}  \bigr) \ud t + \varepsilon k_{t} \ud W_{t}, \quad t \in [0,T], 
%			\\
%			&h_{T}=R^\dagger g(m_{T}),
%		\end{split}
%	\end{equation}
%We claim 
%\begin{equation*} 
%{\mathbb E}^{\boldsymbol h} \Bigl[ \sup_{0 \leq t \leq T} \vert h_t - h_t^\varpi \vert^2 
%\Bigr] \leq 
%\end{equation*} 

%\begin{rem}
%\label{re:BSDE:h}
Importantly (and this is the interest of the result),
the 
solution $(m_t,\varpi h_t^{\varpi})_{0 \le t \le T}$ of the system \eqref{eq:decoupled:limit} should be regarded as a solution of the (original) MFG with common noise
\eqref{eq:state:intro}--\eqref{eq:cost:intro}--\eqref{eq:intro:fixed}, whenever the latter is formulated in the weak form. 
Indeed, 
for ${\boldsymbol m}=(m_{t})_{0 \leq t \leq T}$ 
and
${\boldsymbol h}:= \varpi {\boldsymbol h}^\varpi=(\varpi h^\varpi_{t})_{0 \leq t \leq T}$
as in 
\eqref{eq:decoupled:limit}, 
the optimal path $(X_{t})_{0 \leq t \leq T}$ associated with the minimization problem
\begin{equation}
\label{eq:minimization:problem:correc}
\textrm{\rm inf}_{\boldsymbol \alpha}
{\mathbb E}^{{\boldsymbol h}/\varepsilon}
	\Bigl[ {\mathcal R}^{X_0} \Bigl( {\boldsymbol \alpha}; {\boldsymbol m} ; \textcolor{black}{\varepsilon} {\boldsymbol W}^{{\boldsymbol h}/\varepsilon}\Bigr) \Bigr]
\end{equation}
has exactly ${\boldsymbol m}=(m_{t})_{0 \leq t \leq T}$ as conditional expectation given the common noise, which 
follows from an obvious adaptation of Lemma 
\ref{lem:2.1}. Namely,
\begin{equation*}
m_{t} = {\mathbb E}\bigl[ X_{t} \, \vert \, \sigma({\boldsymbol W}) \bigr], \quad t \in [0,T],
\end{equation*}
where $(X_t)_{0 \le t \le T}$ is the optimal trajectory to the latter cost functional. 
As before, 
this can be rewritten as 
\begin{equation*}
m_{t} = {\mathbb E}^{\boldsymbol h /\varepsilon} \bigl[ X_{t} \, \vert \, \sigma({\boldsymbol W})\bigr], \quad t \in [0,T].
\end{equation*}
Indeed, under ${\mathbb P}^{{\boldsymbol h}/\varepsilon}$, 
the process $({\boldsymbol m},{\boldsymbol h})$ satisfies the following forward-backward system, which characterizes the 
conditional expectation of the
optimal trajectory to \eqref{eq:minimization:problem:correc}:
	\begin{equation}
	\label{eq:bsde:true:underPh}
		\begin{split}
			&\ud  m_{t} = -  \eta_{t}  m_{t} \ud t - h_{t} \ud t + \varepsilon \ud W_{t}^{{\boldsymbol h}/\varepsilon},\\
			&\ud h_{t}  =  \bigl( - Q^\dagger f(m_{t}) + \eta_{t} h_{t} \bigr) \ud t + \varepsilon k_{t} \ud W_{t}^{{\boldsymbol h}/\varepsilon}, \quad t \in [0,T], 
			\\
			&m_0 = {\mathbb E}(X_0), \quad h_{T}=R^\dagger g(m_{T}).
		\end{split}
	\end{equation}
	To derive the above system, it suffices to write it 
	first
under ${\mathbb P}$:
	\begin{equation}
	\label{eq:bsde:true:underP:}
			\begin{split}
			&\ud  m_{t} = -  \eta_{t}  m_{t} \ud t  + \varepsilon \ud W_{t},
			\\
			&\ud h_{t}  =  \bigl( - Q^\dagger f(m_{t}) + \eta_{t} h_{t}  +  k_{t} h_t
			\bigr) \ud t + \varepsilon k_t 		\ud W_{t}, \quad t \in [0,T], 
			\\
			&m_0 = {\mathbb E}(X_0), \quad h_{T}=R^\dagger g(m_{T}).
		\end{split}
	\end{equation}	
	By the changes of variable
	(with the first one being already 
	defined in 
	the paragraph preceding 
	\eqref{eq:minimization:problem:correc}) 
	\begin{equation}
	\label{eq:change:of:variable:varpi} 
	h_t^{\varpi} = \tfrac1{\varpi} h_t, \quad 
	k_t^{\varpi} := \tfrac1{\varpi} k_t, \quad t \in [0,T], 
	\end{equation} 
we indeed recover 
\eqref{eq:decoupled:limit}. As we claimed, this identifies the pair 
$({\boldsymbol m},{\boldsymbol h}=\varpi {\boldsymbol h}^\varpi)$ as an equilibrium of the original MFG 
\eqref{eq:state:intro}--\eqref{eq:cost:intro}--\eqref{eq:intro:fixed}. 
%\end{rem} 
	
%\begin{rem}	
Noticeably, the two systems 
	\eqref{eq:bsde:true:underPh}
	and 
		\eqref{eq:bsde:true:underP:}
provide two distinct representations of the solution $\theta_\varepsilon$ to the PDE 
\eqref{eq:intro:master:1}. 
The connection between the two is given by the identities
\begin{equation}
\label{eq:representation:h:k}
h_{t} = \theta_\varepsilon(t,m_{t}), 
\  t \in [0,T] \, ; 
\quad
k_{t} =  \nabla_{x} \theta_\varepsilon(t,m_{t}), \  t \in [0,T),
\end{equation}
which holds true 
under both  ${\mathbb P}$ and ${\mathbb P}^{{\boldsymbol h} /\varepsilon}$.
By a standard application of the maximum principle (for PDEs), 
$\theta_\varepsilon$ is bounded in terms of $d$, $T$, $\|Q^\dagger f \|_{1,\infty}$
and $\|R^\dagger g \|_{1,\infty}$, and, by Lemma 
\ref{lem:2}
right below, 
$\nabla_{x} \theta_\varepsilon$ is also bounded in terms of $d$, $\varepsilon$, $T$, 
$\|Q^{\dagger} f \|_{1,\infty}$
and $\|R^{\dagger} g \|_{1,\infty}$. 
The relationships stated in \eqref{eq:representation:h:k}
show in fact how to construct easily a solution to 
	\eqref{eq:bsde:true:underPh}
	and 
		\eqref{eq:bsde:true:underP:}
by solving first for $(m_{t})_{0\leq t \leq T}$ in the forward equation and then by expanding 
$(\theta_\varepsilon(t,m_{t}))_{0 \leq t \leq T}$. 
The hence constructed solution $(h_{t},k_{t})_{0\leq t \leq T}$ to the backward equation in 
\eqref{eq:bsde:true:underP:} (equivalently \eqref{eq:bsde:true:underPh}) is bounded. In turn, uniqueness to the backward equation in 
\eqref{eq:bsde:true:underP:} (equivalently \eqref{eq:bsde:true:underPh}) is easily shown to hold true in the class of bounded processes 
$(h_{t},k_{t})_{0\leq t \leq T}$. 
Similarly, 
		\eqref{eq:decoupled:limit}
has a unique solution, which is then obtained by the change of variable 
\eqref{eq:change:of:variable:varpi}.
%\end{rem}

\begin{rem} 
\label{rem:2.5:correc}
The reader will observe that, in the equation 
\eqref{eq:X:n:n+1} for the optimal trajectory in the fictitious play, the intensity of the 
idiosyncratic noise is $\sigma/\varpi$. This is different from 
the intensity of the idiosyncratic noise underpinning the controlled trajectories in the MFG \eqref{eq:state:intro}--\eqref{eq:cost:intro}--\eqref{eq:intro:fixed}, 
which is $\sigma$. In particular, 
the sequence of processes 
$({\boldsymbol X}^{n,\varpi})_{n \geq 1}$, 
defined in 
Lemma 
\ref{lem:2.1} (see also \eqref{eq:X:n:n+1}), cannot be `a good approximation' of 
the process 
${\boldsymbol X}$ that minimizes 
\eqref{eq:minimization:problem:correc}. 
Although this looks paradoxal with the result stated in Theorem 
\ref{main:thm:1}, it must be clear that, implicitly, 
the two bounds \eqref{eq:main:thm:1:bound:infinity}
and	\eqref{eq:main:thm:1:bound}
 rely on the fact that 
the dynamics for the MFG equilibrium ${\boldsymbol m}$ does not depend on 
$\sigma$. 
\end{rem}

\begin{rem}
\label{rem:fictitious:play:rate}
Our construction of a fictitious play
with a geometric learning rate, proportional to $\varpi^{-n}$ 
(see \eqref{eq:def:overline:m:n}), as opposed 
to the harmonic learning rate $1/(n+1)$ 
that is used in the standard version of the fictitious play, 
could be easily extended to 
other (neither harmonic nor geometric) rates. 
There are two key principles that should be followed: 
the first one is to 
make appear, in the dynamics 
\eqref{eq:mmm:n:n+1}, 
a difference 
between the 
component ${\boldsymbol h}^{n+1}$ of the
scheme at iteration $n+1$ and the
component ${\boldsymbol h}^n$ 
of the 
 scheme 
at iteration $n$, 
and the second one is to define the tilted measure depending on the form 
of the finite difference (which is exactly the case in \eqref{eq:opti:n+1:n}).
What is remarkable in this approach is that the tilted measure has a bias. Whereas it would be natural 
to take the tilted measure as ${\mathbb P}^{{\boldsymbol h}^n\hspace{-2pt}/\varepsilon}$, it is here taken as 
${\mathbb P}^{\varpi {\boldsymbol h}^n\hspace{-2pt}/\varepsilon}$ with $\varpi \not =1$. The reader will easily see that, in order to 
recover the `non-biased' setting (which formally corresponds to $\varpi=1$), one needs to replace 
$\varpi h^n_t$ by $h^n_t$ in 
\eqref{eq:X:n:n+1}
and
\eqref{eq:mmm:n:n+1}. 
By summing over $n$ as in 
\eqref{eq:m:n:n+1}, we would then 
obtain 
$1/(n+1)$ as learning rate in 
\eqref{eq:def:overline:m:n}. In other words, the `non-biased' regime corresponds to 
the harmonic fictitious play. 

Although 
Theorem 
\ref{main:thm:1} does not cover the case $\varpi=1$, the interested will easily check that the proof that is given below also works 
when $\varpi=1$. Then, the bound 
in 
\eqref{eq:main:thm:1:bound:infinity}
must be replaced by 
$\exp\bigl(C \varepsilon^{-2}\bigr)/n^2$. Similarly, 
\eqref{eq:main:thm:1:bound}
must be replaced by 
$\exp\bigl(C \varepsilon^{-2}\bigr)/n$. 
In fact, the same remark applies to the forthcoming Theorems 
\ref{main:thm:2}
and
\ref{thl:main:regret}.
\end{rem}

\begin{rem}
\label{rem:back:resolvent:correc}
We notice for later purposes that 
the backward equation 
in \eqref{eq:bsde:true:underPh}
can be `solved' explicitly. 
Indeed, calling $(P_{t})_{0 \leq t \leq T}$ the solution of the linear differential equation 
$\dot{P}_{t} = -  P_{t} \eta_{t}$, for $t \in [0,T]$ with $P_{0} = I_{d}$ as initial condition,  
where $I_{d}$ is the $d$-dimensional identity matrix, it holds
\begin{equation*}
\ud \bigl[P_{t} h_{t} \bigr]
=
- P_{t}Q^{\dagger} f\bigl( m_{t}\bigr)  \ud t + \varepsilon
P_{t} k_{t} \ud W_{t}^{
{\boldsymbol h} /\varepsilon },
\end{equation*}
or equivalently,
\begin{equation*}
\begin{split}
&P_{t}h_{t}
=  P_{T} R^{\dagger} g\bigl( {m}_{T} \bigr) 
  + 
\int_{t}^T 
P_{s}
Q^{\dagger} f\bigl( m_{s} \bigr)  \ud s - 
\varepsilon \int_{t}^T  
P_{s} k_{s}  \ud W_{s}^{{\boldsymbol h}/\varepsilon}, 
\quad t \in [0,T].
\end{split}
\end{equation*}
\end{rem}

\color{black}

%\begin{rem}
%\label{rem:2}
\subsubsection{Discussion about the rate of convergence}
\label{subsubse:rate}
Beside any specific application to learning, this is another of our contributions to provide an explicit bound for the rate of convergence of our variant of the fictitious play for mean field games with common noise. 
We feel worth to point out that, to the best of our knowledge, there are very few results on the rate of convergence for the fictitious play in the absence of common noise, whether the mean field game be potential or monotone (as we already explained, no result is available without common noise outside the potential or monotone cases, except \cite{perrin2020fictitious} which is for a time-continuous version of the fictitious play). 
In most of the existing references, the analysis indeed involves an additional compactness argument
which complicates the computation of the rate. 
Still, the reader can find in \cite{Perrin-3} an explicit rate for the exploitability for a monotone and potential game set in discrete time; the bound is of order $O(1/\sqrt{n})$ (hence weaker than ours). In \cite{perrin2020fictitious}, a bound is shown, also for the exploitability, but for the time-continuous version of the fictitious play, when the game is monotone; it is of order $O(1/n)$, and is 
also weaker than ours in the geometric setting but is hence comparable to ours in the harmonic framework. 

In comparison, the thrust of our analysis is to provide a scheme with a geometric decay. 
Obviously, 
this must be tempered, due to the presence of the multiplicative constant $\exp(C \varepsilon^{-2})$. 
When the intensity $\varepsilon$ is away from zero, the effective decay is really good, but when $\varepsilon$ gets close to $0$ (which is the typical regime when we use the common  noise as an exploration noise), the exponential factor really matters. Of course the bound for the error may be rewritten as follows:
\begin{equation*} 
 \exp \bigl( - n \ln(\varpi) + \tfrac{C}{ \varepsilon^{2}} \bigr), 
\end{equation*} 
which says that 
$n$ should be chosen on a  scale larger than $\varepsilon^{-2}$. We think that 
this is numerically affordable. In comparison (see Remark \ref{rem:fictitious:play:rate}), if one had to work with the standard fictitious play
(i.e., with a harmonic learning rate), the same analysis would lead to a bound of 
order $\exp(C \varepsilon^{-2})/n$, which is obviously much worse. In the first arXiv version \cite{DelarueVasileiadis} of this work, dedicated to the analysis of the sole harmonic regime, we claimed that 
this was possible to remove the exponential factor, but there was a clear mistake in the computations: for convenience reasons, we decided not to indicate the dependence 
upon $\varepsilon$ in the various tilted measures, as a consequence of which we forgot a factor $1/\varepsilon$ in some of the main estimates. 

The presence of the exponential factor comes from the following lemma, whose proof is deferred to the end of the subsection: 

\begin{lem}
\label{lem:2}
There exists a constant $C_{1}$, only depending on $d$, $T$, $\|Q^\dagger f \|_{1,\infty}$
and $\|R^\dagger g \|_{1,\infty}$, such that
\begin{equation*}
\vert \nabla_{x} \theta_\varepsilon(t,x) \vert \leq \frac{C_{1}}{\varepsilon^2}, \quad t \in [0,T), \ x \in {\mathbb R}^d,
\end{equation*}
with $\theta_\varepsilon$ the solution of the PDE \eqref{eq:intro:master:1}.
\end{lem}

The above $L^\infty$ bound for the gradient is known 
to be sharp, but it looks rather poor because the estimate is precisely given in 
$L^\infty$. Also, one may hope for better estimates in different norms. In other words, $\theta_\varepsilon$ may indeed become 
very steep, but maybe only on some localized parts of the space. 
This is the point where things become highly subtle
because 
the process $\overline{\boldsymbol m}^n$ becomes localized itself as the diffusion coefficient 
$\varepsilon$ tends to $0$. So, the challenging question is to decide whether it may stay or not in parts of the space where the gradient is high. 
As exemplified in the 
analysis performed in 
Delarue and Foguen \cite{delfog2019}, this may be a challenging question, even in dimension 1. 
Unless we make additional assumptions on the model
(assuming for instance that $\theta_\varepsilon$ is smooth independently of $\varepsilon$, which is for example the case when 
$T$ is small enough), we must confess that we are not able to provide a more relevant bound for the 
gradient of $\theta_\varepsilon$ (which bound gives in the end a bound for the process ${\boldsymbol k}$ in 
\eqref{eq:bsde:true:underPh}, see 
\eqref{eq:representation:h:k}). 

It must be also stressed that 
	\eqref{eq:main:thm:1:bound}
 provides a bound for the so-called \textit{weak} error
 of the scheme 
and is fully relevant from the practical point of view.  
In our context, the \textit{strong} error, as addressed in 
	\eqref{eq:main:thm:1:bound:infinity}, does not provide the same information. Indeed, because 
the two densities 
${\mathcal E}(\tfrac{1}{\varepsilon} {\boldsymbol h})$ and ${\mathcal E}(\tfrac{\varpi}{\varepsilon} {\boldsymbol h}^n)$
become singular when $\varepsilon$ tends to $0$, the passage from the strong to the weak error is not direct. 
\color{black}

\begin{rem}
\label{rem:thm:1}
The reader may wonder about the scope
of Theorem \ref{main:thm:1} in the higher dimensional framework. 
This question will be addressed from a purely numerical prospect
in the forthcoming Section 
\ref{Se:3}, see in particular Subsection 
\ref{subse:3:2:bis}.
From a more theoretical point of view, the same question can be 
addressed by investigating the dependence of the constant $C$
in \eqref{eq:main:thm:1:bound} upon the dimension $d$. 
Whereas we cannot provide a sharp (or at least reasonable) bound in full generality, we show below that, even in simple cases when 
the matrices $Q$ and $R$ reduce to the $d \times d$ identity matrix
and the coefficients $f$ and $g$ are diagonal, i.e. 
each coordinate $i$ of $f$ (respectively $g$) writes 
$f^i(x) = f_0(x_i)$ 
(respectively  
$g^i(x) = g_0(x_i)$) for a function 
$f_0 : {\mathbb R} \rightarrow {\mathbb R}$ 
(respectively 
$g_0 : {\mathbb R} \rightarrow {\mathbb R}$), 
with $x_i$ denoting the $i^{\rm th}$ coordinate of $x$, then the 
exponential factor in 
		\eqref{eq:main:thm:1:bound:infinity}
		and 	\eqref{eq:main:thm:1:bound}
is typically  
$\exp(C {\mathcal O}(\sqrt{d}) \varepsilon^{-2})$, for $C$ independent of $d$. 
\end{rem}
\color{blue}

\color{black}
\color{black}

\subsubsection{Proof of Theorem
\ref{main:thm:1}}
The proof relies on the following lemma, whose proof is also deferred to the end of the subsection.

\begin{lem}
\label{lem:1}
There exists a constant $C_{2}>0$, only depending on $d$, $T$, $\|Q^\dagger f \|_{1,\infty}$
and $\|R^\dagger g \|_{1,\infty}$, such that, ${\mathbb P}$ almost surely, 
\begin{equation*}
\vert h_{t} \vert \leq C_{2} \, ; \quad \vert h_{t}^n \vert \leq C_{2}, \quad t \in [0,T], \ n \geq 1. 
\end{equation*}
%\color{blue}
%Moreover, there exists $\delta \in (0,1]$, depending only on $d$, $T$, $\|Q^\dagger f \|_{1,\infty}$
%and $\|R^\dagger g \|_{1,\infty}$, such that 
%\begin{equation*}
%\begin{split}
%{\mathbb E} \biggl[ 
%\exp \biggl( \delta \int_0^T \vert k_s \vert^2 \ud s 
% \biggr) \biggr]
%&\leq  
%{\mathbb E} 
%\Bigl[ 
% \exp \Bigl( \delta C_1 \bigl[ 1 \wedge \sup_{0 \leq s \leq T} \vert m_s - n_s^\emptyset \vert^2
%\bigr]\Bigr) 
%\Bigr].
%\end{split}
%\end{equation*}
\end{lem}

\color{black}
%\begin{rem}
%\label{rem:prop:1}
%Similar to the rate of growth of the constant $C$ with respect to the dimension $d$ in the statement of Theorem 
%\ref{main:thm:1}, the rate of growth of the constant $C$ 
%with respect to the dimension $d$ 
%in the statement of Proposition \ref{prop:1} cannot be better than $O(\exp(O(d)))$  in simple cases when 
%the matrices $Q$ and $R$ are the identity matrices.
%This remark is in fact one key step in the proof of Remark \ref{rem:thm:1}.  
%\end{rem}
%\color{black}

\begin{proof}[Proof of Theorem \ref{main:thm:1}.]
Throughout, $C$ is a generic constant that is allowed to vary from line to line, 
as long as it only depends on $d$, $T$, $\|Q^\dagger f \|_{1,\infty}$
and $\|R^\dagger g \|_{1,\infty}$. 
\vspace{4pt} 

\textit{First Step.} 
Invoking Lemma \ref{lem:1}
{and recalling that $\varpi \in (1,\sqrt{2}]$}, we then have
\begin{equation}
\label{eq:bound:1/n}
\Bigl\vert
{\tfrac{(1-\varpi^{-1})\varpi^{-n}}{1- \varpi^{-(n+1)}}  h_t^{n+1}}
  \Bigr\vert
 \leq C \varpi^{-n},
\quad t \in [0,T], 
\end{equation}
from which we deduce (consider the difference between  
\eqref{eq:m:n:n+1} 
and 
the forward equation in		\eqref{eq:decoupled:limit})
that 
\begin{equation}
\label{eq:mn-m}
\sup_{0 \le t \leq T} \bigl\vert \overline m^n_{t} -  m_{t}
\bigr\vert \leq  C \varpi^{-n}.
\end{equation}
We now make the difference between 
the backward equations in
\eqref{eq:system:n} and 
\eqref{eq:decoupled:limit}. We obtain
\begin{align}
&\ud \bigl( h_{t}^{n+1} - h_{t}^\varpi \bigr) 
= - \tfrac1{\varpi} \Bigl(Q^\dagger f\bigl( \overline m_{t}^n)
-Q^\dagger f\bigl(  m_{t})
\Bigr) \ud t + \eta_{t} \bigl( h_{t}^{n+1} - h_{t}^\varpi \bigr) \ud t 
+ \varpi  k_{t}^{n+1}  \bigl( h_{t}^n - h_{t}^\varpi \bigr) \ud t
\nonumber
\\
&\hspace{75pt} + \varpi  \bigl( k_{t}^{n+1} - k_{t}^\varpi \bigr) h_t^\varpi \ud t + \varepsilon \bigl( k_{t}^{n+1} -  k_{t}^\varpi \bigr) \ud W_{t},
\label{eq:revision:new:BSDE}
\\
&h_{T}^{n+1} - h_{T}^\varpi =  \tfrac1{\varpi} \Bigl[ 
R^\dagger g \bigl( \overline m_{T}^n \bigr) - R^\dagger g \bigl(  m_{T} \bigr) \Bigr]. \nonumber
\end{align}
In particular, rewriting the above equation under  {${\boldsymbol W}^{ {\boldsymbol h}/  \varepsilon}=
{\boldsymbol W}^{ \varpi {\boldsymbol h}^{\varpi} \hspace{-2pt} /  \varepsilon}$},  we get
\begin{align}
\ud \bigl( h_{t}^{n+1} - h_{t}^\varpi \bigr) 
&= - \tfrac1{\varpi} \Bigl(Q^{\dagger} f\bigl( \overline m_{t}^n)
-Q^{\dagger} f\bigl(  m_{t})
\Bigr) \ud t + \eta_{t} \bigl( h_{t}^{n+1} - h_{t}^\varpi \bigr) \ud t 
+    \varpi k_{t}^{n+1} \bigl( h_{t}^{n} - h_{t}^\varpi \bigr) \ud t
\nonumber
\\
&\hspace{15pt} %+  \textcolor{blue}{\tfrac{1}{\varepsilon}} \bigl(k_{t}^{n+1} - k_{t} \bigr) \bigl( h_{t}^{n} - h_{t} \bigr) \ud t
+ \varepsilon \bigl( k_{t}^{n+1} - k_{t}^\varpi \bigr) \ud W_{t}^{ {{\boldsymbol h}}/  \varepsilon}, 
\quad t \in [0,T]. 
\label{eq:new:hn+1:h}
\end{align}
Then, taking the square, using 
\eqref{eq:mn-m} and Lemmas 
\ref{lem:2} and 
\ref{lem:1}, expanding by It\^o's formula and applying Young's inequality, we get
%, \textcolor{black}{for $\delta$ as in the statement of Lemma \ref{lem:1}},  
\begin{equation}
\label{eq:proof:delta}
\begin{split}
\ud \bigl\vert h_{t}^{n+1} - h_{t}^\varpi \bigr\vert^2 &\geq - C \varpi^{-2n} \ud t  -  C  \bigl\vert h_{t}^{n+1} - h_{t}^\varpi \bigr\vert^2 \ud t 
- 
C \bigl\vert k_t \bigr\vert \bigl\vert h_{t}^{n} - h_{t}^\varpi \bigr\vert
 \bigl\vert h_{t}^{n+1} - h_{t}^\varpi \bigr\vert
 \ud t 
 \\
 &\hspace{15pt} 
 - 
C \bigl\vert k_t^{n+1} - k_t^\varpi \bigr\vert \bigl\vert h_{t}^{n} - h_{t}^\varpi \bigr\vert
 \bigl\vert h_{t}^{n+1} - h_{t}^\varpi \bigr\vert
 \ud t 
\\
&\hspace{15pt}
+\varepsilon^2 \bigl\vert k_{t}^{n+1} - k_{t}^\varpi \bigr\vert^2 \ud t 
+ 2 \varepsilon \bigl( h_{t}^{n+1} - h_{t}^\varpi \bigr)  \cdot \bigl[ \bigl(  k_{t}^{n+1} - k_{t}^\varpi \bigr) \ud W_{t}^{{ {\boldsymbol h}}  /\varepsilon} \bigr] 
\\
&\geq - C \varpi^{-2n} \ud t  -
\tfrac{C}{\varepsilon^2} \bigl\vert h_{t}^{n} - h_{t}^\varpi \bigr\vert^2 \ud t 
-\tfrac{C}{\varepsilon^2}  \bigl\vert h_{t}^{n+1} - h_{t}^\varpi \bigr\vert^2 \ud t 
\\
&\hspace{15pt}
+ 2 \varepsilon \bigl( h_{t}^{n+1} - h_{t}^\varpi \bigr)  \cdot \bigl[ \bigl(  k_{t}^{n+1} - k_{t}^\varpi \bigr) \ud W_{t}^{{ {\boldsymbol h}}  /\varepsilon} \bigr],
\end{split}
\end{equation}
for $t \in [0,T]$
and for a constant $C$ as in the statement. 
%\textcolor{red}{Older proof}
%In the sequel, we use the key fact that the process $(k_{t})_{0 \le t \le T}$
%is bounded by $C/\varepsilon$: this is exactly the point where we use Lemma \ref{lem:2}. 
%Taking the square and expanding by It\^o's formula, we get
%\begin{equation*}
%\begin{split}
%&\ud \bigl( h_{t}^{n+1} - h_{t} \bigr)^2 \geq - \frac1{n^2} dt  -
%c \bigl\vert h_{t}^{n+1} - h_{t} \bigr\vert^2 dt 
%- c \bigl\vert h_{t}^{n} - h_{t} \bigr\vert^2 dt 
%\\
%&\hspace{80pt}
%+ \frac12 \bigl\vert k_{t}^{n+1} - k_{t} \bigr\vert^2 dt 
%+ 2\bigl( h_{t}^{n+1} - h_{t} \bigr)  \bigl( k_{t} - k_{t}^{n+1} \bigr) dw_{t}.
%\end{split}
%\end{equation*}
%Now,
%we let $K_{t} = \tfrac1{\varepsilon} \int_{0}^t \vert  k_{s} \vert \ud s$, $t \in [0,T]$.
 For a free parameter $\lambda >1$, we obtain 
\begin{align}
\ud \Bigl[ \exp\bigl( \tfrac{C}{\varepsilon^2} \lambda t \bigr) \bigl\vert h_{t}^{n+1} - h_{t}^\varpi \bigr\vert^2\Bigr] &\geq  
\exp\bigl(  \tfrac{C}{\varepsilon^2} \lambda t  \bigr) \Bigl(
\tfrac{C (\lambda-1)}{\varepsilon^2}   \bigl\vert h_{t}^{n+1} - h_{t}^\varpi \bigr\vert^2 
- C \varpi^{-2n} -
 \tfrac{C}{\varepsilon^2}   \bigl\vert h_{t}^{n} - h_{t}^\varpi \bigr\vert^2   \Bigr) \ud t
 \nonumber
 \\
%&\hspace{-15pt} + \bigl( \lambda - 1  \bigr) 
%\textcolor{blue}{\tfrac1{\varepsilon} \vert k_t \vert } 
%\exp\bigl(C t + \lambda K_{t}\bigr)  \bigl\vert h_{t}^{n+1} - h_{t} \bigr\vert^2 \ud t
%\\
%&\hspace{-15pt} +  
%\exp\bigl(C t + \lambda K_{t}\bigr)  \bigl\vert k_{t}^{n+1} - k_{t} \bigr\vert^2 \ud t
%\\
&\hspace{15pt}
+ 2 \varepsilon \exp\bigl( \tfrac{C}{\varepsilon^2} \lambda t   \bigr) \bigl( h_{t}^{n+1} - h_{t}^\varpi \bigr)\cdot \bigl[  \bigl(  k_{t}^{n+1} - k_{t}^\varpi \bigr) \ud W_{t}^{
 {\boldsymbol h}   /\varepsilon} \bigr].
\label{eq:bound:main:proof:thm:2:6}
\end{align}
And then, integrating from $t$ to $T$, taking conditional expectation under ${\mathbb P}^{{\boldsymbol h} /\varepsilon}$ given ${\mathcal F}_t$
and recalling that 
$h_T^{n+1} = R^\dagger g ( \overline m^n_{T} )/\varpi$
and
$h_T^\varpi  = R^\dagger g (   m_{T} )/\varpi$,
we get, for any $t \in [0,T]$, 
\begin{align}
\label{eq:backward:bound}
&\exp\bigl( \tfrac{C}{\varepsilon^2} \lambda t  \bigr) \bigl\vert h_{t}^{n+1} - h_{t}^\varpi \bigr\vert^2  + 
\tfrac{C ( \lambda - 1) }{\varepsilon^2} 
 {\mathbb E}^{{\boldsymbol h}  /\varepsilon}
\biggl[ \int_{t}^T
\exp\bigl(  \tfrac{C}{\varepsilon^2} \lambda s \bigr) 
  \bigl\vert h_{s}^{n+1} - h_{s}^\varpi \bigr\vert^2 \ud s \, \vert \, {\mathcal F}_t \biggr]
\\
&\hspace{15pt}
  \leq 
C \varpi^{-2n}    \Bigl( \exp \bigl(  \tfrac{C}{\varepsilon^2} \lambda T  \bigr)  + 
\int_t^T 
\exp \bigl(  \tfrac{C}{\varepsilon^2} \lambda s  \bigr) \ud s \Bigr)
+
\tfrac{C}{\varepsilon^2} 
 {\mathbb E}^{ {\boldsymbol h}  /\varepsilon}
\biggl[ \int_{t}^T
\exp\bigl( 
\tfrac{C}{\varepsilon^2} \lambda s 
\bigr)  \bigl\vert h_{s}^{n} - h_{s}^\varpi \bigr\vert^2  \ud s \, \vert \, {\mathcal F}_t \biggr].
\nonumber
\end{align}
%Recall that $\delta$ is given by Lemma \ref{lem:1}. 
%We deduce that 
%\begin{equation*}
%{\mathbb E}^{\boldsymbol h} \Bigl[ \exp \bigl( \lambda s +K_{s} \bigr) \Bigr] 
%\leq C_{1} \exp \bigl( \lambda T \bigr), \quad s \in [0,T]. 
%\end{equation*}
Next, we choose $\lambda=4$. This yields 
\begin{equation}
\label{eq:backward:bound:000}
\begin{split}
&
\tfrac{\varepsilon^2}{3C} 
\exp\bigl( \tfrac{4C}{\varepsilon^2} t  \bigr) \bigl\vert h_{t}^{n+1} - h_{t}^\varpi \bigr\vert^2 
+
{\mathbb E}^{ {\boldsymbol h} /\varepsilon}
\biggl[
\int_{t}^T
\exp\bigl(
\tfrac{4C}{\varepsilon^2}  s 
\bigr) 
  \bigl\vert h_{s}^{n+1} - h_{s}^\varpi \bigr\vert^2 \ud s \, \vert \, {\mathcal F}_t \biggr]
  \\
&\hspace{15pt}  \leq 
\tfrac{ \varepsilon^2(1+T)}3 \varpi^{-2n}     \exp \bigl(  \tfrac{4C}{\varepsilon^2}  T  \bigr) + 
\tfrac13 {\mathbb E}^{ {\boldsymbol h}  /\varepsilon}
\biggl[
\int_{t}^T
\exp\bigl(\tfrac{4C}{\varepsilon^2} s  \bigr) 
  \bigl\vert h_{s}^{n} - h_{s}^\varpi \bigr\vert^2 \ud s \, \vert \, {\mathcal F}_t \biggr].
\end{split}
\end{equation}
The above holds true for almost every $\omega \in \Omega$ (under ${\mathbb P}$ and 
${\mathbb P}^{ {\boldsymbol h}  /\varepsilon}$) 
and for every $t \in [0,T]$. By a standard induction argument, 
using in addition Lemma 
\ref{lem:1}, we 
deduce that, for a deterministic constant $C'$ depending on the same parameters as $C$, 
for almost every $\omega \in \Omega$, for every $t \in [0,T]$
and for every $n \geq 1$, 
\begin{equation}
\label{eq:conclusion:proof:main:thm:1}
\begin{split}
{\mathbb E}^{  {\boldsymbol h} / \varepsilon}
\biggl[
\int_{t}^T
\exp\bigl(
\tfrac{4C}{\varepsilon^2}  s 
\bigr) 
  \bigl\vert h_{s}^{n} - h_{s}^\varpi \bigr\vert^2 \ud s \, \vert \, {\mathcal F}_t \biggr]
&\leq C'   \exp \bigl(  \tfrac{4C}{\varepsilon^2}  T  \bigr) \sum_{k=0}^{n} 3^{-k} \varpi^{-2(n-k)}
\\
&= C'   \exp \bigl(  \tfrac{4C}{\varepsilon^2}  T  \bigr) \varpi^{-2n} \sum_{k=0}^{n} \bigl( \tfrac{\varpi^2}3 \bigr)^{k}
\\
&\leq C'  \exp \bigl(  \tfrac{4C}{\varepsilon^2}  T  \bigr)  \sum_{k=0}^{n} \bigl( \tfrac{2}3 \bigr)^{k}
\leq C'  \exp \bigl(  \tfrac{4C}{\varepsilon^2}  T  \bigr)  \varpi^{-2n},
\end{split}
\end{equation}
where we used, in the last line, the fact that $\varpi \in (1,\sqrt{2}]$. 
Because $\varepsilon$ is less than 1, we can easily remove the constant $C'$ by 
increasing the constant $C$. 
And then, 
by 
\eqref{eq:backward:bound}, we obtain 
\eqref{eq:main:thm:1:bound:infinity}.
\vspace{4pt} 

\textit{Second Step.} 
We now turn to the proof of 
\eqref{eq:main:thm:1:bound}. In fact, it 
is a direct consequence of Pinsker's inequality, which says that 
\begin{equation*}
d_{\textrm{\rm TV}}\bigl({\mathbb P}^{ {\boldsymbol h}/\varepsilon},{\mathbb P}^{\varpi {\boldsymbol h}^n \hspace{-2pt} / \varepsilon}
\bigr) \leq \sqrt{2} {\mathbb E}^{ {\boldsymbol h}/\varepsilon} \Bigl[ \ln \Bigl( \frac{{\mathcal E}( {\boldsymbol h}/\varepsilon)}{{\mathcal E}(
\varpi 
{\boldsymbol h}^{n}
\hspace{-2pt} / \varepsilon)} \Bigr) \Bigr]^{1/2}, 
\end{equation*}
where $d_{{\textrm{\rm TV}}}$ in the left-hand side is the total variation distance. 
Now,
\begin{equation*}
\begin{split}
&\ln \Bigl( \frac{{\mathcal E}( {\boldsymbol h}/\varepsilon)}{{\mathcal E}(\varpi {\boldsymbol h}^{n}
\hspace{-2pt} / \varepsilon
)} \Bigr)
= \ln \Bigl( \frac{{\mathcal E}(\varpi {\boldsymbol h}^\varpi/\varepsilon)}{{\mathcal E}(\varpi {\boldsymbol h}^{n}
\hspace{-2pt} / \varepsilon
)} \Bigr)
\\
&=  - \tfrac{\varpi}{\varepsilon} \int_{0}^T \bigl( h_{s}^\varpi - h^n_{s} \bigr) \cdot \ud W_{s}
- \tfrac{\varpi^2}{2 \varepsilon^2} \int_{0}^T \bigl( \vert h_{s}^\varpi \vert^2 - \vert h^n_{s} \vert^2 \bigr) 
\ud s
\\
&=  - \tfrac{\varpi}{\varepsilon} \int_{0}^T \bigl( h_{s}^\varpi - h^n_{s} \bigr) \cdot \ud W_{s}^{ {\boldsymbol h}/\varepsilon}
+ \tfrac{\varpi^2}{\varepsilon^2} \int_{0}^T \bigl( h_{s}^\varpi - h^n_{s} \bigr) \cdot h_{s}^\varpi \ud s
- \tfrac{\varpi^2}{2 \varepsilon^2} \int_{0}^T \bigl( \vert h_{s}^\varpi \vert^2 - \vert h^n_{s} \vert^2 \bigr) 
\ud s
\\
&= - \tfrac{\varpi}{\varepsilon} \int_{0}^T \bigl( h_{s}^\varpi - h^n_{s} \bigr) \cdot \ud W_{s}^{{\boldsymbol h}/\varepsilon} 
+ \tfrac{\varpi^2}{2\varepsilon^2} \int_{0}^T \bigl\vert h_{s}^\varpi - h^n_{s} \bigr\vert^2  \ud s.
\end{split}
\end{equation*}
And then, by \eqref{eq:main:thm:1:bound:infinity},
\begin{equation}
\label{eq:d:TV}
d_{\textrm{\rm TV}}\bigl({\mathbb P}^{  {\boldsymbol h}/\varepsilon},{\mathbb P}^{\varpi {\boldsymbol h}^n \hspace{-2pt} / \varepsilon}
\bigr) \leq \tfrac{C}{\varepsilon} \varpi^{-n} \exp\bigl(\tfrac{C}{2 \varepsilon^{2}}\bigr). 
\end{equation}
Modifying the constant $C$, we can easily get rid of the multiplicative constant $1/\varepsilon$ in the right-hand side. 
%Hence, since ${\boldsymbol h}$ and ${\boldsymbol h}^n$ are bounded, see Lemma
%\ref{lem:1}, 
%\begin{equation*}
%\begin{split}
%{\mathbb E}^{{\boldsymbol h}_n}
%\Bigl[ \sup_{0 \leq t \leq T} \vert h_{t}^n - h_{t} \vert \Bigr]
%&\leq \frac{C}{n} + C d_{\textrm{\rm TV}}\bigl({\mathbb P}^{\boldsymbol h},{\mathbb P}^{{\boldsymbol h}_n}
%\bigr)
% \leq \frac{C}{\varepsilon n}. 
%\end{split}
%\end{equation*}
%Call now 
%$\tilde {\boldsymbol m}^n$ the solution of the SDE
%\begin{equation*}
%\ud \tilde m^n_{t} = - \Bigl( \eta_{t} \tilde m^n_{t} + \theta(t,\tilde m^n_{t}) \Bigr) \ud t +
%\varepsilon \ud W_{t}^{{\boldsymbol h}^n}, \quad t \in [0,T], 
%\end{equation*}
%and write
%\begin{equation*}
%\ud m_{t} = - \Bigl( \eta_{t}  m_{t} + \theta(t,  m_{t}) \Bigr) \ud t +
%\bigl( h_{t} - h_{t}^n \bigr) \ud t +
%\varepsilon \ud W_{t}^{{\boldsymbol h}^n}, \quad t \in [0,T].  
%\end{equation*}
By \eqref{eq:main:thm:1:bound:infinity} again, for any function 
$F$ that is 1-bounded and 1-Lipschitz on the space ${\mathcal C}([0,T];{\mathbb R}^d\times {\mathbb R^d})$, 
\begin{equation*}
\begin{split}
\Bigl\vert {\mathbb E}^{{\boldsymbol h}/\varepsilon} \Bigl[ F\bigl(\overline{\boldsymbol m}^n,{\boldsymbol h}^n\bigr) 
\Bigr] -{\mathbb E}^{{\boldsymbol h}/\varepsilon} \Bigl[ F\bigl({\boldsymbol m},{\boldsymbol h}\bigr) 
\Bigr] \Bigr\vert \leq  \varpi^{-n} \exp\bigl(\tfrac{C}{ \varepsilon^{2}}\bigr),
\end{split}
\end{equation*}
and then, by \eqref{eq:d:TV},
we get 
\eqref{eq:main:thm:1:bound}. 
\end{proof}

\begin{rem}
\label{rem:varpi}
The reader can deduce from  
the display
\eqref{eq:conclusion:proof:main:thm:1}
the reason why we assumed $\varpi$ to be less than $\sqrt{2}$. In fact, even though
$\varpi$ were larger, the geometric decay would remain less 
than $3^{-n}$ because of the factor $1/3$ in 
\eqref{eq:backward:bound:000}. 
Here, the factor $1/3$ must be understood as $1/(\lambda-1)$ for $\lambda=4$. 
For sure, it would be tempting to choose $1/(\lambda-1)=1/\varpi$, but this would lead to 
$\lambda = \varpi+1$. 
The resulting exponential factor 
$\exp( C \varepsilon^{-2} \lambda s)$ in \eqref{eq:backward:bound}
would become much too high with $\varpi$.  
\end{rem}

\subsubsection{Proof of the auxiliary statements}
\label{subsubse:auxiliary:statements}

It now remains to prove Lemmas
\ref{lem:2.1},
 \ref{lem:1} and 
 \ref{lem:2} and Remark 
 \ref{rem:thm:1}.

\begin{proof}[Proof of Lemma \ref{lem:2.1}.]
The proof mainly follows from the stochastic Pontryagin principle, see for instance \cite[chapter 3]{YongZhou}. Here, the stochastic Pontryagin principle provides 
a necessary and sufficient condition on the dynamics of the optimal control because the cost coefficients are convex in the spatial variable.
However, the very fist step of the proof is to get rid of the parameter $\varpi$ in the cost functional 
${\mathcal R}^{\varpi X_0}(\varpi \balpha;\overline{\boldsymbol m}^n;\varpi \varepsilon {\boldsymbol W}^{\varpi {\boldsymbol h}^n/\varepsilon})$ (see \eqref{eq:opti:n+1:n}). 

For a control $\balpha$, the controlled dynamics driven by $\varpi X_0$, $\varpi \balpha$ and the common noise 
$\varpi \varepsilon {\boldsymbol W}^{\varpi {\boldsymbol h}^n \hspace{-2pt} /\varepsilon}$
 write 
\begin{equation*} 
\ud X_t^{\varpi,\balpha} = \varpi \alpha_t \ud t + \sigma \ud B_t + \varpi \varepsilon \ud W_t^{\varpi  {\boldsymbol h}^n\hspace{-2pt}/\varepsilon}, 
\quad t \in [0,T] \ ; \quad X_0^{\varpi,\balpha} = \varpi X_0.
\end{equation*} 
Dividing by $\varpi$, we can write $X_t^{\varpi,\balpha}$ in the form $X_t^{\varpi,\balpha} = \varpi X_t^{\balpha}$, with 
(the notation $X_t^{\balpha}$ is in fact a bit abusive because the dynamics below still depend on $\varpi$)
 \begin{equation*} 
 \ud X_t^{\balpha} = \alpha_t \ud t + \tfrac1{\varpi} \sigma \ud B_t + \varepsilon \ud W_t^{\varpi {\boldsymbol h}^n \hspace{-2pt} / \varepsilon}, \quad 
 t \in [0,T] \ ; \quad X_0^{\balpha} = X_0. 
\end{equation*} 
Accordingly, the cost in 
\eqref{eq:opti:n+1:n}
can be rewritten in the form 
\begin{equation*} 
\begin{split} 
&{\mathbb E}^{ \varpi {\boldsymbol h}^n \hspace{-2pt}/\varepsilon}
	\bigl[ {\mathcal R}^{{\varpi X_0}} \bigl( \varpi {\boldsymbol \alpha}; \overline{\boldsymbol m}^n ; \varpi \textcolor{black}{\varepsilon} {\boldsymbol W}^{ \varpi {\boldsymbol h}^n
	\hspace{-2pt}/\varepsilon} \bigr) \bigr]
	\\
	&= \frac{1}2 {\mathbb E}^{ \varpi {\boldsymbol h}^n \hspace{-2pt}/\varepsilon}
\biggl[
\Bigl\vert \varpi X_T^{\balpha} + g \bigl( \overline{m}_T^n) 
\Bigr\vert^2 +
\int_0^T \Bigl( \varpi^2 
\vert \alpha_t \vert^2 
+ \bigl\vert \varpi X_t^{\balpha} + f\bigl( \overline{m}_t^n \bigr) 
\bigr\vert^2 
\Bigr) \ud t 
\biggr] 
\\
&= \frac{\varpi^2}2 
{\mathbb E}^{ \varpi {\boldsymbol h}^n \hspace{-2pt}/\varepsilon}
\biggl[
\Bigl\vert  X_T^{\balpha} +\tfrac1{\varpi} g \bigl( \overline{m}_T^n) 
\Bigr\vert^2 +
\int_0^T \Bigl(  
\vert \alpha_t \vert^2 
+ \bigl\vert X_t^{\balpha} +\tfrac1{\varpi} f\bigl( \overline{m}_t^n \bigr) 
\bigr\vert^2 
\Bigr) \ud t 
\biggr]
\end{split} 
\end{equation*} 
The leading parameter $\varpi^2$ can be easily removed (because it does not change the minimizer). 
Next, the stochastic Pontryagin principle says that 
the optimal trajectory
to the cost functional in the above right-hand side
--and 
thus the optimal trajectory to the cost functional 
\eqref{eq:opti:n+1:n}
(but up to the rescaling factor $\varpi$)-- is
\begin{equation*} 
\ud X_t^{n+1} = - \bigl( \eta_t X_t^{n+1} + h_t^{n+1} \bigr) \ud t + \tfrac1{\varpi} {\sigma} \ud B_t + 
\varepsilon \ud W_t^{\varpi {\boldsymbol h}^n \hspace{-2pt} / \varepsilon}, \quad t \in [0,T] \ ; 
\quad X_0^{n+1} = X_0,
\end{equation*} 
with $(\eta_t)_{0 \le t \le T}$ and 
${\boldsymbol h}^{n+1}$ as in the statement. 
\end{proof}

\begin{proof}[Proof of Lemma \ref{lem:1}.]
The proof is a straightforward consequence of 
the two equations for 
${\boldsymbol h}$ and ${\boldsymbol h}^{n+1}$ under 
(respectively) the 
probability measures 
${\mathbb P}^{ {\boldsymbol h} /\varepsilon}$
and 
${\mathbb P}^{\varpi {\boldsymbol h}^n \hspace{-2pt} /\varepsilon}$, see 
\eqref{eq:bsde:true:underPh} and
\eqref{eq:h:n+1:new:one}. 
One can make the argument especially clear by using
the formulation presented in Remark 
\ref{rem:back:resolvent:correc}
together with the fact that the coefficients $f$ and $g$ are bounded.
\end{proof}

\begin{proof}[Proof of Lemma \ref{lem:2}.]
In \eqref{eq:intro:master:1}, we perform the change of variable
\begin{equation*} 
\theta_\varepsilon(t,x) = \phi_\varepsilon\bigl( \tfrac{t}{\varepsilon^2}, 
\tfrac{x}{\varepsilon^2} \bigr), \quad (t,x) \in [0,T] \times {\mathbb R}^d. 
\end{equation*} 
We get 
\begin{equation*}
\begin{split}
&\tfrac1{\varepsilon^2} 
\partial_{t} \phi_\varepsilon\bigl( \tfrac{t}{\varepsilon^2}, 
\tfrac{x}{\varepsilon^2} \bigr) + \tfrac{1}{2 \varepsilon^2} 
\Delta^2_{xx} \phi_\varepsilon\bigl( \tfrac{t}{\varepsilon^2}, 
\tfrac{x}{\varepsilon^2} \bigr) - \tfrac1{\varepsilon^2} \bigl( \eta_{t} x + \theta_\varepsilon(t,x) \bigr) \cdot \nabla_{x} \phi_\varepsilon\bigl( \tfrac{t}{\varepsilon^2}, 
\tfrac{x}{\varepsilon^2} \bigr) 
\\
&\hspace{15pt} +Q^\dagger f(x) - \eta_{t} \theta_\varepsilon(t,x) = 0.
\end{split}
\end{equation*}
Multiplying by $\varepsilon^2$ and changing $(t,x)$ into $(\varepsilon^2 t,\varepsilon^2 x)$, we obtain
\begin{equation*}
\begin{split}
&\partial_{t} \phi_\varepsilon(t,x) + \tfrac{1}{2} 
\Delta^2_{xx} \phi_\varepsilon(t,x) -   \Bigl( \eta_{t} \varepsilon^2 x + \theta_\varepsilon\bigl( \varepsilon^2 t,
\varepsilon^2 x\bigr) \Bigr) \cdot \nabla_{x} \phi_\varepsilon(t,x)
\\
&\hspace{15pt} + \varepsilon^2 Q^\dagger f\bigl( \varepsilon^2 x\bigr) - \varepsilon^2  \eta_{t} \theta_\varepsilon \bigl( \varepsilon^2 t, 
\varepsilon^2 x \bigr) = 0.
\end{split}
\end{equation*}
Above, $(t,x)$ belongs to $[0,T/\varepsilon^2] \times {\mathbb R}^d$. The  terminal condition is 
$\phi_\varepsilon(T/\varepsilon^2,x) = g(\varepsilon^2 x)$. 
Moreover, by Lemma 
\ref{lem:1}, the function $\theta_\varepsilon$ can be bounded independently 
of $\varepsilon$. Then, for $t$ at distance less than 1 from $T/\varepsilon^2$, 
we get a bound for $\nabla_x \varphi_\varepsilon$ 
 from 
standard estimates for
systems of nonlinear 
 parabolic PDEs, as used
 in 
\cite{Delarue02}. When $t$ is at distance greater than 1 from $T/\varepsilon^2$, 
we get a bound for $\nabla_x \varphi_\varepsilon$ 
from interior estimates for 
systems of nonlinear 
 parabolic PDEs, see for instance \cite{MR2053051}. 
 The result follows. 
\end{proof}

\vskip 5pt

\begin{proof}[Proof of Remark \ref{rem:thm:1}]
We now add a few lines in order to prove 
the complementary 
Remark 
\ref{rem:thm:1}.
When $R$ and $Q$ are the identity matrices, the solution ${\boldsymbol \eta}$ of the Riccati equation
\eqref{eq:riccati} 
becomes an
homothecy (which we identify with a real-valued function).   
In turn, 
if the functions $f$ and $g$ are `diagonal', in the sense of 
Remark \ref{rem:thm:1}, 
then the solution $\theta_\varepsilon$ to the PDE 
\eqref{eq:intro:master:1} is also diagonal, 
namely $\theta^i_\varepsilon(t,x) = \theta_{0,\varepsilon}(t,x_i)$, with 
$\theta_0 : [0,T] \times {\mathbb R} \rightarrow {\mathbb R}$ being the solution of the $1d$-equation: 
\begin{equation*}
\partial_{t} \theta_{0,\varepsilon}(t,x_0) + \tfrac{\varepsilon^2}2
\partial^2_{x_0x_0} \theta_{0,\varepsilon}(t,x_0) - \bigl( \eta_{t} x_0 + \theta_{0,\varepsilon}(t,x_0) \bigr)  \partial_{x_0} \theta_{0,\varepsilon}(t,x_0) + f_0(x_0) - \eta_{t} \theta_{0,\varepsilon}(t,x_0) = 0,
\end{equation*}
for $(t,x_0) \in [0,T] \times {\mathbb R}$, 
and
with the terminal boundary condition 
$\theta_{0,\varepsilon}(T,x_0) = g_0(x_0)$. 

In particular, 
\begin{equation*} 
\sup_{(t,x) \in [0,T] \times {\mathbb R}^d}
\bigl\vert
\nabla_x \theta_\varepsilon(t,x) \bigr\vert^2  
= 
d\sup_{(t,x_0) \in [0,T] \times {\mathbb R}}
\bigl\vert
\partial_{x_0} \theta_{0,\varepsilon}(t,x_0) \bigr\vert^2. 
\end{equation*} 

So, if we take for granted the fact that we cannot get better than $C_0/\varepsilon^2$ for the right-hand side, then 
the bound for the left-hand side becomes $C_0 d /\varepsilon^4$. Returning back to 
\eqref{eq:proof:delta}, the constant $C$ therein 
grows like $\sqrt{d}$ with $d$, which completes the proof of the claim. 
%
%
%
%
%the resolvent $(P_t)_{0 \leq t \leq T}$ is also an homothety (with a non-trivial and dimension-free ratio). Then, the constant 
%$C_1$ is of order $O(d)$ and its inverse can be assumed to be of order $O(1/d)$
%(which means that $C_1$ scales likes $d$).  
%We obtain 
%that the BMO norm of 
%\begin{equation*}
%\biggl( \int_{0}^t k_{s}^{n+1} dW_{s}^{{\boldsymbol h}^n} \biggr)_{0 \leq t \leq T}
%\end{equation*}
%belongs to $[\sqrt{d}/c_1,c_1 \sqrt{d}]$ for some $c_1>1$ independent of $d$. 
%Therefore, multiplying by $\sqrt{\delta}$, the BMO norm of 
%\begin{equation*}
%\biggl( \sqrt{\delta}  \int_{0}^t k_{s}^{n+1} dW_{s}^{{\boldsymbol h}^n} \biggr)_{0 \leq t \leq T}
%\end{equation*}
%belongs to $[ \sqrt{\delta d}/c_1,\delta c_1 \sqrt{\delta d}]$ for some $c_1>1$ independent of $d$. 
%Back to \cite[Theorem 2.2]{kazamaki}, 
%we (must) choose $ c_1 \sqrt{\delta d} < 1/4$, which fits the claim in 
%Remark 
%\ref{rem:lem:1}.
\end{proof}

\subsection{The common noise as an exploration noise}
\label{subse:2.2}

Following the agenda explained in the introduction, we now 
regard the common noise as an exploration noise for an MFG without common noise. 
In words, $\varepsilon$ is set equal to $0$ in the original MFG
\eqref{eq:state:intro}--\eqref{eq:cost:intro}--\eqref{eq:intro:fixed}
 and this is only in the choice of 
the controls that we restore the presence of the common noise, in the form of a randomization.

\subsubsection{Presentation of the model}
In the absence of common noise, the state dynamics merely write
\begin{equation*}
\ud X_{t} = \beta_{t} \ud t + \sigma \ud B_{t}, \quad t \in [0,T]. 
\end{equation*}
However, we want ${\boldsymbol \beta}=(\beta_{t})_{0 \leq t \leq T}$
to be subjected to a random exploration of the form
\begin{equation*}
	\beta_{t} = \alpha_{t} + \varepsilon \dot{W}_{t}, \quad t \in [0,T], 
\end{equation*}
where ${\boldsymbol \alpha}=(\alpha_{t})_{0 \le t \le T}$ is the control effectively chosen by the agent (or by the `controller') and $\epsilon$ is (strictly) positive (and is kept fixed throughout the subsection). In this expansion, 
$(\dot{W}_{t})_{0 \le t \le T}$ is formally understood as the time-derivative of $(W_{t})_{0 \le t \le T}$. Obviously, the latter does not exist as a function, which makes the above decomposition non tractable. However, it prompts us to introduce a variant based upon a mollification of the common noise. To make it clear, we introduce a family of regular processes $((W_{t}^p)_{0 \le t \le T})_{p \geq 1}$ such that, almost surely (under ${\mathbb P}$),
\begin{equation*}
	\lim_{p \rightarrow \infty}
	\sup_{0 \le t \le T} \vert W_{t} - W_{t}^p \vert 
	= 0, 
\end{equation*}
and, for any $p \geq 1$, the paths of $(W_{t}^p)_{0 \leq t \leq T}$
are continuous and piecewise continuously differentiable. \textcolor{black}{Throughout this subsection and the next one, we use} 
the following piecewise affine interpolation:
\begin{equation}
\label{eq:linear:interpolated}
	W_{t}^p := W_{\tau_{p}(t)} + \frac{p (t- \tau_{p}(t))}{T} \bigl( W_{\tau_{p}(t)+T/p} - W_{\tau_{p}(t)} \bigr),
\end{equation}
where, by definition, $\tau_{p}(t): = \lfloor  pt/T \rfloor   (T/p)$. 
In words $\tau_{p}(t)$ is the unique element of $(T/p) \cdot {\mathbb N}$ such that 
$\tau_{p}(t) \leq t < \tau_{p}(t) + T/p=\tau_{p}(t+T/p)$, namely 
$\tau_p(t)=\ell T/p$, for $t \in [\ell T/p,(\ell+1)T/p)$ and $\ell \in \{0,\cdots,p-1\}$. 
\textcolor{black}{It is worth observing that the time derivative of 
${\boldsymbol W}^p$ writes in the form 
of finite differences of 
${\boldsymbol W}$, which explains our definition of 
${\boldsymbol W}^p$ as a linear interpolation. We elaborate on this observation in Remark 
\ref{rem:linear:interpolation}
below.}

For an ${\mathbb F}^{{\boldsymbol W}}$-adapted and continuous environment 
${\boldsymbol m}=(m_{t})_{0 \leq t \leq T}$, the cost functional, as originally defined in 
\eqref{eq:cost:intro}, is turned into the following discrete time version
\begin{equation*}
%\label{eq:true:cost}
	\widetilde J^p({\boldsymbol \alpha};{\boldsymbol m})
	:= \frac12 {\mathbb E} \biggl[  \bigl\vert R^{\dagger} X_{T} + g(m_{T})\bigr\vert^2 
	+ \int_{0}^T \Bigl\{   \bigl\vert Q^{\dagger}X_{\tau_{p}(t)} + f(m_{\tau_{p}(t)}) \bigr\vert^2 +  \vert \alpha_{t} + \varepsilon \dot{W}_{t}^p \vert^2 \Bigr\}
	\ud t \biggr],
\end{equation*}
where the expectation is taken over both the idiosyncratic and exploration (common) noises. 
Above, we require ${\boldsymbol \alpha}$ to be progressively-measurable with respect to the filtration ${\mathbb F}^{p,X_{0},{\boldsymbol B},{\boldsymbol W}}:=
(\sigma (X_{0},(B_{\tau_{p}(s)},W_{\tau_{p}(s)})_{s \leq t}))_{0 \leq t \leq T}$
and to be constant on any interval $[\ell T/p,(\ell +1)T/p)$, for $\ell \in \{0,\cdots,p-1\}$. 
The above minimization problem can be regarded as a discrete-time control problem. The need for a time discretization is twofold: $(i)$ On the one hand, it permits to avoid any anticipativity problem, since 
the linear interpolation at a time $t \not = \tau_{p}(t)$ anticipates on the future realization of the exploration noise; $(ii)$ On the other hand, it is more adapted to numerical purposes.  
\textcolor{black}{To the best of our knowledge, the formulation of 
$\widetilde J^p({\boldsymbol \alpha};{\boldsymbol m})$ in the form of a cost functional featuring 
an additive randomization of the control is new.} 

% This requirement is here needed (differently from what is used in usual reinforcement learning because, at this stage, 
%we don't require $\alpha_{t}$ to be in feedback form; by, the way, feedback form here would mean $\alpha_{t}=- \eta_{t} x_{t} + \psi(m_{t})$ and the fact  
%that $\psi$ is a function of $m_{t}$ (and not of $x_{t}$). Also introduce notations for the filtration? 
%${\mathbb F}^{p,{\boldsymbol W}}$? J'ai bien l'impression que les processus sur lesquels on optimise ont une structure particuli\`ere}

In the presence of the randomization, the cost functional might become very large as $p$ tends to $\infty$, even for a control ${\boldsymbol \alpha}$ of finite energy. This prompts us to renormalize $\tilde J^p$. 
As a result of the adaptability  constraint, we indeed have  
\color{black}
\begin{equation*}
\begin{split}
	{\mathbb E} \int_{0}^T \alpha_{t} \cdot \dot{W}_{t}^p \ud t &= {\mathbb E}
	\sum_{\ell=0}^{p-1}  \int_{\ell T/p}^{(\ell+1)T/p} \alpha_t \dot{W}_t^p \ud t 
	 = \sum_{\ell=0}^{p-1}{\mathbb E} \Bigl[  \alpha_{\ell T/p} \cdot \bigl( {W}_{(\ell+1)T/p} -  {W}_{\ell T/p}  \bigr) \bigr] =0,
\end{split}
\end{equation*}
with the last line following from the fact that $\alpha_{\ell T/p}$ is ${\mathcal F}_{\ell T/p}^{\boldsymbol W}$-measurable. 
\color{black}
Also, since ${\boldsymbol W}^p$ is piecewise linear, we have 
\begin{equation*}
	\begin{split}
		{\mathbb E} 
		\int_{0}^T   
		\vert \dot{W}_{t}^p \vert^2 
		\ud t &= \sum_{\ell=0}^{p-1} \int_{\ell T/p}^{(\ell+1)T/p}
		{\mathbb E}
		\Bigl[ 
		\bigl\vert \frac{p}{T} \bigl( W_{(\ell+1)T/p} - W_{\ell T/p} \bigr)
		\bigr\vert^2 \Bigr] \ud t 
%		\\
%		&= \sum_{k=1}^{p-1} \int_{k T/p}^{(k+1)T/p}
%		{\mathbb E}
%		\Bigl[ 
%		\bigl\vert \frac{p}{T} \bigl( W_{(k+1)T/p} - W_{kT/p} \bigr)
%		\bigr\vert^2 \Bigr] \ud t
= d \, p \frac{T}{p} \frac{p^2}{T^2}
		\frac{T}{p} = d \, p.
	\end{split}
\end{equation*}
So, we must subtract to the cost a diverging term to recover the original cost functional. To make it clear, 
the effective cost must be 
\begin{equation}
\label{eq:effective:cost}
\begin{split}
J^p({\boldsymbol \alpha};{\boldsymbol m}) 
:=& \, \widetilde J^p({\boldsymbol \alpha};{\boldsymbol m})- \tfrac12 \varepsilon^2 d \, p
\\
=& \, \frac12 {\mathbb E} \biggl[  \bigl\vert R^\dagger X_{T} + g(m_{T})\bigr\vert^2 
	+ \int_{0}^T \Bigl\{   \bigl\vert Q^\dagger X_{\tau_{p}(t)} + f(m_{\tau_{p}(t)}) \bigr\vert^2 +  \vert \alpha_{\tau_{p}(t)} \vert^2 \Bigr\}
	\ud t \biggr].
\end{split}
\end{equation}
Noticeably, we recover (up to the time discretization) the same cost functional $J$ as in 
\eqref{eq:cost:intro}, which explains why we have removed 
the tilde over the symbol $J^p$. 
Anyway, 
$J^p(\cdot;{\boldsymbol m})$
and 
$\widetilde J^p(\cdot;{\boldsymbol m})$
have the same minimizers. 
%
%\begin{rem}
%Think about p in the section of the numerics 
%\end{rem}
%\textbf{Assumptions }
%\subsubsection{Girsanov's transform}
\vskip 4pt

Before we provide the form of the corresponding fictitious play, we need to clarify the 
notion of tilted measure. 
We assume that we are given a process ${\boldsymbol h}=(h_{t})_{0 \le t \le T}$ that is piecewise constant and ${\mathbb F}^{p,{\boldsymbol W}}$-adapted, with 
${\mathbb F}^{p,{\boldsymbol W}}
:=
{\mathbb F}^{{\boldsymbol W}^p}$: For the same $p$
as before, the process ${\boldsymbol h}$
is constant on each interval $([\ell T/p,(\ell+1)T/p))_{\ell=0,\cdots,p-1}$ and each 
$h_{\ell T/p}$ is 
${\mathcal F}^{p,{\boldsymbol W}}_{\ell T/p}$-measurable, 
with ${\mathcal F}_{\ell T/p}^{p,{\boldsymbol W}}=\sigma(W^p_{s},s \leq \ell T/p)$. Then, as before, we can let 
${\boldsymbol W}^{{\boldsymbol h}/\varepsilon}$:
\begin{equation*}
	W^{{\boldsymbol h}/\varepsilon}_{t} 
	= W_{t} + \frac1{\varepsilon} \int_{0}^t  h_{s} \ud s.
\end{equation*}
We compute the $p$-piecewise linear interpolation ${\boldsymbol W}^{p,{\boldsymbol h}/\varepsilon}$ of ${\boldsymbol W}^{{\boldsymbol h}/\varepsilon}$:
\begin{equation*}
	\begin{split}
		W^{p,{\boldsymbol h}/\varepsilon}_{t}
		&= 
		W^{p,{\boldsymbol h}/\varepsilon}_{\tau_{p}(t)}
		+
		\frac{p (t- \tau_{p}(t))}{T} \bigl( W^{{\boldsymbol h}/\varepsilon}_{\tau_{p}(t)+T/p} - W^{{\boldsymbol h}/\varepsilon}_{\tau_{p}(t)} \bigr)
		\\
		&= W^{p,{\boldsymbol h}/\varepsilon}_{\tau_{p}(t)}
		+
		\frac{p (t- \tau_{p}(t))}{T} \bigl( W_{\tau_{p}(t)+T/p} - W_{\tau_{p}(t)} \bigr)
		+  \frac1{\varepsilon}
		\frac{p (t- \tau_{p}(t))}{T} \frac{T}p h_{\tau_{p}(t)}
		\\
		&= W^{p,{\boldsymbol h}/\varepsilon}_{\tau_{p}(t)}
		+
		\frac{p (t- \tau_{p}(t))}{T} \bigl( W_{\tau_{p}(t)+T/p} - W_{\tau_{p}(t)} \bigr)
		+ 
		 \frac1{\varepsilon} \bigl(t- \tau_{p}(t)\bigr) h_{\tau_{p}(t)}
		\\
		&= W^{p,{\boldsymbol h}/\varepsilon}_{\tau_{p}(t)} + \int_{\tau_{p}(t)}^t 
		\ud W_{s}^{p} +  \frac1{\varepsilon} \int_{\tau_{p}(t)}^t {h}_{s} \ud s,
	\end{split}
\end{equation*} 
from which we deduce that 
%\textcolor{red}{ça, c'est une propriété qui est vraie parce que l'on a choisi l'interpolation linéaire}. 
\begin{equation}
\label{eq:relation:Wph:Wp}
	W^{p,{\boldsymbol h}/\varepsilon}_{t}
	= W^{p}_{t} + \frac1{\varepsilon}  \int_{0}^t h_{s} \ud s, \quad t \in [0,T]. 
\end{equation}
For sure, under the probability ${\mathcal E}({\boldsymbol h}/\varepsilon) \cdot {\mathbb P}$, 
the process 
${\boldsymbol W}^{p,{\boldsymbol h}/\varepsilon}$ is the piecewise linear interpolation of
${\boldsymbol W}^{{\boldsymbol h}/\varepsilon}$ and the latter is a Brownian motion. Also, 
\begin{equation*}
	\dot{W}^{p,{\boldsymbol h}/\varepsilon}_{t}
	= \dot{W}^{p}_{t} +  \frac1{\varepsilon} h_{t},
\end{equation*}
which permits to say that 
the  trajectories controlled by 
$(\alpha_t+\varepsilon \dot{W}^{p}_t)_{0 \le t \le T}$ satisfy 
\begin{equation*}
	\ud X_{t} = \bigl( \alpha_{t} - h_{t} \bigr) \ud t + \sigma \ud B_{t} + \varepsilon \ud W_{t}^{p,{\boldsymbol h}/\varepsilon},
\end{equation*}
which coincides with \eqref{eq:state:intro}. Importantly, since 
${\boldsymbol h}$ is piecewise constant, 
${\mathcal E}({\boldsymbol h}/\varepsilon)$ can be expressed in terms of the sole ${\boldsymbol h}$ and 
${\boldsymbol W}^p$.

\color{black}
\begin{rem}
\label{rem:linear:interpolation}
We feel useful to comment more on our choice to work with the linear interpolation $(W^p_t)_{0 \leq t \leq T}$ 
of $(W_t)_{0 \leq t \leq T}$. While it may look somewhat arbitrary, this choice is in fact dictated by the structure of the
model and the form of the final result that is provided next. Indeed, the main result of this subsection, see Theorem \ref{main:thm:2} below, yields a bound for the weak error (in a convenient sense) 
between the solution of a time-discrete version of the mean field game (with a common noise) 
and the corresponding time-discrete variant of the fictitious play, 
with the time mesh being given by $(k T/p)_{k=0,\cdots,p}$. 

Obviously, working with a discrete-time version of the game makes perfect sense from a practical point of view. Here, it is also especially adapted to our objective. 
As we already explained, 
we want to regard the instantaneous value of the control at time 
$t$ as being \textit{corrupted}, at least formally, by the time derivative of the common noise. 
Because the latter derivative does not exist as a true function, our strategy is to work instead with finite differences of the 
common noise. This is one first reason for working with $(W^p_t)_{0 \leq t \leq T}$, because the linear interpolation  
is completely characterized by the finite differences of $(W_t)_{0 \leq t \leq T}$ at times 
$0$, $T/p$, $2T/p$, $\cdots$, $T$. 
Another related reason is that, for ${\boldsymbol h}$ a piecewise constant process that is 
${\mathbb F}^{p,{\boldsymbol W}}$-adapted, 
the relationship 
\eqref{eq:relation:Wph:Wp}
holds true because  
${\boldsymbol W}^p$ and ${\boldsymbol W}^{p,{\boldsymbol h}}$ are piecewise linear. 
This is another strong case for working with the piecewise linear interpolation because
\eqref{eq:relation:Wph:Wp}
makes it possible to apply Girsanov theorem directly, with the latter being the key ingredient of the whole analysis. 

In a nutshell, given the class of time-discrete games we address for approximating the original game
(using in particular finite differences to perturb the controls), 
the linear interpolation is the most natural one to reconstruct time-continuous dynamics from time-discrete observations. 
As such, the linear interpolation does not directly impact the bound on the weak error that is stated below, because the weak error
is defined \textit{a priori} within the given class of approximating games, regardless of the choice of the interpolation. 
However, working with the linear interpolation makes the proof much easier. 
\end{rem}
\color{black}

\subsubsection{Fictitious play}

The analysis from the previous paragraph leads to the following new scheme, which should be regarded as the discrete-time analogue of 
the scheme presented in 
Subsection \ref{subse:5:2:fictitious:play}.
Fix  
a real $\varpi \in (1,\sqrt{2}]$ and 
 an integer $p \geq 1$
and, for the same 
initialization ${\boldsymbol m}^0=(m^0_{t}={\mathbb E}(X_{0}))_{0 \le t \le T}$
and 
${\boldsymbol h}^0=(h^0_{t}=0)_{0 \le t \le T}$ as in the continuous-time setting, 
 assume that we have defined two families of proxies $({\boldsymbol m}^1,\cdots,{\boldsymbol m}^n)$
and $({\boldsymbol h}^1,\cdots,{\boldsymbol h}^n)$ with the additional assumption 
that each process ${\boldsymbol h}^k$  is constant on each interval $[\ell T/p,(\ell +1)T/p)$, for 
$\ell \in \{0,\cdots,p-1\}$ and $k \in \{1,\cdots,n\}$. And we assume each $(m_{\ell T/p}^k,h_{\ell T/p}^k)$ to be measurable with respect to 
${\mathcal F}_{\ell T/p}^{p,{\boldsymbol W}}$
 if $\ell \in \{0,\cdots,p\}$. Then, we solve for
\begin{equation}
 \label{eq:optimization:exploration}
 \begin{split}
	{\boldsymbol \alpha}^{(p),n+1,\varpi}
	&= \textrm{\rm argmin}_{\boldsymbol \alpha}
	\biggl( {\mathbb E}^{\varpi {\boldsymbol h}^n \hspace{-2pt} / \varepsilon}
	\Bigl[ {\mathcal R}^{p,\varpi X_0}  \Bigl( \varpi {\boldsymbol \alpha} + \varpi \varepsilon \dot{\boldsymbol W}^{p,\varpi {\boldsymbol h}^n \hspace{-2pt} /\varepsilon} ; \overline{\boldsymbol m}^n;0 \Bigr) \Bigr] - \tfrac12 d\, \varpi^2 \varepsilon^2 p \biggr) 
	\\
	&=\textrm{\rm argmin}_{\boldsymbol \alpha}
	\biggl( {\mathbb E}^{\varpi {\boldsymbol h}^n \hspace{-2pt} / \varepsilon}
	\Bigl[ {\mathcal R}^{p,\varpi X_0} \Bigl( \varpi {\boldsymbol \alpha}  ; \overline{\boldsymbol m}^n; \varpi \varepsilon {\boldsymbol W}^{p,\varpi {\boldsymbol h}^n/\varepsilon} \Bigr) \Bigr] \biggr), 
	\end{split}
	\end{equation}
with the same measurability rules as those explained above and where 
${\mathcal R}^{p,x_0}$ is the obvious discrete-time version of 
${\mathcal R}^{x_0}$, namely
\begin{equation}
\label{eq:mathcalR:p}
{\mathcal R}^{p,x_0}({\boldsymbol \alpha};{\boldsymbol n};\textcolor{black}{\varepsilon} {\boldsymbol w}) = 
\frac12
\biggl[ 
\bigl\vert R x_T + g(n_T) \bigr\vert^2 
+
\sum_{\ell=0}^{p-1}
\Bigl\{ 
\bigl\vert Q x_{\ell T/p} + f(n_{\ell T/p}) \bigr\vert^2 
+
\bigl\vert \alpha_{\ell T/p} \bigr\vert^2
\Bigr\} \biggr],
\end{equation}
where 
$x_{(\ell +1)T/p}
=
x_{\ell T/p}
+ (T/p) \alpha_{\ell T/p} + \sigma (B_{(\ell+1)T/p}- B_{\ell T/p})+
\varepsilon  
(w_{(\ell+1)T/p}- w_{\ell T/p})$.
As before, 
$\overline{m}^n_t$ in 
\eqref{eq:optimization:exploration}
is 
$\overline{m}^n_t = 
 \tfrac{\varpi(1-\varpi^{-1})}{1-\varpi^{-n}} \sum_{k=1}^{n} \varpi^{-k} m_{t}^{k} $.
%\begin{rem}\textit{to think about later}\\
%	It looks like $p$ counts for nothing here since it does change the minimizer.
%	
%\end{rem}
%
%Call then ${\boldsymbol x}^{n+1}$ the optimal trajectory and let 
%\begin{equation*}
%	m^{n+1}_{t} = {\mathbb E}^0 \bigl[ x_{t}^{n+1} \bigr],
%\end{equation*}
%together with 
%\begin{equation*}
%	\overline{m}^{n+1}_{t} = \frac1{n+1} \sum_{k=1}^{n+1} m_{t}^{n} = \frac1{n+1} m_{t}^{n+1} + \frac{n}{n+1}
%	\overline{m}^n_{t}.
%\end{equation*}
%
%
%\subsubsection{Iterative construction of the $h$-process }

The first step here is to solve the optimization problem
\eqref{eq:optimization:exploration}.
Very similar to Lemma \ref{lem:2.1}, we 
can provide an explicit form for the optimal feedback
through the solution of the following time-discrete 
Riccati equation:
\begin{equation}
\begin{split}
&\frac{p}{T}\Bigl( \eta_{(\ell+1)T/p}^{(p)}
-\eta_{\ell T/p}^{(p)}
\Bigr)
- 
 \eta_{(\ell+1) T/p}^{(p)} 
  \eta_{\ell T/p}^{(p)} 
 +Q^\dagger Q \Bigl( I_d -  \frac{T}p \eta_{\ell T/p} \Bigr) =0, 
 \\
 &\hspace{250pt} \quad \ell \in \{0,\cdots,p-2\} \ ; 
 \\
&\frac{p}{T}\Bigl( \eta_{T}^{(p)}
-\eta_{(p-1)T/p}^{(p)}
\Bigr)
- 
\eta_{T}^{(p)}
\eta_{(p-1) T/p}^{(p)}  =0 \ ; \quad \eta_{T}^{(p)}=R^\dagger R.
\label{eq:discretized:Riccati}
\end{split}
\end{equation}
The analysis of the latter (see the earlier reference 
\cite{Dorato1971OptimalLR}) goes through 
the auxiliary Riccati equation
\begin{equation}
\label{eq:Riccati:aux}
\frac{p}{T}\Bigl( P_{(\ell+1)T/p}^{(p)}
- P_{\ell T/p}^{(p)}
\Bigr)
- P_{(\ell+1)T/p}^{(p)}
\Bigl( I_d
+
\frac{T}p
P_{(\ell+1)T/p}^{(p)}
\Bigr)^{-1}
P_{(\ell+1)T/p}^{(p)}
+Q^\dagger Q=0, 
\end{equation}
for $\ell \in \{0,\cdots,p-1\}$, 
with $P_{T}^{(p)}=R^\dagger R$ as boundary condition.
The Riccati equation \eqref{eq:Riccati:aux} can be solved inductively: the solution is symmetric
 and non-negative\footnote{
 Symmetry is obvious. Non-negativity is a bit more demanding. Assuming that $P^{(p)}_{(\ell+1)T/p}$ is non-negative, non-negativity of 
 $P_{\ell T/p}^{(p)}$ is proved as follows.
 For any vector $x \in {\mathbb R}^d$, we have 
 \begin{equation*} 
 \begin{split}
 \bigl( P_{\ell T/p}^{(p)} x \bigr) \cdot x& =
 \bigl( P_{(\ell+1) T/p}^{(p)} x \bigr) \cdot x 
 - \tfrac{T}{p} \bigl( P_{(\ell +1)T/p}^{(p)} x \bigr) \cdot \Bigl[ \bigl( I_d + \tfrac{T}p P_{(\ell+1)T/p}^{(p)} 
 \bigr)^{-1} P_{(\ell+1)T/p}^{(p)} x \Bigr]
 +\tfrac{T}p \vert Q x \vert^2
 \\
 &\geq 
 \bigl( P_{(\ell+1) T/p}^{(p)} x \bigr) \cdot x 
 -  \bigl( P_{(\ell +1)T/p}^{(p)} x \bigr) \cdot x 
 +  \bigl( P_{(\ell +1)T/p}^{(p)} x \bigr)
 \cdot
 \Bigl[ \bigl( I_d + \tfrac{T}p P_{(\ell+1)T/p}^{(p)} 
 \bigr)^{-1}   x \Bigr],
 \end{split}
 \end{equation*}
 and the right-hand side is obviously non-negative.} , which guarantees 
that the inverse right above is well-defined.  
Then, 
\begin{equation}
\label{eq:discretized:Riccati:sol}
\eta_{\ell T/p}^{(p)} = \Bigl( I_d + \frac{T}p P^{(p)}_{(\ell+1)T/p} 
\Bigr)^{-1} P^{(p)}_{(\ell+1)T/p}, \quad \ell \in \{0,\cdots,p-1\}, 
\end{equation}
solves 
\eqref{eq:discretized:Riccati}. Notice that the above left-hand side is symmetric and non-negative because the two matrices in the right-hand side commute.

\begin{lem}
\label{lem:4}
Under the above assumptions, the 
minimization problem 
	\eqref{eq:optimization:exploration}
	has a unique solution
${\boldsymbol \alpha}^{(p),n+1,\varpi}$, which writes 
\begin{equation}
\label{eq:lem:4:feedback:form}
\alpha_{t}^{(p),n+1,\varpi} = - \biggl( \eta_{\tau_{p}(t)}^{(p)} X_{\tau_{p}(t)}^{(p),n+1,\varpi} + \tilde h_{\tau_p(t)}^{n+1}   \biggr), \quad t \in [0,T], 
\end{equation}
where 
$\tilde{\boldsymbol h}^{n+1}$ solves the backward SDE:
\begin{equation}
\label{eq:bsde:exploration}
\begin{split}
		&\ud  \tilde h_{t}^{n+1} 
		= \Bigl\{ - \tfrac1{\varpi} 
		{\mathbb E} 
		\Bigl[ 
		Q^\dagger f\bigl(\overline m_{\tau_{p}(t)+T/p}^n\bigr)\, \vert \, 
		{\mathcal F}^{p,{\boldsymbol W}}_{\tau_p(t)} \Bigr]
		 {\mathbf 1}_{\{ t \leq (p-1)T/p\}}  
		 \\
&\hspace{100pt}		  + 
		\Bigl( \eta_{\tau_{p(t)}+T/p}^{(p)}
		  + \frac{T}p Q^\dagger Q 
		  {\mathbf 1}_{\{ t \leq (p-1)T/p\}}
		  	\Bigr) 		  
		   \tilde h_{\tau_p(t)}^{n+1} \Bigr\}
 \ud t  
		  \\
&\hspace{50pt}		  + \varpi k_{t}^{n+1} h_{t}^n \ud t 
		+ \varepsilon k_{t}^{n+1} \ud W_{t}, \quad t \in [0,T),
%\nonumber
		%+ \ud W_{t} - \ud W_{t}^p, \quad t \in [0,T], 
		\\
		&\tilde h_{T}^{n+1} = \tfrac1{\varpi} R^\dagger g \bigl( \overline m^{n}_{T} \bigr). 
			\end{split}
			\end{equation}
Accordingly, the optimal path ${\boldsymbol X}^{(p),n+1,\varpi}$ solves 
(up to a rescaling factor $\varpi$)
the forward SDE:
\begin{equation}
\label{eq:bsde:exploration:forward:equation:with:varpi}
\ud X_{t}^{(p),n+1,\varpi} = \alpha_{t}^{(p),n+1,\varpi}  \ud t + \tfrac1{\varpi} \sigma \ud B_{t} + \varepsilon \ud W_{t}^{p,\varpi {\boldsymbol h}^n \hspace{-2pt}/\varepsilon}, \quad t \in [0,T] ;
\quad X_0^{n+1} = X_0. 
\end{equation}
\end{lem}

\begin{rem} 
Although 
it is not indicated in 
the notations 
$\tilde {\boldsymbol h}^{n}$, 
${\boldsymbol h}^{n}$,
${\boldsymbol k}^{n}$
and $\overline{\boldsymbol m}^n$, 
it should be clear to the reader that 
the latter four processes depend on $p$, $\varpi$ and $\varepsilon$. 
\end{rem}

\begin{rem}
\label{rem:main:thm:2}
The convergence of 
$(P^{(p)}_{\tau_p(t)})_{0 \le t \le T}$ to 
$(\eta_t)_{0 \le t \le T}$ in Lemma \ref{lem:2.1} is obvious. 

By the symmetric and non-negative structure of 
the matrices $(P^{(p)}_{\ell})_{\ell=0,\cdots,p}$, we can easily 
get uniform bounds on 
the latter. In turn, 
we can
regard the equation \eqref{eq:Riccati:aux} as an Euler scheme for 
a matricial differential equation driven by a Lipschitz vector field. The rate of convergence is linear in 
$p$. By 
\eqref{eq:discretized:Riccati:sol}, we deduce that 
\begin{equation}
\label{eq:rem:main:thm:2:correc}
\sup_{0 \leq t \leq T} \vert \eta_t^{(p)} - \eta_t \vert \leq \frac{C}p,
\end{equation}
for a constant $C$ independent of $p$, \textcolor{black}{but possibly depending on the dimension $d$.  Typically, the bound on  $(P^{(p)}_{\ell})_{\ell=0,\cdots,p}$ should depend on the norms of 
$Q^\dagger Q$ and $R^\dagger R$ and is thus expected to depend on $d$ in a polynomial way. In turn, stability arguments for finite difference equations say
that 
those bounds should propagate to the constant $C$ in the above estimate for the distance between $(\eta^{(p)}_t)_{0 \leq t \leq T}$ and 
$(\eta_t)_{0 \leq t \leq T}$. However, this argument may not be sharp: when $Q$ and $R$ are the identity matrix, the Riccati equation 
\eqref{eq:Riccati:aux}
reduces to a scalar equation and the constant $C$ should just scale like $\sqrt{d}$.}
\end{rem}

\begin{rem}
\label{rem:main:thm:2:b}
The backward SDE \eqref{eq:bsde:exploration} is in fact a mere discrete-time equation. Indeed,
\begin{equation}
\label{eq:rem:main:thm:2:b}
\begin{split}
		\tilde h_{\ell T/p}^{n+1} 
		&= \tilde h_{(\ell+1)T/p}^{n+1} +   \tfrac{T}{\varpi p}  
		Q^\dagger f(\overline m_{(\ell+1)T/p}^n) 
		{\mathbf 1}_{\{\ell \leq p-2\}}
		\\
		&\hspace{15pt} 
		-  \tfrac{T}p  
	\Bigl(	
		\eta_{(\ell+1)T/p}^{(p)} + \tfrac{T}p Q^\dagger Q {\mathbf 1}_{\{\ell \leq p-2\}} \Bigr)
		  \tilde h_{\ell T/p}^{n+1} 
		 - \varepsilon \int_{\ell T/p}^{(\ell+1)T/p} k_{s}^{n+1} \ud W_{s}^{\varpi {\boldsymbol h}^n \hspace{-2pt} / \varepsilon},
		%+ \ud W_{t} - \ud W_{t}^p, \quad t \in [0,T],
\end{split}
\end{equation}
for 
$\ell \in \{0,\cdots,p-1\}$. Taking 
conditional expectation given ${\mathcal F}_{\ell T/p}^{{\boldsymbol W}}$, 
we can solve for 
\begin{equation*}
\Bigl( I_d + \tfrac{T}p 
  \eta_{(\ell+1)T/p}^{(p)}
  + \tfrac{T^2}{p^2} Q^\dagger Q {\mathbf 1}_{\{\ell \leq p-2\}}
  \Bigr)
	\tilde h_{\ell T/p}^{n+1}. 
\end{equation*}
It is easy to deduce $\tilde{h}^{n+1}_{\ell T/p}$. 
We get
\begin{equation}
\label{eq:recurrence:cas:discret}
\begin{split}
\tilde h_{\ell T/p}^{n+1} 
 &= 
 \Bigl( I_d + \tfrac{T}p 
  \eta_{(\ell+1)/p}^{(p)}
    + \tfrac{T^2}{p^2} Q^\dagger Q {\mathbf 1}_{\{\ell \leq p-2\}}
  \Bigr)^{-1}
  \\
&\hspace{15pt} \times  
 {\mathbb E}^{\varpi {\boldsymbol h}^n \hspace{-2pt} / \varepsilon} \biggl[
 \tilde h_{(\ell+1)T/p}^{n+1} 
  +   \tfrac{T}{\varpi p}  
		Q^\dagger f\bigl(\overline m_{(\ell+1)T/p}^n\bigr) {\mathbf 1}_{\{\ell \leq p-2\}}
 \, \vert \, {\mathcal F}_{\ell T/p}^{{\boldsymbol W}}\biggr],
\end{split}
 \end{equation}
  which proves, by induction, that 
  $\tilde h_{\ell T/p}^{n+1}$
  is 
  ${\mathcal F}_{\ell T/p}^{p,{\boldsymbol W}}$-measurable.   
 It suffices to observe that, for $Z$ an ${\mathcal F}_{(\ell+1) T/p}^{p,{\boldsymbol W}}$-measurable random variable, the conditional expectation of $Z$ given ${\mathcal F}_{\ell T/p}^{\boldsymbol W}$
 is in fact 
${\mathcal F}_{\ell T/p}^{p,{\boldsymbol W}}$-measurable. 
\end{rem}

Taking for granted Lemma \ref{lem:4} (the proof is given at the end of the subsection), we put 
\begin{equation}
\label{eq:hn+1:tildeh:n+1}
	h^{n+1}_{t} := \tilde{h}^{n+1}_{\tau_{p}(t)}, 
\end{equation}
which satisfies the required measurability constraints thanks to
Remark  
\ref{rem:main:thm:2:b}.
Then, we let 
\begin{equation}
\label{eq:mn+1:m:n+1:correc}
\begin{split}
	m^{n+1}_{t} &:= {\mathbb E} \bigl[ X_{t}^{(p),n+1,\varpi} \, \vert \, \sigma({\boldsymbol W}) \bigr] =
	  {\mathbb E}^{\varpi {\boldsymbol h}^n/\varepsilon} \bigl[ X_{t}^{(p),n+1,\varpi} \, \vert \,  \sigma({\boldsymbol W}) \bigr], \quad t \in [0,T], 
	\end{split}
\end{equation}
together with 
\begin{equation}
\label{eq:update}
	\overline{m}^{n+1}_{t} = \tfrac{\varpi(1-\varpi^{-1})}{1-\varpi^{-(n+1)}} \sum_{k=1}^{n+1} \varpi^{-k} m_{t}^{k} = 
	\tfrac{\varpi^{-n}(1-\varpi^{-1})}{1- \varpi^{-(n+1)}}
	m^{n+1}_t +  \Bigl( 1 - 
	\tfrac{\varpi^{-n}(1-\varpi^{-1})}{1- \varpi^{-(n+1)}}
	 \Bigr) 
	\overline{m}^n_{t}.
\end{equation}
Importantly, notice that 
\begin{equation*}
m^{n+1}_{\ell T/p} = 
{\mathbb E} \bigl[ X^{(p),n+1,\varpi}_{\ell T/p} \, \vert \, \sigma({\boldsymbol W}^p) \bigr]
=
{\mathbb E} \bigl[ X^{(p),n+1,\varpi}_{\ell T/p} \, \vert \, {\mathcal F}_{\ell T/p}^{p,{\boldsymbol W}} \bigr]
, \quad \ell \in \{0,\cdots,p\},
\end{equation*}
which proves in particular that the left-hand side is 
${\mathcal F}_{\ell T/p}^{p,{\boldsymbol W}}$-measurable. 
 \vskip 4pt
 
The analogue of Theorem 
\ref{main:thm:1} becomes:

\begin{thm}
\label{main:thm:2}
	There exists 
	a threshold $c>0$, depending on  $d$, $T$, the norms
$\| f\|_{1,\infty}$ and $\| g \|_{1,\infty}$
and the norms $\vert Q\vert$ and $\vert R\vert$ of the matrices $Q$ and $R$,	 such that, 
for $p \varepsilon^{2} \geq c $, the 
	 scheme
	\eqref{eq:hn+1:tildeh:n+1}--\eqref{eq:update}
	converges 
	to 
	$({\boldsymbol m}^{(p)},\tilde{\boldsymbol h}^{(p)} \hspace{-2pt} /\varpi,{\boldsymbol k}^{(p)} \hspace{-2pt} /\varpi)$, 
	where $({\boldsymbol m}^{(p)},\tilde{\boldsymbol h}^{(p)},{\boldsymbol k}^{(p)})$ is the unique solution of the decoupled discrete-time FBSDE system:
	\begin{equation}
		\label{eq:decoupled:limit:discrete}
		\begin{split}
			&m_{(\ell+1)T/p}^{(p)} = 
				m_{\ell T/p}^{(p)}	
				- \tfrac{T}p \eta_{\ell T/p}^{(p)} m_{\ell T/p}^{(p)}   
				 + \varepsilon \bigl( W_{(\ell+1)T/p} - W_{\ell T/p} \bigr),
				\\
			&\tilde h_{\ell T/p}^{(p)} = \tilde h_{(\ell+1)T/p}^{(p)} +   \tfrac{T}{p}  
		Q^\dagger f\bigl(m_{(\ell+1)T/p}^{(p)}\bigr) 
		{\mathbf 1}_{\{\ell \leq p-2\}}
		\\
		&\hspace{30pt} 
		-  \tfrac{T}p  
	\Bigl(	
		\eta_{(\ell+1)T/p}^{(p)} + \tfrac{T}p Q^\dagger Q {\mathbf 1}_{\{\ell \leq p-2\}} \Bigr)
		  \tilde h_{\ell T/p}^{(p)} 
		  -  \biggl( \int_{\ell T/p}^{(\ell+1)T/p} k_{s}^{(p)}  \ud s \biggr) \tilde h_{\ell T/p}^{(p)}
		  \\
		&\hspace{30pt} 
		 - \varepsilon \int_{\ell T/p}^{(\ell+1)T/p} k_{s}^{(p)} \ud W_{s},  \quad \ell \in \{0,\cdots,p-1\}, 
			\\
			&m_{0}^{(p)} = {\mathbb E}(X_0), \quad 			\tilde h_{T}^{(p)}=  R^\dagger g\bigl(m_{T}^{(p)}\bigr),
		\end{split}
	\end{equation}
	with an explicit bound on the rate of convergence, namely
	\begin{equation}
		\label{eq:main:thm:1:bound:infinity:2:17}
	\textrm{\rm essup}_{\omega \in \Omega} \Bigl[ \sup_{\ell =0,\cdots,p} \Bigl( \vert m_{\ell T/p}^{(p)} - \overline m_{\ell T/p}^n \vert^2 + \vert 
	\varpi^{-1} \tilde h_{\ell T/p}^{(p)}  - h_{\ell T/p}^n \vert^2 \Bigr) \Bigr] \leq  \varpi^{-2n} \exp\bigl(C \varepsilon^{-2}\bigr),
	\end{equation}
for a constant $C$ that also depends on  $d$, $T$, 
$\| f\|_{1,\infty}$, $\| g \|_{1,\infty}$, $\vert Q\vert$ and $\vert R\vert$.

Moreover, 
if we extend ${\boldsymbol m}^{(p)}$ by continuous interpolation to the entire $[0,T]$
and if we call ${\boldsymbol h}^{(p)}$ the piecewise constant extension of $\tilde{\boldsymbol h}^{(p)}$
to the entire $[0,T]$, then, 
up to a modification of the constant $C$, the weak error of the scheme for the
	Fortet-Mourier distance satisfies
			\begin{equation}
			\label{eq:main:thm:1:bound:infinity:2}
\begin{split}
	  &\sup_{F} \Bigl\vert {\mathbb E}^{\varpi {\boldsymbol h}^n \hspace{-2pt}/\varepsilon}
	  \Bigl[ F\bigl(\overline{\boldsymbol m}^n,{\boldsymbol h}^n\bigr) \Bigr] - 
	  {\mathbb E}^{ {\boldsymbol h}^{(p)} \hspace{-2pt} /\varepsilon}
	  \Bigl[ F\bigl({\boldsymbol m}^{(p)}, \varpi^{-1} {\boldsymbol h}^{(p)} \bigr) \Bigr]
\Bigr\vert 
 \leq   \varpi^{- n} \exp\bigl(C \varepsilon^{-2}\bigr),
 \end{split}
 \end{equation}	
%\textcolor{red}{should be updated}
%
%With $({\boldsymbol m},{\boldsymbol h})$ denoting the solution of 
%		\eqref{eq:decoupled:limit} (or, equivalently, the unique equilibrium of the game with common noise) \textcolor{black}{and with the notations 
%		\eqref{eq:hn+1:tildeh:n+1}
%		and
%\eqref{eq:update}}, the weak error of the scheme for the
%	Fortet-Mourier distance is upper bounded as follows: 
%	\begin{equation}
%	\label{eq:main:thm:2}
%	  \sup_{F} \Bigl\vert {\mathbb E}^{{\boldsymbol h}^n}
%	  \Bigl[ F\bigl(\overline{\boldsymbol m}^n,{\boldsymbol h}^n\bigr) \Bigr] - 
%	  {\mathbb E}^{{\boldsymbol h}}
%	  \Bigl[ F\bigl({\boldsymbol m},{\boldsymbol h}\bigr) \Bigr]
%\Bigr\vert \leq \frac{C}{\varepsilon} \Bigl( \frac{1}{n} + \frac{(1+\ln(p))^{1/2}}{p^{1/2}} \Bigr),	
%\end{equation}
the supremum being taken over all the functions 
$F$ on ${\mathcal C}([0,T];{\mathbb R}^d \times {\mathbb R}^d)$ that are 
bounded by $1$ and $1$-Lipschitz continuous.
\end{thm}
The following comments are in order:
\begin{itemize}
\item 

 In 
		\eqref{eq:decoupled:limit:discrete}, 
		${\boldsymbol m}^{(p)}$ and $\tilde{\boldsymbol h}^{(p)}$
		are implicitly understood as 
		${\boldsymbol m}^{(p)}=(m_{\ell T/p}^{(p)})_{\ell =0,\cdots,p}$ 
		and 
		$\tilde{\boldsymbol h}^{(p)}=(h_{\ell T/p}^{(p)})_{\ell =0,\cdots,p}$,
		with $m_{\ell T/p}^{(p)}$
		 and $h_{\ell T/p}^{(p)}$ being ${\mathbb R}^d$-valued and 
		 ${\mathcal F}^{p,{\boldsymbol W}}_{\ell T/p}$-measurable for each $\ell \in \{0,\cdots,p\}$. 
		 
		 \item The process ${\boldsymbol k}^{(p)}$ is understood as $(k_s^{(p)})_{0 \leq s \leq T}$. It is ${\mathbb R}^{d \times d}$-valued and 
		 ${\mathbb F}^{\boldsymbol W}$-progressively measurable. 
		 \end{itemize} 
		 Existence and uniqueness of a solution to 
		\eqref{eq:decoupled:limit:discrete} is in fact ensured by the following proposition, whose proof is deferred to 
		\S \ref{subsubse:2.2.4:correc}. The statement also explains the need for a threshold on the product $p \varepsilon^2$
		in the two bounds \eqref{eq:main:thm:1:bound:infinity:2:17}
			and \eqref{eq:main:thm:1:bound:infinity:2}.  
	
	\begin{prop}
\label{prop:!:FBSDE:p-discrete}
There exists a constant $c$, depending on  $d$, $T$, 
$\| f\|_{1,\infty}$, $\| g \|_{1,\infty}$, $\vert Q\vert$ and $\vert R\vert$, 
such that for $p \varepsilon^2 \geq c$, 
the backward equation in 
		\eqref{eq:decoupled:limit:discrete} 
		admits a unique solution 
		$(\tilde h_{\ell T/p}^{(p)})_{\ell=0,\cdots,p}$, 
		with $\tilde h_{\ell T/p}^{(p)} \in L^2(\Omega,{\mathcal F}_{\ell T/p}^{p,{\boldsymbol W}},{\mathbb P};{\mathbb R}^d)$
		for each $\ell \in \{0,\cdots,p\}$. 
		
		The process $(k_t^{(p)})_{0 \leq t \leq T}$ is ${\mathbb F}^{\boldsymbol W}$-progressively measurable and is square integrable over $[0,T] \times \Omega$ ($\Omega$ being equipped with ${\mathbb P}$). 
\end{prop}
	 
We also feel useful to clarify the meaning of the continuous interpolation of  ${\boldsymbol m}^{(p)}$ in \eqref{eq:main:thm:1:bound:infinity:2}. In fact, the definition
 is similar to \eqref{eq:linear:interpolated}:
 \begin{equation} 
\label{eq:mvarpi:extended}
 m_{t}^{(p)} := m_{\tau_{p}(t)}^{(p)} + \frac{p (t- \tau_{p}(t))}{T} \bigl( m^{(p)}_{\tau_{p}(t)+T/p} - m^{(p)}_{\tau_{p}(t)} \bigr),
 \quad t \in [0,T].
 \end{equation} 
Moreover, the definition of the piecewise constant extension of $\tilde{\boldsymbol h}^{(p)}$
is similar to 
\eqref{eq:hn+1:tildeh:n+1}:
\begin{equation}
\label{eq:hpvarpi:extended}
	h^{(p)}_{t} := \tilde{h}^{(p)}_{\tau_{p}(t)}, \quad t \in [0,T].
\end{equation}

It is worth noting  that $({\boldsymbol m}^{(p)}, {\boldsymbol h}^{(p)})$ (hence extended to the entire $[0,T]$ as 
above) is the unique solution of the
following two-step fixed point problem, which is nothing but 
the
discrete-time  version of the MFG
with common noise addressed in Theorem 
\ref{main:thm:1}: 

$(i)$ The solution of 
\begin{equation}
\label{eq:optimization:exploration:276}
\textrm{\rm argmin}_{\boldsymbol \alpha}
	\biggl( {\mathbb E}^{ {\boldsymbol h}^{(p)} \hspace{-2pt} / \varepsilon}
	\Bigl[ {\mathcal R}^{p,X_0}  \Bigl(  {\boldsymbol \alpha} +  \varepsilon \dot{\boldsymbol W}^{p,{\boldsymbol h}^{(p)} \hspace{-2pt} /\varepsilon} ;  {\boldsymbol m}^{(p)};0 \Bigr) \Bigr] - \tfrac12 d\,    \varepsilon^2 p \biggr),
	\end{equation}
	which is also 
\begin{equation} 
\label{eq:mfg:p:1}
\textrm{\rm argmin}_{\boldsymbol \alpha}
	\biggl( {\mathbb E}^{ {\boldsymbol h}^{(p)} \hspace{-2pt}  / \varepsilon}
	\Bigl[ {\mathcal R}^{p,   X_0} \Bigl(    {\boldsymbol \alpha}  ;  {\boldsymbol m}^{(p)};   \varepsilon {\boldsymbol W}^{p, {\boldsymbol h}^{(p)}
	\hspace{-2pt} /\varepsilon} \Bigr) \Bigr] \biggr)
	\end{equation}
is given by 
\begin{equation}
\label{eq:mfg:p:3}
\alpha_{t}^{(p),\star}  = - \biggl( \eta_{\tau_{p}(t)}^{(p)} X_{\tau_{p}(t)}^{(p),\star}  +  h_{t}^{(p)}   \biggr), \quad t \in [0,T], 
\end{equation}
where 
\begin{equation}
\label{eq:mfg:p:3:bbis}
\begin{split} 
\ud X_{t}^{(p),\star}  
&=
 -  \eta_{\tau_{p}(t)}^{(p)} X_{\tau_{p}(t)}^{(p),\star}  
 \ud t +   \sigma \ud B_{t} + \varepsilon \ud W_{t}^{p}
 \\
&= \alpha_{t}^{(p),\star} \ud t +   \sigma \ud B_{t} + \varepsilon \ud W_{t}^{p,{\boldsymbol h}^{(p)} \hspace{-2pt}/\varepsilon}, \quad t \in [0,T], 
\end{split}
\end{equation}
with $X_0$ as initial condition.

 $(ii)$ The process obtained by taking the conditional expectation of ${\boldsymbol X}^{(p),\star}$ given $\sigma({\boldsymbol W}^p)$, namely 
 $({\mathbb E}[X_t^{(p),\star} \, \vert \, \sigma({\boldsymbol W}^p)])_{0 \leq t \leq T}$, solves the forward equation in 
\eqref{eq:decoupled:limit:discrete}, i.e.,
\begin{equation}
\label{eq:mfg:p:2} 
{\mathbb E}\bigl[X_{\ell T/p}^{(p),\star} \, \vert \, \sigma({\boldsymbol W}^p) \bigr] = m^{(p)}_{\ell T/p}, \quad \ell =0,\cdots,p. 
\end{equation}   

Indeed, by a straightforward adaptation of Lemma 
\ref{lem:4},
the system
		\eqref{eq:decoupled:limit:discrete}
 can be shown 
 to characterize the solution to the fixed point problem associated with the two items 
$(i)$ and $(ii)$ right above. This proves in particular that 
the discrete-time MFG
constructed on the top of the cost functional 
\eqref{eq:optimization:exploration:276} has a unique solution when $p \varepsilon^2 \geq c$, which is
given by Proposition 
\ref{prop:!:FBSDE:p-discrete}. 

\begin{rem}
\label{rem:another:extension:tildeh}
As made in clear in the forthcoming proof of Proposition 
\ref{prop:!:FBSDE:p-discrete}, there is another conceivable extension for 
$\tilde{\boldsymbol h}^{(p)}$ (in addition to the extension 
defined in \eqref{eq:hpvarpi:extended}). 
Indeed, very similar to 
\eqref{eq:bsde:exploration}, given ${\boldsymbol h}^{(p)}$
(from 
\eqref{eq:hpvarpi:extended}), one can 
define the extension of 
$\tilde{\boldsymbol h}^{(p)}$
to the entire $[0,T]$ as the solution
of the 
BSDE: 
\begin{equation}
\label{eq:bsde:exploration:p,varpi}
\begin{split}
		&\ud  \tilde h_{t}^{(p)} 
		= \Bigl\{ -  
		{\mathbb E} 
		\Bigl[ 
		Q^\dagger f\bigl(m_{\tau_{p}(t)+T/p}^{(p)} \bigr)\, \vert \, 
		{\mathcal F}^{p,{\boldsymbol W}}_{\tau_p(t)} \Bigr]
		 {\mathbf 1}_{\{ t \leq (p-1)T/p\}}  
		 \\
&\hspace{100pt}		  + 
		\Bigl( \eta_{\tau_{p(t)}+T/p}^{(p)}
		  + \tfrac{T}p Q^\dagger Q 
		  {\mathbf 1}_{\{ t \leq (p-1)T/p\}}
		  	\Bigr) 		  
		    h_{t}^{(p)} \Bigr\}
 \ud t  
		  \\
&\hspace{50pt}		  +  k_{t}^{(p)} h_{t}^{(p)} \ud t 
		+ \varepsilon k_{t}^{(p)} \ud W_{t}, \quad t \in [0,T),
%\nonumber
		%+ \ud W_{t} - \ud W_{t}^p, \quad t \in [0,T], 
		\\
		&\tilde h_{T}^{(p)} =   R^\dagger g \bigl( m^{(p)}_{T} \bigr). 
			\end{split}
			\end{equation}
Consistency of this extension is explained in the proof of 
Proposition 
\ref{prop:!:FBSDE:p-discrete}. The time-continuous 
process $\tilde{\boldsymbol h}^{p}$ defined as the solution of 
\eqref{eq:bsde:exploration:p,varpi}
coincides at times $(\ell T/p)_{\ell=0,\cdots,p}$
with the solution 
$(\tilde{h}_{\ell T/p}^{(p)})_{\ell=0,\cdots,p}$ of
		\eqref{eq:decoupled:limit:discrete}. Moreover, 
		the process ${\boldsymbol k}^{(p)}$ in 
		\eqref{eq:bsde:exploration:p,varpi} coincides with 
${\boldsymbol k}^{(p)}$ in 
\eqref{eq:decoupled:limit:discrete}. 
\end{rem}

\begin{rem}
\label{rem:thm:2:00}
Similar to the proof of Theorem 
\ref{main:thm:1}, 
the proof of 
Theorem \ref{main:thm:2}
also shows that 
\begin{equation*}
d_{\textrm{\rm TV}}\bigl({\mathbb P}^{\varpi {\boldsymbol h}^n \hspace{-2pt}  /\varepsilon},{\mathbb P}^{ {\boldsymbol h}^{(p)} \hspace{-2pt} /\varepsilon}
\bigr) \leq 
  \varpi^{- n} \exp\bigl(C \varepsilon^{-2}\bigr).
\end{equation*}
\end{rem}

\color{black}
\begin{rem}
\label{rem:thm:2}
Following 
Remark 
\ref{rem:thm:1}
about the scope of Theorem 
\ref{main:thm:1} to the higher dimensional setting, we could think of tracking the dependence of the constant 
$C$ (in 		\eqref{eq:decoupled:limit}
			and	\eqref{eq:main:thm:1:bound:infinity}), which is here independent of $p$, upon the dimension $d$. 
Although we prefer not to address this question in full detail, we 
insist on the fact that the conclusion of Remark \ref{rem:thm:1}
also holds true here, meaning that the constant $C$ cannot be better than 
$O(\exp(O(\sqrt{d})))$ in simple cases when $Q$ and $R$ are the identity matrices and $f$ and $g$ are diagonal. 
Intuitively, we cannot expect a constant that would be, asymptotically in $d$ and uniformly in $p$, better than the constant appearing 
in the statement of 
Theorem 
\ref{main:thm:1}, as otherwise we could take the limit $p \rightarrow \infty$ in the statement of Theorem 
\ref{main:thm:2} and then get a better estimate in Theorem 
\ref{main:thm:1}. 
%In fact, the same conclusion can be drawn by repeating directly the proof of 
%Remark \ref{rem:thm:1}: 
%the respective analogues of Lemma 
%\ref{lem:1} and Proposition 
%\ref{prop:1} are
%Lemma 
%\ref{lem:1:b}
%and Proposition 
%\ref{prop:2}. A careful inspection of the proofs of the latter 
%would permit to adapt the proofs of Remarks \ref{rem:lem:1} 
%and 
% \ref{rem:prop:1}, on which the proof of Remark \ref{rem:thm:1} relies. 
\end{rem}
\color{black}

\subsubsection{Analysis of the convergence}
Conditioning on 
${\boldsymbol W}$ in 
\eqref{eq:bsde:exploration:forward:equation:with:varpi} and recalling 
the two notations 
\eqref{eq:hn+1:tildeh:n+1}
and
\eqref{eq:mn+1:m:n+1:correc}, we obtain 
\begin{equation*}
\ud m_{t}^{n+1} = - \Bigl( \eta^{(p)}_{\tau_p(t)} 
m_{\tau_p(t)}^{n+1} 
+ h_t^{n+1} \Bigr)  \ud t +  \varepsilon \ud W_{t}^{p,\varpi {\boldsymbol h}^n \hspace{-2pt}/\varepsilon}, \quad t \in [0,T] ;
\quad m_0^{n+1} ={\mathbb E}(X_0). 
\end{equation*} 
Thanks to 
\eqref{eq:update},  the identity \eqref{eq:m:n:n+1} becomes
\begin{equation}
\label{eq:mn+1:p}
\ud \overline m_{t}^{n+1} = - \bigl( \eta_{\tau_p(t)}^{(p)} \overline m_{\tau_p(t)}^{n+1} + \tfrac{(1-\varpi^{-1}) \varpi^{-n}}{1-\varpi^{-(n+1)}}  h_{t}^{n+1} \bigr) \ud t + \varepsilon \ud W_{t}^p, \quad t \in [0,T]. 
\end{equation}
As in the 
analysis of \eqref{eq:system:n}, the second term in the middle disappears asymptotically at a geometric rate. 
This follows from the next result, which is an improved version of Lemma \ref{lem:1} and which is also used next  
in place of Lemma 
\ref{lem:2} (the latter relying on a Markovian structure that is not satisfied here):
%(using 
%\eqref{eq:recurrence:cas:discret}
%in order to control inductively the norm of ${\boldsymbol h}^n$
%and, for reasons that will become clear in the proof, yielding the estimate on ${\boldsymbol k}$ in a conditional form):
\begin{lem}
\label{lem:1:b}
With $(\overline{\boldsymbol m}^{n},\tilde {\boldsymbol h}^{n})$ 
being defined as in 
\eqref{eq:bsde:exploration},
%\eqref{eq:hn+1:tildeh:n+1}, 
\eqref{eq:mn+1:m:n+1:correc} 
and \eqref{eq:update}
and 
  $({\boldsymbol m}^{(p)},\tilde{\boldsymbol h}^{(p)})$
as in  
		\eqref{eq:decoupled:limit:discrete}
		and
\eqref{eq:bsde:exploration:p,varpi},
there exists a constant $C_{1}$, only depending on $d$, $T$, 
$\| f\|_{1,\infty}$, $\| g \|_{1,\infty}$, $\vert Q\vert$ and $\vert R\vert$, such that, ${\mathbb P}$ almost surely, 
\begin{equation}
\label{eq:lem:1:b:1}
\vert h_{t}^n \vert \leq C_{1},
\quad t \in [0,T], \ n \geq 1 \, ; 
\quad
\vert \tilde h_{t}^{(p)} \vert \leq C_{1}, \quad t \in [0,T]. 
\end{equation}
Moreover, for the same constant $C_{1}$, there exists $\delta \in (0,1]$, depending on the same parameters as $C_1$, such that, with probability 1 under 
${\mathbb P}$, 
\begin{equation}
\label{eq:lem:1:b:2}
\forall t \in [0,T], \quad {\mathbb E}^{ {\boldsymbol h}^{(p)} \hspace{-2pt} /\varepsilon} \biggl[ \exp \biggl( \delta \varepsilon^2 \int_{t}^T \vert k_{s}^{(p)} \vert^2 \ud s 
\biggr) \, \vert \, {\mathcal F}_{t}^{{\boldsymbol W}} \biggr] \leq  C_1.  
\end{equation}
\end{lem}
The proof of Lemma \ref{lem:1:b} is deferred to \S \ref{subsubse:2.2.4:correc}. 
Taking for granted the statement, we now complete the proof of Theorem 
\ref{main:thm:2}. 

\begin{proof}[Proof of Theorem \ref{main:thm:2}.]
Taking for granted Proposition 
\ref{prop:!:FBSDE:p-discrete}, we assume that the system 
		\eqref{eq:decoupled:limit:discrete}
		has a unique solution. Implicitly, this requires $p \varepsilon^2 \geq c$, but this is indeed one of the assumption 
		 of Theorem \ref{main:thm:2}.
			\vskip 4pt
			
			\textit{First Step.} 
			For any integer $n \geq 1$, we thus compare $(\overline{\boldsymbol m}^{n+1},\tilde{\boldsymbol h}^{n+1})$ and 
			$(\overline{\boldsymbol m}^{(p)},\tilde{\boldsymbol h}^{(p)} \hspace{-2pt} /\varpi)$, 
with the latter two ones being			extended to the entire $[0,T]$ as explained in 
			\eqref{eq:mvarpi:extended} and 
			\eqref{eq:bsde:exploration:p,varpi}. To do so, we first notice that 
			\begin{equation*} 
			\ud m_t^{(p)} = - \eta_{\tau_p(t)}^{(p)} m_{\tau_p(t)}^{(p)} \ud t + \varepsilon \ud W_t^{p}, \quad t \in [0,T].  
			\end{equation*}  
			By 
			\eqref{eq:mn+1:p}
			and 
			\eqref{eq:lem:1:b:1},
			\begin{equation*} 
			\ud \bigl( \overline{m}^{n+1}_t - {m}_t^{(p)} \bigr) 
			= 
			 -  \eta_{\tau_p(t)}^{(p)} \bigl( \overline m_{t}^{n+1} - 
			 {m}_t^{(p)}
		\bigr) \ud t 	 
			 -  O \bigl( \varpi^{-n} \bigr) \ud t, \quad t \in [0,T],  
			\end{equation*} 
			where $\vert O ( \varpi^{-n} ) \vert \leq C \varpi^{-n}$ for a constant $C$ as in the statement of Theorem \ref{main:thm:2}. 
			We easily get
			\begin{equation}
			\label{eq:proof:main:thm:2:1} 
			\sup_{0 \le t \le T} \vert \overline{m}^{n+1}_t - {m}_t^{(p)} \vert \leq C \varpi^{-n}. 
			\end{equation}

			\textit{Second Step.} We turn to the difference $ \tilde{\boldsymbol h}^{n+1} - \tilde{\boldsymbol h}^{(p)} \hspace{-2pt} /\varpi$. 
			We rewrite
\eqref{eq:bsde:exploration} as 
\begin{equation*} 
\begin{split} 
&\ud  \tilde h_{t}^{n+1} 
		= \Bigl\{ - \tfrac1{\varpi} 
		{\mathbb E} 
		\Bigl[ 
		Q^\dagger f\bigl(\overline m_{\tau_{p}(t)+T/p}^{n}\bigr)\, \vert \, 
		{\mathcal F}^{p,{\boldsymbol W}}_{\tau_p(t)} \Bigr]
		 {\mathbf 1}_{\{ t \leq (p-1)T/p\}}  
		 \\
&\hspace{100pt}		  + 
		\Bigl( \eta_{\tau_{p(t)}+T/p}^{(p)}
		  + \tfrac{T}p Q^\dagger Q 
		  {\mathbf 1}_{\{ t \leq (p-1)T/p\}}
		  	\Bigr) 		  
		     h_{t}^{n+1} \Bigr\}
 \ud t  
		  \\
&\hspace{50pt}		  + \varpi k_{t}^{n+1} \bigl( h_{t}^{n} - \tfrac1{\varpi}  {h}_t^{(p)} \bigr) \ud t 
		+ \varepsilon k_{t}^{n+1} \ud W_{t}^{ {\boldsymbol h}^{(p)} \hspace{-2pt} /\varepsilon}, \quad t \in [0,T]. 
\end{split} 
\end{equation*} 
Similarly, we rewrite 
\eqref{eq:bsde:exploration:p,varpi}
as
\begin{equation*}
\begin{split}
		&\ud \bigl(  \tfrac1{\varpi} \tilde h_{t}^{(p)} \bigr) 
		= \Bigl\{ - \tfrac1{\varpi} 
		{\mathbb E} 
		\Bigl[ 
		Q^\dagger f\bigl(m_{\tau_{p}(t)+T/p}^{(p)} \bigr)\, \vert \, 
		{\mathcal F}^{p,{\boldsymbol W}}_{\tau_p(t)} \Bigr]
		 {\mathbf 1}_{\{ t \leq (p-1)T/p\}}  
		 \\
&\hspace{100pt}		  + 
		\Bigl( \eta_{\tau_{p(t)}+T/p}^{(p)}
		  + \frac{T}p Q^\dagger Q 
		  {\mathbf 1}_{\{ t \leq (p-1)T/p\}}
		  	\Bigr) 
			\bigl( \tfrac1{\varpi}		  
		    h_{t}^{(p)} \bigr) \Bigr\}
 \ud t  
		  \\
&\hspace{50pt}		+  \tfrac{\varepsilon}{\varpi} k_{t}^{(p)} \ud W_{t}^{\varpi {\boldsymbol h}^{(p),\varpi} \hspace{-2pt} /\varepsilon}, \quad t \in [0,T),
%\nonumber
		%+ \ud W_{t} - \ud W_{t}^p, \quad t \in [0,T], 
		\\
		&\tfrac1{\varpi} \tilde h_{T}^{(p)} = \tfrac1{\varpi} R^\dagger g \bigl( m^{(p)}_{T} \bigr). 
			\end{split}
			\end{equation*}
			Forming and squaring the difference between the two equations and 
using
			\eqref{eq:proof:main:thm:2:1} (together with Remark
			\ref{rem:main:thm:2},  which gives a bound for ${\boldsymbol \eta}^{(p)}$ independent of $p$), we deduce
			(very like as in 
			\eqref{eq:new:hn+1:h}
and
\eqref{eq:proof:delta})
%\begin{equation*} 
%\begin{split} 
%\ud \bigl\vert \tilde h_t^{n+1} - \tilde h_t^{n+q+1} \bigr\vert^2  
%&\geq - C \varpi^{-2n} \ud t - C 
%\bigl\vert h^{n+1}_t - h_t^{n+q+1} \bigr\vert^2
%\ud t 
%  - C 
%\bigl\vert \tilde h^{n+1}_t - \tilde h_t^{n+q+1} \bigr\vert^2
%\ud t 
%\\
%&\hspace{10pt} -  C \bigl\vert \tilde h_t^{n+1} - \tilde h_t^{n+q+1} \bigr\vert \, \bigl\vert k_t^{n+1} h_t^n - k_t^{n+q+1} h_t^{n+q} \bigr\vert 
% \ud t
% + \varepsilon^2 \bigl\vert k_t^{n+1} - k_t^{n+q+1} \bigr\vert^2 \ud t
%\\
%&\hspace{10pt} + \varepsilon \bigl( h_t^{n+1} - h_t^{n+q+1}\bigr) \cdot \bigl[ \bigl( k_t^{n+1} - k_t^{n+q+1} \bigr) \ud W_t \bigr], \quad t \in [0,T],
%\end{split} 
%\end{equation*} 
%Therefore,  
%using 
%\eqref{eq:lem:1:b:1} once again, 
\begin{equation*} 
\begin{split} 
&\ud \Bigl[ \bigl\vert \tilde h_t^{n+1} - \tfrac1{\varpi}   \tilde h_t^{(p)} \bigr\vert^2  \Bigr]
\\
&\geq - C \varpi^{-2n} \ud t - C 
\bigl\vert h^{n+1}_t - \tfrac1{\varpi}   h_t^{(p)} \bigr\vert^2
\ud t 
  - C 
\bigl\vert \tilde h^{n+1}_t - \tfrac1{\varpi}   \tilde h_t^{(p)} \bigr\vert^2
\ud t 
\\
&\hspace{15pt} -  C \vert k_t^{(p)} \vert \, \bigl\vert \tilde h_t^{n+1} - \tfrac1{\varpi}   \tilde h_t^{(p)} \bigr\vert \, \bigl\vert  h_t^n - 
\tfrac1{\varpi}   h_t^{(p)} \bigr\vert 
 \ud t - C \bigl\vert  \tilde h_t^{n+1} - \tfrac1{\varpi}  \tilde h_t^{(p)} \bigr\vert 
  \bigl\vert  k_t^{n+1} -  \tfrac1{\varpi}   k_t^{(p)} \bigr\vert \ud t 
  \\
&\hspace{15pt}  
 + \varepsilon^2 \bigl\vert k_t^{n+1} - \tfrac1{\varpi}   k_t^{(p)} \bigr\vert^2 \ud t
 + \varepsilon \bigl( h_t^{n+1} - \tfrac1{\varpi}   \tilde h_t^{(p)} \bigr) \cdot \bigl[ \bigl( k_t^{n+1} - \tfrac1{\varpi}   k_t^{(p)} \bigr) \ud W_t^{ {\boldsymbol h}^{(p)} \hspace{-2pt} /\varepsilon} \bigr],
\end{split} 
\end{equation*} 
with the constant $C$ being allowed to vary from line to line as long as it depends on the same parameters as those indicated in the statement of 
Theorem 
\ref{main:thm:2}. 
With $\delta$ as in 
\eqref{eq:lem:1:b:2}, this leads to 
\begin{equation*} 
\begin{split} 
&\ud \Bigl[ \bigl\vert \tilde h_t^{n+1} - \tfrac1{\varpi} \tilde h_t^{(p)} \bigr\vert^2  \Bigr]
\\
&\geq - C \varpi^{-2n} \ud t - C 
\bigl\vert h^{n+1}_t - \tfrac1{\varpi}  h_t^{(p)} \bigr\vert^2
\ud t 
  - C 
\bigl\vert \tilde h^{n+1}_t -\tfrac1{\varpi}   \tilde h_t^{(p)} \bigr\vert^2
\ud t 
\\
&\hspace{15pt} -  \tfrac{C}{\delta \varepsilon^2}    \bigl\vert  h_t^n -  \tfrac1{\varpi}   h_t^{(p)} \bigr\vert^2 
 \ud t - \bigl(  \delta \varepsilon^2 \vert k_t^{(p)} \vert^2 + \tfrac{C}{\varepsilon^2} \bigr)\bigl\vert  \tilde h_t^{n+1} - 
 \tfrac1{\varpi}   \tilde h_t^{(p)} \bigr\vert^2  \ud t 
  \\
&\hspace{15pt}  
 + \tfrac{\varepsilon^2}2 \bigl\vert k_t^{n+1} - \tfrac1{\varpi}   k_t^{(p)} \bigr\vert^2 \ud t
  + \varepsilon \bigl( h_t^{n+1} - \tfrac1{\varpi}   \tilde h_t^{(p)} \bigr) \cdot \bigl[ \bigl( k_t^{n+1} - \tfrac1{\varpi}   k_t^{(p)} \bigr) \ud W_t^{{\boldsymbol h}^{(p)} \hspace{-2pt} /\varepsilon} \bigr]. 
\end{split} 
\end{equation*}
Using the fact that, by convention, $\varepsilon$ is taken
in $(0,1)$, we can get rid of the last term on the first line (which is dominated by the last term on the second line). Then, letting
\begin{equation} 
\label{eq:etp:def}
E_t^{(p)} := \exp \biggl( \int_0^t 
 \bigl(  \delta \varepsilon^2 \vert k_s^{(p)} \vert^2 + \tfrac{C}{\delta \varepsilon^2} \bigr) \ud s \biggr), \quad t \in [0,T],
\end{equation} 
and assuming without any loss of generality that $\delta \in (0,1)$, 
we obtain 
\begin{align} 
&\ud \Bigl[
E_t^{(p)}
 \bigl\vert \tilde h_t^{n+1} - 
 \tfrac1{\varpi}   \tilde h_t^{(p)} \bigr\vert^2  
 \Bigr] \nonumber
 \\
&\geq - C 
E_t^{(p)} 
 \Bigl[ 
\varpi^{-2n}
+
\bigl\vert h^{n+1}_t - \tfrac1{\varpi}   h_t^{(p)} \bigr\vert^2
+  \tfrac{1}{\delta \varepsilon^2}    \bigl\vert  h_t^n - \tfrac1{\varpi}   h_t^{(p)} \bigr\vert^2 
\Bigr]
 \ud t
  + \tfrac{\varepsilon^2}2 E_t^{(p)} \bigl\vert k_t^{n+1} - \tfrac1{\varpi}   k_t^{(p)} \bigr\vert^2 \ud t \nonumber
 % - \bigl(  \tfrac12 \varepsilon^2 \vert k_t^{n+1} \vert^2 + \tfrac{C}{\varepsilon^2} \bigr)\bigl\vert  \tilde h_t^{n+1} - \tilde h_t^{n+q+1} \bigr\vert^2  \ud t 
  \\
&\hspace{10pt}  
 +  \varepsilon  E_t^{(p)}\bigl( h_t^{n+1} - \tfrac1{\varpi}   h_t^{(p)} \bigr) \cdot \Bigl[ \bigl( k_t^{n+1} - \tfrac1{\varpi}   k_t^{(p)} \bigr) \ud W_t^{{\boldsymbol h}^{(p)}
 \hspace{-2pt} 
 /\varepsilon} \Bigr]. 
 \label{eq:BDG:correc}
 \end{align} 
Then, using in addition the fact that 
$\tilde h_T^{n+1} - 
   \tilde h_T^{(p)} / \varpi 
 = [ g ( \overline m_{T}^{n})  - g( m_T^{(p)})]/\varpi$, we get
\begin{equation*} 
\begin{split} 
&E_t^{(p)} 
 \bigl\vert \tilde h_t^{n+1} - \tfrac1{\varpi}   \tilde h_t^{(p)} \bigr\vert^2
 + \tfrac{\varepsilon^2}2 
 {\mathbb E}^{{\boldsymbol h}^{(p)} \hspace{-2pt} /\varepsilon} 
 \biggl[ \int_t^T 
  E_s^{(p)} \bigl\vert k_s^{n+1} - \tfrac1{\varpi}   k_s^{(p)} \bigr\vert^2 \ud s \, \vert \, {\mathcal F}_t^{\boldsymbol W} 
  \biggr]
 \\
 &\hspace{5pt} \leq C {\mathbb E}^{ {\boldsymbol h}^{(p)}/\varepsilon} 
 \biggl[ 
 E_T^{(p)} \varpi^{-2n}
 +
 \int_t^T E_s^{(p)} 
  \Bigl[ 
\varpi^{-2n}
+
\bigl\vert h^{n+1}_s - \tfrac1{\varpi}   h_s^{(p)} \bigr\vert^2
+  \tfrac{1}{\delta \varepsilon^2}    \bigl\vert  h_s^n -  \tfrac1{\varpi}   h_s^{(p)} \bigr\vert^2 
\Bigr]
 \ud s \, \vert \, {\mathcal F}_t^{\boldsymbol W} \biggr]. 
 \end{split}
 \end{equation*} 
 Take now $t = \ell T/p$, for some $\ell \in \{0,\cdots,p-1\}$. Then, 
 with $C$ and $\delta$ as above, 
\begin{equation*} 
\begin{split}
&E^{(p)}_{\ell T/p} 
 \bigl\vert h_{\ell T/p} ^{n+1} -  \tfrac1{\varpi}   h_{\ell T/p}^{(p)} \bigr\vert^2
 \\
 &\leq C {\mathbb E}^{ {\boldsymbol h}^{(p)} \hspace{-2pt} /\varepsilon} 
 \biggl[ 
 E_{T}^{(p)}  
\varpi^{-2n}
 +
 \int_{\ell T/p}^T E_{s}^{(p)} 
  \Bigl[ 
\bigl\vert h^{n+1}_s - \tfrac1{\varpi}   h_s^{(p)} \bigr\vert^2
+  \tfrac{1}{\delta \varepsilon^2}    \bigl\vert  h_s^n -  \tfrac1{\varpi}   h_s^{(p)} \bigr\vert^2 
\Bigr]
 \ud s \, \vert \, {\mathcal F}_{\ell T/p}^{\boldsymbol W} \biggr]
 \\
 &\leq 
C \exp \biggl(   \delta \varepsilon^2  \int_0^{\ell T/p}
\vert k_s^{(p)} \vert^2   \ud s \biggr)
{\mathbb E}^{ {\boldsymbol h}^{(p)} \hspace{-2pt} /\varepsilon} 
\biggl[ 
\exp \biggl(   \delta \varepsilon^2  \int_{\ell T/p}^T 
\vert k_s^{(p)} \vert^2   \ud s \biggr) \, \vert \, {\mathcal F}_{\ell T/p}^{\boldsymbol W} 
\biggr]
\\
&\hspace{1pt} \times
\biggl( 
 \exp \bigl( 
C 
\tfrac{T}{\delta \varepsilon^2} 
\bigr)
\varpi^{-2n} 
  +
  \tfrac{T}{p} \sum_{k=\ell}^{p-1}  
 \exp \bigl( 
C 
\tfrac{k T}{p \delta \varepsilon^2} 
\bigr)
 \underset{\small \omega \in \Omega}{  \textrm{\rm essup}}
 \Bigl[ 
% \underset{\small \omega \in \Omega}{  \textrm{\rm essup}}
 \bigl\vert h_{k T/p} ^{n+1} - \tfrac1{\varpi}   h_{k T/p}^{(p)} \bigr\vert^2
 +
 \tfrac{1}{\delta \varepsilon^2} 
 %\underset{\small \omega \in \Omega}{  \textrm{\rm essup}}
 \bigl\vert h_{k T/p} ^{n} - \tfrac1{\varpi}   h_{k T/p}^{(p)} \bigr\vert^2
 \Bigr] \biggr). 
 \end{split}
 \end{equation*}  
By 
\eqref{eq:lem:1:b:2}, we
have a bound for the expectation on the 
penultimate line. Also, recalling the
form of $E^{(p)}_{\ell T/p}$ in 
\eqref{eq:etp:def}, we 
can divide 
both sides of the inequality
by 
$\exp (   \delta \varepsilon^2  \int_0^{\ell T/p}
\vert k_s^{(p)} \vert^2   \ud s)$. We 
 get 
\begin{equation*} 
\begin{split}
&\exp \bigl( 
C 
\tfrac{\ell T}{p \delta \varepsilon^2} 
\bigr) \underset{\small \omega \in \Omega}{  \textrm{\rm essup}}
 \bigl\vert h_{\ell T/p} ^{n+1} -  \tfrac1{\varpi}   h_{\ell T/p}^{(p)} \bigr\vert^2  \phantom{\Bigl(}
 \\
 &\leq 
C \exp \bigl( 
C 
\tfrac{T}{\delta \varepsilon^2} 
\bigr)
\varpi^{-2n} 
 + \tfrac{CT}{p}
 \sum_{k=\ell}^{p-1} 
 \exp \bigl( 
C 
\tfrac{k T}{p \delta \varepsilon^2} 
\bigr)
 \underset{\small \omega \in \Omega}{  \textrm{\rm essup}} \Bigl[ 
 \bigl\vert h_{k T/p} ^{n+1} -  \tfrac1{\varpi}   h_{k T/p}^{(p)} \bigr\vert^2
  +
 \tfrac{1}{\delta \varepsilon^2} 
 \bigl\vert h_{k T/p} ^{n} - \tfrac1{\varpi}   h_{k T/p}^{(p)} \bigr\vert^2
 \Bigr]. 
 \end{split}
 \end{equation*}
 For $C T/p \leq 1/2$ (which assumption is consistent with the lower bound $p \varepsilon^2 \geq c$), we get 
 (for a possibly new value of $C$) 
 \begin{equation*} 
 \begin{split} 
 &\exp \bigl( 
C 
\tfrac{\ell T}{p \delta \varepsilon^2} 
\bigr)  \underset{\small \omega \in \Omega}{  \textrm{\rm essup}}
 \bigl\vert h_{\ell T/p} ^{n+1} -  \tfrac1{\varpi}   h_{\ell T/p}^{(p)} \bigr\vert^2
  \leq  A_\ell + 
  \tfrac{CT}{p}
 \sum_{k=\ell+1}^{p-1} 
 \exp \bigl( 
C 
\tfrac{k T}{p \delta \varepsilon^2} 
\bigr)
 \underset{\small \omega \in \Omega}{  \textrm{\rm essup}}
 \bigl\vert h_{k T/p} ^{n+1} -  \tfrac1{\varpi}   h_{k T/p}^{(p)} \bigr\vert^2,
 \end{split} 
 \end{equation*} 
 with the notation 
 \begin{equation*} 
 A_\ell := 
C  \exp \bigl( 
C 
\tfrac{T}{\delta \varepsilon^2} 
\bigr)
+ 
 \tfrac{CT}{p \delta \varepsilon^2}
 \sum_{k=\ell}^{p-1} 
 \exp \bigl( 
C 
\tfrac{k T}{p \delta \varepsilon^2} 
\bigr)
 \textrm{\rm essup}_{\omega \in \Omega}  
 \bigl\vert h_{k T/p} ^{n} -  \tfrac1{\varpi}   h_{k T/p}^{(p)} \bigr\vert^2.
 \end{equation*}
 By the discrete version of Gronwall's lemma, we obtain
 (for a possibly new value of $C$):
  \begin{equation} 
  \label{eq:thm:2:17:correc}
\begin{split}
&\exp \bigl( 
C 
\tfrac{\ell T}{p \delta \varepsilon^2} 
\bigr) \textrm{\rm essup}_{\omega \in \Omega}  
 \bigl\vert h_{\ell T/p} ^{n+1} -  \tfrac1{\varpi}   h_{\ell T/p}^{(p)} \bigr\vert^2
 \\
 &\hspace{15pt} \leq C A_\ell \leq 
 C  \exp \bigl( 
C 
\tfrac{T}{\delta \varepsilon^2} 
\bigr)
 +
  \tfrac{CT}{p\delta \varepsilon^2}
 \sum_{k=\ell}^{p-1} 
 \exp \bigl( 
C 
\tfrac{k T}{p \delta \varepsilon^2} 
\bigr)
 \textrm{\rm essup}_{\omega \in \Omega}  
 \bigl\vert h_{k T/p} ^{n} - \tfrac1{\varpi}   h_{k T/p}^{(p)} \bigr\vert^2. 
 \end{split}
 \end{equation}
 And then, for some real $\lambda >0$, 
\begin{equation*} 
\begin{split}
&\sum_{\ell=0}^{p-1} 
\exp \bigl( 
\lambda 
\tfrac{\ell T}{p \delta \varepsilon^2} 
\bigr)
\exp \bigl( 
C 
\tfrac{\ell T}{p \delta \varepsilon^2} 
\bigr)  \underset{\small \omega \in \Omega}{  \textrm{\rm essup}} 
 \bigl\vert h_{\ell T/p} ^{n+1} - \tfrac1{\varpi}    h_{\ell T/p}^{(p)} \bigr\vert^2
 \\
 &\leq 
 C
 \exp \bigl( 
C 
\tfrac{T}{\delta \varepsilon^2} 
\bigr)
 \sum_{\ell=0}^{p-1} 
\exp \bigl( 
\lambda 
\tfrac{\ell T}{p \delta \varepsilon^2} 
\bigr) \varpi^{-2n}
\\
&\hspace{15pt} 
 + 
  \tfrac{C T}{p \delta \varepsilon^2}
 \sum_{k=0}^{p-1} 
 \biggl(\sum_{\ell=0}^k 
\exp \bigl( 
\lambda 
\tfrac{\ell T}{p \delta \varepsilon^2} 
\bigr) 
 \biggr)
 \exp \bigl( 
C 
\tfrac{k T}{p \delta \varepsilon^2} 
\bigr)
  \underset{\small \omega \in \Omega}{  \textrm{\rm essup}}
 \bigl\vert h_{k T/p} ^{n} - \tfrac1{\varpi}   h_{k T/p}^{(p)} \bigr\vert^2
 \\
 &\leq 
  \tfrac{C}{\exp(\lambda T/(p \delta \varepsilon^2))-1}
 \exp \bigl( 
(C+\lambda) 
\tfrac{T}{\delta \varepsilon^2} 
\bigr)
\varpi^{-2n}
\\
&\hspace{15pt} 
 +    
  \tfrac{CT}{p \delta \varepsilon^2} 
 \tfrac{ \exp(\lambda T/(p \delta \varepsilon^2))}{\exp(\lambda T/(p \delta \varepsilon^2))-1}
 \sum_{k=0}^{p-1}  
 \exp \bigl( 
(C+\lambda) 
\tfrac{k T}{p \delta\varepsilon^2} 
\bigr)
 \underset{\small \omega \in \Omega}{  \textrm{\rm essup}}
 \bigl\vert h_{k T/p} ^{n} - \tfrac1{\varpi}   h_{k T/p}^{(p)} \bigr\vert^2.
 \end{split}
 \end{equation*} 
% Finally, 
% \begin{equation*}
% \begin{split}
%&\sum_{\ell=0}^{p-1} 
%\exp \bigl( 
%\lambda 
%\tfrac{\ell T}{p \delta \varepsilon^2} 
%\bigr)
%\exp \bigl( 
%C 
%\tfrac{\ell T}{p \delta \varepsilon^2} 
%\bigr)  \underset{\small \omega \in \Omega}{  \textrm{\rm essup}}
% \bigl\vert h_{\ell T/p} ^{n+1} -  \tfrac1{\varpi}   h_{\ell T/p}^{(p)} \bigr\vert^2
%\\
% &\leq 
%  \tfrac{C T}{p \delta \varepsilon^2} 
%  \tfrac{\exp(\lambda T/(p \delta \varepsilon^2))}{\exp(\lambda T/(p \delta \varepsilon^2))-1}
% \sum_{\ell=0}^{p-1} 
%\exp \bigl( 
%\lambda 
%\tfrac{\ell T}{p \delta \varepsilon^2} 
%\bigr)
%\exp \bigl( 
%C 
%\tfrac{\ell T}{p \delta  \varepsilon^2} 
%\bigr)  \underset{\small \omega \in \Omega}{  \textrm{\rm essup}}
% \bigl\vert h_{\ell T/p} ^{n} -  \tfrac1{\varpi}   h_{\ell T/p}^{(p)}  \bigr\vert^2
% \\
% &\hspace{15pt} + 
%  \tfrac{CT}{p} 
%  \tfrac{ \exp(\lambda T/(p \delta \varepsilon^2))}{\exp(\lambda T/(p \delta  \varepsilon^2))-1}
% \tfrac1{\exp((C+\lambda) T/(p \delta  \varepsilon^2))-1}
% \exp \bigl( [C+\lambda] \tfrac{T}{\delta \varepsilon^2} \bigr)  \varpi^{-2n}.  
% \end{split}
% \end{equation*}
 Now, 
 \begin{equation*}
 \begin{split} 
 \tfrac{T}{p \delta  \varepsilon^2} 
  \tfrac{\exp(\lambda T/(p \delta  \varepsilon^2))}{\exp(\lambda T/(p \delta  \varepsilon^2))-1}
  &=  \tfrac{T}{p \delta  \varepsilon^2}  
 \tfrac{1}{1-\exp(-\lambda T/(p  \delta  \varepsilon^2))}
 = \tfrac{1}{\lambda} \varphi \bigl(  \tfrac{ \lambda T}{p \delta  \varepsilon^2} \bigr),
 \end{split} 
 \end{equation*} 
 where 
 \begin{equation*}
 \varphi(x) = \tfrac{x}{1- \exp(-x)}, \quad x>0. 
 \end{equation*} 
 Clearly, $\varphi(x) \rightarrow 1$ as $x \rightarrow 0$. 
 Take now $\lambda=6C$ and $p \varepsilon^2$ large enough so that 
 $\varphi(\tfrac{ \lambda T}{p \delta  \varepsilon^2} )
 =
 \varphi(\tfrac{ 6 C T}{p \delta  \varepsilon^2} )
 $ is less than $2$. 
 We have 
  \begin{equation*}
 \begin{split}
&\sum_{\ell=0}^{p-1} 
\exp \bigl( 
7 C
\tfrac{\ell T}{p \delta \varepsilon^2} 
\bigr)
%\exp \bigl( 
%C 
%\tfrac{\ell T}{p \delta \varepsilon^2} 
%\bigr) 
 \underset{\small \omega \in \Omega}{  \textrm{\rm essup}}
 \bigl\vert h_{\ell T/p} ^{n+1} -  \tfrac1{\varpi}   h_{\ell T/p}^{(p)} \bigr\vert^2
\\
 &\leq 
  \tfrac{1}{3} 
 \sum_{\ell=0}^{p-1} 
\exp \bigl( 
7 C
\tfrac{\ell T}{p \delta \varepsilon^2} 
\bigr)
%\exp \bigl( 
%C 
%\tfrac{\ell T}{p \delta  \varepsilon^2} 
%\bigr) 
 \underset{\small \omega \in \Omega}{  \textrm{\rm essup}}
 \bigl\vert h_{\ell T/p} ^{n} -  \tfrac1{\varpi}   h_{\ell T/p}^{(p)} \bigr\vert^2
 + 
  \tfrac{C}{1- \exp(- 6 C T/(p \delta \varepsilon^2))}
 \exp \bigl( 7 C \tfrac{T}{\delta \varepsilon^2} \bigr)  \varpi^{-2n}.
 \end{split}
 \end{equation*} 
Here, we notice that (assuming without any loss of generality that $C \geq 1$) 
\begin{equation*} 
\begin{split}
&  \tfrac{C}{1- \exp(- 6 C T/(p \delta \varepsilon^2))}
= \tfrac{p \delta \varepsilon^2}{6T} 
\varphi \Bigl( 
\tfrac{6 C T}{p \delta \varepsilon^2}
\Bigr) 
\leq  \tfrac{p \delta \varepsilon^2}{3T}
%\\
%& \tfrac1{\exp( 7C T/(p \delta  \varepsilon^2))-1}
%\leq
%\tfrac1{1 - \exp(- 7C T/(p \delta  \varepsilon^2))} 
%\leq 
%\tfrac1{1 - \exp(- 6C T/(p \delta  \varepsilon^2))} 
%\leq  \tfrac{ p \delta \varepsilon^2}{3T},
\end{split} 
\end{equation*} 
and then, 
\begin{equation*}
  \begin{split}
  &\sum_{\ell=0}^{p-1} 
\exp \bigl( 
7 C
\tfrac{\ell T}{p \delta \varepsilon^2} 
\bigr)
 \underset{\small \omega \in \Omega}{  \textrm{\rm essup}}
 \bigl\vert h_{\ell T/p} ^{n+1} -  h_{\ell T/p}^{(p),\varpi} \bigr\vert^2
\\
 &\leq 
  \tfrac{1}{3} 
 \sum_{\ell=0}^{p-1} 
\exp \bigl( 
7 C 
\tfrac{\ell T}{p \delta \varepsilon^2} 
\bigr)
 \underset{\small \omega \in \Omega}{  \textrm{\rm essup}}
 \bigl\vert h_{\ell T/p} ^{n} -  h_{\ell T/p}^{(p),\varpi} \bigr\vert^2
 + 
C' p  \exp \bigl( 7C \tfrac{T}{\delta \varepsilon^2} \bigr)  \varpi^{-2n},
 \end{split}
 \end{equation*}
 for a constant $C'$ depending on the same parameters as $C$. 
 The above inequality is very similar to 
 \eqref{eq:backward:bound:000}. Similar to 
\eqref{eq:conclusion:proof:main:thm:1}, we obtain 
   \begin{equation*} 
   \sum_{\ell=0}^{p-1} 
\exp \bigl( 
7 C 
\tfrac{\ell T}{p \delta \varepsilon^2} 
\bigr)  \underset{\small \omega \in \Omega}{  \textrm{\rm essup}}
 \bigl\vert h_{\ell T/p} ^{n+1} -  \tfrac1{\varpi}   h_{\ell T/p}^{(p)} \bigr\vert^2
 \leq C'
 p  \exp \bigl( 7 C \tfrac{T}{\delta \varepsilon^2} \bigr)  \varpi^{-2n}, 
\end{equation*} 
for a possibly new value of $C'$. 
Back to 
  \eqref{eq:thm:2:17:correc}, we obtain 
\eqref{eq:main:thm:1:bound:infinity:2:17}. 
Inequality 
			\eqref{eq:main:thm:1:bound:infinity:2}
			is proven as 	\eqref{eq:main:thm:1:bound}. 
%\vspace{4pt} 
%
%
%\textit{Third Step.} 
%Back to \eqref{eq:BDG:correc}, we easily get 
%\begin{equation*}
%\lim_{n \rightarrow \infty} 
%\sup_{q \geq 1} 
%{\mathbb E}^{\varpi {\boldsymbol h}^n \hspace{-2pt} / \varepsilon}
%\biggl[ 
%\int_0^T 
%\vert k_t^{n+1} - k_t^{n+q+1} \vert^2 \ud t 
%\biggr] = 0. 
%\end{equation*} 
%Observing that we have a bound on the $L^2$-norm of 
%$d {\mathbb P}^{\varpi {\boldsymbol h}^n \hspace{-2pt} /\varepsilon}/d {\mathbb P}$, we deduce that 
%\begin{equation*}
%\lim_{n \rightarrow \infty} 
%\sup_{q \geq 1} 
%{\mathbb E} 
%\biggl[ 
%\biggl( 
%\int_0^T 
%\vert k_t^{n+1} - k_t^{n+q+1} \vert^2 \ud t 
%\biggr)^{1/2} 
%\biggr] = 0. 
%\end{equation*} 
%By Burkholder-Davis-Gundy inequality, we get that 
%\begin{equation*} 
%\lim_{n \rightarrow \infty} 
%\sup_{ q \geq 1} 
%\textrm{\rm essup}_{\omega \in \Omega} \sup_{0 \leq t \leq T} 
%\vert \tilde h_t^{n+1} - \tilde h_t^{n+q} \vert =0. 
%\end{equation*} 
%By Cauchy criterion, we deduce that the sequence $(\tilde{\boldsymbol h}^n,{\boldsymbol k}^n)_{n \geq 0}$ is convergent in the space 
%$L^\infty(\Omega,{\mathbb P};{\mathcal C}([0,T];{\mathbb R}^d)) \times L^1([0,T],L^2(\Omega,{\mathbb P};{\mathbb R}^d))$. There is no difficulty for passing to the limit in 
%\eqref{eq:bsde:exploration}
%and for proving that the limit induces a solution to 
%		\eqref{eq:decoupled:limit}. The 
%		bound 
%				\eqref{eq:main:thm:1:bound:infinity}
%				is then easily established. 
%				The bound 
%				\eqref{eq:main:thm:1:bound:infinity:2}
%			is proven as in 	
%	the proof of Theorem
%\ref{main:thm:1}.
\end{proof}

\subsubsection{Convergence to the MFG}

We now study the distance between $({\boldsymbol m}^{(p)},{\boldsymbol h}^{(p)})$ and $({\boldsymbol m},{\boldsymbol h})$
(see 		\eqref{eq:decoupled:limit})
 as $p$ tends to $\infty$. 
While this looks a natural question, 
the following result, which is given for reader's interest only, has a secondary role in our study. 
The limit $p \rightarrow \infty$ will be addressed
in the next subsection but from another point of view.

\begin{prop}
\label{prop:2:19:correc}
With $({\boldsymbol m}^{(p)},\tilde{\boldsymbol h}^{(p)})$ 
as in the statement of Theorem 
\ref{main:thm:2}
(with the same extension as in Remark 
\ref{rem:another:extension:tildeh}) 
and with 
$({\boldsymbol m},{\boldsymbol h})$
as in the statement of
Theorem  
\ref{main:thm:1} and 
\eqref{eq:minimization:problem:correc}, 
there exists
a constant $C$, only depending on 
$d$, $T$, 
$\| f\|_{1,\infty}$, $\| g \|_{1,\infty}$, $\vert Q\vert$ and $\vert R\vert$, 
such that 
\begin{equation*} 
\sup_{0 \le t \le T}
{\mathbb E} \Bigl[ \vert h_t - \tilde h_t^{(p)} \vert^2 
\Bigr] \leq \exp \Bigl( \tfrac{C}{\varepsilon^2}\Bigr) \tfrac{\ln(p)}{p}.
\end{equation*} 

\end{prop}

\begin{proof} 
For an integer $p \geq 1$, 
we write
\begin{equation*} 
\ud m_t^{(p)} = - \eta^{(p)}_{\tau_p(t)} m^{(p)}_{\tau_p(t)}  \ud t + \varepsilon \ud W_t^p, \quad t \in [0,T]. 
\end{equation*}
We notice that 
\begin{align} 
\label{eq:p:mfg:correc}
\bigl\vert m_t^{(p)} - m^{(p)}_{\tau_p(t)} \bigr\vert 
&\leq C \tfrac{T}p  \vert m_{\tau_p(t)}^{(p)} \vert + \varepsilon 
\sup_{0 \leq s- \tau_p(t) \leq T/p} \vert W_s - W_{\tau_p(t)} \vert
\\
&\leq C \tfrac{T}p \sup_{0 \leq s \leq T} \vert m_s^{(p)} \vert + \varepsilon  \max_{k=0,\cdots,p-1}
\sup_{kT/p \leq s    \leq (k+1)T/p} \vert W_s - W_{kT/p} \vert,  
\quad t \in [0,T].  
\nonumber
\end{align} 
Moreover, 
\begin{equation*} 
\begin{split} 
\ud \bigl( m_t^{(p)} - m_t \bigr) &= - \bigl( \eta^{(p)}_{\tau_p(t)}  m^{(p)}_{\tau_p(t)} - \eta_t m_t \bigr) 
\ud t + \ud (W_t^p - W_t) 
\\
&=  -   \bigl( \eta^{(p)}_{\tau_p(t)}   m^{(p)}_{\tau_p(t)} -  
\eta^{(p)}_t 
 m^{(p)}_t \bigr) \ud t
  - \bigl( \eta_t^{(p)} - \eta_t \bigr) m_t^{(p)} \ud t 
 - \eta_t \bigl( m_t^{(p)} - m_t \bigr) \ud t
 \\
 &\hspace{15pt} + \ud (W_t^p - W_t). 
\end{split} 
\end{equation*} 
By combining the last two displays with \eqref{eq:rem:main:thm:2:correc}, we deduce from Gronwall's lemma that
\begin{equation} 
\label{eq:p:mfg:correc:19:0}
\begin{split} 
\vert m_t^{(p)} - m_t \vert &\leq  C \tfrac{T}p \sup_{0 \leq s \leq T} \vert m_s^{(p)} \vert + \varepsilon 
\sup_{0 \leq s    \leq T} \vert W_s - W_s^p \vert
\\
&\hspace{15pt} + \varepsilon 
  \max_{k=0,\cdots,p-1}
\sup_{kT/p \leq s    \leq (k+1)T/p} \vert W_s - W_{kT/p} \vert
\\
  &\leq  C \tfrac{T}p \sup_{0 \leq s \leq T} \vert m_s^{(p)} \vert + C \varepsilon 
  \max_{k=0,\cdots,p-1}
\sup_{kT/p \leq s    \leq (k+1)T/p} \vert W_s - W_{kT/p} \vert,
 \end{split} 
 \end{equation} 
 for 
$t \in [0,T]$. 
By Lemma 
\ref{lem:2} 
(the proof of which also provides a bound for the time-derivative of $\theta_\varepsilon$) 
and  
	\eqref{eq:change:of:variable:varpi}--\eqref{eq:representation:h:k} and by the analogue of 
	\eqref{eq:p:mfg:correc} but for ${\boldsymbol m}$, we also have 
 \begin{equation}
 \label{eq:p:mfg:correc:19}
 \begin{split}
\vert h_{\tau_p(t)} - h_t \vert  
&=  \bigl\vert \theta_\varepsilon\bigl( \tau_p(t), m_{\tau_p(t)}\bigr) 
- 
\theta_\varepsilon\bigl(t, m_t \bigr) 
 \bigr\vert 
 \\
 &\leq \tfrac{C}{\varepsilon^2}
 \Bigl( 
 \tfrac{T}{p} + 
 \vert 
 m_{\tau_p(t)}
 - 
 m_{t}
 \vert 
 \Bigr)
 \\
 &\leq 
  \tfrac{C}{\varepsilon^2}
 \Bigl( 
 \tfrac{T}{p} + 
\tfrac{T}p \sup_{0 \leq s \leq T} \vert m_s  \vert + \varepsilon 
 \max_{k=0,\cdots,p-1}
\sup_{kT/p \leq s    \leq (k+1)T/p} \vert W_s - W_{kT/p} \vert
 \Bigr)
  \\
 &\leq 
  \tfrac{C}{\varepsilon^2}
 \Bigl( 
 \tfrac{T}{p} + 
\tfrac{T}p \sup_{0 \leq s \leq T} \vert m_s^{(p)}  \vert + \varepsilon 
 \max_{k=0,\cdots,p-1}
\sup_{kT/p \leq s    \leq (k+1)T/p} \vert W_s - W_{kT/p} \vert
 \Bigr). 
 \end{split}
 \end{equation} 
 Next, we rewrite 
the first term in 
the right-hand side in 
\eqref{eq:decoupled:limit:discrete}
 as
 (which follows from 
 the last paragraph in Remark 
 \ref{rem:main:thm:2:b})
 \begin{equation*} 
 		{\mathbb E} 
		\Bigl[ 
		Q^\dagger f\bigl( m^{(p)}_{\tau_{p}(t)+T/p}\bigr)\, \vert \, 
		{\mathcal F}^{p,{\boldsymbol W}}_{\tau_p(t)} \Bigr]
		=
		 		{\mathbb E} 
		\Bigl[ 
		Q^\dagger f\bigl(  m_{\tau_{p}(t)+T/p}^{(p)} \bigr)\, \vert \, 
		{\mathcal F}^{{\boldsymbol W}}_{\tau_p(t)} \Bigr]. 
		\end{equation*} 
Meanwhile, 
\begin{equation*} 
\begin{split}
&\Bigl\vert 
{\mathbb E} 
		\Bigl[ 
		Q^\dagger f\bigl( m^{(p)}_{\tau_{p}(t)+T/p} \bigr)\, \vert \, 
		{\mathcal F}^{{\boldsymbol W}}_{\tau_p(t)} \Bigr]
		- 
		Q^\dagger f\bigl( m_{t} \bigr)
		\Bigr\vert
\\
&\leq 		
\Bigl\vert 
{\mathbb E} 
		\Bigl[ 
		Q^\dagger f\bigl( m^{(p)}_{\tau_{p}(t)+T/p} \bigr)\, \vert \, 
		{\mathcal F}^{{\boldsymbol W}}_{\tau_p(t)} \Bigr]
		- 
		Q^\dagger f\bigl( m_{\tau_{p}(t)}^{(p)} \bigr)
		\Bigr\vert
		+ \Bigl\vert 
		Q^\dagger f\bigl( m_{\tau_{p}(t)}^{(p)} \bigr)		
		-
				Q^\dagger f\bigl( m_{t} \bigr)
		\Bigr\vert.
		\end{split}
		\end{equation*}
		By the first line in 
		\eqref{eq:p:mfg:correc}, 
		\begin{equation*} 
		\begin{split} 
		&\Bigl\vert 
{\mathbb E} 
		\Bigl[ 
		Q^\dagger f\bigl( m^{(p)}_{\tau_{p}(t)+T/p} \bigr)\, \vert \, 
		{\mathcal F}^{{\boldsymbol W}}_{\tau_p(t)} \Bigr]
		- 
		Q^\dagger f\bigl( m_{\tau_{p}(t)}^{(p)} \bigr)
		\Bigr\vert
		\leq C \Bigl( 
		 \tfrac{T}p \sup_{0 \leq s \leq T} \vert m_s^{(p)} \vert + 
\varepsilon \sqrt{ \tfrac{T}p } \Bigr).
\end{split} 
\end{equation*} 
Together with \eqref{eq:p:mfg:correc:19:0}, this gives
\begin{equation*} 
\begin{split}
&\Bigl\vert 
{\mathbb E} 
		\Bigl[ 
		Q^\dagger f\bigl( m^{(p)}_{\tau_{p}(t)+T/p} \bigr)\, \vert \, 
		{\mathcal F}^{{\boldsymbol W}}_{\tau_p(t)} \Bigr]
		- 
		Q^\dagger f\bigl( m_{t} \bigr)
		\Bigr\vert
		\\
&		\leq C \biggl( 
		 \tfrac{T}p \sup_{0 \leq s \leq T} \vert m_s^{(p)} \vert + \varepsilon    \max_{k=0,\cdots,p-1}
\sup_{kT/p \leq s    \leq (k+1)T/p} \vert W_s - W_{kT/p} \vert + 
\varepsilon \sqrt{ \tfrac{T}p } \biggr).
\end{split}
\end{equation*} 
By the above display and by 
\eqref{eq:p:mfg:correc:19}, we can rewrite the equation for ${\boldsymbol h}$ in 
\eqref{eq:bsde:true:underP:} in the form:
\begin{equation*} 
\begin{split} 
\ud h_t  &= \Bigl\{ 
- 
{\mathbb E} \Bigl[ Q^\dagger f \bigl( m^{(p)}_{\tau_p(t) + T/p} \bigr) \, \vert \, {\mathcal F}_{\tau_p(t)}^{\boldsymbol W} \Bigr] 
{\mathbf 1}_{\{ t \leq (p-1) T/p\}} 
\\
&\hspace{30pt} + \Bigl( \eta_{\tau_p(t) + T/p}^{(p)} + \tfrac{T}p Q^\dagger Q {\mathbf 1}_{\{ t \leq (p-1) T/p \}}
\Bigr) h_{\tau_p(t)} \Bigr\} \ud t + e_t \ud t
  + k_t h_t \ud t + \varepsilon k_t  \ud W_t,
\end{split} 
\end{equation*} 
for $t \in [0,T]$, where
\begin{align} 
\label{eq:bound:e:correc} 
\vert e_t \vert  
&\leq C  {\mathbf 1}_{\{ (p-1) T/p < t \leq T\}} 
+
 C \varepsilon \sqrt{ \tfrac{T}p }
\\
&\hspace{5pt} 
+  C
 \Bigl( 
 \tfrac{T}{p} + 
 \tfrac{T}p \sup_{0 \leq s \leq T} \vert m_s^{(p)} \vert   + \varepsilon 
    \max_{k=0,\cdots,p-1}
\sup_{kT/p \leq s    \leq (k+1)T/p} \vert W_s - W_{kT/p} \vert
   \Bigr). \nonumber
\end{align} 
Next, 
by 
\eqref{eq:decoupled:limit:discrete}, 
\begin{equation*} 
\begin{split} 
		\ud \bigl[  h_t  - \tilde h_t^{(p)} \bigr]
		&=  
		\Bigl( \eta_{\tau_{p(t)}+T/p}^{(p)}
		  + \tfrac{T}p Q^\dagger Q 
		  {\mathbf 1}_{\{ t \leq (p-1)T/p\}}
		  	\Bigr) 	
			\bigl( h_{\tau_p(t)}  - \tilde h_{\tau_p(t)}^{(p)} \bigr) 
 \ud t  + e_t \ud t
		  \\
&\hspace{15pt}		  +   \bigl( k_{t} h_{t} - k_t^{(p)} \tilde h_t^{(p)} \bigr)  \ud t 
		+ \varepsilon \bigl( k_{t} -k_t^{(p)} \bigr) \ud W_{t}, \quad t \in [0,T]. 
		\end{split}
\end{equation*} 
Then, 
proceeding as in 
			\eqref{eq:new:hn+1:h}
and
\eqref{eq:proof:delta} (using Lemma 
\ref{lem:2}), 
\begin{equation*} 
\begin{split}
\ud \vert h_t - \tilde h_t^{(p)} \vert^2 
&\geq - C 
\vert h_t - \tilde  h_t^{(p)} \vert^2 \ud t 
- C 
\vert h_{\tau_p(t)}  - \tilde  h_{\tau_p(t)}^{(p)} \vert^2 \ud t 
- 
2
\vert h_t  - \tilde  h_t^{(p)} \vert \vert e_t \vert \ud t
\\
&\hspace{15pt} - C \vert h_t   - \tilde h_t^{(p)} \vert \, \vert k_t  - k_t^{(p)} \vert - C \vert k_t  \vert \vert h_t  - 
\tilde h_t^{(p)} \vert^2 \ud t
+ \varepsilon^2 \vert k_t  - k_t^{(p)} \vert^2 \ud t 
\\
&\hspace{15pt} 
+2 \varepsilon \bigl(  h_t - \tilde h_t^{(p)} \bigr) \cdot \bigl[ 
\bigl(  k_t - k_t^{(p)}
\bigr) \ud W_t \bigr]
\\
&\geq -\tfrac{C}{\varepsilon^2} 
\vert h_t - \tilde h_t^{(p)} \vert^2 \ud t - \varepsilon^2 \vert e_t \vert^2 \ud t 
- C 
\vert h_{\tau_p(t)} - h_{\tau_p(t)}^{(p)} \vert^2 \ud t 
\\
&\hspace{15pt} + 2 \varepsilon \bigl(  h_t - \tilde h_t^{(p)} \bigr) \cdot \bigl[ 
\bigl(  k_t - k_t^{(p)}
\bigr) \ud W_t \bigr].
\end{split} 
\end{equation*} 
There is one small subtlety here to handle the last term on the penultimate line because it is indexed by 
time $\tau_p(t)$ (and not $t$). The strategy is to take expectation on both sides in the above inequality and then to integrate in time
between $\tau_p(t)$ and $t$ for a given $t \in [0,T)$.
As for the dynamics of the expectation, we
have
\begin{equation} 
\label{eq:backward:taup:tt}
\ud {\mathbb E}\bigl[ \vert h_t  - \tilde h_t^{(p)} \vert^2  \bigr] 
\geq -\tfrac{C}{\varepsilon^2} 
{\mathbb E}\bigl[ \vert h_t - \tilde h_t^{(p)} \vert^2 \bigr] \ud t -   {\mathbb E}\bigl[ 
\vert e_t \vert^2 \bigr] \ud t 
- C 
{\mathbb E}\bigl[ 
\vert h_{\tau_p(t)}  - h_{\tau_p(t)}^{(p)} \vert^2 \bigr] \ud t.
\end{equation}
By 
\eqref{eq:bound:e:correc}, notice that 
\begin{equation}
\label{eq:bound:e:correc:expectation}
{\mathbb E}\bigl[\vert e_t \vert^2 \bigr]\leq C  {\mathbf 1}_{\{ (p-1) T/p < t \leq T\}} 
 +   \tfrac{C}{p^2}
 + 
 C
 {\mathbb E} 
 \Bigl[
  \max_{k=0,\cdots,p-1}
\sup_{kT/p \leq s    \leq (k+1)T/p} \vert W_s - W_{kT/p} \vert^2
 \Bigr].
\end{equation}
We admit for a while that
\begin{equation} 
\label{eq:proof:log(p)/p}
 {\mathbb E} 
 \Bigl[
  \max_{k=0,\cdots,p-1}
\sup_{kT/p \leq s    \leq (k+1)T/p} \vert W_s - W_{kT/p} \vert^2
 \Bigr] \leq C \tfrac{\ln(p)}{p}.
\end{equation}
And then, by 
integrating 
\eqref{eq:backward:taup:tt}
between $\tau_p(t)$ and $t$ 
and by invoking
boundedness of the
two processes ${\boldsymbol h}$ and $\tilde {\boldsymbol h}^{(p)}$ 
(see Lemmas 
\ref{lem:1} 
and
\ref{lem:1:b}), we deduce that% (using in addition the lower bound $p \varepsilon^2 \geq c$) 
\begin{equation*} 
\begin{split}
{\mathbb E}\bigl[ \vert h_{\tau_p(t)} - \tilde h_{\tau_p(t)}^{(p)} \vert^2  \bigr] 
&\leq 
C {\mathbb E}\bigl[ \vert h_{t} - \tilde h_{t}^{(p)} \vert^2  \bigr] 
+
 \tfrac{C}{\varepsilon^2} 
\int_{\tau_p(t)}^{t}  {\mathbb E}\bigl[ \vert h_{t} - \tilde h_{t}^{(p)} \vert^2  \bigr] \ud s
+
 \tfrac{C \ln(p)}{p} 
 \\
  &\leq 
C {\mathbb E}\bigl[ \vert h_{t} - \tilde h_{t}^{(p)} \vert^2  \bigr] 
+
 \tfrac{C \ln(p)}{\varepsilon^2 p} .
 \end{split}
\end{equation*}
Inserting the above estimate into 
\eqref{eq:backward:taup:tt} and then invoking 
\eqref{eq:bound:e:correc:expectation}, we obtain 
\begin{equation*} 
\ud {\mathbb E}\bigl[ \vert h_t - \tilde h_t^{(p)} \vert^2  \bigr] 
\geq -\tfrac{C}{\varepsilon^2} 
{\mathbb E}\bigl[ \vert h_t - \tilde h_t^{(p)} \vert^2 \bigr] \ud t - 
C \Bigl(
{\mathbf 1}_{\{ (p-1) T/p < t \leq T\}}  
+ \tfrac{\ln(p)}{\varepsilon^2 p} 
\Bigr) \ud t.
\end{equation*}
By Gronwall's lemma
(and by 
\eqref{eq:p:mfg:correc:19:0},
which allows us to control the terminal condition), we get 
\begin{equation*} 
{\mathbb E} \bigl[ \vert h_t - \tilde h_t^{(p)} \vert^2 
\bigr] \leq \exp \Bigl( \tfrac{C}{\varepsilon^2}\Bigr) \tfrac{C \ln(p)}{\varepsilon^2 p}, 
\quad t \in [0,T]. 
\end{equation*}
Changing the value of $C$, we easily complete the proof of the statement.

It remains to check
\eqref{eq:proof:log(p)/p}. We have, for any $r>0$, 
\begin{equation*} 
\begin{split} 
&{\mathbb P} 
\Bigl( \Bigl\{   \max_{k=0,\cdots,p-1}
\sup_{kT/p \leq s    \leq (k+1)T/p} \vert W_s - W_{kT/p} \vert^2 \geq r \Bigr\} \Bigr)
\\
&= 1 - {\mathbb P} 
\Bigl( \Bigl\{   \max_{k=0,\cdots,p-1}
\sup_{kT/p \leq s    \leq (k+1)T/p} \vert W_s - W_{kT/p} \vert^2 < r \Bigr\} \Bigr)
\\
&= 1 - {\mathbb P} 
\Bigl( \Bigl\{   
\sup_{0 \leq s    \leq  T/p} \vert W_s  \vert^2 < r \Bigr\}   \Bigr)^p
\\
&\leq 1 - \bigl( 1 - \exp (- c p r) \bigr)^p   
 \leq p \exp(-cpr), 
\end{split}
\end{equation*}
for a constant $c$ only depending on $T$. Replacing $r$ by $\ln(p) r/p$, we easily deduce that 
\begin{equation*} 
 {\mathbb E} 
 \Bigl[
  \max_{k=0,\cdots,p-1}
\sup_{kT/p \leq s    \leq (k+1)T/p} \vert W_s - W_{kT/p} \vert^2
 \Bigr] \leq C \tfrac{\ln(p)}{p},
\end{equation*}
which is 
\eqref{eq:proof:log(p)/p}. 
\end{proof}
\color{black}

\subsubsection{Proof of auxiliary results}
\label{subsubse:2.2.4:correc}

\begin{proof}[Proof of Lemma \ref{lem:4}.]
{ \ }

\textit{First Step.} We first notice that, for a given choice of $(\overline{\boldsymbol m}^n,{\boldsymbol h}^n)$, the problem \eqref{eq:optimization:exploration} has at least one minimizer. Indeed, 
the cost 
${\boldsymbol \alpha} \mapsto {\mathbb E}^{\varpi {\boldsymbol h}^n\hspace{-2pt}/\varepsilon}
	[ {\mathcal R}^{p,\varpi X_0}  ( \varpi {\boldsymbol \alpha} + \varpi \varepsilon \dot{\boldsymbol W}^{p,\varpi {\boldsymbol h}^n \hspace{-2pt} /\varepsilon} ; \overline{\boldsymbol m}^n;0 )  ]$  is lower semicontinuous with respect to ${\boldsymbol \alpha}$ 
	when the latter is identified with 
a collection of random variables $(\alpha_0,\cdots,\alpha_{(p-1)T/p}) 
\in 
\times_{\ell=0}^{p-1}
L^2(\Omega,{\mathcal F}_{\ell T/p}^{p,X_0,{\boldsymbol B},{\boldsymbol W}},{\mathbb P}^{\varpi {\boldsymbol h}^n
\hspace{-2pt}/\varepsilon};{\mathbb R}^d)$, with the product space being equipped with the weak topology. The identification is made possible by the fact that the controls ${\boldsymbol \alpha}$ are chosen to be piecewise constant. 
Moreover, the cost functional blows up when 
the $L^2$-norm of 
${\boldsymbol \alpha}$ tends to $\infty$. The minimization problem can hence be restricted to a bounded subset of 
$\times_{\ell=0}^{p-1}
L^2(\Omega,{\mathcal F}_{\ell T/p}^{p,X_0,{\boldsymbol B},{\boldsymbol W}},{\mathbb P}^{\varpi {\boldsymbol h}^n
\hspace{-2pt}/\varepsilon};{\mathbb R}^d)$, the latter being obviously compact for the weak topology. Existence of a minimizer easily follows. 

We check below that any critical point of the cost functional is a solution of the forward-backward system 
(\ref{eq:lem:4:feedback:form})--(\ref{eq:bsde:exploration}). Since the latter is uniquely solvable, 
this indeed provides a characterization of the minimizer. 
\vskip 4pt

\textit{Second Step.}
In order to prove that critical points solve (\ref{eq:lem:4:feedback:form})--(\ref{eq:bsde:exploration}), 
we follow the usual lines of the Pontryagin principle. 
Although the result is certainly not new, we feel better to provide the complete proof, since the formulation we use is tailor-made to our needs. We start to notice that, under the probability
${\mathbb P}^{\varpi {\boldsymbol h}^n
\hspace{-2pt}/\varepsilon}$, 
the path driven by 
the control 
$\varpi {\boldsymbol \alpha}$ 
and the initial condition 
$\varpi X_0$ reads
$(\varpi X_t)_{0 \le t \le T}$ with 
\begin{equation*}
	\ud X_{t} = \alpha_{t} \ud t + \tfrac1{\varpi} \sigma \ud B_{t} + \varepsilon \ud W_{t}^{p,
	\varpi 
	{\boldsymbol h}^n \hspace{-2pt} /\varepsilon}, \quad t \in [0,T]. 
\end{equation*}
We then introduce an additive perturbation of the cost and hence replace ${\boldsymbol \alpha}=(\alpha_{t})_{0 \le t \le T}$ by 
${\boldsymbol \alpha}+ \delta {\boldsymbol \beta}=(\alpha_{t} + \delta \beta_{t})_{0 \le t \le T}$ in the above expansion, with 
${\boldsymbol \alpha}$ and ${\boldsymbol \beta}$ being identified as before as elements of 
$\times_{\ell=0}^{p-1}
L^2(\Omega,{\mathcal F}_{\ell T/p}^{p,X_0,{\boldsymbol B},{\boldsymbol W}},{\mathbb P}^{\varpi {\boldsymbol h}^n
\hspace{-2pt}/\varepsilon};{\mathbb R}^d)$. We then write ${\boldsymbol X}^{\delta}=(X_{t}^{\delta})_{0 \le t \le T}$
in order to emphasize the dependence of ${\boldsymbol X}$ upon $\delta$. 
When $\delta=0$, we merely write ${\boldsymbol X}$ instead of ${\boldsymbol X}^0$. 
The formal derivative with respect to $\delta$ at $\delta=0$ reads
\begin{equation*}
	\ud \bigl( \partial X_{t} \bigr) = \beta_{t} \ud t, \quad t \in [0,T] \, ; 
	\quad \partial X_{0} =0. 
\end{equation*}
In turn, 
the derivative of the cost
${\mathbb E}^{\varpi {\boldsymbol h}^n \hspace{-2pt} / \varepsilon}
	[ {\mathcal R}^{p,\varpi X_0}  ( \varpi [{\boldsymbol \alpha}+\delta {\boldsymbol \beta}] + \varpi \varepsilon \dot{\boldsymbol W}^{p,\varpi {\boldsymbol h}^n \hspace{-2pt} /\varepsilon} ; \overline{\boldsymbol m}^n;0 ) ]$, 
with respect to $\delta$, 
 is  
\begin{equation*}
\begin{split}
&\frac{\ud}{\ud \delta}
\Bigl\{ 
 {\mathbb E}^{\varpi {\boldsymbol h}^n \hspace{-2pt} / \varepsilon}
	\Bigl[ {\mathcal R}^{p,\varpi X_0}  \Bigl( \varpi {\boldsymbol \alpha} + \varpi \varepsilon \dot{\boldsymbol W}^{p,\varpi {\boldsymbol h}^n \hspace{-2pt} /\varepsilon} ; \overline{\boldsymbol m}^n;0 \Bigr) \Bigr] 
	\Bigr\}_{\vert\delta =0}
% \tilde J^p\bigl({\boldsymbol \alpha}+\delta {\boldsymbol \beta};{\boldsymbol m}\bigr)_{\vert \delta=0}
\\
&= \varpi^2 {\mathbb E}^{{\boldsymbol h}^n} \biggl[ \Bigl(R X_{T}+ \tfrac1{\varpi} g\bigl( \overline m^n_{T}\bigr)
	\Bigr) \cdot R \partial X_{T} 
\\
&\hspace{15pt} 	+ \int_{0}^T \Bigl\{ \Bigl(Q X_{\tau_{p}(t)} + \tfrac1{\varpi}  f\bigl(\overline m^n_{\tau_{p}(t)} \bigr)
	\Bigr) 
	\cdot  Q \partial X_{\tau_t(p)} + \alpha_{\tau_{p}(t)} \cdot \beta_{\tau_{p}(t)} \Bigr\} \ud t \biggr].
\end{split}
\end{equation*}
We now solve the BSDE (under ${\mathbb P}^{\varpi {\boldsymbol h}^n \hspace{-2pt} /\varepsilon}$)
\begin{equation*}
\begin{split}
	&\ud Y_{t} = - \bigl(  Q^\dagger Q X_{\tau_{p}(t)} + \tfrac1{\varpi}  Q^\dagger f(\overline m^n_{\tau_{p}(t)}) \bigr) \ud t + Z_{t}^{\boldsymbol B} \ud B_{t} + Z_{t}^{\boldsymbol W} \ud W_{t}^{\varpi  {\boldsymbol h}^n \hspace{-2pt}/\varepsilon}, \quad t \in [0,T], 
	\\
	&Y_{T}=R^\dagger R X_{T}+ \tfrac1{\varpi}  R^\dagger g\bigl(\overline m^n_{T}\bigr).
\end{split} 
\end{equation*} 
By discrete integration by parts, 
we have
\begin{equation*}
\begin{split}
&{\mathbb E}^{\varpi  {\boldsymbol h}^n \hspace{-2pt}/\varepsilon} \bigl[ Y_T \cdot \partial X_{T} 
\bigr]
\\
&=
\sum_{\ell=0}^{p-1} 
{\mathbb E}^{\varpi  {\boldsymbol h}^n \hspace{-2pt}/\varepsilon}\Bigl[
\Bigl( 
Y_{(\ell+1)T/p} - Y_{\ell T/p} \Bigr) \cdot \partial X_{\ell T/p} 
+ Y_{(\ell +1)T/p} \cdot 
\Bigl( \partial X_{(\ell+1)T/p} - \partial X_{\ell T/p} 
\Bigr) 
\Bigr]
\\
&= 
\tfrac{T}p \sum_{\ell=0}^{p-1} 
{\mathbb E}^{\varpi  {\boldsymbol h}^n \hspace{-2pt}/\varepsilon} \Bigl[
- \Bigl( 
Q^\dagger Q X_{\ell T/p} + \tfrac1{\varpi}  Q^\dagger  
f(\overline m^n_{\ell T/p})  \Bigr) \cdot \partial X_{\ell T/p} 
+ Y_{(\ell+1)T/p} \cdot 
\beta_{\ell T/p}
\Bigr].
\end{split}
\end{equation*}
Rewriting the above sum as an integral, we obtain
\begin{equation*}
\begin{split}
	&\frac{\ud}{\ud \delta}
\Bigl\{ 
 {\mathbb E}^{\varpi {\boldsymbol h}^n \hspace{-2pt} / \varepsilon}
	\Bigl[ {\mathcal R}^{p,\varpi X_0}  \Bigl( \varpi {\boldsymbol \alpha} + \varpi \varepsilon \dot{\boldsymbol W}^{p,\varpi {\boldsymbol h}^n \hspace{-2pt} /\varepsilon} ; \overline{\boldsymbol m}^n;0 \Bigr) \Bigr] 
	\Bigr\}_{\vert\delta =0}
	\\
	&= \varpi^2  {\mathbb E}^{\varpi  {\boldsymbol h}^n \hspace{-2pt}/\varepsilon} \biggl[ \int_{0}^T 
	\Bigl(  Y_{\tau_p(t)+T/p}   + \alpha_{\tau_{p}(t)} \Bigr) \cdot \beta_{\tau_{p}(t)} \ud t \biggr],
\end{split}
\end{equation*}
which means that the optimizer 
(which we merely write 
	${\boldsymbol \alpha}^{n+1}$ without specifying the indices $p$ and $\varpi$)
satisfies 
\begin{equation}
\label{eq:minimiser:p:problem:lemma2.10}
	\alpha_{\ell T/p}^{n+1} = 
	-  {\mathbb E}^{\varpi {\boldsymbol h}^n \hspace{-2pt} / \varepsilon} \Bigl[ Y_{(\ell +1)T/p}^{n+1} \, \vert \,  {\mathcal F}_{\ell T/p}^{p,X_0,{\boldsymbol B},{\boldsymbol W}}\Bigr]
	= 
	- Y_{\ell T/p}^{n+1}
	 + \tfrac{T}{p} \Bigl( Q^\dagger Q X_{\ell T/p}^{n+1} + 
	 \tfrac1{\varpi}
	 Q^\dagger f(\overline m^n_{\ell T/p}) \Bigr),
	%	- \frac{p}{T} \int_{\tau_{p}(t)}^{\tau_{p}(t)+T/p} {\mathbb E}^{{\boldsymbol h}^n} \Bigl[ Y_{s}^{n+1} \vert {\mathcal F}_{\tau_{p}(t)}^{p,{\boldsymbol B},{\boldsymbol W}}\Bigr] \ud s,
\end{equation}
for $\ell \in \{0,\cdots,p-1\}$, 
where 
\begin{equation*}
	\begin{split}
		&\ud X_{t}^{n+1} = \alpha_{t}^{n+1} \ud t +\tfrac1{\varpi} \sigma \ud B_{t} + \varepsilon \ud W_{t}^{p,\varpi {\boldsymbol h}^n \hspace{-2pt} / \varepsilon},
		\\
		&\ud Y_{t}^{n+1} = - \bigl(  Q^\dagger Q X_{\tau_{p}(t)}^{n+1} + \tfrac1{\varpi} Q^\dagger f ( \overline m_{\tau_{p}(t)}^n)\bigr) \ud t  + Z_{t}^{n+1,{\boldsymbol B}} \ud B_{t} + Z_{t}^{n+1,{\boldsymbol W}} \ud W_{t}^{\varpi {\boldsymbol h}^n \hspace{-2pt} / \varepsilon},
		\quad t \in [0,T], 
		\\
		&Y_{T}^{n+1} = R^\dagger R X_{T}^{n+1} + \tfrac1{\varpi} R^\dagger g(\overline m^n_{T}).
	\end{split}
\end{equation*}
Importantly, we stress that, by construction, ${\boldsymbol X}_{\tau_{p}(\cdot)}^{n+1}$ is $({\mathbb F}^{p,X_0,{\boldsymbol B},{\boldsymbol W}})_{0 \le t \le T}$-adapted. We then let, for the same solution ${\boldsymbol \eta}^{(p)}$ as in 
\eqref{eq:discretized:Riccati}--\eqref{eq:discretized:Riccati:sol},
\begin{equation}
\label{eq:discrete:time:h}
\begin{split}
	\tilde h_{\ell T/p}^{n+1}
	&=
	{\mathbb E}^{{\varpi {\boldsymbol h}^n \hspace{-2pt} / \varepsilon}} \Bigl[ Y_{(\ell+1)T/p}^{n+1}\, \vert \, {\mathcal F}_{\ell T/p}^{X_0,{\boldsymbol B},{\boldsymbol W}}\Bigr]
	- \eta_{\ell T/p}^{(p)} X^{n+1}_{\ell T/p}
	\\
	&= Y_{\ell T/p}^{n+1}
	- 
 \tfrac{T}{p} \Bigl( Q^\dagger Q X_{\ell T/p}^{n+1} + \tfrac1{\varpi}
	 Q^\dagger f(\overline m^n_{\ell T/p}) \Bigr)	
	 - \eta_{\ell T/p}^{(p)} X^{n+1}_{\ell T/p}, 
	 	 \end{split}
\end{equation}
for 
$\ell \in \{0,\cdots,p-1\}$.
We obtain, for $\ell \in \{0,\cdots,p-2\}$, %(using the fact that $\eta$ is in fact equal to 1 in our case)
\begin{equation*}
\begin{split}
&\tilde h_{(\ell+1)T/p}^{n+1} - \tilde h_{\ell T/p}^{n+1}
= - \tfrac{T}{p} \Bigl( Q^\dagger Q X_{(\ell+1)T/p}^{n+1} +
\tfrac1{\varpi}
Q^\dagger f \bigl( \overline m_{(\ell+1)T/p}^n\bigr)\Bigr) 
 - \tfrac{T}p  \eta_{\ell T/p}^{(p)}   \alpha^{n+1}_{\ell T/p}  
\\
&\hspace{15pt} - \Bigl( \eta_{(\ell +1)T/p}^{(p)} - \eta_{\ell T/p}^{(p)} \Bigr) 
 X_{(\ell +1)T/p}^{n+1} 
+ \int_{\ell T/p}^{(\ell+1)T/p}
 k_{s}^{n+1,{\boldsymbol B}} \ud B_{s} +
 \int_{\ell T/p}^{(\ell+1)T/p}
 k_{s}^{n+1,{\boldsymbol W}} \ud W_{s}^{{\boldsymbol h}^n}, %+ \varepsilon \eta_{\tau_{p}(t)}^{(p)}\bigl( ,
\end{split}
\end{equation*}
for some square integrable and ${\mathbb F}^{X_0,{\boldsymbol B},{\boldsymbol W}}$-progressively measurable processes ${\boldsymbol k}^{n+1,{\boldsymbol B}}$
and 
${\boldsymbol k}^{n+1,{\boldsymbol W}}$. 
Above, we used the fact that ${W}_{\tau_{p}(t)+T/p}^{p,{\varpi {\boldsymbol h}^n \hspace{-2pt} / \varepsilon}}
- 
{W}_{\tau_{p}(t)}^{p,{\varpi {\boldsymbol h}^n \hspace{-2pt} / \varepsilon}}=
{W}_{\tau_{p}(t)+T/p}^{{\varpi {\boldsymbol h}^n \hspace{-2pt} / \varepsilon}}
- 
{W}_{\tau_{p}(t)}^{{\varpi {\boldsymbol h}^n \hspace{-2pt} / \varepsilon}}$. 
Now, 
 using 
\eqref{eq:minimiser:p:problem:lemma2.10}
 and  conditioning on 
${\mathcal F}_{\tau_{p}(t)}^{p,X_0,{\boldsymbol B},{\boldsymbol W}}$
in 
\eqref{eq:discrete:time:h}, we obtain
\begin{equation*}
	\begin{split}
		\alpha_{\ell T/p}^{n+1} = - \eta_{\ell T/p}^{(p)} X^{n+1}_{\ell T/p} - 
		 {\mathbb E}^{{\varpi {\boldsymbol h}^n \hspace{-2pt} / \varepsilon}} \Bigl[ \tilde h_{\ell T/p}^{n+1} \, \vert \, {\mathcal F}_{\ell T/p}^{p,X_0,{\boldsymbol B},{\boldsymbol W}}\Bigr], \quad \ell \in \{0,\cdots,p-1\},
	\end{split}
\end{equation*}
which permits to write
\begin{equation}
\label{eq:xn+1:xn:p}
\begin{split}
X^{n+1}_{(\ell+1)T/p} 
&= X^{n+1}_{\ell T/p} + \tfrac{T}p 
	\alpha_{\ell T/p}^{n+1} + \tfrac1{\varpi} \sigma \Bigl( B_{(\ell+1)T/p}- 
	B_{\ell T/p} \Bigr) + \varepsilon \Bigl( W^{{\varpi {\boldsymbol h}^n \hspace{-2pt} / \varepsilon}}_{(\ell+1)T/p}
	-
	W^{{\varpi {\boldsymbol h}^n \hspace{-2pt} / \varepsilon}}_{\ell T}
	\Bigr)
\\
&= \Bigl( I_d - \tfrac{T}p \eta_{\ell T/p}^{(p)} \Bigr) 
	X^{n+1}_{\ell T/p} - \tfrac{T}{p}  {\mathbb E}^{{\varpi {\boldsymbol h}^n \hspace{-2pt} / \varepsilon}} \Bigl[ \tilde h_{\ell T/p}^{n+1} \, \vert \,  {\mathcal F}_{\ell T/p}^{p,X_0,{\boldsymbol B},{\boldsymbol W}}\Bigr]
	\\
&\hspace{15pt}	+ \tfrac{1}{\varpi} \sigma \Bigl( B_{(\ell+1)T/p}- 
	B_{\ell T/p} \Bigr) + \varepsilon \Bigl( W^{{\varpi {\boldsymbol h}^n \hspace{-2pt} / \varepsilon}}_{(\ell+1)T/p}
	-
	W^{{\varpi {\boldsymbol h}^n \hspace{-2pt} / \varepsilon}}_{\ell T}
	\Bigr).
\end{split}	
\end{equation}
Modifying the processes ${\boldsymbol k}^{n+1,{\boldsymbol B}}$ and ${\boldsymbol h}^{n+1,{\boldsymbol W}}$
and 
using in addition the first equation 
for
${\boldsymbol \eta}^{(p)}$ in \eqref{eq:discretized:Riccati}, we
 end-up with 
\begin{equation}
\label{eq:bsde:pontryagin:proof:p:000}
\begin{split}
\tilde h_{(\ell+1)T/p}^{n+1} - \tilde h_{\ell T/p}^{n+1}
&= - \tfrac{T}{p \varpi}  
Q^\dagger f ( \overline m_{(\ell+1)T/p}^n)  +\tfrac{T}p  
\Bigl( \eta_{\ell T/p}^{(p)}
+ \rho_{\ell T/p}^{(p)} 
\Bigr)
  {\mathbb E}^{{\varpi {\boldsymbol h}^n \hspace{-2pt} / \varepsilon}} \Bigl[ \tilde h_{\ell T/p}^{n+1} \, \vert \,  {\mathcal F}_{\ell T/p}^{p,X_0,{\boldsymbol B},{\boldsymbol W}}\Bigr] 
 \\
 &\hspace{15pt}+ \int_{\ell T/p}^{(\ell+1)T/p}
 k_{s}^{n+1,{\boldsymbol B}} \ud B_{s} +
 \int_{\ell T/p}^{(\ell+1)T/p}
 k_{s}^{n+1,{\boldsymbol W}}  \ud W_{s}^{{\boldsymbol h}^n},
\end{split}
\end{equation}
for $\ell \in \{0,\cdots,p-2\}$, where 
\begin{equation*}
\rho_{\ell T/p}^{(p)}
= 
\tfrac{T}{p \varpi} Q^\dagger Q 
+
  \Bigl( \eta_{(\ell +1)T/p}^{(p)} - \eta_{\ell T/p}^{(p)} \Bigr).
\end{equation*}
By
recalling that 
$ \overline m_{(\ell+1)T/p}^n$ is 
${\mathcal F}_{(\ell+1)T/p}^{p,{\boldsymbol W}}$-measurable and by 
 modifying the process 
${\boldsymbol k}^{n+1,{\boldsymbol W}}$
accordingly, 
we can rewrite 
\eqref{eq:bsde:pontryagin:proof:p:000}
in the form 
\begin{equation}
\label{eq:bsde:pontryagin:proof:p}
\begin{split}
\tilde h_{(\ell+1)T/p}^{n+1} - \tilde h_{\ell T/p}^{n+1}
&= - \tfrac{T}{p}  {\mathbb E}^{{\varpi {\boldsymbol h}^n \hspace{-2pt} / \varepsilon}}
\Bigl[ 
\tfrac1{\varpi}
Q^\dagger f ( \overline m_{(\ell+1)T/p}^n) \, \vert \, 
{\mathcal F}_{\ell T/p}^{p,{\boldsymbol W}}
\Bigr]
\\
&\hspace{15pt}   +\tfrac{T}p  
\Bigl( \eta_{\ell T/p}^{(p)}
+ \rho_{\ell T/p}^{(p)} 
\Bigr)
  {\mathbb E}^{{\varpi {\boldsymbol h}^n \hspace{-2pt} / \varepsilon}} \Bigl[ \tilde h_{\ell T/p}^{n+1} \vert {\mathcal F}_{\ell T/p}^{p,X_0,{\boldsymbol B},{\boldsymbol W}}\Bigr] 
 \\
 &\hspace{15pt}+ \int_{\ell T/p}^{(\ell+1)T/p}
 k_{s}^{n+1,{\boldsymbol B}} \ud B_{s} +
 \int_{\ell T/p}^{(\ell+1)T/p}
 k_{s}^{n+1,{\boldsymbol W}}  \ud W_{s}^{{\boldsymbol h}^n},
\end{split}
\end{equation}
for $\ell \in \{0,\cdots,p-2\}$. 
When $\ell=p-1$, we deduce
from
\eqref{eq:discrete:time:h}
and
\eqref{eq:xn+1:xn:p}
that
\begin{equation*}
\begin{split}
&\tilde h^{n+1}_{(p-1)T/p}
\\
&= R^\dagger R \, {\mathbb E}^{{\varpi {\boldsymbol h}^n \hspace{-2pt} / \varepsilon}} 
\bigl[ X_T^{n+1} \vert 
{\mathcal F}_{(p-1)T/p}^{X_0,{\boldsymbol B},{\boldsymbol W}} 
\bigr] + \tfrac1{\varpi} R^\dagger
 {\mathbb E}^{{\varpi {\boldsymbol h}^n \hspace{-2pt} / \varepsilon}} 
\bigl[
g\bigl(\overline m^n_{T}\bigr)\, \vert \,
{\mathcal F}_{(p-1)T/p}^{X_0,{\boldsymbol B},{\boldsymbol W}} 
\bigr]   
- \eta_{(p-1)T/p}^{(p)} X^{n+1}_{(p-1)T/p}
\\
&=
\Bigl[ 
R^\dagger R
\Bigl( I_d - \frac{T}p \eta_{(p-1)T/p}^{(p)}  \Bigr)
- \eta_{(p-1)T/p}^{(p)} 
\Bigr] X^{n+1}_{(p-1)T/p}
\\
&\hspace{15pt} + \tfrac1{\varpi} R^\dagger  
 {\mathbb E}^{{\varpi {\boldsymbol h}^n \hspace{-2pt} / \varepsilon}} 
\bigl[
g\bigl(\overline m^n_{T}\bigr)\, \vert \,
{\mathcal F}_{(p-1)T/p}^{X_0,{\boldsymbol B},{\boldsymbol W}} 
\bigr]   
 - \tfrac{T}{p} R^\dagger R \,  {\mathbb E}^{{\varpi {\boldsymbol h}^n \hspace{-2pt} / \varepsilon}} \Bigl[ \tilde h_{(p-1) T/p}^{n+1} \vert {\mathcal F}_{(p-1) T/p}^{p,X_0,{\boldsymbol B},{\boldsymbol W}}\Bigr],
\end{split}
\end{equation*}
which yields, using the boundary condition in \eqref{eq:discretized:Riccati} together with the fact that 
$\overline m^n_{T}$ is 
${\mathcal F}_T^{{\boldsymbol W}}$-measurable
\begin{equation}
\label{eq:bsde:pontryagin:proof:p:boundary}
\tilde h^{n+1}_{(p-1)T/p}
= \tfrac1{\varpi} R^\dagger  
 {\mathbb E}^{{\varpi {\boldsymbol h}^n \hspace{-2pt} / \varepsilon}} 
\bigl[
g\bigl(\overline m^n_{T}\bigr)\, \vert \,
{\mathcal F}_{(p-1)T/p}^{{\boldsymbol W}} 
\bigr]    - \tfrac{T}{p} R^\dagger R \,  {\mathbb E}^{{\varpi {\boldsymbol h}^n \hspace{-2pt} / \varepsilon}} \Bigl[ \tilde h_{(p-1) T/p}^{n+1} \vert {\mathcal F}_{(p-1) T/p}^{p,X_0,{\boldsymbol B},{\boldsymbol W}}\Bigr].
\end{equation}
We
now recall that, by construction,  
$\tilde h_{\ell T/p}^{n+1}$ 
is 
${\mathcal F}_{\ell T/p}^{X_0,{\boldsymbol B},{\boldsymbol W}}$-measurable. 
Arguing as in Remark 
\ref{rem:main:thm:2:b}, we can prove inductively (over $\ell$) that 
$\tilde h_{\ell T/p}^{n+1}$ 
is 
in fact 
${\mathcal F}_{\ell T/p}^{p,{\boldsymbol W}}$-measurable.
As a consequence,  
${\boldsymbol k}^{n+1,{\boldsymbol B}}$ has to be zero.
By combining 
\eqref{eq:bsde:pontryagin:proof:p}
and
\eqref{eq:bsde:pontryagin:proof:p:boundary}, we hence 
 recover BSDE \eqref{eq:bsde:exploration}.

\end{proof}

\begin{proof}[Proof of Lemma
\ref{lem:1:b}.]
The bound
\eqref{eq:lem:1:b:1} on ${\boldsymbol h}^n$ is a direct consequence of 
\eqref{eq:recurrence:cas:discret}, which yields: 
\begin{equation*}
\begin{split}
\textrm{\rm essup}_{\omega \in \Omega}
\vert \tilde h^{n+1}_{\ell T/p} \vert \leq 
\bigl( 1+ \tfrac{C_1}{p} \bigr)
\Bigl[ 
\textrm{\rm essup}_{\omega \in \Omega}
\vert \tilde h^{n+1}_{(\ell+1) T/p} \vert
+ \tfrac{C_1}p \Bigr], \quad \ell \in \{0,\cdots,p-1\}, 
			\end{split}
			\end{equation*}
			for $C_1$ as in the statement. 
			Observing that 
$\textrm{\rm essup}_{\omega \in \Omega}
\vert \tilde h^{n+1}_{T} \vert$ is bounded by $C_1$, we get the result by means of 
a straightforward backward induction. 		
The bound on $\tilde {\boldsymbol h}^{(p)}$ is proven in a similar way. 

In order to prove the second claim 
in the statement (see \eqref{eq:lem:1:b:2}), we invoke \cite[Theorem 2.2]{kazamaki}. 
Indeed, we notice that the ${\mathbb P}^{{\boldsymbol h}^{(p)}/\varepsilon }$-martingale
\begin{equation*}
\biggl( \varepsilon \int_{0}^t  k_{s}^{(p)}  \ud W_{s}^{  {\boldsymbol h}^{(p)}/\varepsilon } \biggr)_{0 \leq t \leq T}
\end{equation*}
has a finite BMO norm, see for instance (2.1) with $p=2$ in the book \cite{kazamaki} by Kazamaki.
The proof is as follows. Rewriting \eqref{eq:bsde:exploration:p,varpi} in the form  
\begin{equation*} 
\begin{split}
		&\ud  \tilde h_{t}^{(p)} 
		= \Bigl\{ -  
		{\mathbb E} 
		\Bigl[ 
		Q^\dagger f\bigl(m_{\tau_{p}(t)+T/p}^{(p)} \bigr)\, \vert \, 
		{\mathcal F}^{p,{\boldsymbol W}}_{\tau_p(t)} \Bigr]
		 {\mathbf 1}_{\{ t \leq (p-1)T/p\}}  
		 \\
&\hspace{50pt}		  + 
		\Bigl( \eta_{\tau_{p(t)}+T/p}^{(p)}
		  + \frac{T}p Q^\dagger Q 
		  {\mathbf 1}_{\{ t \leq (p-1)T/p\}}
		  	\Bigr) 		  
		    h_{t}^{(p)} \Bigr\}
 \ud t  
 	  + \varepsilon   k_{t}^{(p)}   \ud W_{t}^{{\boldsymbol h}^{(p)} \hspace{-2pt} /\varepsilon}, \quad t \in [0,T),
%\nonumber
		%+ \ud W_{t} - \ud W_{t}^p, \quad t \in [0,T], 
		\\
		&\tilde h_{T}^{(p)} = R^\dagger g \bigl( m^{(p)}_{T} \bigr), 
			\end{split}\end{equation*} 
we easily deduce that,
for any stopping 
$\tau$,
\begin{equation*}
\biggl\vert 
{\mathbb E}^{ {\boldsymbol h}^{(p)} \hspace{-2pt} /\varepsilon} \Bigl[ \int_{\tau}^T 
\varepsilon k_{s}^{(p)} \ud W_s^{{\boldsymbol h}^{(p)} \hspace{-2pt} /\varepsilon} 
\big\vert \, {\mathcal F}_{\tau}^{{\boldsymbol W}} \Bigr] \biggr\vert \leq C_{1},
\end{equation*}
for a possibly new value of $C_1$. We conclude by means of \cite[Theorem 2.2]{kazamaki}.
\end{proof}

\begin{proof}[Proof of Proposition \ref{prop:!:FBSDE:p-discrete}]
The equation for 
$(\tilde h_{\ell T/p}^{(p)})_{\ell=0,\cdots,p}$ in \eqref{eq:decoupled:limit:discrete} can be rewritten in 
the form: 
\begin{equation}
\label{eq:backward:p:one:step}
\begin{split}
			&\Bigl[ I_d + 
			 \tfrac{T}p  
	\Bigl(	\eta_{(\ell+1)T/p}^{(p)} + \tfrac{T}p Q^\dagger Q {\mathbf 1}_{\{\ell \leq p-2\}} \Bigr)
	\Bigr] \tilde h_{\ell T/p}^{(p)} 
	\\
	&= 
	{\mathbb E}
	\biggl[ {\mathcal E}\bigl( \tfrac{1}{\varepsilon} 
	\tilde h_{\ell T/p}^{(p)}\bigr)
	  \Bigl(
	\tilde h_{(\ell+1)T/p}^{(p)} +   \tfrac{T}{ p}  
		Q^\dagger f\bigl(m_{(\ell+1)T/p}^{(p)}\bigr) 
		{\mathbf 1}_{\{\ell \leq p-2\}} \Bigr) \, \bigl\vert \, 
		{\mathcal F}_{\ell T/p}^{\boldsymbol W}
		 \biggr], 
\end{split}
\end{equation} 
for $\ell \in \{0,\cdots,p-1\}$, 
with the shorten notation 
\begin{equation*} 
{\mathcal E}\bigl( \tfrac{1}{\varepsilon} 
	\tilde h_{\ell T/p}^{(p)}\bigr)
	:=
\exp \Bigl(  - \tfrac{1}{\varepsilon}  
	\tilde h_{\ell T/p}^{(p)} \cdot \bigl( W_{(\ell+1)T/p} - W_{\ell T/p} \bigr)
	- \tfrac{T}{2 \varepsilon^2 p} 
	\vert 
	\tilde h_{\ell T/p}^{(p)} \vert^2 
	\Bigr).
\end{equation*} 
The challenge is thus to 
find, for each $\ell \in \{0,\cdots,p-1\}$,
a variable $\tilde h_{\ell T/p}^{(p)} \in L^2(\Omega,{\mathcal F}_{\ell T/p}^{p,{\boldsymbol W}},{\mathbb P};{\mathbb R}^d)$
that solves
\eqref{eq:backward:p:one:step}, 
when  
$\tilde h_{(\ell+1) T/p}^{(p)}$
is given  
in $L^2(\Omega,{\mathcal F}_{(\ell+1) T/p}^{p,{\boldsymbol W}},{\mathbb P};{\mathbb R}^d)$. 

In order to do so, we consider the mapping 
\begin{equation*} 
\begin{split}
\Phi_\ell : L^2(\Omega,{\mathcal F}_{\ell T/p}^{p,{\boldsymbol W}},{\mathbb P};{\mathbb R}^d)
\ni y 
&\mapsto 
	\Bigl[ I_d + 
			 \tfrac{T}p  
	\Bigl(	\eta_{(\ell+1)T/p}^{(p)} + \tfrac{T}p Q^\dagger Q {\mathbf 1}_{\{\ell \leq p-2\}} \Bigr)
	\Bigr]^{-1} 
\\
&\hspace{5pt} \times	{\mathbb E}
	\biggl[ {\mathcal E}\bigl( \tfrac{1}{\varepsilon} 
	y \bigr)
	  \Bigl(
	\tilde h_{(\ell+1)T/p}^{(p)} +   \tfrac{T}{p}  
		Q^\dagger f\bigl(m_{(\ell+1)T/p}^{(p)}\bigr) 
		{\mathbf 1}_{\{\ell \leq p-2\}} \Bigr)  \, \bigl\vert \, 
		{\mathcal F}_{\ell T/p}^{p,{\boldsymbol W} }\biggr]
		\end{split}
\end{equation*} 
and then look for a fixed point to it. 

%In the above conditional expectation, 
%the random variable $y$ 
%is ${\mathcal F}_{\ell T/p}^{p,{\boldsymbol W}}$-measurable 
%and 
%the random variable 
%$\tilde h_{(\ell+1)T/p}^{(p),\varpi} +   [T/(\varpi p) ] 
%		Q^\dagger f\bigl(m_{(\ell+1)T/p}^{(p)}\bigr) 
%		{\mathbf 1}_{\{\ell \leq p-2\}}$
%	is ${\mathcal F}_{(\ell+1) T/p}^{p,{\boldsymbol W}}$-measurable. 
%	In particular, the conditional expectation reduces to a mere expectation with respect
%	to the sole increment $W_{(\ell+1)T/p} - W_{\ell T/p}$. 
% 
%
%
Following the second step
of Lemma  
\ref{lem:1} (see in particular \eqref{eq:d:TV}), we have, for any two 
$y,y' \in L^2(\Omega,{\mathcal F}_{\ell T/p}^{p,{\boldsymbol W}},{\mathbb P};{\mathbb R}^d)$, 
with probability 1 under ${\mathbb P}$, 
\begin{equation*} 
\begin{split} 
&d_{\rm TV} 
\Bigl( {\mathbb P}_y \bigl( \cdot \, \vert \, {\mathcal F}_{\ell T/p}^{p,\boldsymbol W} \bigr), 
 {\mathbb P}_{y'} \bigl( \cdot \, \vert \, {\mathcal F}_{\ell T/p}^{p,\boldsymbol W} \bigr)\Bigr)
\leq \sqrt{2} \Bigl( 
\tfrac{1}{2 \varepsilon^2}
\tfrac{T}{p} 
\vert y - y' \vert^2
\Bigr)^{1/2}
\leq  \bigl(
 \tfrac{T}{p \varepsilon^2} \bigr)^{1/2} \vert y - y' \vert, 
\end{split} 
\end{equation*} 
where 
$ {\mathbb P}_y = {\mathcal E}(y) \cdot {\mathbb P}$ (and similarly for $y'$). 
Above, 
${\mathbb P}_y( \cdot \, \vert \, {\mathcal F}_{\ell T/p}^{p,\boldsymbol W})$
should be understood as a regular conditional probability 
of ${\mathbb P}_y$ on $(\Omega, {\mathcal F}_{(\ell+1) T/p}^{p,\boldsymbol W})$
given 
$ {\mathcal F}_{\ell T/p}^{p,\boldsymbol W}$ (and similarly for $y'$). 

Recalling that $\tilde h_{(\ell+1)T/p}^{(p)} +   [T/p ] 
	Q^\dagger f(m_{(\ell+1)T/p}^{(p)}) 
		{\mathbf 1}_{\{\ell \leq p-2\}}$ is bounded (by a constant independent of $\varepsilon$, $\ell$ and $p$), 
		we deduce that
		there exists a constant $c$ (depending on the same parameters as those quoted in the statement) 
		such that, with probability 1 under ${\mathbb P}$,
\begin{equation*} 
\vert \Phi_\ell(y') - \Phi_\ell(y) \vert \leq c  
 \tfrac{1}{\sqrt{p} \varepsilon} 
 \vert y' - y \vert.
\end{equation*} 		
And then, 
\begin{equation*} 
{\mathbb E} 
\bigl[
\vert \Phi_\ell(y') - \Phi_\ell(y) \vert^2 \bigr] \leq 
 \tfrac{c^2}{p \varepsilon^2} 
{\mathbb E} 
\bigl[ \vert y' - y \vert^2 \bigr],
\end{equation*} 		
which proves that $\Phi_\ell$ is a contraction in 
$ L^2(\Omega,{\mathcal F}_{\ell T/p}^{p,{\boldsymbol W}},{\mathbb P};{\mathbb R}^d)$
when $ \sqrt{p} \varepsilon>c$. 
By (backward) induction on $\ell$ (following the proof of Lemma 
\ref{lem:1:b}), we know that 
each $\tilde h_{\ell T/p}^{(p)}$ is bounded in $L^\infty$
by a constant $C$ 
(depending on the same parameters as those quoted in the statement).

Once 
$\tilde h_{\ell T/p}^{(p)}$ has been found in 
\eqref{eq:decoupled:limit:discrete}, it remains to find 
the process $(k_s^{(p)})_{\ell T/p \leq s \leq (\ell+1) T/p}$. This can be done by solving the following standard BSDE
(with the pair process $({\boldsymbol Y},{\boldsymbol Z})=(Y_s,Z_s)_{\ell T/p \leq s \leq (\ell+1) T/p}$ as unknown, taking values in 
${\mathbb R}^d \times {\mathbb R}^{d \times d}$):
\begin{equation*} 
\begin{split} 
&\ud Y_s
= - 	Q^\dagger f\bigl(m_{(\ell+1)T/p}^{(p)}\bigr) 
		{\mathbf 1}_{\{\ell \leq p-2\}} \ud s
		+  
	\Bigl(	
		\eta_{(\ell+1)T/p}^{(p)} + \tfrac{T}p Q^\dagger Q {\mathbf 1}_{\{\ell \leq p-2\}} \Bigr)
		  \tilde h_{\ell T/p}^{(p)}  \ud s 
		  \\
&\hspace{15pt}		  +   Z_s \tilde h_{\ell T/p}^{(p)} \ud s
		  + \varepsilon Z_s \ud W_{s},  
\end{split}
\end{equation*}
for $s \in [\ell T/p,(\ell+1)T/p]$, with the boundary condition $Y_{(\ell+1)T/p} = 
\tilde h_{(\ell+1)T/p}^{(p)}$ (at time $(\ell+1)T/p$).
Existence and uniqueness of a solution is standard because 
$\tilde h_{\ell T/p}^{(p)}$ is in $L^\infty$. 
Following 
the last sentence in Remark
\ref{rem:main:thm:2:b}
(about measurability of the conditioning 
of an 
${\mathcal F}_{(\ell+1) T/p}^{p,\boldsymbol W}$-measurable random variable 
given 
${\mathcal F}_{\ell T/p}^{\boldsymbol W}$), 
we deduce that $Y_{\ell T/p}$ coincides with 
$\tilde h_{\ell T/p}^{(p)}$ and eventually 
$(Z_s)_{\ell T/p \leq s \leq (\ell+1) T/p}$ 
solves the equation with $(k_s^{(p)})_{\ell T/p \leq s \leq (\ell+1) T/p}$ 
as unknown in
\eqref{eq:decoupled:limit:discrete}. Uniqueness of this choice 
can be easily checked by observing (from It\^o's isometry) that, for another 
solution, say $(k_s^{(p),\prime})_{\ell T/p \leq s \leq (\ell+1) T/p}$,
\begin{equation*}
\begin{split}
\varepsilon^2 {\mathbb E} \int_{\ell T/p}^{(\ell+1) T/p} 
\bigl\vert k_s^{(p)} - k_s^{(p),\prime} \bigr\vert^2 
\ud s 
&\leq 
C {\mathbb E} \biggl[ \biggl\vert 
 \int_{\ell T/p}^{(\ell+1) T/p}
 \bigl( k_s^{(p)} - k_s^{(p),\prime} \bigr) 
 \ud s 
 \biggr\vert^2 \biggr]
 \\
&\leq 
\tfrac{CT}{p} {\mathbb E} \biggl[   
 \int_{\ell T/p}^{(\ell+1) T/p}
 \bigl\vert k_s^{(p)} - k_s^{(p),\prime} \bigr\vert^2 
 \ud s 
 \biggr],
 \end{split}
\end{equation*} 
for a possibly new choice of the constant $C$. Uniqueness easily follows 
if $CT/p<\varepsilon^2$. 
\end{proof}

\subsection{Exploitation versus exploration}
\label{subse:exploitation:exploration}
We now address the \textcolor{black}{exploitability}, when we play the equilibrium strategy given by the 
fictitious play. 

\subsubsection{Approximate Nash equilibrium formed by 
the solution of the MFG with common noise}
We first prove (which is easier) that the 
solution of the 
$p$-discrete MFG with common noise
(as defined in the statement of Theorem 
\ref{main:thm:2} and in 
\eqref{eq:mfg:p:1}--\eqref{eq:mfg:p:2})
 forms an approximate Nash equilibrium of the game without common noise. 
In order to do so, we recall
from 
\eqref{eq:mfg:p:3} 
that
$ {\boldsymbol \alpha}^{(p),\star}$ is the 
equilibrium strategy of the 
$p$-discrete MFG associated with  
the cost functional 
\eqref{eq:mfg:p:1}
and that the optimal trajectory is given by 
\eqref{eq:mfg:p:3:bbis}.

We claim:
\begin{prop}
\label{prop:epsilon:Nash}
There exists a constant $C$, only depending on the parameters $d$,  $T$, $\| f\|_{1,\infty}$, $\| g \|_{1,\infty}$,
$\vert Q\vert$, $\vert R\vert$ \textcolor{black}{and ${\mathbb E}[\vert X_0\vert^2]$}, such that, for any integer $p \geq 1$, 
\begin{equation*}
\inf_{\boldsymbol \alpha}
	{\mathbb E}^{{\boldsymbol h}^{(p)} \hspace{-2pt} /\varepsilon} \Bigl[ {\mathcal R}^{p,X_0}\bigl({\boldsymbol \alpha};{\boldsymbol m}^{(p)};0\bigr)\Bigr]
	\geq 
	{\mathbb E}^{{\boldsymbol h}^{(p)} \hspace{-2pt} /\varepsilon} \Bigl[ {\mathcal R}^{p,X_0} \bigl({\boldsymbol \alpha}^{(p),\star};{\boldsymbol m}^{(p)};0\bigr)\Bigr]
	 - C \varepsilon,
\end{equation*}
the infimum being taken over 
${\mathbb F}^{p,X_{0},{\boldsymbol B},{\boldsymbol W}}=
(\sigma (X_{0},(B_{\tau_{p}(s)},W_{\tau_{p}(s)})_{s \leq t}))_{0 \leq t \leq T}$-progressively measurable 
and square-integrable process
$({\boldsymbol \alpha}_{t})_{0 \leq t \leq T}$ that is piecewise constant on any interval $[\ell T/p,(\ell +1)T/p)$, for $\ell \in \{0,\cdots,p-1\}$. 
\end{prop}

We start with the following lemma:

\begin{lem}
\label{lem:3:0}
There exists a constant $C$, only depending on  $d$, $T$, $\| f\|_{1,\infty}$, $\| g \|_{1,\infty}$,
$\vert Q\vert$ and $\vert R\vert$, such that, for any integer $p \geq 1$ and any 
${\mathbb F}^{p,X_{0},{\boldsymbol B},{\boldsymbol W}}$-progressively measurable 
and square-integrable process
${\boldsymbol \alpha} = ({\alpha}_{t})_{0 \leq t \leq T}$ that is piecewise constant on any interval $[\ell T/p,(\ell +1)T/p)$, for $\ell \in \{0,\cdots,p-1\}$, 
\begin{equation*}
\begin{split}
&\biggl\vert 
{\mathbb E}^{{\boldsymbol h}^{(p)} \hspace{-2pt} /\varepsilon} \Bigl[ {\mathcal R}^{X_0}\bigl({\boldsymbol \alpha};{\boldsymbol m}^{(p)};\varepsilon {\boldsymbol W}^{p, {\boldsymbol h}^{(p)} \hspace{-2pt} /\varepsilon}\bigr) \Bigr]
- 
{\mathbb E}^{ {\boldsymbol h}^{(p)} \hspace{-2pt} /\varepsilon} \Bigl[ {\mathcal R}^{X_0} \bigl({\boldsymbol \alpha};{\boldsymbol m}^{(p)};0\bigr)   \Bigr]
\biggr\vert 
\\
&\hspace{15pt} \leq   
C \biggl( 1+   {\mathbb E} \bigl[ \vert X_{0} \vert^2 \bigr]^{1/2} 
+
{\mathbb E}^{ {\boldsymbol h}^{(p)} \hspace{-2pt} /\varepsilon}
\biggl[ \int_{0}^T \vert \alpha_{t} \vert^2 \ud t  
\biggr]^{1/2} \biggr)
\varepsilon. 
\end{split}
\end{equation*}
\end{lem}
The symbol ${\mathbb E}$ instead of 
${\mathbb E}^{{\boldsymbol h}^{(p)} \hspace{-2pt} /\varepsilon}$ for the $L^2$ moment of the initial condition is 
fully justified by the fact that $X_{0}$ is ${\mathcal F}_{0}$-measurable. 
\begin{proof}
For the same initial condition $X_{0}$ and the same ${\boldsymbol \alpha}$ as in the statement, we let 
\begin{equation*}
\begin{split}
&\ud X_{t}^{(p),\varepsilon} = \alpha_{t} \ud t + \sigma \ud B_{t} + \varepsilon \ud W_{t}^{p, {\boldsymbol h}^{(p)} \hspace{-2pt} / \varepsilon},
\quad t \in [0,T]. 
\end{split}
\end{equation*}
In particular, we write ${\boldsymbol X}^{(p),0}=(X_{t}^{(p),0})_{0 \le t \le T}$
when there is no common noise. Then, we obviously have
\begin{equation*}
 \sup_{0 \leq t \leq T} \vert X_{t}^{\textcolor{black}{\varepsilon}} - X_{t}^0 \vert \leq C 
\varepsilon \sup_{0 \leq t \le T} \vert W_{t}^{p,{\boldsymbol h}^{(p)} \hspace{-2pt} / \varepsilon} \vert
\leq C \varepsilon \sup_{0 \leq t \le T} \vert W_{t}^{{\boldsymbol h}^{(p)} \hspace{-2pt} / \varepsilon} \vert. 
\end{equation*}
Meanwhile, we observe that 
\begin{equation*}
{\mathbb E}^{{\boldsymbol h}^{(p)} \hspace{-2pt} / \varepsilon}
\Bigl[ \sup_{0 \leq t \leq T} \bigl( \vert X_{t}^{\textcolor{black}{\varepsilon}} \vert^2 
+ 
\vert X_{t}^0 \vert^2 
\bigr)  \Bigr] \leq C \biggl( 1 + {\mathbb E} \bigl[ \vert X_{0} \vert^2 \bigr] + 
{\mathbb E}^{{\boldsymbol h}^{(p)} \hspace{-2pt} / \varepsilon} \biggl[ \int_{0}^T \vert \alpha_{t} \vert^2 \ud t  \biggr]
 \biggr).
\end{equation*}
Recalling
\eqref{eq:mathcalR:p},  
we get the result. 
\end{proof}
The next lemma says that we can easily restrict ourselves to controls with a finite energy.

\begin{lem}
\label{lem:3:1}
There exists a constant $A$, only depending on the parameters $d$,  $T$, $\| f\|_{1,\infty}$,
$\| g \|_{1,\infty}$, $\vert Q\vert $, $\vert R\vert$
and \textcolor{black}{${\mathbb E}[\vert X_0 \vert^2]$}
 such that, for any 
$\varepsilon \in (0,1]$, any integer $p \geq 1$ and any 
${\mathbb F}^{p,X_{0},{\boldsymbol B},{\boldsymbol W}}$-progressively measurable 
and square-integrable process
${\boldsymbol \alpha} = ({\alpha}_{t})_{0 \leq t \leq T}$ that is piecewise constant on any interval $[\ell T/p,(\ell +1)T/p)$, for $\ell \in \{0,\cdots,p-1\}$, 
\begin{equation*}
\begin{split}
&{\mathbb E}^{ {\boldsymbol h}^{(p)} \hspace{-2pt} /\varepsilon}
\biggl[ \int_{0}^T \vert \alpha_{t} \vert^2 \ud t 
\biggr] \geq A
\\
&\hspace{15pt}  \Rightarrow 
{\mathbb E}^{  {\boldsymbol h}^{(p)} \hspace{-2pt} /\varepsilon} \Bigl[ {\mathcal R}^{p,X_0}\bigl({\boldsymbol \alpha};{\boldsymbol m}^{(p)};0\bigr)  \Bigr] \geq
{\mathbb E}^{  {\boldsymbol h}^{(p)} \hspace{-2pt} /\varepsilon} \Bigl[ {\mathcal R}^{p,X_0}\bigl(0;{\boldsymbol m}^{(p)};\varepsilon {\boldsymbol W}^{p, {\boldsymbol h}^{(p)} \hspace{-2pt} /\varepsilon}\bigr)  \Bigr] + 1. 
\end{split}
\end{equation*}
\end{lem}

\begin{proof}
It suffices to observe that 
\begin{equation*}
{\mathbb E}^{ {\boldsymbol h}^{(p)} \hspace{-2pt} /\varepsilon} \Bigl[ {\mathcal R}^{p,X_0}\bigl({\boldsymbol \alpha};{\boldsymbol m}^{(p)};0\bigr)   \Bigr]
\geq \frac12 
{\mathbb E}^{  {\boldsymbol h}^{(p)} \hspace{-2pt} /\varepsilon}
\biggl[ \int_{0}^T \vert \alpha_{t} \vert^2 \ud t  
\biggr].
\end{equation*}
Meanwhile, 
\begin{equation*}
{\mathbb E}^{ {\boldsymbol h}^{(p)} \hspace{-2pt} /\varepsilon} \Bigl[ {\mathcal R}^{p,X_0}\bigl(0;{\boldsymbol m}^{(p)};\varepsilon {\boldsymbol W}^{p,  {\boldsymbol h}^{(p)} \hspace{-2pt} /\varepsilon}\bigr)  \Bigr]
\leq C.
\end{equation*}
The proof is easily completed. 
\end{proof}

We now complete the

\begin{proof}[Proof of Proposition \ref{prop:epsilon:Nash}.]
By Lemma \ref{lem:3:1}, 
we can reduce ourselves to 
the set ${\mathcal E}_{A}$ of processes ${\boldsymbol \alpha}$ (as in the statement of Proposition \ref{prop:epsilon:Nash}) such that 
\begin{equation*}
{\mathbb E}^{ {\boldsymbol h}^{(p) } \hspace{-2pt} /\varepsilon}
\biggl[ \int_{0}^T \vert \alpha_{t} \vert^2 \ud t 
\biggr] \leq A.
\end{equation*}
We then apply 
Lemma 
\ref{lem:3:0}, which says that 
\begin{equation*}
\sup_{{\boldsymbol \alpha} \in {\mathcal E}_{A}}
	\Bigl\vert {\mathbb E}^{  {\boldsymbol h}^{(p)} \hspace{-2pt} /\varepsilon}\Bigl[ {\mathcal R}^{p,X_0} \bigl({\boldsymbol \alpha};{\boldsymbol m}^{(p)};\varepsilon {\boldsymbol W}^{p,  {\boldsymbol h}^{(p)} \hspace{-2pt} /\varepsilon} \bigr)\Bigr]
	-
		{\mathbb E}^{  {\boldsymbol h}^{(p)}  \hspace{-2pt}  /\varepsilon} \Bigl[ {\mathcal R}^{p,X_0}  \bigl({\boldsymbol \alpha};{\boldsymbol m};0\bigr)\Bigr]
		\Bigr\vert \leq C \varepsilon. 
\end{equation*} 
The proof is easily completed, using 
\eqref{eq:mfg:p:1}--\eqref{eq:mfg:p:3},
from which we get that, for any 
${\boldsymbol \alpha} \in {\mathcal E}_{A}$, 
\begin{equation*}
\begin{split}
{\mathbb E}^{{\boldsymbol h}^{(p)}  \hspace{-2pt} /\varepsilon} \Bigl[ {\mathcal R}^{p,X_0} \bigl(
{\boldsymbol \alpha}^{(p),\star};{\boldsymbol m}^{(p)}; \varepsilon {\boldsymbol W}^{p, {\boldsymbol h}^{(p)}  \hspace{-2pt} /\varepsilon}\bigr)\Bigr]
&\leq {\mathbb E}^{ {\boldsymbol h}^{(p)}  \hspace{-2pt} /\varepsilon} \Bigl[ {\mathcal R}^{p,X_0} \bigl({\boldsymbol \alpha};{\boldsymbol m}^{(p)};  \varepsilon {\boldsymbol W}^{p, {\boldsymbol h}^{(p) }  \hspace{-2pt} /\varepsilon}\bigr)\Bigr]
\\
&\leq
{\mathbb E}^{ {\boldsymbol h}^{(p) }  \hspace{-2pt} /\varepsilon}  \Bigl[ {\mathcal R}^{p,X_0} \bigl({\boldsymbol \alpha};{\boldsymbol m}^{(p)};0 \bigr)\Bigr] + C \varepsilon. 
\end{split}
\end{equation*} 
This completes the proof. 
\end{proof}

\subsubsection{Approximate Nash equilibrium formed by the return of the fictitious play}
We now state a similar result, but for the approximation returned by the fictitious play
subjected to the $\varepsilon$-randomization and to the $p$-discretization. 
Throughout this paragraph, 
${\boldsymbol h}^{n+1}$ and 
${\boldsymbol m}^{n+1}$ are as in 
\eqref{eq:hn+1:tildeh:n+1}
and
\eqref{eq:mn+1:m:n+1:correc}. We recall that these two processes depend on $p$ (although it is not indicated in the notation). 

The point is then to consider 
the strategy ${\boldsymbol \alpha}^{(p),n,\diamond}$ defined by 
\begin{equation}
\label{eq:X:n+1:nearlyNash:000}
\alpha_t^{(p),n,\diamond} := - \eta_{\tau_p(t)}^{(p)} X_{\tau_p(t)}^{(p),\star} - \varpi h_t^{n}, 
\end{equation} 
where we recall 
(see
\eqref{eq:mfg:p:3:bbis})
\begin{equation}
\label{eq:X:n+1:nearlyNash}
\begin{split}
\ud X^{(p),\star}_{t} 
&= - \eta_{\tau_p(t)}^{(p)} X_{\tau_p(t)}^{(p),\star} \ud t +  \sigma \ud B_{t} + \varepsilon \ud W_{t}^{p}
\\
&= \alpha_{t}^{(p),n,\diamond} \ud t +  \sigma \ud B_{t} + \varepsilon \ud W_{t}^{p,\varpi {\boldsymbol h}^n \hspace{-2pt} /\varepsilon}, \quad t \in [0,T]; 
\quad X_0^{(p),n,\diamond} = X_0. 
\end{split}
\end{equation}
Following 
\eqref{eq:mn+1:m:n+1:correc}, 
we have
\begin{equation}
\label{eq:X:n+1:nearlyNash:2}
{\mathbb E}^{\varpi {\boldsymbol h}^n \hspace{-2pt} / \varepsilon}
\Bigl[ X_{t}^{(p),\star} \, \vert \,  \sigma({\boldsymbol W}) \Bigr]=
{\mathbb E} 
\Bigl[ X_{t}^{(p),\star} \, \vert \,  \sigma({\boldsymbol W}) \Bigr]  = m_t^{(p)}, \quad t \in [0,T]. 
\end{equation}
In other words, the conditional state under the candidate strategy 
is the environment itself, which is a pre-requisite in mean field games. 
Notice that 
${\boldsymbol \alpha}^{(p),n,\diamond}$
and ${\boldsymbol h}^n$ and  do depend on $\varepsilon$. 

\begin{thm}
\label{thl:main:regret}
\textcolor{black}{Assume that the law of the initial condition has sub-Gaussian tails, i.e. 
${\mathbb P}(\{ \vert X_0 \vert \geq r \}) \leq a^{-1} \exp( - a r^2)$, for some 
$a>0$ and for any $r >0$.} Then, 
there exist two positive constants $c$ and $C$, only depending on the parameters $a$, $d$, $T$, $\| f\|_{1,\infty}$,  $\| g \|_{1,\infty}$, 
$\vert Q\vert$ and $\vert R\vert$, such that, {for any 
$\varepsilon \in (0,1]$, any integer $p \geq 1$ satisfying $p \varepsilon^2 \geq c$ 
and any integer $n \geq 1$}, 
\begin{equation}
\label{eq:bd:th:main:regret}
\begin{split}
\inf_{\boldsymbol \alpha}
	{\mathbb E}^{\varpi {\boldsymbol h}^{n} \hspace{-2pt} /\varepsilon} \Bigl[ {\mathcal R}^{X_0}\bigl({\boldsymbol \alpha};{\boldsymbol m}^{(p)};0\bigr)\Bigr]
	&\geq 
	{\mathbb E}^{\varpi {\boldsymbol h}^{n} \hspace{-2pt} /\varepsilon} \Bigl[ {\mathcal R}^{X_0} \bigl({\boldsymbol \alpha}^{(p),n,\diamond};{\boldsymbol m}^{(p)};0\bigr)\Bigr]  
	\\
&\hspace{15pt} -  C \varepsilon   - C  p^{-1} -  \exp(C \varepsilon^{-2})  \varpi^{-n},
\end{split}
\end{equation}
the infimum in the left-hand side being taken over 
${\mathbb F}^{X_0,{\boldsymbol B},{\boldsymbol W}}$-progressively measurable 
and square-integrable processes
$({\boldsymbol \alpha}_{t})_{0 \leq t \leq T}$. 

\textcolor{black}{The difference between the term in the left-hand side and the first term in the right-hand side of 
\eqref{eq:bd:th:main:regret} is called the
${\mathbb P}^{\varpi {\boldsymbol h}^{n} \hspace{-2pt} /\varepsilon}$-mean exploitability of the policy ${\boldsymbol \alpha}^{(p),n,\diamond}$. (Obviously, this difference is non-positive.)}
\end{thm}

We feel useful to comment on the meaning and scope of Theorem \ref{thl:main:regret}. We first observe that, in the 
functional ${\mathcal R}$ that appears in ${\mathcal R}^{X_0}({\boldsymbol \alpha};{\boldsymbol m}^{(p)};0)$
and ${\mathcal R}^{X_0}({\boldsymbol \alpha}^{(p),n,\diamond};{\boldsymbol m}^{(p)};0)$, on both sides of the main inequality
\eqref{eq:bd:th:main:regret}, there is no time discretization and no common noise (since the third input is $0$). In other words, both the costs to
$({\boldsymbol \alpha},{\boldsymbol m}^{(p)})$  
and
$({\boldsymbol \alpha}^{(p),n,\diamond},{\boldsymbol m}^{(p)})$  
are computed according to the time-continuous original model without common noise (even though the control ${\boldsymbol \alpha}^{(p),n,\diamond}$ is piecewise constant and random as an output of the scheme). In particular, for a given realization of the common noise ${\boldsymbol W}$ 
(which does manifest here because ${\boldsymbol \alpha}$, ${\boldsymbol \alpha}^{(p),n,\diamond}$, ${\boldsymbol m}^{(p)}$ and ${\boldsymbol h}^{n}$ 
depend on ${\boldsymbol W}$), the conditional expectations 
  ${\mathbb E}^{\varpi {\boldsymbol h}^{n} \hspace{-2pt} /\varepsilon} [{\mathcal R}^{X_0}({\boldsymbol \alpha};{\boldsymbol m}^{(p)};0) \, \vert \, 
  \sigma({\boldsymbol W})]$
  and 
  ${\mathbb E}^{\varpi {\boldsymbol h}^{n} \hspace{-2pt} /\varepsilon} [{\mathcal R}^{X_0}({\boldsymbol \alpha}^{(p),n,\diamond};{\boldsymbol m}^{(p)};0)\, \vert \, 
  \sigma({\boldsymbol W})]$
  coincide   with the 
costs 
$J({\boldsymbol \alpha}(\cdot,\cdot,{\boldsymbol W}) ; {\boldsymbol m}^{(p)}({\boldsymbol W}) )$
and
$J({\boldsymbol \alpha}^{(p),n,\diamond}(\cdot,\cdot,{\boldsymbol W}) ; {\boldsymbol m}^{(p)}({\boldsymbol W}) )$, for 
$J$ as 
in 
\eqref{eq:cost:intro} (with the expectation in the latter being just performed over $(X_0,{\boldsymbol B})$). 
In particular, the notations 
${\boldsymbol \alpha}(\cdot,\cdot,{\boldsymbol W})$, 
${\boldsymbol \alpha}^{(p),n,\diamond}(\cdot,\cdot,{\boldsymbol W})$
and
${\boldsymbol m}^{(p)}({\boldsymbol W})$ are used here to emphasize that, in the computation of $J$, the realization of 
the common noise is kept frozen 
in the inputs 
${\boldsymbol \alpha}(\cdot,\cdot,{\boldsymbol W})$, 
${\boldsymbol \alpha}^{(p),n,\diamond}(\cdot,\cdot,{\boldsymbol W})$
and
${\boldsymbol m}^{(p)}({\boldsymbol W})$. 
In turn, the two expectations in \eqref{eq:bd:th:main:regret}
can be rewritten as 
${\mathbb E}^{\varpi {\boldsymbol h}^{n} \hspace{-2pt} /\varepsilon} [J({\boldsymbol \alpha}(\cdot,\cdot,{\boldsymbol W});{\boldsymbol m}^{(p)}(\cdot,{\boldsymbol W}))]$
and
${\mathbb E}^{\varpi {\boldsymbol h}^{n} \hspace{-2pt} /\varepsilon} [ J({\boldsymbol \alpha}^{(p),n,\diamond}(\cdot,\cdot,{\boldsymbol W});{\boldsymbol m}^{(p)}(\cdot,{\boldsymbol W}))]$, 
with the expectation ${\mathbb E}^{\varpi {\boldsymbol h}^{n} \hspace{-2pt} /\varepsilon}$ being now taken under the sole common noise. 
This can be rephrased quite simply: 
Theorem 
\ref{thl:main:regret}
provides a bound for the ${\mathbb P}^{\varpi {\boldsymbol h}^{n} \hspace{-2pt} /\varepsilon}$-mean  \textcolor{black}{exploitability} associated with the original MFG, 
the mean being here taken with respect to 
the tilted law of the common noise (i.e., the law of ${\boldsymbol W}$ under ${\mathbb P}^{\varpi {\boldsymbol h}^{n} \hspace{-2pt} /\varepsilon}$). 
The bound that is given for the mean \textcolor{black}{exploitability} depends on the three parameters $\varepsilon$, $n$ and $p$. Typically, we want to choose 
$\varepsilon$ small and $p$ large, which is well understood: the equilibrium $({\boldsymbol m}^{(p)},{\boldsymbol \alpha}^{(p),n,\diamond})$
is learnt for the $p$-discrete MFG with an $\varepsilon$-common noise; if $\varepsilon$ is large or $p$ is small, the equilibrium that is learnt cannot be
 a `good' approximate equilibrium of the original MFG. This is exactly what the terms $-C \varepsilon$ and $-C/p$ say in 
\eqref{eq:bd:th:main:regret}. 
As for the last term in 
\eqref{eq:bd:th:main:regret}, it
becomes smaller and smaller as $n$ increases. 
This is also consistent with the intuition: the mean \textcolor{black}{exploitability} becomes small when the number 
$n$ of iterations of the fictitious play is large. 
For sure, there is a price to pay: as $\varepsilon$ tends $0$, the impact of the common noise becomes smaller and 
$n$ has to be chosen larger to attain the same accuracy for the mean \textcolor{black}{exploitability}. 
%This is what the last term says 
%in
%\eqref{eq:bd:th:main:regret}.

The reader must \textcolor{black}{also} realize that, differently from what is done 
in	\eqref{eq:optimization:exploration}, the environment used in the cost functional 
is the conditional state of the 
reference particle when using the strategy ${\boldsymbol \alpha}^{(p),n,\diamond}$. This looks subtle, but 
this makes a big difference in the analysis. 
\textcolor{black}{Indeed, if we had to 
follow
\eqref{eq:optimization:exploration}
and thus  
use $\overline{\boldsymbol m}^n$ as environment in 
the cost functional, then 
%the minimum in the left-hand side would be reached at 
%${\boldsymbol \alpha}={\boldsymbol \alpha}^{(p),n+1}$. However, 
the 
resulting minimizing path
would NOT be typical of the 
environment, meaning that its (theoretical) conditional expectation
(under ${\mathbb P}^{\varpi {\boldsymbol h}^n  \hspace{-2pt} / \varepsilon} $ or ${\mathbb P}$) 
given the common noise would NOT match
%\footnote{\color{red} 
%The interested reader will observe that, even though 
%the minimizing path is not typical of the environment, 
%it is in fact `not that far'. Indeed, using Proposition 
%\ref{prop:2:b} below, one would be able 
%to prove that the mean distance (under ${\mathbb P}^{\boldsymbol h}$) between 
%${\boldsymbol m}^{n+1}$ (which is the conditional expectation 
%of the minimizing path given the common noise, see 
%\eqref{eq:X:n+1:nearlyNash:2}) and 
%$\overline{\boldsymbol m}^n$
%tends
%to $0$ as 
%$n$ and $p$ tend to $\infty$. However, 
%some work would be needed in order to state a similar property but under 
%the probability ${\mathbb P}^{{\boldsymbol h}^n}$, 
%the distance between  
%${\mathbb P}^{{\boldsymbol h}^n}$
%and
%${\mathbb P}^{{\boldsymbol h}}$
%being upper bounded
%by 
%Remark 
%\ref{rem:thm:2}
%We leave the details to the reader since 
%the proof of 
%Theorem \ref{thl:main:regret} below relies on a similar sketch.}
$\overline{\boldsymbol m}^n$. 
Differently, our choice 
to use ${\boldsymbol m}^{(p)}$ as environment in the 
statement of Theorem 
\ref{thl:main:regret}
allows for the following property: 
the 
dynamics 
\eqref{eq:X:n+1:nearlyNash}
are typical of the environment, meaning that 
the environment is exactly given by the conditional state given 
the realization of the common noise, which is exactly what 
\eqref{eq:X:n+1:nearlyNash:2} says. 
In this framework,  
Theorem 
\ref{thl:main:regret} asserts that, by deviating unilaterally from the rest of the population, 
a reference player could hardly 
increase her cost, at best by a remainder that 
is small when $p$ and $n$ are large and $\varepsilon$ is fixed.}

We end-up this discussion with the following two observations, which may echo possible questions of the reader. 
First, one may wonder about the practical implementation of 
\eqref{eq:X:n+1:nearlyNash}, as it requires to know the value of 
${\boldsymbol \eta}^{(p)}$. In fact, as it is clarified in the next section, any efficient method 
that would be used in the implementation of the fictitious play 
should learn, at the same time, the solution 
to the Riccati equation \eqref{eq:discretized:Riccati}, even though the coefficients $f$ and $g$ are not known. This is somehow a pre-requisite 
in the implementation of the fictitious play. Second, the fact that the dynamics 
\eqref{eq:X:n+1:nearlyNash}
are independent of 
$n$ (under the historical measure) may be rather intriguing. Of course, one must recall in this regard that, on a statistical point of view, the dynamics that truly matter are
in fact 
those  computed under the tilted probability measure. Indeed, the two costs
in 
\eqref{eq:bd:th:main:regret}
are averaged under the tilted measure. Should we represent those costs in terms of an infinite cloud of particles,  
then all the particles forming the approximate Nash equilibrium 
provided by  
Theorem 
\ref{thl:main:regret}
should be subjected to a common noise that would be sampled under the tilted probability measure. 
This is exactly the point where the dependence on $n$ manifests.

We start with the following 
refinement of 
Theorem 
\ref{main:thm:2}: 

\begin{lem}
\label{prop:2:b}
With the same notation as in the statement of Theorem 
\ref{main:thm:2}, 
there exist two positive constants $c$ and $C$, only depending on the parameters $a$, $d$, $T$, $\| f\|_{1,\infty}$,  $\| g \|_{1,\infty}$, 
$\vert Q\vert$ and $\vert R\vert$, such that, {for any 
$\varepsilon \in (0,1]$, any integer $p \geq 1$ satisfying $p \varepsilon^2 \geq c$ 
and any integer $n \geq 1$}, 
%
%
%there exists a constant $C$, only depending on $d$, $T$, the norms
%$\| f\|_{1,\infty}$ and $\| g \|_{1,\infty}$
%and the norms $\vert Q\vert$ and $\vert R\vert$ of the matrices $Q$ and $R$,
%such that, for any integer $p \geq 1$ and 
%any $\varepsilon \in (0,1)$ with $p\varepsilon^{-2} \geq 1$, 
\begin{equation*}
\begin{split}
&{\mathbb E}^{ {\boldsymbol h}^{(p)} \hspace{-2pt} /\varepsilon}
\Bigl[
 \sup_{0 \le t \le T} 
\vert \alpha_{t}^{(p),n,\diamond}
-\alpha_{t}^{(p), \star}
\vert^2 
\Bigr]
 \leq   \varpi^{-2n} \exp(C \varepsilon^{-2}),
\end{split}
\end{equation*}
where we recall
that  
${\boldsymbol \alpha}^{(p),\star}$ is the 
equilibrium strategy of the 
$p$-discrete 
MFG with an $\varepsilon$ common noise, as given by 
\eqref{eq:mfg:p:3}.
%, and 
%${\boldsymbol X}^{(p),\star}$ is the corresponding controlled trajectory, see
%\eqref{eq:mfg:p:3:bbis}. 
\end{lem}

\begin{proof}[Proof of Lemma \ref{prop:2:b}]
We recall
from 
\eqref{eq:mfg:p:3}
and 
\eqref{eq:X:n+1:nearlyNash:000}
%\eqref{eq:X:n+1:nearlyNash:000}
%and
 that 
\begin{equation*}
\alpha_t^{(p),n,\diamond}
-
\alpha_t^{(p),\star} = \varpi h^n_t - h^{(p)}_t, \quad t \in [0,T], 
\end{equation*} 
and the bound in the statement follows
from Theorem 
\ref{main:thm:2}. 
\end{proof}

We now 
turn to 
\begin{proof}[Proof of Theorem \ref{thl:main:regret}]
{ \ }

\textit{First Step.}
Recalling the shape of the optimizer ${\boldsymbol \alpha}^{(p),n+1,\diamond}$
in 
\eqref{eq:X:n+1:nearlyNash:000}--\eqref{eq:X:n+1:nearlyNash}
, together with 
the fact that ${\boldsymbol h}^{n+1}$
is bounded, uniformly in $n \geq 1$, see 
Lemma \ref{lem:1:b} (recalling that $C_1$ therein is independent of $n$, as mentioned in the statement), we deduce that 
\begin{equation*}
\Bigl\vert 	{\mathbb E}^{\varpi {\boldsymbol h}^{n} \hspace{-2pt} / \varepsilon} \Bigl[ {\mathcal R}^{X_0}\bigl({\boldsymbol \alpha}^{(p),n,\diamond};{\boldsymbol m}^{(p)};0\bigr)\Bigr]
\Bigr\vert \leq C, 
\end{equation*}
where the constant $C$ is independent $n$. In the rest of the proof, the value of $C$ may vary from line to line provided it remains independent 
of $n$. Then, we must invoke the fact that the problem
$$\inf_{\boldsymbol \alpha}
	{\mathbb E}^{\varpi {\boldsymbol h}^{n} \hspace{-2pt} / \varepsilon} \Bigl[ {\mathcal R}^{X_0}\bigl({\boldsymbol \alpha};{\boldsymbol m}^{(p)};0\bigr)\Bigr]$$
	is a linear quadratic-problem in a random environment. 
	Similar to the minimization problem 
studied in Lemma \ref{lem:2.1}, its solution is given in feedback form. If we call $\hat{\boldsymbol \alpha}^{n+1}$ the minimizer, it reads 
\begin{equation}
\label{eq:5th:step:correc:1:e} 
\hat{\alpha}_{t}^{n} = - \bigl( \eta_{t} \hat X_{t}^{n} + \hat{h}_{t}^{n} \bigr), \quad t \in [0,T],
\end{equation}
with 
\begin{equation}
\label{eq:5th:step:correc:1:c} 
\ud \hat X_{t}^{n} = - \bigl( \eta_{t} \hat X_{t}^{n} + \hat{h}_{t}^{n} \bigr)
\ud t +\sigma  \ud B_{t},% +\varepsilon \ud W_{t}^{{\boldsymbol h}^n},
\quad t \in [0,T] ; \quad \hat{X}_0^{n} = X_0, 
\end{equation}
and where%, ${\mathbb P}^{{\boldsymbol h}^n}$ almost surely,
\begin{equation}
\label{eq:5th:step:correc:1:d} 
\begin{split} 
&\ud \hat{h}_{t}^{n} = \Bigl( - Q^\dagger f\bigl( m^{(p)}_{t}\bigr) + \eta_{t} \hat{h}_{t}^{n} \Bigr) \ud t + \varepsilon \hat{ k}_{t}^{n} \ud W_{t}^{\varpi {\boldsymbol h}^n \hspace{-2pt} /\varepsilon}, \quad t \in [0,T],
\\
&\hat{h}_T^{n} = R^\dagger g \bigl( m^{(p)}_T \bigr). 
\end{split} 
\end{equation}
We start with a series of results whose proof is given in the last step below. 
We first claim that, for each $k \in \{0,\cdots,p\}$, 
$\hat{h}_{kT/p}^{n}$ is 
${\mathcal F}_{kT/p}^{{\boldsymbol W}^p}$-measurable. 
Moreover, 
\begin{equation}
\label{eq:5th:step:correc:1:f} 
\textrm{\rm essup}_{\omega \in \Omega}  \sup_{0 \le t \le T} \vert \hat{h}_{t}^{n} \vert \leq C. 
\end{equation}
As a result, with probability 1 under ${\mathbb P}^{\varpi {\boldsymbol h}^n \hspace{-2pt} /\varepsilon}$
(and also under ${\mathbb P}$),
\begin{equation}
\label{eq:bound:hat:X:alpha}
\begin{split}
&\sup_{0 \leq t \leq T} 
\vert \hat{X}_{t}^{n} \vert \leq C \Bigl( 1 + 
\vert X_0 \vert+
%\sup_{0 \leq t \leq T} \vert \varepsilon W_{t}^{{\boldsymbol h}^n} \vert + 
\sup_{0 \leq t \leq T} \vert B_{t} \vert \Bigr),
\\
&\sup_{0 \leq t \leq T} 
\vert \hat{\alpha}_{t}^{n} \vert \leq C \Bigl( 1 + 
\vert X_0 \vert+
%\sup_{0 \leq t \leq T} \vert\varepsilon W_{t}^{{\boldsymbol h}^n} \vert + 
\sup_{0 \leq t \leq T} \vert B_{t} \vert \Bigr),
\end{split}
\end{equation}
which implies 
\begin{equation}
\label{eq:bound:hat:X:alpha:correc}
 {\mathbb E}^{\varpi {\boldsymbol h}^{n} \hspace{-2pt} / \varepsilon}  \Bigl[ 
\sup_{0 \leq t \leq T} \bigl\vert \hat{X}_{t}^{n} \bigr\vert^2
\Bigr] \leq C. 
\end{equation}
and
\begin{equation}
\label{eq:bound:hat:X:alpha:2}
{\mathcal R}^{X_0}\bigl( \hat{\boldsymbol \alpha}^{n};{\boldsymbol m}^{(p)};0\bigr) 
 \leq C \Bigl( 1 + 
 \vert X_0 \vert^2 + 
%\sup_{0 \leq t \leq T} \vert\varepsilon W_{t}^{{\boldsymbol h}^n} \vert^2 + 
\sup_{0 \leq t \leq T} \vert B_{t} \vert^2 \Bigr). 
\end{equation}

\textit{Second Step.}
By the first step,
\begin{equation*}
\begin{split}
\inf_{\boldsymbol \alpha}
	{\mathbb E}^{\varpi {\boldsymbol h}^{n} \hspace{-2pt} / \varepsilon} \Bigl[ {\mathcal R}^{X_0}\bigl({\boldsymbol \alpha};{\boldsymbol m}^{(p)};0\bigr)\Bigr]
 &= 	{\mathbb E}^{\varpi {\boldsymbol h}^{n} \hspace{-2pt} / \varepsilon} \Bigl[ {\mathcal R}^{X_0}\bigl(\hat{\boldsymbol \alpha}^{n};{\boldsymbol m}^{(p)};0\bigr)\Bigr].
\end{split}
\end{equation*}
Here, 
%\textcolor{red}{ici, il y a une incertitude sur le choix de l'environnement : continu ou discret ?}
\begin{equation*} 
\begin{split}
& {\mathbb E}^{\varpi {\boldsymbol h}^{n} \hspace{-2pt} / \varepsilon} \Bigl[ {\mathcal R}^{X_0} \bigl(\hat{\boldsymbol \alpha}^{n};{\boldsymbol m}^{(p)};0\bigr)\Bigr]
\\
&= \tfrac12 {\mathbb E}^{\varpi {\boldsymbol h}^{n} \hspace{-2pt} / \varepsilon}
\biggl[ \vert R^{\dagger} \hat X_T^{n}
+ g(m_T^{(p)}) \vert^2 
+ \int_0^T \Bigl\{ \vert Q^{\dagger} \hat{X}_t^{n} 
+ f (m_t^{(p)}) \vert^2 + \vert \hat{\alpha}^{n}_t \vert^2 
\Bigr\} \ud t
\biggr],
\end{split}
\end{equation*} 
and the purpose is to lower bound the above cost by ${\mathcal R}^{p,X_0}(\hat{\boldsymbol \alpha}^{(p),n};{\boldsymbol m}^{(p)};0)$
for some well-chosen
${\mathbb F}^{p,X_{0},{\boldsymbol B},{\boldsymbol W}}$-progressively measurable 
and piecewise constant 
control $\hat{\boldsymbol \alpha}^{(p),n}$ and with 
${\mathcal R}^{p,X_0}$ as in 
\eqref{eq:mathcalR:p}. In order to do so, we let 
\begin{equation}
\label{eq:widehat:h:taupt:p,n+1} 
\hat{h}_{\tau_p(t)}^{(p),n}
:= \tfrac{p}T 
 {\mathbb E}^{\varpi {\boldsymbol h}^{n} \hspace{-2pt} / \varepsilon}
 \biggl[ 
\int_{\tau_p(t)}^{\tau_p(t) + T/p}   \hat{h}_s^{n}   \ud s \, \vert \, 
{\mathcal F}_{\tau_p(t)}^{p,X_0,{\boldsymbol B},{\boldsymbol W}} 
\biggr], \quad t \in [0,T], 
\end{equation} 
which allows us to define by induction:
\begin{equation}
\label{eq:fourth:step:correc:2}  
\hat{X}_{(k+1)T/p}^{(p),n} 
:= 
\hat{X}_{k T/p}^{(p),n}
-\tfrac{T}p 
\eta_{kT/p}  
\hat{X}_{k T/p}^{(p),n}
-\tfrac{T}p
\hat{h}_{kT/p}^{(p),n} + 
\sigma \bigl( B_{(k+1)T/p} 
 - B_{kT/p} \bigr),
\end{equation} 
for 
$k \in \{0,\cdots,p-1\}$ (with $X_0$ as initial condition). 
We prove in the fourth step below 
that \begin{equation}
\label{eq:fourth:step:correc:0} 
\begin{split}
&{\mathbb E}^{\varpi {\boldsymbol h}^{n} \hspace{-2pt} / \varepsilon}
\Bigl[ 
\sup_{k=0,\cdots,p} 
\bigl\vert 
\hat{X}_{kT/p}^{n} 
-
\hat{X}_{k T/p}^{(p),n} 
\bigr\vert^2 \Bigr]
\leq 
 \tfrac{C}p,
\end{split} 
\end{equation} 
and
\begin{equation}
\label{eq:fourth:step:correc:1} 
\begin{split}
&\sup_{0 \leq t \leq T} {\mathbb E}^{\varpi {\boldsymbol h}^{n} \hspace{-2pt} / \varepsilon}
\Bigl[ 
\bigl\vert 
\hat{X}_t^{n} 
-
\hat{X}_{\tau_p(t)}^{(p),n} 
\bigr\vert^2 \Bigr]
\leq 
 \tfrac{C}p,
\end{split} 
\end{equation}
for a constant $C$ as in the statement of 
Theorem \ref{thl:main:regret}. 
By 
\eqref{eq:rem:main:thm:2:correc}, we then also have 
\begin{equation} 
\label{eq:5th:step:correc:1} 
\begin{split}
&\sup_{0 \leq t \leq T} {\mathbb E}^{\varpi {\boldsymbol h}^{n} \hspace{-2pt} / \varepsilon}
\Bigl[ 
\bigl\vert 
\eta_t
\hat{X}_t^{n} 
-
\eta_{\tau_p(t)}^{(p)} \hat{X}_{\tau_p(t)}^{(p),n} 
\bigr\vert^2 \Bigr]
\leq 
 \tfrac{C}p,
\end{split} 
\end{equation}
Together with 
\eqref{eq:bound:hat:X:alpha:correc}, we deduce from the latter displays that 
\begin{align}
&\tfrac12 
 {\mathbb E}^{\varpi {\boldsymbol h}^{n} \hspace{-2pt} / \varepsilon}
 \Bigl[ \vert 
 Q^\dagger \hat{X}_t^{n} 
 + f(m^{(p)}_t) 
 \vert^2
 \Bigr]
 \nonumber
\\
&\geq
\tfrac12 
 {\mathbb E}^{\varpi {\boldsymbol h}^{n} \hspace{-2pt} / \varepsilon}
 \Bigl[ \vert 
 Q^\dagger \hat{X}_{\tau_p(t)}^{(p),n} 
 + f\bigl(m^{(p)}_{\tau_p(t)}\bigr) 
 \vert^2
 \Bigr] \nonumber
 \\
&\hspace{5pt}  - 
 {\mathbb E}^{\varpi {\boldsymbol h}^{n} \hspace{-2pt} / \varepsilon}
 \Bigl[ 
 \Bigl( 
 \vert 
 Q^\dagger \bigl(  \hat{X}_t^{n} 
 - 
 \hat{X}_{\tau_p(t)}^{(p),n} 
 \bigr) 
 \vert 
+ 
 \vert 
f (  m_t^{(p)} )
 - 
 f( m_{\tau_p(t)}^{(p)})
 \vert  
 \Bigr) 
 \, 
 \vert 
  Q^\dagger  
 \hat{X}_{\tau_p(t)}^{(p),n} 
 + f(m^{(p)}_{\tau_p(t)}) 
 \vert
 \Bigr] \nonumber
 \\
 &\geq \tfrac12 
 {\mathbb E}^{\varpi {\boldsymbol h}^{n} \hspace{-2pt} / \varepsilon}
 \Bigl[ \vert 
 Q^\dagger \hat{X}_{\tau_p(t)}^{(p),n} 
 + f(m^{(p)}_{\tau_p(t)}) 
 \vert^2
 \Bigr]  - \tfrac{C}{\sqrt{p}}. \label{eq:5th:step:correc:101} 
  \end{align} 
Proceeding similarly for the other terms in the cost functional, we obtain 
\begin{equation*} 
\begin{split}
 {\mathbb E}^{\varpi {\boldsymbol h}^{n} \hspace{-2pt} / \varepsilon} \Bigl[ {\mathcal R}^{X_0} \bigl(\hat{\boldsymbol \alpha}^{n};{\boldsymbol m}^{(p)};0\bigr)\Bigr]
&\geq \tfrac12 {\mathbb E}^{\varpi {\boldsymbol h}^{n} \hspace{-2pt} / \varepsilon}
\biggl[ \vert R^{\dagger} \hat X_T^{(p),n}
+ g(m_T^{(p)}) \vert^2 
\\
&\hspace{-40pt} + \int_0^T \Bigl\{ \vert Q^{\dagger} \hat{X}_{\tau_p(t)}^{(p),n} 
+ f (m_{\tau_p(t)}^{(p)}) \vert^2 + \vert 
\eta_{\tau_p(t)}^{(p)}
   \hat{X}^{(p),n}_{\tau_p(t)}
+ \hat{h}_t^{n}
 \vert^2 
\Bigr\} \ud t
\biggr] - \tfrac{C}{\sqrt{p}},
\end{split}
\end{equation*}
and by
conditioning 
the last term in the integral 
by 
${\mathcal F}_{\tau_p(t)}^{p,X_0,{\boldsymbol B},{\boldsymbol W}}$, 
we deduce from 
\eqref{eq:widehat:h:taupt:p,n+1} 
and
 Jensen's inequality that 
\begin{equation}
\label{eq:5th:step:correc:1:g}  
\begin{split}
{\mathbb E}^{\varpi {\boldsymbol h}^{n} \hspace{-2pt} / \varepsilon} \Bigl[ {\mathcal R}^{X_0}\bigl(\hat{\boldsymbol \alpha}^{n};{\boldsymbol m}^{(p)};0\bigr)\Bigr]
&\geq \tfrac12 {\mathbb E}^{\varpi {\boldsymbol h}^{n} \hspace{-2pt} / \varepsilon}
\biggl[ \vert R^{\dagger} \hat X_T^{(p),n}
+ g(m_T^{(p)}) \vert^2 
\\
&\hspace{-30pt}  + \int_0^T \Bigl\{ \vert Q^{\dagger} \hat{X}_{\tau_p(t)}^{(p),n} 
+ f (m_{\tau_p(t)}^{(p)}) \vert^2 + \vert 
   \hat{\alpha}^{(p),n}_{t}
 \vert^2 
\Bigr\} \ud t
\biggr] - \tfrac{C}{\sqrt{p}},
\end{split}
\end{equation}
%
%
%
%notice that 
%\begin{equation*} 
%\begin{split}
% {\mathbb E}^{\varpi {\boldsymbol h}^{n} \hspace{-2pt} / \varepsilon}
%\biggl[  
% \int_0^T   \vert \hat{\alpha}^{n+1}_t \vert^2  \ud t
%\biggr]
%\geq 
%  {\mathbb E}^{\varpi {\boldsymbol h}^{n} \hspace{-2pt} / \varepsilon}
%\biggl[  
% \int_0^T   \vert \hat{\alpha}^{(p),n+1}_t \vert^2  \ud t
%\biggr],
%\end{split} 
%\end{equation*} 
where
\begin{equation*} 
\begin{split}
\hat{\alpha}^{(p),n}_t
&= - \eta_{\tau_p(t)}^{(p)}
   \hat{X}^{(p),n}_{\tau_p(t)} 
-  \hat{h}_{\tau_p(t)}^{(p),n}, 
\end{split}  
\end{equation*} 
for $t \in [0,T]$.
%, with 
%\begin{equation}
%\label{eq:5th:step:correc:1:b} 
%\tilde{h}_{\tau_p(t)}^{(p),n+1}
%= \tfrac{p}T 
% {\mathbb E}^{\varpi {\boldsymbol h}^{n} \hspace{-2pt} / \varepsilon}
% \biggl[ 
%\int_{\tau_p(t)}^{\tau_p(t) + T/p}   \hat{h}_s^{n+1}   \ud s \, \vert \, 
%{\mathcal F}_{\tau_p(t)}^{(p),X_0,{\boldsymbol B},{\boldsymbol W}} 
%\biggr]. 
%\end{equation} 
Clearly, the above conditioning is licit because 
the process $(\hat{X}_{kT/p}^{(p),n})_{k=0,\cdots,p}$
is
$({\mathcal F}_{kT/p}^{p,X_{0},{\boldsymbol B},{\boldsymbol W}})_{k=0,\cdots,p}$-progressively measurable. 
%Recalling that, for each $k \in \{0,\cdots,p\}$, 
%$\hat{h}_{kT/p}^{n+1}$ is 
%${\mathcal F}_{kT/p}^{(p),X_0,{\boldsymbol B},{\boldsymbol W}}$-measurable, 
By the way, we deduce
from the latter that the process 
$(\hat{\alpha}^{(p),n}_{t})_{0 \leq t \leq T}$
is
${\mathbb F}^{p,X_0,{\boldsymbol B},{\boldsymbol W}}$-progressively measurable, square-integrable 
and 
piecewise constant on any interval $[\ell T/p,(\ell +1)T/p)$, for $\ell \in \{0,\cdots,p-1\}$. 
In turn, the right-hand side in  
\eqref{eq:5th:step:correc:1:g}
can be rewritten
\begin{equation*}
{\mathbb E}^{\varpi {\boldsymbol h}^{n} \hspace{-2pt} / \varepsilon} \Bigl[ {\mathcal R}^{p,X_0} \bigl(\hat{\boldsymbol \alpha}^{(p),n};{\boldsymbol m}^{(p)};0\bigr)\Bigr] - \tfrac{C}{\sqrt{p}}.
\end{equation*}
We deduce that, for any $\varrho >1$, 
\begin{equation*}
\begin{split}
&\inf_{\boldsymbol \alpha}
	{\mathbb E}^{\varpi {\boldsymbol h}^{n} \hspace{-2pt} / \varepsilon} \Bigl[ {\mathcal R}^{X_0}\bigl({\boldsymbol \alpha};{\boldsymbol m}^{(p)};0\bigr)\Bigr]
 \\
 &\hspace{15pt} \geq 
 {\mathbb E}^{\varpi {\boldsymbol h}^{n} \hspace{-2pt} / \varepsilon}\Bigl[ {\mathcal R}^{{p,X_0}}\bigl(\hat{\boldsymbol \alpha}^{(p),n};{\boldsymbol m}^{(p)};0\bigr) {\mathbf 1}_{\{ 
 {\mathcal R}^{{p,X_0}}(\hat{\boldsymbol \alpha}^{(p),n};{\boldsymbol m}^{(p)};0) \leq \varrho 
 \}}\Bigr] - \tfrac{C}{\sqrt{p}}.
\end{split}
\end{equation*}
Using 
the upper bound on $
d_{\textrm{\rm TV}}\bigl({\mathbb P}^{{\boldsymbol h}^{(p)} \hspace{-2pt} /\varepsilon},{\mathbb P}^{\varpi {\boldsymbol h}^n \hspace{-2pt} /\varepsilon}
\bigr)$ (see Remark \ref{rem:thm:2:00}), we obtain
\begin{equation*}
\begin{split}
&\inf_{\boldsymbol \alpha}
	{\mathbb E}^{\varpi {\boldsymbol h}^{n}\hspace{-2pt} /\varepsilon} \Bigl[ {\mathcal R}^{X_0}\bigl({\boldsymbol \alpha};{\boldsymbol m}^{(p)};0\bigr)\Bigr]
\\
&\hspace{5pt} \geq 
 {\mathbb E}^{ {\boldsymbol h}^{(p)}\hspace{-2pt} /\varepsilon}  \Bigl[ {\mathcal R}^{{p,X_0}}\bigl(\hat{\boldsymbol \alpha}^{(p),n};{\boldsymbol m}^{(p)};0\bigr) {\mathbf 1}_{\{ 
 {\mathcal R}^{{p},X_0} (\hat{\boldsymbol \alpha}^{(p),n};{\boldsymbol m}^{(p)};0 ) \leq \varrho 
 \}}\Bigr]  - \varrho \exp(C \varepsilon^{-2})  \varpi^{-n} - \tfrac{C}{\sqrt{p}}.  
\end{split}
\end{equation*}
And then, by  
\eqref{eq:bound:hat:X:alpha:2} and \textcolor{black}{the sub-Gaussian property of the law of $X_0$}, 
we obtain
\begin{equation}
\label{eq:bound:hat:X:alpha:3}
\begin{split}
\inf_{\boldsymbol \alpha}
	{\mathbb E}^{\varpi {\boldsymbol h}^{n} \hspace{-2pt} / \varepsilon} \Bigl[ {\mathcal R}^{X_0}\bigl({\boldsymbol \alpha};{\boldsymbol m}^{(p)};0\bigr)\Bigr]
&\geq 
 {\mathbb E}^{{\boldsymbol h}^{(p)}\hspace{-2pt} /\varepsilon}\Bigl[ {\mathcal R}^{{p,X_0}}\bigl(\hat{\boldsymbol \alpha}^{(p),n};{\boldsymbol m}^{(p)};0\bigr) \Bigr] 
 \\
&\hspace{15pt} - C \exp \bigl( - \eta \varrho \bigr) - \varrho \exp(C \varepsilon^{-2}) \varpi^{-n}  -   \tfrac{C}{\sqrt{p}},
\end{split}
\end{equation}
for some $\eta>0$. 
%Moreover, by Proposition \ref{prop:2:b} (and with an argument similar to \eqref{eq:5th:step:correc:101}), we get
%\begin{equation*}
%{\mathbb E}^{{\boldsymbol h}^{(p)} \hspace{-2pt} / \varepsilon}
%\Bigl[ \bigl\vert
%{\mathcal R}^{{p,X_0}}\bigl(
%\hat{\boldsymbol \alpha}^{(p),n}
%;{\boldsymbol m}^{(p)};0\bigr)
%-
%{\mathcal R}^{{p,X_0}}\bigl(
%\hat{\boldsymbol \alpha}^{(p),n};{\boldsymbol m}^{(p)};0\bigr)
%\bigr\vert
%\Bigr] \leq \exp(C \varepsilon^{-2}) \varpi^{-n}.
%\end{equation*}
%Plugging the latter into 
%\eqref{eq:bound:hat:X:alpha:3}, we obtain
%\begin{equation*}
%\begin{split}
%\inf_{\boldsymbol \alpha}
%	{\mathbb E}^{\varpi {\boldsymbol h}^{n} \hspace{-2pt} / \varepsilon} \Bigl[ {\mathcal R}^{X_0} \bigl({\boldsymbol \alpha};{\boldsymbol m}^{(p)};0\bigr)\Bigr]
%&\geq 
% {\mathbb E}^{{\boldsymbol h}^{(p)} \hspace{-2pt} / \varepsilon}\Bigl[ {\mathcal R}^{{p},X_0}\bigl( 
% \hat{\boldsymbol \alpha}^{(p),n}
% ;{\boldsymbol m}^{(p)} ;0\bigr) \Bigr] 
%\\
%&\hspace{15pt} - C \exp \bigl( - \eta \varrho \bigr) - \rho \exp(C \varepsilon^{-2})  \varpi^{-n}  -   \tfrac{C}{\sqrt{p}},
%\end{split}
%\end{equation*}
Finally, by 
Proposition 
\ref{prop:epsilon:Nash},
\begin{align}
\label{eq:bound:hat:X:alpha:5}
&\inf_{\boldsymbol \alpha}
	{\mathbb E}^{\varpi {\boldsymbol h}^{n} \hspace{-2pt} / \varepsilon} \Bigl[ {\mathcal R}^{X_0}\bigl({\boldsymbol \alpha};{\boldsymbol m}^{(p)};0\bigr)\Bigr]
\\
&\hspace{10pt} \geq 
 {\mathbb E}^{{\boldsymbol h}^{(p)} \hspace{-2pt} / \varepsilon}\Bigl[ {\mathcal R}^{p,X_0}\bigl({\boldsymbol \alpha}^{(p),\star};{\boldsymbol m}^{(p)} ;0\bigr) \Bigr] 
- C \varepsilon - C \exp \bigl( - \eta \varrho \bigr) - \rho \exp(C \varepsilon^{-2}) \varpi^{-n}  -   \tfrac{C}{\sqrt{p}}.
\nonumber
\end{align}
\vskip 4pt

\textit{Third Step.}
We now revert the computations achieved in the second step. 
%To do so, we observe that 
%the bounds 
%\eqref{eq:bound:hat:X:alpha}
%and
%\eqref{eq:bound:hat:X:alpha:2}
%remain true
%\textcolor{black}{up to an additional $\sup_{0 \leq t \leq T} \vert W_t \vert$},
% if we replace
%$\hat{\boldsymbol X}^{n+1}$
%by 
%${\boldsymbol X}^{n+1}$
%and similarly
%$\hat{\boldsymbol \alpha}^{n+1}$
%by 
%${\boldsymbol \alpha}^{n+1}$, 
%with 
%${\boldsymbol X}^{n+1}$
%and
%${\boldsymbol \alpha}^{n+1}$
%being defined as in Lemma \ref{lem:4}. 
%By Proposition 
%\ref{prop:2:b}, 
%\eqref{eq:bound:hat:X:alpha:5}
%yields 
%\begin{equation*}
%\begin{split}
%&\inf_{\boldsymbol \alpha}
%	{\mathbb E}^{\varpi {\boldsymbol h}^{n} \hspace{-2pt} / \varepsilon}  \Bigl[ {\mathcal R}^{X_0} \bigl({\boldsymbol \alpha};{\boldsymbol m}^{(p)};0\bigr)\Bigr]
%	\\
%&\hspace{15pt} \geq  {\mathbb E}^{{\boldsymbol h}^{(p)} \hspace{-2pt} / \varepsilon}\Bigl[ {\mathcal R}^{p,X_0} \bigl({\boldsymbol \alpha}^{(p),\star};{\boldsymbol m}^{(p)} ;0\bigr) \Bigr] 
%- C \varepsilon - C \exp \bigl( - \eta \varrho \bigr) - \rho \exp(C \varepsilon^{-2}) \varpi^{-n}  -   \tfrac{C}{\sqrt{p}}
%\\
%&\hspace{15pt} \geq 
% {\mathbb E}^{ {\boldsymbol h}^{(p)}  \hspace{-2pt} /\varepsilon}  \Bigl[ {\mathcal R}^{p,X_0} \bigl({\boldsymbol \alpha}^{(p),\star};{\boldsymbol m}^{(p)};0\bigr) \Bigr] 
%- C \varepsilon - C \exp \bigl( - \eta \varrho \bigr)  - \rho \exp(C \varepsilon^{-2})  \varpi^{-n}  - 
% \tfrac{C}{\sqrt{p}}.
% \end{split}
% \end{equation*} 
%Here, 
Following 
\eqref{eq:5th:step:correc:1:g}
and then 
\eqref{eq:5th:step:correc:101} and the proof of Lemma 
\ref{lem:3:0}, the cost on the last line
can be lower bounded as follows: 
\begin{equation*} 
\begin{split} 
& {\mathbb E}^{ {\boldsymbol h}^{(p)} \hspace{-2pt} /\varepsilon}  \Bigl[ {\mathcal R}^{p,X_0} \bigl({\boldsymbol \alpha}^{(p),\star};{\boldsymbol m}^{(p)};0\bigr) \Bigr] 
\\
&= \tfrac12 {\mathbb E}^{{\boldsymbol h}^{(p)} \hspace{-2pt} /\varepsilon}
\biggl[ \vert R^{\dagger}   X_T^{(p),\star,0}
+ g(m_T^{(p)}) \vert^2 
   + \int_0^T \Bigl\{ \vert Q^{\dagger} {X}_{\tau_p(t)}^{(p),\star,0} 
+ f (m_{\tau_p(t)}^{(p)}) \vert^2 + \vert 
    {\alpha}^{(p),\star}_{t}
 \vert^2 
\Bigr\} \ud t
\biggr]
\\
&\geq \tfrac12 {\mathbb E}^{ {\boldsymbol h}^{(p)} \hspace{-2pt} /\varepsilon}
\biggl[ \vert R^{\dagger}  X_T^{(p),\star}
+ g(m_T^{(p)}) \vert^2 
   + \int_0^T \Bigl\{ \vert Q^{\dagger} {X}_{t}^{(p),\star} 
+ f (m_{t}^{(p)}) \vert^2 + \vert 
    {\alpha}^{(p),\star}_{t}
 \vert^2 
\Bigr\} \ud t
\biggr]
\\
&\hspace{15pt} - \tfrac{C}{\sqrt{p}} - C \varepsilon, 
\end{split} 
\end{equation*} 
where ${\boldsymbol X}^{(p),\star,0}$ on the second line denotes the time-continuous continuous process defined by 
\begin{equation*} 
\begin{split} 
 {X}^{(p),\star,0}_t &= 
 {X}^{(p),\star,0}_{kT/p} +  \bigl( 
t- \tfrac{kT}p \bigr) 
 {\alpha}_{kT/p}^{(p),\star}  + 
\sigma \bigl( B_{t} 
 - B_{kT/p} \bigr),
 \quad t \in [\tfrac{kT}p,\tfrac{(k+1)T}p], \quad k \in \{0,\cdots,p-1\}. 
\end{split} 
\end{equation*} 
In particular, 
$ {\boldsymbol X}^{(p),\star,0}$ is the time-continuous process driven
by $ {\boldsymbol \alpha}^{(p),\star}$ when there is no common noise. 
Notice that 
${\boldsymbol X}^{(p),\star,0}$ differs from 
the process 
${\boldsymbol X}^{(p),\star}$
that appears on the third line and that is defined
 in 
\eqref{eq:mfg:p:3:bbis}. 

In turn, 
by Proposition 
\ref{prop:2:b}
and for $\rho$ as in
\eqref{eq:bound:hat:X:alpha:5}, 
\begin{equation*}
\begin{split}
&\inf_{\boldsymbol \alpha}
	{\mathbb E}^{\varpi {\boldsymbol h}^{n} \hspace{-2pt} / \varepsilon}  \Bigl[ {\mathcal R}^{X_0}\bigl({\boldsymbol \alpha};{\boldsymbol m}^{(p)};0\bigr)\Bigr]
\\
&\geq 
 {\mathbb E}^{{\boldsymbol h}^{(p)} \hspace{-2pt} /\varepsilon}  \Bigl[ {\mathcal R}^{X_0} \bigl({\boldsymbol \alpha}^{(p),\star};{\boldsymbol m}^{(p)};0\bigr)   \Bigr] 
- C \varepsilon - C \exp \bigl( - \eta \varrho \bigr)  - \rho \exp(C \varepsilon^{-2})   \varpi^{-n}  - \tfrac{C}{\sqrt{p}}
\\
&\geq 
 {\mathbb E}^{ {\boldsymbol h}^{(p)} \hspace{-2pt} /\varepsilon}  \Bigl[ {\mathcal R}^{X_0}\bigl({\boldsymbol \alpha}^{(p),n,\diamond};{\boldsymbol m}^{(p)};0\bigr)   \Bigr] 
- C \varepsilon - C \exp \bigl( - \eta \varrho \bigr)  - \rho \exp(C \varepsilon^{-2})  \varpi^{-n}  - \tfrac{C}{\sqrt{p}}
\\
&\geq 
 {\mathbb E}^{{\boldsymbol h}^{(p)} \hspace{-2pt} /\varepsilon}  \Bigl[ {\mathcal R}^{X_0} \bigl({\boldsymbol \alpha}^{(p),n,\diamond};{\boldsymbol m}^{(p)};0\bigr) {\mathbf 1}_{\{ 
 {\mathcal R}^{X_0} ({\boldsymbol \alpha}^{(p),n,\diamond};{\boldsymbol m}^{(p)};0 ) \leq \varrho 
 \}} \Bigr] 
- C \varepsilon - C \exp \bigl( - \eta \varrho \bigr) 
\\
&\hspace{15pt} - \rho \exp(C \varepsilon^{-2})  \varpi^{-n}  -  \tfrac{C}{\sqrt{p}}.
 \end{split}
\end{equation*}
And then, as in the derivation of 
\eqref{eq:bound:hat:X:alpha:3},
\begin{equation*}
\begin{split}
&\inf_{\boldsymbol \alpha}
	{\mathbb E}^{\varpi {\boldsymbol h}^{n} \hspace{-2pt} / \varepsilon} \Bigl[ {\mathcal R}^{X_0}\bigl({\boldsymbol \alpha};{\boldsymbol m}^{(p)};0\bigr)\Bigr]
\\
& \geq  
{\mathbb E}^{\varpi {\boldsymbol h}^{n} \hspace{-2pt} / \varepsilon}  \Bigl[ {\mathcal R}^{X_0}\bigl({\boldsymbol \alpha}^{(p),n,\diamond};{\boldsymbol m}^{(p)};0\bigr) {\mathbf 1}_{\{ 
 {\mathcal R}^{X_0} ({\boldsymbol \alpha}^{(p),n,\diamond};{\boldsymbol m}^{(p)};0 ) \leq \varrho 
 \}} \Bigr] 
- C \varepsilon - C \exp \bigl( - \eta \varrho \bigr) 
\\
&\hspace{15pt} - \rho \exp(C \varepsilon^{-2})   \varpi^{-n}  -  \tfrac{C}{\sqrt{p}}
\\
&  \geq  
 {\mathbb E}^{\varpi {\boldsymbol h}^{n} \hspace{-2pt} / \varepsilon}\Bigl[ {\mathcal R}^{X_0} \bigl({\boldsymbol \alpha}^{(p),n,\diamond};{\boldsymbol m}^{(p)};0\bigr)   \Bigr] 
- C \varepsilon - C \exp \bigl( - \eta \varrho \bigr)  - \rho \exp(C \varepsilon^{-2})  \varpi^{-n}  -  \tfrac{C}{\sqrt{p}}.
 \end{split}
\end{equation*}

Choosing 
$\eta \varrho = 
- \ln(\varepsilon)$
and modifying the value of $C$, 
we complete the proof. 
\vskip 4pt

\textit{Fourth Step.}
It now remains to prove some auxiliary results that are used in the previous steps. 

We start with the analysis of 
$\hat{\boldsymbol h}^{n}$
in 
\eqref{eq:5th:step:correc:1:d}.
We recall from 
\eqref{eq:hn+1:tildeh:n+1} that, 
that, for each $t \in [0,T]$, 
${h}_t^{n}$ is 
$(\sigma( (W^p_{k})_{k \leq \lfloor t p/T \rfloor T/p} ))_{0 \leq t \leq T}$-measurable. 
Hence, from 
\eqref{eq:X:n+1:nearlyNash:000}
and
\eqref{eq:X:n+1:nearlyNash}, 
each 
${m}_t^{(p)}$ is 
$(\sigma( (W^p_{k})_{k \leq \lceil t p/T \rceil T/p} ))_{0 \leq t \leq T}$-measurable (notice the use of the ceil part instead of the floor part in the time index of the filtration).
Further, we notice from 
\eqref{eq:5th:step:correc:1:d} 
and 
Remark
\ref{rem:back:resolvent:correc}
that 
\begin{equation*}
\begin{split} 
&P_{kT/p} 
\hat{h}_{k T/p}^{n} = 
{\mathbb E}^{\varpi {\boldsymbol h}^n \hspace{-2pt} /\varepsilon}
\biggl[ 
P_T
R^\dagger g \bigl( m^{(p)}_T \bigr)
+
\int_{kT/p}^T 
 P_s Q^\dagger f\bigl( m^{(p)}_{t}\bigr)  \ud t 
 \, \vert \, 
 {\mathcal F}^{\boldsymbol W}_{kT/p}
\biggr], \quad k \in \{0,\cdots,p\}. 
\end{split} 
\end{equation*}
 Following the analysis of 
\eqref{eq:recurrence:cas:discret}, the left-hand side is 
${\mathcal F}_{kT/p}^{{\boldsymbol W}^p}$-measurable. Noting 
that each $P_t$, for $t \in [0,T]$, is invertible, we deduce that 
$\hat{h}_{k T/p}^{n}$
is 
${\mathcal F}_{kT/p}^{{\boldsymbol W}^p}$-measurable, which property is used below.

And then, 
\eqref{eq:5th:step:correc:1:f} follows from the above representation of 
$(P_{kT/p} 
\hat{h}_{k T/p}^{n})_{k=0,\cdots,p}$ (with a similar representation 
of
$P_{t} 
\hat{h}_{t}^{n}$
for any $t \in [0,T]$). The two bounds in 
\eqref{eq:bound:hat:X:alpha}
easily follow.

We
now turn to the proof of 
\eqref{eq:fourth:step:correc:0} and
\eqref{eq:fourth:step:correc:1}. By 
\eqref{eq:fourth:step:correc:2} ,  
we know that 
\begin{equation*} 
\hat{X}_{(k+1)T/p}^{(p),n} 
= 
\hat{X}_{k T/p}^{(p),n}
-\tfrac{T}p 
\eta_{kT/p}  
\hat{X}_{k T/p}^{(p),n}
-\tfrac{T}p
\hat{h}_{kT/p}^{(p),n} + 
\sigma \bigl( B_{(k+1)T/p} 
 - B_{kT/p} \bigr), \quad k \in \{0,\cdots,p-1\}.  
\end{equation*} 
Meanwhile,
recalling \eqref{eq:5th:step:correc:1:c}, we write
\begin{equation*} 
\begin{split} 
\hat{X}_{(k+1)T/p}^{n} 
&=
\hat{X}_{kT/p}^{n}  
- 
 \int_{kT/p}^{(k+1)T/p} 
 \bigl[
\eta_s \hat{X}_s^{n} 
+ 
\hat{h}_s^{n} 
\bigr] \ud s  + 
\sigma \bigl( B_{(k+1)T/p} - B_{kT/p} \bigr)
\\
&= \hat{X}_{kT/p}^{n}  
- 
\biggl( 
\tfrac{T}p 
\eta_{kT/p} 
\hat{X}_{kT/p}^{n} 
+ \int_{kT/p}^{(k+1)T/p} 
\hat{h}_s^{n}  
 \ud s \biggr) 
-  \tfrac{T}p 
\hat{A}_{k}^{n}   + 
\sigma \bigl( B_{(k+1)T/p} - B_{kT/p} \bigr),
\end{split} 
\end{equation*} 
with 
\begin{equation} 
\label{eq:correc:A:k:n+1}
\begin{split} 
\hat{A}_{k}^{n}
=  \tfrac{p}T
\int_{kT/p}^{(k+1)T/p}
\bigl[ \eta_s \hat{X}_s^{n} - \eta_{kT/p} \hat{X}^{n}_{kT/p} 
\bigr] \ud s. 
\end{split} 
\end{equation} 
And then,
\begin{equation*} 
\begin{split} 
\hat{X}_{(k+1)T/p}^{(p),n} 
-
\hat{X}_{(k+1)T/p}^{n} 
&= 
\hat{X}_{k T/p}^{(p),n}
-
\hat{X}_{k T/p}^{n} 
-\tfrac{T}p 
\eta_{kT/p}  
\bigl( 
\hat{X}_{k T/p}^{(p),n}
-
\hat{X}_{k T/p}^{n} 
\bigr) 
\\
&\hspace{15pt} 
- \int_{kT/p}^{(k+1)T/p} 
\bigl[
\hat{h}_{\tau_p(s)}^{(p),n} - \hat{h}_s^{n} 
\bigr] \ud s +  \tfrac{T}p 
\hat{A}_{k}^{n}.
\end{split} 
\end{equation*} 
By discrete Gronwall's lemma, there exists a constant $C$ (depending on the same parameters as in the statement of 
Theorem \ref{thl:main:regret}) such that 
\begin{equation*} 
\sup_{\ell=0,\cdots,p} 
\bigl\vert
\hat{X}_{\ell T/p}^{(p),n} 
-
\hat{X}_{\ell T/p}^{n+1} 
\bigr\vert
\leq C \sum_{\ell=0}^{p-1} 
 \biggl( 
 \biggl\vert \int_{\ell T/p}^{(\ell+1)T/p} 
\bigl[
\hat{h}_{\tau_p(s)}^{(p),n} - \hat{h}_s^{n} 
\bigr]  \ud s \biggr\vert + \tfrac{T}p 
\bigl\vert \hat{A}_{\ell}^{n}
\bigr\vert 
\biggr).
\end{equation*} 
And then, 
\begin{equation}
\label{eq:5th:step:correc} 
\begin{split}
&{\mathbb E}^{\varpi {\boldsymbol h}^{n} \hspace{-2pt} / \varepsilon}
\Bigl[ 
\sup_{k=0,\cdots,p} 
\bigl\vert 
\hat{X}_{kT/p}^{(p),n} 
-
\hat{X}_{k T/p}^{n} 
\bigr\vert^2 \Bigr]
\\ 
&\leq 
 C
 \tfrac{T}p 
\sum_{\ell=0}^{p-1} 
 {\mathbb E}^{\varpi {\boldsymbol h}^{n} \hspace{-2pt} / \varepsilon}
\biggl[  \tfrac{p^2}{T^2} \biggl\vert \int_{\ell T/p}^{(\ell+1)T/p} 
\bigl[
\hat{h}_{\tau_p(s)}^{(p),n} - \hat{h}_s^{n} 
\bigr]  \ud s \biggr\vert^2 + 
\bigl\vert \hat{A}_{\ell}^{n}
\bigr\vert^2 
\biggr].
\end{split} 
\end{equation} 
Here, we observe 
from 
\eqref{eq:widehat:h:taupt:p,n+1} that 
\begin{equation}
\label{eq:5th:step:correc:999} 
\begin{split} 
& \tfrac{T}p 
\sum_{\ell=0}^{p-1} 
 \tfrac{p^2}{T^2} 
   {\mathbb E}^{\varpi {\boldsymbol h}^{n} \hspace{-2pt} / \varepsilon}
 \biggl[ \biggl\vert \int_{\ell T/p}^{(\ell+1)T/p} 
\bigl[
\hat{h}_{\tau_p(s)}^{(p),n} - \hat{h}_s^{n} 
\bigr]  \ud s \biggr\vert^2\biggr]
\\
&= \tfrac{p}{T} \sum_{\ell=0}^{p-1} 
{\mathbb E}^{\varpi {\boldsymbol h}^{n} \hspace{-2pt} / \varepsilon}
\biggl[ 
 \biggl\vert 
 {\mathbb E}^{\varpi {\boldsymbol h}^{n} \hspace{-2pt} / \varepsilon}
 \biggl[ 
\int_{\ell T/p}^{(\ell +1) T/p}   \hat{h}_s^{n}   \ud s \, \vert \, 
{\mathcal F}_{\ell T/p}^{{\boldsymbol W}^p} 
\biggr]
-
\int_{\ell T/p}^{(\ell+1)T/p} 
\hat{h}_s^{n} 
\ud s
\biggr\vert^2
\biggr]
\\
&\leq  \sum_{\ell=0}^{p-1} 
\int_{\ell T/p}^{(\ell +1) T/p} 
{\mathbb E}^{\varpi {\boldsymbol h}^{n} \hspace{-2pt} / \varepsilon}
\Bigl[ 
 \bigl\vert
\hat{h}_s^{n} 
-   
 {\mathbb E}^{\varpi {\boldsymbol h}^{n} \hspace{-2pt} / \varepsilon}
 \bigl[    \hat{h}_s^{n} \, \vert \, 
{\mathcal F}_{\ell T/p}^{{\boldsymbol W}^p} 
\bigr]
\bigr\vert^2
\Bigr]
\ud s.
\end{split} 
\end{equation} 
By
\eqref{eq:5th:step:correc:1:d}, we have, for 
$\ell \in \{0,\cdots,p-1\}$
and for $s \in [\ell T/p,(\ell+1)T/p]$, 
\begin{equation*}
\begin{split} 
\hat{h}_{s}^{n} = 
\hat{h}_{\ell T/p}^{n} + 
\int_{\ell T/p}^{s} \bigl[ -   Q^\dagger f\bigl( m^{(p)}_{r}\bigr) + \eta_{r} \hat{h}_{r}^{n} \bigr]  \ud r + \varepsilon
\int_{\ell T/p}^{s}  \hat{ k}_{r}^{n} \ud W_{r}^{\varpi {\boldsymbol h}^n \hspace{-2pt} /\varepsilon}. 
\end{split} 
\end{equation*}
Taking the conditional expectation given ${\mathcal F}_{\ell T/p}^{{\boldsymbol W}^p}$ under ${\mathbb P}^{\varpi {\boldsymbol h}^{n} \hspace{-2pt} / \varepsilon}$, we obtain 
\begin{equation*}
\begin{split} 
 {\mathbb E}^{\varpi {\boldsymbol h}^{n} \hspace{-2pt} / \varepsilon}
 \bigl[  \hat{h}_{s}^{n} \, \vert \, {\mathcal F}_{\ell T/p}^{{\boldsymbol W}^p}
 \bigr]= 
\hat{h}_{\ell T/p}^{n} + 
 {\mathbb E}^{\varpi {\boldsymbol h}^{n} \hspace{-2pt} / \varepsilon}
\biggl[ 
\int_{\ell T/p}^{s} \bigl[ -   Q^\dagger f\bigl( m^{(p)}_{r}\bigr) + \eta_{r} \hat{h}_{r}^{n} \bigr]  \ud r 
\, \vert \, {\mathcal F}_{\ell T/p}^{{\boldsymbol W}^p}
 \biggr],
\end{split} 
\end{equation*}
and then 
\begin{equation*} 
\begin{split} 
&{\mathbb E}^{\varpi {\boldsymbol h}^{n} \hspace{-2pt} / \varepsilon}
\Bigl[ 
 \bigl\vert
\hat{h}_s^{n} 
-   
 {\mathbb E}^{\varpi {\boldsymbol h}^{n} \hspace{-2pt} / \varepsilon}
 \bigl[    \hat{h}_s^{n} \, \vert \, 
{\mathcal F}_{\ell T/p}^{{\boldsymbol W}^p} 
\bigr]
\bigr\vert^2
\Bigr] \leq \tfrac{C}{p^2} + C \varepsilon^2 
 {\mathbb E}^{\varpi {\boldsymbol h}^{n} \hspace{-2pt} / \varepsilon}
 \biggl( \int_{\ell T/p}^s 
  \vert   \hat{k}_r^{n} \vert^2  \ud r\biggr). 
\end{split}
\end{equation*} 
Finally, 
by 
\eqref{eq:5th:step:correc:999}, 
\begin{equation*} 
\begin{split} 
& \tfrac{T}p 
\sum_{\ell=0}^{p-1} 
 \tfrac{p^2}{T^2} 
   {\mathbb E}^{\varpi {\boldsymbol h}^{n} \hspace{-2pt} / \varepsilon}
 \biggl[ \biggl\vert \int_{\ell T/p}^{(\ell+1)T/p} 
\bigl[
\hat{h}_{\tau_p(s)}^{(p),n} - \hat{h}_s^{n} 
\bigr]  \ud s \biggr\vert^2\biggr]
\\
&\leq 
\tfrac{C}{p^2} 
+
\sum_{\ell=0}^{p-1} 
\int_{\ell T/p}^{(\ell +1) T/p} 
{\mathbb E}^{\varpi {\boldsymbol h}^{n} \hspace{-2pt} / \varepsilon}
\biggl( \int_{\ell T/p}^s
\varepsilon^2
 \vert   \hat{k}_r^{n} \vert^2  \ud r
\biggr) \, \ud s
\\
&\leq \tfrac{C}{p^2} 
+ \tfrac{C}{p}
{\mathbb E}^{\varpi {\boldsymbol h}^{n} \hspace{-2pt} / \varepsilon}
\biggl(
\int_{0}^{T} 
 \varepsilon^2 \vert   \hat{k}_r^{n} \vert^2  \ud r
\biggr). 
\end{split} 
\end{equation*} 
Taking the square in 
\eqref{eq:5th:step:correc:1:d}
and using 
\eqref{eq:5th:step:correc:1:f}, we can prove that 
\begin{equation*} 
{\mathbb E}^{\varpi {\boldsymbol h}^{n} \hspace{-2pt} / \varepsilon}
\biggl(
\int_{0}^{T} 
 \varepsilon^2 \vert   \hat{k}_r^{n} \vert^2  \ud r
\biggr) 
\leq C.
\end{equation*} 
Back to 
\eqref{eq:correc:A:k:n+1} 
and \eqref{eq:5th:step:correc}, we get 
\begin{equation*} 
\begin{split}
&{\mathbb E}^{\varpi {\boldsymbol h}^{n} \hspace{-2pt} / \varepsilon}
\Bigl[ 
\sup_{k=0,\cdots,p} 
\bigl\vert 
\hat{X}_{kT/p}^{(p),n} 
-
\hat{X}_{k T/p}^{n} 
\bigr\vert^2 \Bigr]
\leq 
 \tfrac{C}p,
\end{split} 
\end{equation*} 
which is 
\eqref{eq:fourth:step:correc:0}.

By
\eqref{eq:5th:step:correc:1:c}, we easily obtain 
\eqref{eq:fourth:step:correc:1}.
\end{proof}

\section{Numerical experiments}
\label{Se:3}

This section is devoted to several numerical experiments 
based on our learning algorithm. We first provide
in Subsection \ref{Subse:3:1}
 a theoretical analysis of a benchmark example that we use throughout the section. 
In Subsection \ref{Subse:3:2}, we present several variants of the implemented version of the algorithm. We also explain how to compute a reference solution. 
Then, we 
study in Subsection \ref{Subse:3:3:a} the numerical behavior of the algorithm 
when the 
intensity $\varepsilon$ of the common noise is equal to 1.
In Subsection \ref{subse:3:2:bis}, we explain some of the difficulties that arise in the large dimension setting.
 Finally, in 
Subsection 
\ref{subse:3:3}, we address the small viscosity regime. 
In both Subsections 
\ref{Subse:3:3:a} 
and 
\ref{subse:3:3}, we discuss the influence of the parameter $\varpi$. In particular, we compare the two harmonic ($\varpi=1$) and 
geometric ($\varpi>1$) versions of the algorithm. In this regard, it is worth recalling Remark 
\ref{rem:fictitious:play:rate}: all our results hold for the harmonic version of the 
fictitious play with the geometric decay $\varpi^{-n}$ being replaced by $1/n$. 

\subsection{A benchmark example}
\label{Subse:3:1}
As a particular instance of the cost functional $J$ in 
\eqref{eq:cost:intro}, we focus here on 
\begin{equation}
\label{eq:cost:numerical}
J\bigl({\boldsymbol \alpha} ; {\boldsymbol m} \bigr) = \frac12  {\mathbb E}
\biggl[ \bigl\vert X_{T} + g(m_{T}) \bigr\vert^2 + 
\int_{0}^T   
\vert \alpha_{t} \vert^2
  \ud t \biggr],
\end{equation}
which implicitly means that $f=0$. 
The examples that are addressed below are in dimension $d=1$, $d=2$, \textcolor{black}{$d=12$ and $d=20$}, with $T=1$ in any cases.
For simplicity, we also work with $X_{0}=0$.  

In dimension $d=1$, we choose
\begin{equation}
\label{eq:benchmark:d=1}
g(x) = \cos ( \kappa x), \quad x \in {\mathbb R}, 
\end{equation}
for a free positive parameter $\kappa$ that we tune from smaller to larger values. 
Obviously, the Lipschitz constant of $g$ becomes larger with $\kappa$. Accordingly, the coupling between the two forward and backward equations in 
\eqref{eq:intro:FBSDE} becomes stronger as $\kappa$ gets larger, which makes it more difficult to solve, especially when $\varepsilon=0$. We illustrate this in Lemma \ref{lem:uniqueness:MFG:f=0:g:Lipschitz:small} below: it says that, when $\varepsilon=0$,  
\eqref{eq:intro:FBSDE} is uniquely solvable when $\vert \kappa \vert<2$. Even more, the analysis performed in 
\cite{ChassagneuxCrisanDelarue} shows that the more standard (and obviously simpler) Picard scheme
\eqref{eq:backward:step:n}--\eqref{eq:forward:step:n}
 would converge when $\kappa$ is small, whether there is a common noise or not. Our result is thus especially relevant when $\kappa$ gets larger. 

In dimension $2$, we work with a similar terminal cost $g=(g_{1},g_{2})$:
\begin{equation}
\label{eq:benchmark:d=2}
\begin{split}
&g_{1}(x_{1},x_{2}) = \cos(\kappa x_{1}) \cos(\kappa x_{2}),
\quad g_{2}(x_{1},x_{2}) = \sin(\kappa x_{1}) \sin (\kappa x_{2}), 
\quad x_{1}, \ x_{2} \in {\mathbb R},
\end{split}
\end{equation}
where, as in dimension 1, $\kappa$ is a free positive parameter that we tune from smaller to larger values. 

\color{black}
As we just said, we also address higher dimensional examples, with $d=12$ or $d=20$. In order to encode the function $g$ in a systematic manner, we 
then choose:
\begin{equation}
\label{eq:benchmark:d:higher}
\begin{split}
&g_i(x) = \cos \biggl( \sum_{j=1}^d \Theta_{i,j} x_j \biggr)
\quad x \in {\mathbb R}^d, \quad i = 1,\cdots, d, 
\end{split}
\end{equation}
for a fixed square matrix $\Theta$ of size $d \times d$, whose choice depends on the examples. 
For instance, we may take $\Theta_{i,j}=(i+j)/(2d)$. The impact of $\Theta$ is then quite similar to the impact of 
$\kappa$ in the low dimensional case as it induces highly oscillatory phenomena, which make the coupling in 
\eqref{eq:intro:FBSDE} stronger. 

\color{black}
Regardless of the choice of $g$, 
the Riccati equation 
\eqref{eq:riccati}
associated with 
\eqref{eq:cost:numerical} writes 
\begin{equation}
\label{eq:riccati:test}
\dot{\eta}_{t} - \eta_{t}^2 = 0, \quad t \in [0,1] \ ; \quad \eta_{1} = 1, 
\end{equation}
the solution of which is given by 
\begin{equation}
\label{eq:riccati:test:solution}
\eta_{t} = \frac{1}{2-t}, \quad t \in [0,1]. 
\end{equation}
\textcolor{black}{Obviously, the 
solution to \eqref{eq:riccati} is here identified with a scalar-valued function while, formally, it takes values in 
the set of square $d \times d$ matrices.
This follows from the fact that $Q$ and $R$ in 
\eqref{eq:riccati} are just the identity matrix. 
Numerically, this makes both the code and the analysis slightly easier, but the resulting restriction on the scope of the results is in fact limited, 
the real challenge being to learn ${\boldsymbol h}^n$ in 
 \eqref{eq:optimization:exploration}.
}

\subsubsection{Solutions without common noise}
Even though this is not needed for all the examples we handle below, we feel useful 
to have a preliminary discussion about the shape of the solutions when there is no common noise, keeping in mind that, originally, the true model of interest is the MFG without common noise. Recalling \eqref{eq:intro:FBSDE}, 
we 
indeed have that, 
whenever there is no common noise (i.e., $\varepsilon=0$), the 
equilibria are given as the solutions of the deterministic system 
\begin{equation}
\label{eq:coupled:limit}
\dot{m}_{t} = - \bigl( \eta_{t}  m_{t} + h_{t} \bigr), \quad \dot{h}_{t}  =   \eta_{t} h_{t}, \quad t \in [0,1] \ ;  
\quad h_{1}=g(m_{1}),
\end{equation}
with $m_{0}=0$, 
which prompts us to perform the changes of variable: 
\begin{equation}
\label{eq:change:variable:m,h}
\tilde{m}_{t} = \frac{m_{t}}{2-t}, \quad \tilde{h}_{t} = (2-t) h_{t}, \quad 
t \in [0,1]. 
\end{equation}
We then have that $(m_{t},h_{t})_{0 \le t \le 1}$
solves 
\eqref{eq:coupled:limit}
if and only if 
\begin{equation*}
\begin{split}
&\dot{\tilde{m}}_{t} = - \frac1{2-t} h_{t} =
- \frac1{(2-t)^2} \tilde h_{t}, \quad \dot{\tilde{h}}_{t} = 0, \quad t \in [0,1] \ ;
\quad \tilde{h}_{1} = g\bigl( \tilde m_{1} \bigr),
\end{split}
\end{equation*}
from which we deduce the following simpler characterization  (recalling that $m_{0}=0$) 
\begin{equation}
\label{eq:mean_field:d:1:equilibrium}
\tilde m_{1} =  - g\bigl( \tilde m_{1} \bigr) \int_{0}^1 \frac{\ud t}{(2-t)^2}
= - \frac{g(\tilde m_{1})}2.
\end{equation}
There are as many equilibria as solutions to the equation
$x = -g(x)/2$. 
We thus have the following obvious lemma:
\begin{lem}
\label{lem:uniqueness:MFG:f=0:g:Lipschitz:small}
When $\varepsilon=0$ and regardless of the dimension, the solutions of the MFG associated with the cost functional 
\eqref{eq:cost:numerical}
are given by the roots $\tilde m_{1}$ of the equation $2x+g(x)=0$ and then by the changes of variable
\eqref{eq:change:variable:m,h}. In particular, 
if the Lipschitz constant of the function $g$ is strictly less than 2, then the MFG has one and only one solution.
\end{lem}

\subsubsection{Potential structure when $\varepsilon=0$}
\label{subse:potential}
Following 
\eqref{eq:potential:structure:1}
and
\eqref{eq:potential:structure:2}, we may associate 
a mean field control problem with the MFG (with $\varepsilon=0$) if the function $g$ derives from a potential 
$G$.  
The cost functional is given by 
\begin{equation*}
{\mathcal J}({\boldsymbol \alpha}) 
= {\mathbb E} \biggl[ \frac12 \vert X_{1} \vert^2 + G\bigl({\mathbb E}(X_{1})\bigr) + \frac12 \int_{0}^1 \vert \alpha_{t} \vert^2 \ud t \biggr],
\end{equation*}
where, in the right-hand side, 
\begin{equation*}
X_{1} = \int_{0}^1 \alpha_{t} \ud t.  
\end{equation*} 
The analysis of the minimization  problem $\inf_{{\boldsymbol \alpha}} {\mathcal J}({\boldsymbol \alpha})$ is straightforward. 
Indeed, by an obvious convexity argument, we observe that 
\begin{equation*}
{\mathcal J}({\boldsymbol \alpha})  \geq 
{\mathcal J}\biggl({\mathbb E} \int_{0}^1 \alpha_{t} \ud t \biggr),
\end{equation*}
where, in the right-hand side, it is obviously understood that 
the argument of 
${\mathcal J}$ is a constant control. 
This shows that 
the minimizers 
are deterministic and constant in time. 
Using the generic notation $\beta$ for $\int_{0}^1 \alpha_{t} \ud t$, the problem is thus to minimize
\begin{equation*}
{\mathscr J}(\beta)= \beta^2 + G(\beta ), \quad \beta \in {\mathbb R}. 
\end{equation*}
Of course, we observe that any minimizer of  ${\mathscr J}$ is also a solution of the equation 
$\beta=-g(\beta)/2$: we recover the fact that 
any solution to the mean field control problem is an MFG solution. 
Obviously, the converse may not be true. Accordingly,
the set of solutions to the mean field control problem may be strictly included in the set 
of equilibria of the corresponding mean field game. This principle is illustrated by 
Figure \ref{fig:potential mathscrJ} below with 
$g$ as in \eqref{eq:benchmark:d=1}: For all the values of 
$\kappa \in \{2,3,\cdots,10\}$, the potential ${\mathscr J}$ has a unique minimizer; 
but, for $\kappa \in \{7,8,9,10\}$, the derivative of ${\mathscr J}$ has several zeros, hence proving that the MFG has several solutions.

\begin{figure}[h!]
	\begin{subfigure}[b]{0.48\linewidth}
		\includegraphics[width=\linewidth]{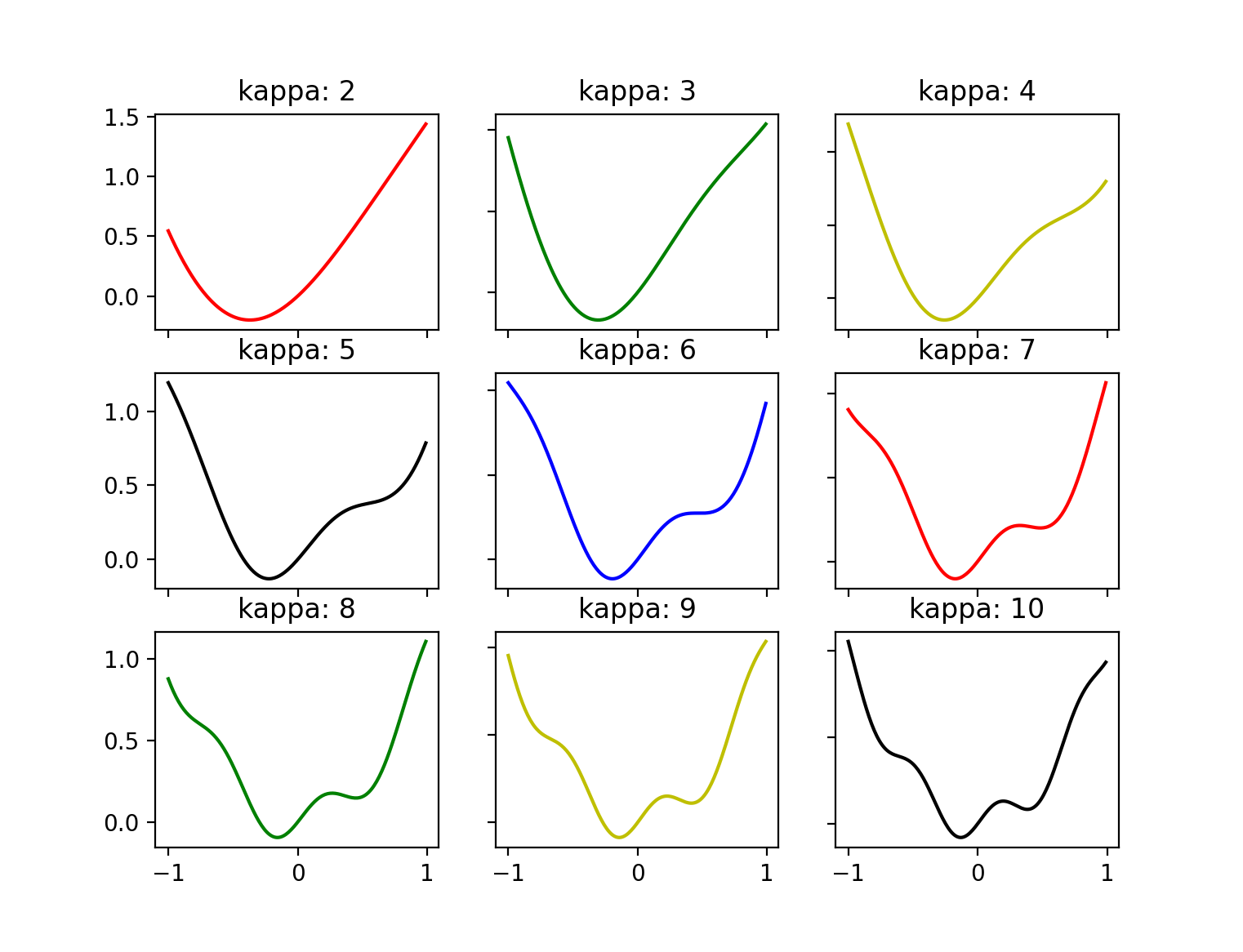}
		\caption{Evolution of ${\mathscr J}$ with $\kappa$. }
	\end{subfigure}
	\begin{subfigure}[b]{0.48\linewidth}
		\includegraphics[width=\linewidth]{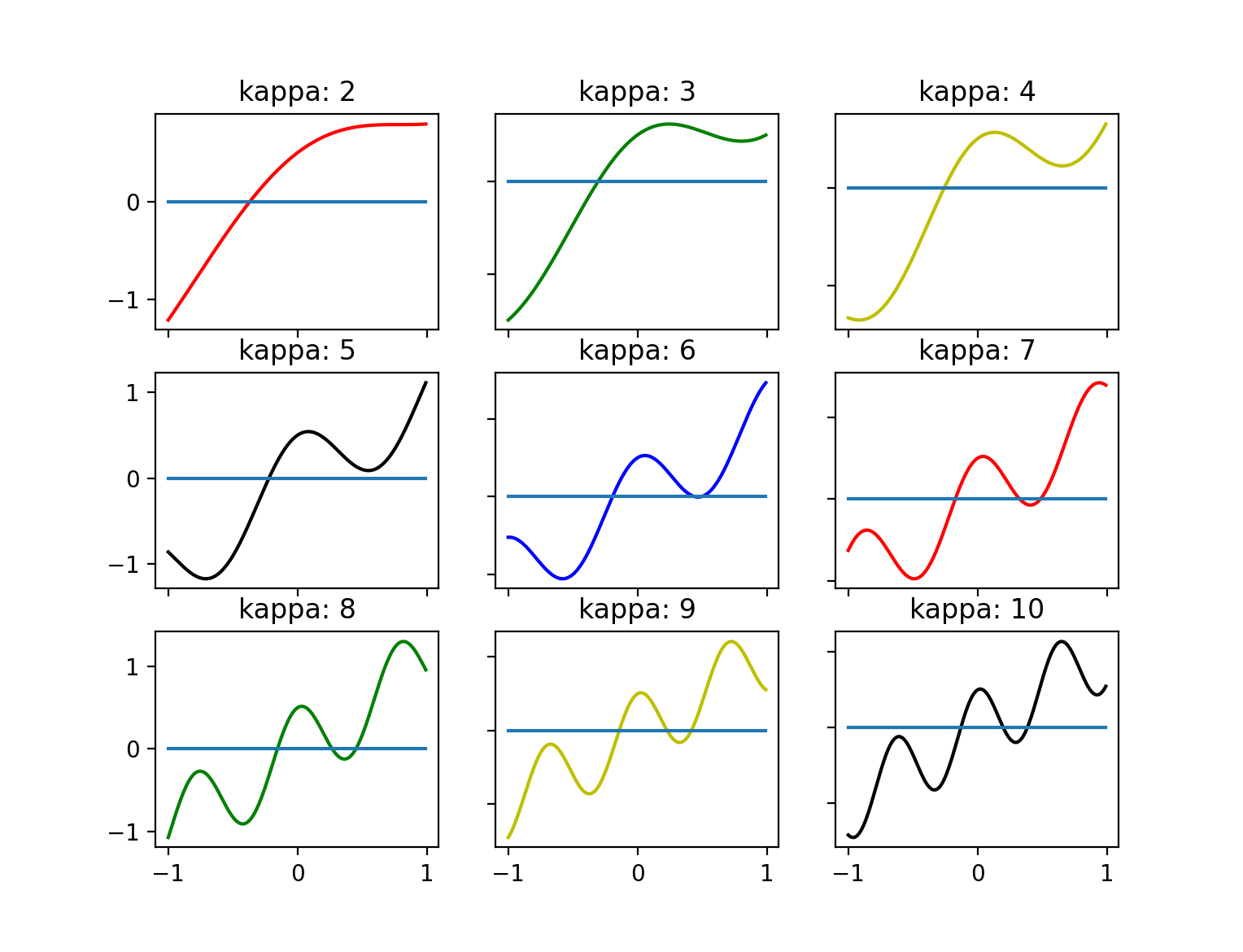}
		\caption{Evolution of $\frac{\ud}{\ud \beta}{\mathscr J}$ with $\kappa$. }
	\end{subfigure}
	\caption{Minimizers of the mean field control problem and equilibria of the mean field game for $\kappa \in \{2,\cdots,10\}$ when $d=1$ and $g(x)=\cos(\kappa x)$:
	Solutions of the mean field control problem are the minimizers of ${\mathcal J}$; Solutions of the mean field game are the roots of $\frac{\ud}{\ud \beta}{\mathscr J}$.}
	\label{fig:potential mathscrJ}
\end{figure}

Interestingly, 
the recent contribution \cite{cecdel2019}
says that, when the MFG is potential, the equilibria 
that are not associated with a minimizer of the potential should be ruled out. Equivalently, 
only the MFG solutions that minimize (globally) the potential should be selected. 
Even more, the guess (which has been rigorously established in  
\cite{cecdel2019,delfog2019} but in different settings that the one addressed here) 
is that the minimizers of the potential are precisely those that appear 
by forcing uniqueness with a common noise and then by tuning down the intensity of the common noise. We check this prediction 
numerically in Subsection \ref{subse:3:3}, by using our learning method.

\subsection{Algorithms for a fixed intensity of the common noise}
\label{Subse:3:2}

We here explain how to implement the fictitious play  
in the two benchmark examples 
\eqref{eq:benchmark:d=1}
and
\eqref{eq:benchmark:d=2} along the lines of Figure \ref{fig:4}. 
Throughout the subsection, the intensity of the common noise is fixed. For simplicity, 
the intensity of the common noise is chosen as $\varepsilon=1$
and the intensity of the independent noise is chosen as $\sigma=0$ or $1$. 
Numerical experiments are discussed in the next subsection. 

\subsubsection{Numerical reference solution}
\label{subse:reference}
We first explain how to compute a reference solution by solving numerically the decoupled forward-backward equation \eqref{eq:decoupled:limit}. 
Observe that there is no contradiction in using a numerical method
for testing our learning algorithm: the numerical method relies on the explicit knowledge of 
the coefficient $g$, whereas the learning algorithm is intended to work on the sole observations of the costs. 

The numerical method we use is tailor-made to our problem. Indeed, we observe that 
the solution to the backward equation in \eqref{eq:bsde:true:underPh}
 writes
\begin{equation}
\label{eq:numerical:guess:0}
h_{t} = {\mathbb E}^{{\boldsymbol h}} \biggl[ \exp  \biggl( - \int_{t}^1 
\eta_{s} \ud s \biggr) g \bigl( m_{1} \bigr) \, \big\vert \, {\mathcal F}^{{\boldsymbol W}}_{t} \biggr], \quad t \in [0,1], 
\end{equation}			
where
$(m_{t})_{0 \le t \le 1}$ is an Ornstein-Uhlenbeck process 
(with independent coordinates) 
\begin{equation}
\label{eq:numerical:guess:forward}
\ud m_{t} = -  \eta_{t}  m_{t} \ud t +  \ud W_{t}, \quad t \in [0,1] \ ; 
\quad m_{0}=0. 
\end{equation}
We then employ a Picard scheme for solving the fixed point equation 
 \eqref{eq:numerical:guess:0}, by iterating
\begin{equation}
\label{eq:numerical:guess}
h_{t}^{\textrm{\rm ref},n+1} = {\mathbb E}^{{\boldsymbol h}^{\textrm{\rm ref},n}} \biggl[ \exp  \biggl( - \int_{t}^1 
\eta_{s} \ud s \biggr) g \bigl( m_{1} \bigr) \, \big\vert \, {\mathcal F}^{{\boldsymbol W}}_{t} \biggr], \quad t \in [0,1], 
\end{equation}			  
 Numerically, we use a time grid $\{t_{k}=k/p\}_{k}$ with $p$ uniform steps (for the same $p$ as the one 
 used in the learning method, which makes the comparison easier)
 \textcolor{black}{and, following the seminal work of 
\cite{GobetLemorWarin}, 
  we employ a regression method in order to approximate the conditional expectation appearing in the right-hand side.
 Precisely,} 
  we find, 
 at each iteration $n$ of the Picard sequence and at each node $t_{j} >0$ of the time mesh, an approximation of
 the conditional expectation in the right-hand side of 
\eqref{eq:numerical:guess}
in the form of 
a (deterministic) function of $m_{t_{j}}$ (the forward component in \eqref{eq:numerical:guess})\textcolor{black}{, with the deterministic function being chosen 
in a given class of functions from ${\mathbb R}^d$ into itself. 
The point is then to choose a convenient class within which 
the regression is achieved.}

\textcolor{black}{Here, we address two approaches, already reported in the literature. The first one
consists in performing regression on linear combinations of Hermite polynomials. 
It
is inspired from 
\cite{BriandLabart}, which offers (in a somewhat more complicated framework) some theoretical bounds 
on the regression error. 
The second approach is taken from 
the work
\cite{EHanJentzen}, which paved the way for a systematic use of neural networks in the 
computation of numerical solutions to nonlinear PDEs and related Markovian BSDEs. 
The point in this second approach is thus to approximate 
the conditional expectation in 
\eqref{eq:numerical:guess}, at time $t=t_j$, by means of a neural network with $m_{t_j}$ as entry.}

\color{black}
For the sake of completeness, we provide of a short overview of the shape of the functions 
within these two classes.
\vskip 5pt

\textit{Using Hermite polynomials.}
The real-valued components of the regression functions are taken in the linear span of
$\{H_{\ell}^d(\cdot/\sigma_{t_{j}})\}_{\vert \ell \vert \leq D}$, 
where $\sigma_{t_{j}}$ is the common standard deviation of the coordinates of $m_{t_{j}}$ (which is explicitly computable) 
and $\{H_{\ell}^d(\cdot)\}_{\vert \ell \vert \leq D}$ is the collection of 
Hermite polynomials of dimension $d$ ($d=1,2$ in our case) and of degree less than $D$, 
for a given integer $D$  (here 
$\ell$ is a $d$-tuple of integers and $\vert \ell \vert$ is the sum of its entries). 
We recall that, when $d=1$, 
$$H_\ell^1(x)= \frac{(-1)^\ell}{\sqrt{2^\ell \ell!}} e^{x^2}\frac{\ud}{\ud x^\ell}[e^{-x^2}], \quad x \in {\mathbb R}.$$  
When $d=2$, 
$$H_{(\ell_{1},\ell_{2})}^2(x_{1},x_{2})= H_{\ell_{1}}^1(x_{1}) H_{\ell_{2}}^1(x_{2}), 
\quad x_{1},x_{2} \in {\mathbb R}.$$
\vskip 5pt

\textit{Using a neural network.}
For a number of layers $L$, 
the functions used for the regression have the 
form $\psi^{L+1} \circ \varphi^{L} \circ \psi^{L} \cdots \circ \psi^2 \circ \varphi^1 \circ \psi^1$, with 
$\psi^\ell$, for each $\ell=1,\cdots,L+1$, being an affine function from ${\mathbb R}^{d_{\ell-1}}$ to 
${\mathbb R}^{d_\ell}$, where $d_0=d_{L+1}=d$. For each $\ell=1,\cdots,L$, $\varphi^\ell$ is a function from ${\mathbb R}^{d_\ell}$ into itself that 
maps $(x_1,\cdots,x_{d_\ell})$ onto $(\varsigma(x_1),\cdots,\varsigma(x_{d_\ell}))$ for some activation function 
$\varsigma$ (which, for simplicity, is taken to be the same for any $\ell$). 
The activation function is fixed \textit{a priori}. Standard examples for it are the ReLu or Sigmoid functions. 
We then call neurons of the networks the coefficients encoding the linear mappings 
$(\psi^\ell)_{\ell=1,\cdots,L+1}$. In clear, the principle is to tune, for each $\ell$, matrices $A^\ell\in {\mathbb R}^{d_{\ell-1} \times d_\ell}$
and vectors $B^\ell \in {\mathbb R}^{d_\ell}$ in the decomposition of $\psi_\ell$ as
\begin{equation*}
\psi^\ell(x) = A^\ell x + B^\ell, \quad x \in {\mathbb R}^{d_{\ell-1}}. 
\end{equation*} 
Neural networks are notoriously known to (possibly) behave well
in higher dimension. Below, we provide examples with $d=12$ and $d=20$. 
\vskip 5pt

Given a finite dimensional class ${\mathcal C}$ of regression functions (obtained by  Hermite polynomials or neural networks), 
we use the following algorithm for the 
the computation of a reference solution.
\vskip 5pt

\color{black}

\noindent {\bf Algorithm 1.} [Reference solution]

\begin{enumerate}
\item[Input:] Introduce
$\Delta_{t_{k}} w^{(j)}$ with $j \in \{1,\cdots,N\}$
and
$k \in \{1,\cdots,p\}$, 
a collection 
of $N$ realizations  of independent Gaussian variables ${\mathcal N}(0,I_{d})$. 
\item[Task:] For each $i$, compute 
$(m^{(j)}_{t_{k}})_{j=0,\cdots,p}$ the realizations  of
the Euler scheme associated with  
\eqref{eq:numerical:guess:forward} and driven by the simulations 
$$\biggl(w^{(j)}_{t_{k}}= \frac1{\sqrt{p}}\Bigl[\Delta_{t_{1}} w^{(j)} + \cdots+ \Delta_{t_{k}} w^{(j)}\Bigr]\biggr)_{k=1,\cdots,p}.$$
\item[Loop:] One iteration consists of one Picard iteration. The input at iteration number $n$ is encoded in the form of an array $(h_{t_{k}}^{n,(j)})_{j,k}$
with entries in ${\mathbb R}^d$.
The following procedure returns the next input for iteration number $n+1$.
\begin{enumerate}
\item Look for 
$(h_{t_{k}}^{n+1,(j)})_{j,k}$
in terms of 
$(h_{t_{k}}^{n,(j)})_{j,k}$
by solving, for each $k=0,\cdots,p-1$, the minimization  problem \color{black}
\begin{equation*}
\min_{{\mathfrak h} \in {\mathcal C}}
\frac1N \sum_{j=1}^N 
{\mathcal E}^{n,(j)}_{k} \biggl\vert \exp \biggl( - 
 \sum_{s=k}^{p-1} \eta_{t_{s}}
\biggr) g(m_{1}^{(j)}) - {\mathfrak h} \bigl( m_{t_{k}}^{(j)} \bigr)
\biggr\vert^2,
\end{equation*}
with 
\begin{equation*}
{\mathcal E}^{n,(j)}_{k} =\exp \biggl( -  \sum_{s=k}^{p-1} h_{t_{s}}^{n,(j)}
\cdot \Delta_{t_{k+1}} 
w^{(j)} - \frac1{2p} 
\sum_{s=k}^{p-1} \vert h_{t_{s}}^{n,(j)} \vert^2
\biggr).
\end{equation*}
\color{black}
\item To enforce some stability, update each coordinate of 
$(h_{t_{k}}^{n+1,(j)})_{j,k}$
by projecting it 
onto $[-1,1]$.
\end{enumerate}
 \end{enumerate}

\begin{rem}
\label{rem:normalisation:hermite}
For instance, when Hermite polynomials are used, 
the minimization step in 
Algorithm 1 can be rewritten as 
 \begin{equation*}
\min_{{\boldsymbol c}=(c_{\ell})_{\vert \ell \vert \leq D}}
\frac1N \sum_{j=1}^N 
{\mathcal E}^{n,(j)}_{k} \biggl\vert \exp \biggl( - 
 \sum_{s=k}^{p-1} \eta_{t_{s}}
\biggr) g(m_{1}^{(j)}) - \sum_{\vert \ell \vert \leq D} c_{\ell}
H^d_{\ell} \bigl( \tfrac1{\sigma_{t_k}} m_{t_{k}}^{(j)} \bigr)
\biggr\vert^2,
\end{equation*}
where each $c_{\ell}$ in 
${\boldsymbol c}=(c_{\ell})_{\vert \ell \vert \leq D}$ is in ${\mathbb R}^d$.
(When $k=0$, only $c_{\boldsymbol 0}$, where $\vert {\boldsymbol 0} \vert=0$, matters and it is given by a mere empirical mean.)

Here, the coordinates in 
\eqref{eq:numerical:guess:forward}
 have the same marginal variances because $(\eta_{t})_{0 \leq t \leq T}$
 is scalar-valued (or equivalent is
 a $d$-diagonal matrix).  
This is no longer true 
 when 
 $R$ and  
 $Q$ in
 \eqref{eq:cost:intro}
are general matrices. In such a case (and more generally when the components are not centered), 
the components of $(m_{t})_{0 \leq t \leq T}$ are no longer independent
and 
the `normalized' Hermite polynomials
$\{H_{\ell}^d(\cdot/\sigma_{t_{k}})\}_{\vert \ell \vert \leq D}$ used  
above should be instead replaced by 
$\{H_{\ell}^d({\mathbb K}^{-1/2}(m_{t_{k}}) (\cdot - {\mathbb E}(m_{t_{k}}))\}_{\vert \ell \vert \leq D}$, where 
${\mathbb K}(m_{t_{j}})$
is the covariance matrix of $m_{t_{k}}$ 
and 
${\mathbb K}^{1/2}(m_{t_{k}})$ is any root of it (e.g., as given by the Cholesky decomposition). 
\end{rem}

\subsubsection{Numerical approximation by fictitious play}
We now address the implementation  of the fictitious play according to the principle stipulated by 
Figure \ref{fig:4}. The black-box therein is discretized in a suitable manner. It is used 
to solve numerically the 
optimization  problem 
\eqref{eq:optimization:exploration},
and then to implement the 
updating rule 
\eqref{eq:update} for the environment. 
Since Lemma 
\ref{lem:4} says that the corresponding optimal law has an affine Markov feedback form, 
we \textcolor{black}{can} perform the optimization  in 
\eqref{eq:optimization:exploration}
over closed loop controls 
that are affine in the state variable, with the linear coefficient being a scalar only depending on time and with the intercept possibly depending on the environment.
From the practical point of view, this choice is fully justified if the model is expected to be linear-quadratic in the space/action variables (as it is in 
\eqref{eq:state:intro}--\eqref{eq:cost:intro}, with $Q$ and $R$ being scalars), but this does not require the coefficients $Q$, $R$, $f$ and $g$ to be known explicitly.
Mathematically, this writes as follows. 
We \textcolor{black}{can} restrict ourselves to controls
${\boldsymbol \alpha}=(\alpha_{t_{k}})_{{k =1,\cdots,p}}$ of the form 
\begin{equation}
\label{eq:numerical:scheme:alphatk:xtk:ctk}
\alpha_{t_{k}} = a_{t_{k-1}} X_{t_{k-1}} + C_{t_{k-1}},
\end{equation}
where $a_{t_{k-1}}$ is a scalar and $C_{t_{k-1}}$ is
a $d$-dimensional  random vector \textcolor{black}{(which is typically adapted to the 
increments $(\Delta_{t_\ell} w)_{\ell \leq k-1}$, but the measurability properties of which 
are specified  
in a finer manner right below)}. 
In comparison with the formula \eqref{eq:lem:4:feedback:form}, $a_{t_{k-1}}$ 
should be understood as a proxy for 
$-\eta_{t_{k-1}}$ and 
$C_{t_{k-1}}$
as a proxy 
for the intercept therein. Part of the numerical difficulty is thus 
to have a tractable regression method for capturing the randomness of 
$C_{t_{k-1}}$.
Recalling that the intercept ${\boldsymbol h}$ 
in the solution of the mean field game with common noise is a function of 
the environment $(m_{t})_{0 \leq t \leq 1}$, see \eqref{eq:representation:h:k}, a key point in our numerical method is to 
look for 
$C_{t_{k-1}}$
as a $\sigma(\overline m_{t_{k-1}})$-measurable random variable, where 
$\overline{m}_{t_{k-1}}$ is the current proxy for the state of the environment at time 
$t_{k-1}$. Numerically, we use a regression method, 
\textcolor{black}{very much in the spirit of Algorithm 1: 
either we use a regression based on 
a $d$-dimensional linear combination of Hermite polynomials of degree less than $D$, for a fixed value of $D$, or we use a neural network
with $L$ layers and with a prescribed numbers of neurons.
Of course, the reader may wonder about another approach 
to 
\eqref{eq:numerical:scheme:alphatk:xtk:ctk}, in which we optimize over general (possibly non-affine) controls in Markov feedback form. In fact, while this seems an attractive way to bypass any \textit{a priori} knowledge about the shape of the model, this approach requires an extra step for updating the intercept ${\boldsymbol h}$ at the next iteration of the fictitious play. Intuitively, the new intercept value can be found by linearly regressing the controls on the states. %The precise form of this additional linear regression will be clarified later. 
At this point, however, we assert that the presence of this additional step in the numerical procedure can only be fully justified if the linear-quadratic structure of the model is known. This makes the advantage of not postulating 
\eqref{eq:numerical:scheme:alphatk:xtk:ctk} very limited.}

We now make this construction explicit when $\sigma=\varepsilon=1$ \textcolor{black}{and the postulate 
\eqref{eq:numerical:scheme:alphatk:xtk:ctk} is in force}, noticing that the algorithm must rely on simulations for the independent and common noises.
Throughout, the learning parameter $\varpi$ is arbitrary: it may be strictly greater than 1 (in which case we are using the geometric variant of the fictitious play)
or equal to 1 (in which case we are using the harmonic variant). 
Below, we call $M$ the number of particles (given a realization of the common noise) and $N$ the number of simulations of the whole system (or equivalently of the common noise), 
with $i$ denoting the generic label for a particle and $j$ denoting the generic label for a realization. 
We are thus given 
a collection
\begin{equation}
\label{eq:noises:simulation}
\Delta_{t_{k}} b^{(i,j)}, \ \Delta_{t_{k}} w^{(j)}, \quad i \in \{1,\cdots,M\}, \ j \in \{1,\cdots,N\}, \
k \in \{1,\cdots,p\}, 
\end{equation}
of realizations  of independent ${\mathcal N}(0,I_{d})$ Gaussian variables. 
At iteration number $n$ of the fictitious play, the current proxy for the environment is thus given in the form of 
a collection 
$(\overline{m}^{n,(j)}_{t_{k}})_{j,k}$, with $j$ running from $1$ to $N$ and 
$k$ running from $0$ to $p$.  Similarly, 
the proxy for the intercept in the optimal law \eqref{eq:lem:4:feedback:form}
is given in the form of a collection 
$(h^{n,(j)}_{t_{k}})_{j,k}$, also with $j$ running from $1$ to $N$ and 
$k$ running from $0$ to $p-1$. 
The fact that the two proxies are independent of $i$ is fully consistent with the fact that
$\overline{\boldsymbol m}^n$ and ${\boldsymbol h}^n$ in 
Remark 
\ref{rem:main:thm:2:b}
are adapted to the realization  of the common noise.
Following our introductory discussion, we solve, as an approximation of the stochastic control problem 
\eqref{eq:optimization:exploration}, the minimization  problem:
\begin{equation}
\label{eq:approx:optimisation}
\min \biggl\{ 
\frac1{N}
\sum_{j=1}^N \biggl(
{\mathcal E}^{n,(j)}
\frac1M \sum_{i=1}^M
\biggl[  \frac{1}{2p} \sum_{k=1}^{p} 
\varpi^2 \vert \alpha_{t_{k}}^{(i,j)} \vert^2  + \frac12 \Bigl\vert \varpi x^{(i,j)}_{1} + g\Bigl(\overline{m}_{1}^{n,(j)}\Bigr)
\Bigr\vert^2
\biggr]
\biggr)
\biggr\},
\end{equation}
where 
\begin{equation*}
{\mathcal E}^{n,(j)} = \exp \Biggl( - \varpi \sqrt{\frac{1}{p}} \sum_{k=0}^{p-1}
h^{n,(j)}_{t_{k}}
\cdot \Delta_{t_{k+1}} w^{(j)} - \frac{\varpi^2}{2 p}
\sum_{k=0}^{p-1}
\vert 
h^{n,(j)}_{t_{k}}
\vert^2
\Biggr).
\end{equation*}
Moreover, in the minimization  problem,
$\alpha_{t_{k}}^{(i,j)}$, $x_{t_{k}}^{(i,j)}$ are required to be of the form
\begin{equation}
\label{eq:dynamics:numerics}
\begin{split}
&x_{t_{k}}^{(i,j)} = x_{t_{k-1}}^{(i,j)} + \frac{1}{p}
\alpha_{t_{k}}^{(i,j)} + {\frac{1}{\varpi \sqrt{p}}} \Delta_{t_{k}} b^{(i,j)}, 
\quad \ell=1,\cdots,p \ ; 
\quad x_{0}^{(i,j)}=x_{0},
\\
&\alpha_{t_{k}}^{(i,j)} = a_{t_{k-1}} x_{t_{k-1}}^{(i,j)} + C_{t_{k-1}}^{(j)} + \varpi  h^{n,(j)}_{t_{k-1}} +
\sqrt{p} \Delta_{t_{k}} w^{(j)}, \quad k=1,\cdots,p.
\end{split}
\end{equation}
In particular,
$(x_{t_{k}}^{(i,j)})_{k=0,\cdots,p}$ solves the following Euler scheme: 
\begin{equation*}
x_{t_{k}}^{(i,j)} = x_{t_{k-1}}^{(i,j)} + \frac{1}{p}
\Bigl( a_{t_{k-1}} x_{t_{k-1}}^{(i,j)} + C_{t_{k-1}}^{(j)} + \varpi h^{n,(j)}_{t_{k-1}} \Bigr) 
+
{\frac{1}{\varpi \sqrt{p}}} \Delta_{t_{k}} b^{(i,j)} +  
{\frac{1}{\sqrt{p}}} \Delta_{t_{k}} w^{(j)}, \quad k=1,\cdots,p.
\end{equation*}
\color{black}
Moreover, 
$C^{(j)}_{t_{k}}$ is required to be of the form 
\begin{equation}
\label{eq:approx:general}
C^{(j)}_{t_{k}} = {\mathfrak h}_{t_k} \bigl( 
\overline{m}^{n,(j)}_{t_{k}}
\bigr),
\end{equation}
for a function ${\mathfrak h}_{t_k}$ within one of the two classes 
${\mathcal C}$ described in 
Subsection 
\ref{subse:reference}, depending on whether we use Hermite polynomials or 
neural networks. 
As already explained, 
neural networks may provide better results in higher dimension (or, to put it differently, 
Hermite polynomials may be of an intractable complexity in higher dimension). 

As before, we feel useful to exemplify \eqref{eq:approx:general}
within each of the two aforementioned case. 
\vskip 5pt

\textit{Using Hermite polynomials.}
In this case, we use a slightly modified form of 
\eqref{eq:approx:general}, as we also center 
$ \overline{m}^{n,(j)}_{t_{k}}$ by means of the empirical mean in the argument of the Hermite polynomials.
Definition  \eqref{eq:approx:general} thus becomes:
\begin{equation}
\label{eq:approx:hermite}
C^{(j)}_{t_{k}} = \sum_{\vert \ell \vert \leq D}
c_{t_{k}}(\ell) H^{d}_{\ell}
\biggl(\bigl(U_{t_{k}}^n\bigr)^{-1} \Bigl( \overline{m}^{n,(j)}_{t_{k}}
- \frac1{N}
\sum_{r=1}^N \overline{m}^{n,(r)}_{k}
\Bigr)
\biggr),
\end{equation}
where $c_{k}(\ell) \in {\mathbb R}^d$ and $U_{t_{k}}^n$ is the diagonal matrix obtained 
by taking the roots of the diagonal of the empirical covariance matrix:
\begin{equation*}
\Sigma_{t_{k}}^n = \frac1N \sum_{j=1}^N
 \Bigl( \overline{m}^{n,(j)}_{t_{k}}
- \frac1{N}
\sum_{r=1}^N \overline{m}^{n,(r)}_{t_{k}}
\Bigr)
\otimes \Bigl( \overline{m}^{n,(j)}_{t_{k}}
- \frac1{N}
\sum_{r=1}^N \overline{m}^{n,(r)}_{t_{k}}
\Bigr).
\end{equation*}
\begin{rem}
Following 
Remark \ref{rem:normalisation:hermite}, 
we can define $U_{t_{k}}^n$
(in a more general fashion) 
as the 
upper triangular matrix given by the Cholesky decomposition of 
$\Sigma_{t_{k}}^n$ in the form $\Sigma_{t_{k}}^n=(U_{t_{k}}^n)^{\dagger} U_{t_{k}}^n$.
\end{rem}
The minimization  problem 
in 
\eqref{eq:approx:optimisation}
is thus defined (when using Hermite polynomials) over
${\boldsymbol a}=(a_{t_{k}})_{k} \in {\mathbb R}^p$
and
${\boldsymbol c}=(c_{t_k}(\ell))_{k,\ell} \in ({\mathbb R}^d)^{p \times L}$, where 
$L$ is the number of distinct $d$-tuples of (non-negative) integers whose sum is less than $D$. 
\vskip 5pt

\textit{Using neural networks.}
With the same definition as in
Subsection  
\ref{subse:reference}, 
  \eqref{eq:approx:general} takes the form 
   \begin{equation}
\label{eq:approx:neural}
C^{(j)}_{t_{k}} = 
\psi^{L+1}_{t_k} \circ \varphi^{L}_{t_k} \circ \psi^{L}_{t_k} \cdots \circ \psi^2_{t_k} \circ \varphi^1_{t_k} \circ \psi^1_{t_k} \bigl( 
\overline{m}^{n,(j)}_{t_{k}}
\bigr).
\end{equation}
Typically, the number of layers $L$, the dimensions of the various layers and the activation functions are 
the same for any time $t_k$, $k=0,\cdots,N-1$. In other words, we can write 
$\varphi^{L}_{t_k}=\varphi^L$, $\varphi^{L-1}_{t_k} = \varphi^{L-1}$, $\cdots$. 
Each $\psi_{t_k}^{\ell}$, for $\ell=1,\cdots,L+1$, writes as
\begin{equation*}
\psi_{t_k}^\ell(x) = A_k^\ell x + B_k^\ell, \quad x \in {\mathbb R}^{d_{\ell-1}}. 
\end{equation*} 
The minimization  problem 
in 
\eqref{eq:approx:optimisation}
is thus defined over
${\boldsymbol a}=(a_{t_{k}})_{k} \in {\mathbb R}^p$, 
${\boldsymbol A}=(A_{k}^\ell)_{k,\ell}$
and 
${\boldsymbol B}=(B_{k}^\ell)_{k,\ell}$, with 
$A_k^\ell\in {\mathbb R}^{d_{\ell-1} \times d_\ell}$
and $B_k^\ell \in {\mathbb R}^{d_\ell}$.
\vskip 5pt

\color{black}
We can summarize the algorithm in the following form.
It has the same form whether we work with the harmonic ($\varpi=1$) or geometric ($\varpi>1$) 
version of the fictitious play. 
\vskip 4pt

\noindent {\bf Algorithm 2.} [Learning with two noises, $\sigma=\varepsilon=1$]
\begin{enumerate}
\item[Input:] Take as input the realizations  
\eqref{eq:noises:simulation}, which 
are the same at any iteration of the algorithm.
\item[Loop:] At rank $n+1$ of the fictitious play:
\begin{enumerate}
\item Take as input 
the proxies
$(\overline{m}^{n,j}_{t_{k}})_{j,k}$
and
$(h^{n,j}_{t_{k}})_{j,k}$
	for (respectively) the environment and the intercept. 	
	\item \color{black}Solve the minimization  problem 
	\eqref{eq:approx:optimisation}
	over 
	${\boldsymbol a}=(a_{t_{k}})_{k}$ in 
	\eqref{eq:dynamics:numerics}
	and
%	${\boldsymbol c}=(c_{t_{k}}(\ell))_{k,\ell}$ in 
${\boldsymbol {\mathfrak h}}=({\mathfrak h}_{t_{k}})_{k}$
in
	\eqref{eq:approx:general}.
	Call 
	${\boldsymbol a}^{n+1}=(a^{n+1}_{t_{k}})_{k}$ 
	and
	${\boldsymbol {\mathfrak h}}^{n+1}=({\mathfrak h}^{n+1}_{t_{k}})_{k}$ the optimal points (returned by any optimization  method). 
	\item With ${\boldsymbol a}^{n+1}=(a_{t_{k}}^{n+1})_{k}$ 
	and ${\boldsymbol {\mathfrak h}}^{n+1}=({\boldsymbol {\mathfrak h}}^{n+1}_{t_{k}})_{k}$, associate $(x^{n+1,(i,j)}_{t_{k}})_{i,j,k}$ as in 
		\eqref{eq:dynamics:numerics}
		and $(C_{t_{k}}^{n+1,(j)})_{k,j}$ as in 
				\eqref{eq:approx:general}. \color{black}
				\item  Update the proxies by letting
				\begin{equation*}
				\begin{split}
&m^{n+1,(j)}_{t_k}= \frac1M \sum_{i=1}^M 
x^{n+1,(i,j)}_{t_k}, \quad \overline{m}^{n+1,(j)}_{t_k}
= 
\pi_{n+1}(\varpi)
\times 
m^{n+1,(j)}_{t_k}
+ 
\bigl( 1 - \pi_{n+1}(\varpi)\bigr) \times  \overline m^{n,(j)}_{t_k}, 
\\
&h^{n+1,(j)}_{t_k}= - C_{t_k}^{n+1,(j)},
\end{split}
\end{equation*}
\end{enumerate} 
where
$\displaystyle 
\pi_n(\varpi)
=
\left\{
\begin{array}{ll}
1/n \quad &{\rm if} \ \varpi = 1,
\\
\tfrac{\varpi^{-(n-1)} (1-\varpi^{-1})}{1-\varpi^{-n}} \quad &{\rm if} \ \varpi > 1. 
\end{array}
\right.$
\end{enumerate} 

\color{black}
\begin{rem}
\label{rem:bij}
The construction of Algorithm 2 relies on 
the various noises 
\eqref{eq:noises:simulation}. In fact, it is worth mentioning that
the algorithm can be reformulated in a similar manner when the increments 
$(\Delta_{t_{k}} b^{(i,j)})_{i,j,k}$ 
defining the idiosyncratic noises 
are assumed to just depend on $(i,k)$ (equivalently they are equal for the same value of 
$i$ but for two different values of $j$). As before, the resulting 
random variables $(\Delta_{t_k} b^{(i)})_{i=1,\cdots,M,k=1,\cdots,p}$ are required to be independent and identically distributed. 

Obviously, the smallest family $(\Delta_{t_k} b^{(i)})_{i=1,\cdots,M,k=1,\cdots,p}$
is of a cheapest numerical cost. 
However, it creates additional correlations between the particles: for two different indices $j$ and $j'$, 
the corresponding two empirical means over $i \in \{1,\cdots,N\}$ in 
\eqref{eq:approx:optimisation} are dependent, with the correlations 
vanishing asymptotically as $N$ tends to $\infty$. Even though we do not provide any 
bound for the numerical error in this section, it is worth mentioning that 
the presence of additional correlations would make the Monte-Carlo error harder to estimate. 
As reported below, we tested the two approaches. For the range of parameters 
we used, we have not seen any major difference.  
\end{rem}
\color{black}

In some of the numerical examples below, we will use  variants of Algorithm 2. 
We 
first focus on the shape of the algorithm when $\sigma=0$, recalling that our results do not require the presence of an idiosyncratic noise. 
Implicitly, we then have a single particle only, i.e. $M=1$. 
Moreover, 
$\Delta_{t_{k}} b^{(i,j)}$ no longer appears in 
\eqref{eq:dynamics:numerics}.
We then have 
the following variant of Algorithm 2, which is of lower complexity:
\vskip 4pt

\noindent {\bf Algorithm 3.} [Learning with common noise only, $\sigma=0$, $\varepsilon=1$]
\begin{enumerate}
\item[Input:] Take as input the realizations  
$(\Delta_{t_{k}} w^{(j)})_{j,k}$ in 
\eqref{eq:noises:simulation}, which 
are the same at any iteration of the algorithm.
\item[Loop:] At rank $n+1$ of the fictitious play, apply the same loop as in Algorithm 2, but 
with $M=1$ and with 
$\Delta_{t_{k}} b^{(1,j)}=0$ in 
\eqref{eq:dynamics:numerics}.
\end{enumerate} 

We end up our presentation with the case when there is no common noise, i.e. 
$\sigma=1$ and $\varepsilon=0$. In that case, our strategy no longer applies. 
The point is thus to implement the standard version of the fictitious play. 
Accordingly, 
there is no Girsanov density in 
\eqref{eq:approx:optimisation}
and 
$\alpha_{t_{k}}^{(i,j)}$ in 
\eqref{eq:dynamics:numerics}
merely writes
\begin{equation*}
\alpha_{t_{k}}^{(i,j)} = a_{t_{k-1}} x_{t_{k-1}}^{(i,j)} + C_{t_{k-1}}, \quad k=1,\cdots,p.
\end{equation*}
with $C_{t_{k}}$ being a deterministic $d$-dimensional vector (which is independent of $j$). In particular, 
there is no need to use the proxy 
$(h^{n,j}_{t_{k}})_{j,k}$ for the intercept. Equivalently, 
we can assume that $N=1$ \textcolor{black}{and ${\mathfrak h}_{t_k}$ in 
\eqref{eq:approx:general}
to be a mere constant function (i.e., the function ${\mathfrak h}_{t_k}$ is identified with 
${\mathfrak h}_{t_k}(0)$)}.
The rest of the algorithm is similar. 
It may be written as follows.  
\vskip 4pt

\noindent {\bf Algorithm 4.} [Learning with idiosyncratic noise only, $\sigma=1$, $\varepsilon=0$]
\begin{enumerate}
\item[Input:] Take as input the realizations  
$(\Delta_{t_{k}} b^{(i,1)})_{k}$ in 
\eqref{eq:noises:simulation}, which 
are the same at any iteration of the algorithm.
\item[Loop:] At rank $n+1$ of the fictitious play:
\begin{enumerate}
\item Take as input 
the proxy
$(\overline{m}^{n}_{t_{k}})_{k}$
for the environment and the intercept. 	
	\item \color{black}Solve the minimization  problem 
	\eqref{eq:approx:optimisation}
	over 
	${\boldsymbol a}=(a_{t_{k}})_{k}$ in 
	\eqref{eq:dynamics:numerics}
	and
	${\boldsymbol {\mathfrak h}}=({\mathfrak h}_{t_{k}})$ in 
	\eqref{eq:approx:general}, 
	with 
	${\mathcal E}^{n,(j)}=1$ in 
	\eqref{eq:approx:optimisation}
	and 
	with the last line in 
\eqref{eq:dynamics:numerics}
being replaced by 
\begin{equation*}
\alpha_{t_{k}}^{(i,1)}= a_{t_{k-1}} x_{t_{k-1}}^{(i,1)} + {\mathfrak h}_{t_{k-1}}(0), \quad k=1,\cdots,p.
\end{equation*}
Call 
	${\boldsymbol a}^{n+1}=(a^{n+1}_{t_{k}})_{k}$ 
	and
	${\boldsymbol {\mathfrak h}}^{n+1}(0)=({\mathfrak h}^{n+1}_{t_{k}}(0))_{k}$ the optimal points. 
	\item With ${\boldsymbol a}^{n+1}=(a_{t_{k}}^{n+1})_{k}$ 
	and ${\boldsymbol {\mathfrak h}}^{n+1}=({\mathfrak h}^{n+1}_{t_{k}}(0))_{k}$, associate $(x^{n+1,(i,1)}_{t_{k}})_{i,k}$ as in 
		\eqref{eq:dynamics:numerics} (with the same prescription as above).\color{black}
				\item Update the proxy by letting
				\begin{equation*}
				\begin{split}
&m^{n+1,(1)}_{t_k}= \frac1M \sum_{i=1}^M 
x^{n+1,(i,1)}_{t_k}, \quad \overline{m}^{n+1,(1)}_{t_k}
= 
\pi_{n+1}(\varpi)
m^{n+1,(1)}_{t_k}
+ 
\bigl( 1- 
\pi_{n+1}(\varpi)
\bigr)
 \overline m_{t_k}^{n,(1)},
\end{split}
\end{equation*}
with $\pi_n(\varpi)$ being defined as in Algorithm 2. 
\end{enumerate} 
\end{enumerate}

In all the numerical experiments, we use ADAM optimizer (as implemented in \texttt{TensorFlow}) in order to solve the optimization  step in Algorithms 2, 3 and 4. We recall from the discussion in 
Subsection \ref{subse:num:ex}
that the code in \texttt{TensorFlow} relies on some automatic differentiation and thus makes an explicit use of the linear-quadratic form of the coefficients. In this sense, our numerical implementation requires part of the model to be known, but, as we already explained in 
Subsection \ref{subse:num:ex}, this does not really affect the scope of our conclusions: 
Provided the optimization  method used in Algorithms 2, 3 and 4 is sufficiently accurate,
the whole works well and demonstrates the interest for exploring the state space by means of the common noise.

\subsection{Numerical experiments in \textcolor{black}{low dimension}}
\label{Subse:3:3:a}
The numerical examples that we provide below are here to illustrate several features of our form of fictitious play. In all the examples treated in this subsection, we choose 
$d=2$, $f \equiv 0$ and $g$ as in \eqref{eq:benchmark:d=2} with $\kappa=10$.
\textcolor{black}{The intercept ${\boldsymbol h}^n$ in 
 \eqref{eq:optimization:exploration}
 is approximated by means of Hermite polynomials, as explained in 
 Subsection  
 \ref{subse:reference}. Similar conclusions to the ones that are reported here could be drawn if the regression of 
 ${\boldsymbol h}^{n+1}$ onto
 $\overline{\boldsymbol m}^n$ (see \eqref{eq:approx:general}) was done by using a neural network. 
 For brevity, we feel however more appropriate to restrict the exposition of the numerical results obtained with neural networks to the higher dimensional setting only, which is done in 
 the next section.} 

It is worth mentioning that, whatever the method that is used, our aim is to demonstrate, numerically, the positive impact of the common noise, but not to compare the numerical rates of convergence 
obtained in the numerical experiments with the bounds obtained in the theoretical analysis. 
The reason is that the algorithms we present below include additional features (like numerical optimization  methods and regressions) that are not addressed in the theoretical analysis. 
This makes difficult any precise comparison. Also, it is fair to say that the time of execution of the full algorithm becomes
quite long (several hours for the longest experiments), even for low values of $n$ and $p$ ($n$ less than 50 and $p$ less than 100). 

As for the parameter $\varpi$, we address its influence in several manners. To do so, we implement both the harmonic and geometric variants of the fictitious play. 
Even in the harmonic regime (i.e., $\varpi=1$), 
 we can easily see the impact of the common noise, especially when the latter has intensity 1 (which is our choice in this subsection). 
Obviously, this is very good. A related observation is that, although our  
theoretical guarantees (given by the analogues of Theorems
\ref{main:thm:1}
and
 \ref{main:thm:2} in the harmonic regime, see Remark \ref{rem:fictitious:play:rate}) just provide a harmonic decay when $\varpi=1$, the observed rate of convergence 
may be better. In particular, the difference with the geometric version is rather tiny on the examples that are tested in this subsection, for a common noise with intensity 1 (as we will see in the forthcoming Subsection 
\ref{subse:3:3}, the picture is different when the viscosity is small). 
 There are even situations where, for numerical reasons, the harmonic variant has better results than the  geometric one. For sure, this could be rather surprising for the reader (and even disappointing in some sense), but it is clear that our (theoretical) estimates rely on rather generic properties of the dynamics in hand. Obviously, there might be more subtle features of the equations that should be taken into account in order to explain the behavior of the algorithm
 as the number $n$ of iterations 
 of the fictitious play increases. 
 Having a sharp understanding of 
 the algorithm 
 when $n$ is large and, at the same time, $\varepsilon$ is small is even more difficult. As we already explained, the constant $\exp(C \varepsilon^{-2})$ that appears in all of our estimates is certainly non-optimal in many situations. 
 %That said, we will see in the forthcoming 
 %Subsection \ref{subse:3:3} that the geometric variant of the fictitious seems to become especially relevant (in comparison with the harmonic one) 
%in the small viscosity regime. This is in contrast with the results reported in this subsection (which are for $\varepsilon=1$) and this is consistent with our initial intuition that there should be a competition between $n$ and $\varepsilon$.  

\subsubsection{Evolution of the learnt cost with one type of noise only}
\label{subsubse:M=1N=1}
We first test the influence of each of the two types of noises onto the behavior  of the algorithm. We thus compare the evolution of the cost learnt by the harmonic and geometric fictitious plays in three scenarios (Figure \ref{fig:0000}): In the first scenario, there is an independent noise, but no common noise (Algorithm 4); in the second scenario, there is a common noise but no independent noise and $\varpi=1$ (Algorithm 3, harmonic fictitious play); in the third scenario, there is a common noise but no independent noise and $\varpi=1.1$ (Algorithm 3, geometric fictitious play).

\begin{figure}[h!]
\begin{subfigure}[b]{0.3\linewidth}
\includegraphics[width=\linewidth]{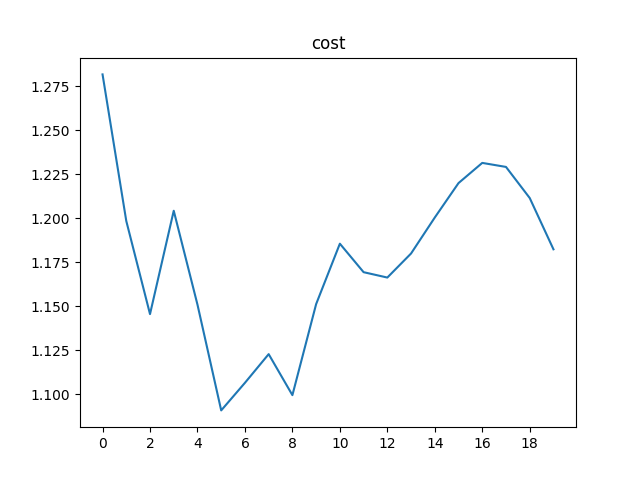}
\caption{With independent but no common noise}
\end{subfigure}
\begin{subfigure}[b]{0.3\linewidth}
\includegraphics[width=\linewidth]{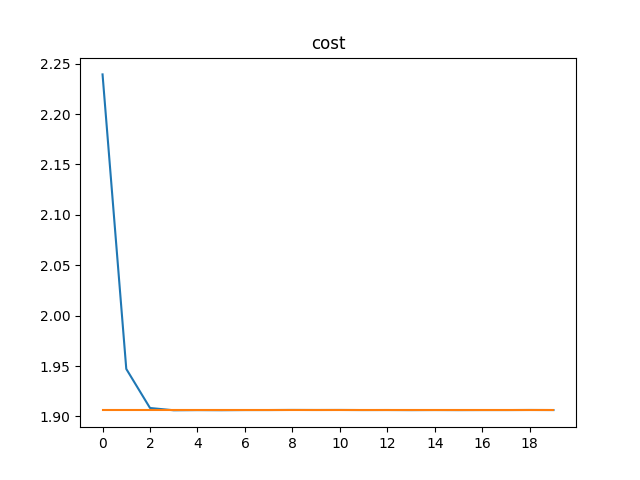}
\caption{With common but no independent noise, $\varpi=1$}
\end{subfigure}
\begin{subfigure}[b]{0.3\linewidth}
\includegraphics[width=\linewidth]{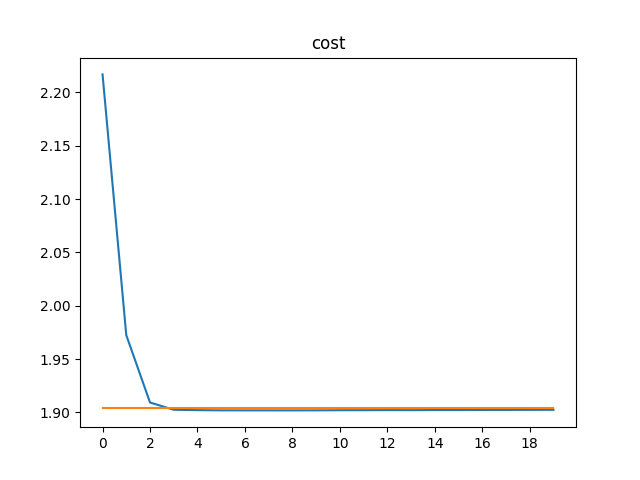}
\caption{With common but no independent noise, $\varpi=1.1$}
\end{subfigure}
\caption{Harmonic and geometric fictitious plays. Comparison of the learnt cost depending on the type of noise and the value of $\varpi$. In $x$-axis: number of iterations; In $y$-axis; learnt cost.}
\label{fig:0000}
\end{figure}
The experiments are computed with: 
$n=20$ learning iterations, $p=30$ time steps, 
$M =4\times 10^5$, $N=1$, 
$\sigma=1$ and $\varepsilon =0$ 
in case (A) 
and
$n=p=20$, 
$M=1$, 
$N=4\times 10^5$, 
$\sigma=0$, $\varepsilon=1$ and $D=4$ in cases (B) and (C). In ADAM method, the learning rate is $0.01$, with 15 epochs and one batch.

In plots (B) and (C), the orange line is the theoretical equilibrium cost as computed by the BSDE method explained in Subsection 
\ref{subse:reference}. In plot (A), there is no computed reference cost. As we already explained, there might not be a unique equilibrium and the notion of reference cost no longer makes sense. Notice by the way that the equilibrium cost in cases (B) and (C) is not an equilibrium cost in case (A) because the problems are different (as being set over different forms of dynamics). Anyway, the conclusion is clear: the learnt cost \textcolor{black}{exhibits an oscillatory behavior} in case (A), whereas it does not in cases (B) and (C). This is an evidence of the numerical impact of the common noise onto the behavior  of the
fictitious play. 
Of course, it would be desirable 
to have more theoretical guarantees on the possible divergence of the fictitious play in absence of the common noise. 
Unfortunately, we are not aware of such results. 
In our numerical experiments (including some that are not reported here), we have been able to reproduce the oscillatory behavior 
characterizing the pane (A) in Figure \ref{fig:0000} in other two-dimensional examples without common noise, meaning that, for a given batch (even of large size), 
the learnt (or training) cost may feature some non-trivial oscillations. However, we have not been able to reproduce\footnote{\color{black} To be complete on this point, 
the experiments that we did are for a class of coefficients of the same nature as in the $2d$ example, in particular with 
a terminal boundary condition exhibiting oscillations of the same frequency, 
see
\eqref{eq:benchmark:d=1} 
with $\kappa$ between $1$ and $10$.} a similar phenomenon 
in dimension 1 
(in which case the fictitious play is known to converge), even when the MFG is known to have multiple equilibria. As for panes (B) and (C), it is worth observing that the 
results are consistent. 
As announced, the
convergence is
fast, even in case (B), for which our theoretical guarantees just provide a harmonically decaying bound for the error. 
%As for the latter point, we refer to the preliminary discussion we gave at the beginning of this subsection.}

\subsubsection{Error in the optimal path/intercept with common noise and no independent noise}
We now compute the error achieved by the learning algorithm. 
Using the same notation as in the statement of Theorem \ref{main:thm:2}, we focus on the following $L^2$ error:
\begin{equation*}
\biggl[ {\mathbb E}
\int_{0}^1  \Bigl( 
\vert \overline m_{t}^{n}
- m_{t}^{(p)} \vert^2
+
\vert \varpi h_{t}^{n} - h_{t}^{(p)} \vert^2 
\Bigr) 
\ud t \biggr]^{1/2}.
\end{equation*}
Focusing here on the time-averaged error is obviously more advantageous from the numerical point of view, but it is sufficient to do so in order to demonstrate the efficiency of the algorithm and also to identify some of its limitations. Notice also that the expectation in the above error is taken under 
the non-tilted expectation. 
This looks a reasonable choice because $\varepsilon$ is here equal to 1. In particular, 
the Girsanov density 
${\mathcal E}({\boldsymbol h})$ 
and its inverse have bounded moments of any order (independently of $n$), as a consequence of which
integrating under either  measure should not make a big difference.  
%
%
%. 
%Also, since we have bounds for the moments of $\overline {\boldsymbol m}^n$
%and ${\boldsymbol m}$, 
%Proposition 
%\ref{prop:2} shows that the above expectation tends to 0 as $n$ tends to $\infty$. This is sufficient for our purpose since we do not try to identify the numerical rate of convergence precisely. 

Numerically, the error is approximated by 
\begin{equation}
\label{eq:numerical:error}
\biggl[
\frac1{Np} 
\sum_{j=1}^N 
\sum_{k=0}^{p-1}
\Bigl( 
\vert \overline m_{t_{k}}^{n,(j)}
- m_{t_k}^{\star,(j)} \vert^2
+
\vert h_{t_{k}}^{n,(j)} - \varpi h_{t_k}^{\star,(j)} \vert^2\Bigr) \biggr]^{1/2}, 
\end{equation}
where $(\overline m_{t_{k}}^{n,(j)})_{j,k}$ and 
$(h_{t_{k}}^{n,(j)})_{j,k}$ are the returns of Algorithm 2 or 3 (depending on the value of 
$\sigma$) and 
$(\overline m_{t_{k}}^{\star,(j)})_{j,k}$ and 
$(h_{t_{k}}^{\star,(j)})_{j,k}$
are the reference solutions computed with Algorithm 1 (with a sufficiently high number of Picard iterations: 10 in practice).

\begin{figure}[h!]
\begin{subfigure}[b]{0.3\linewidth}
\includegraphics[width=\linewidth]{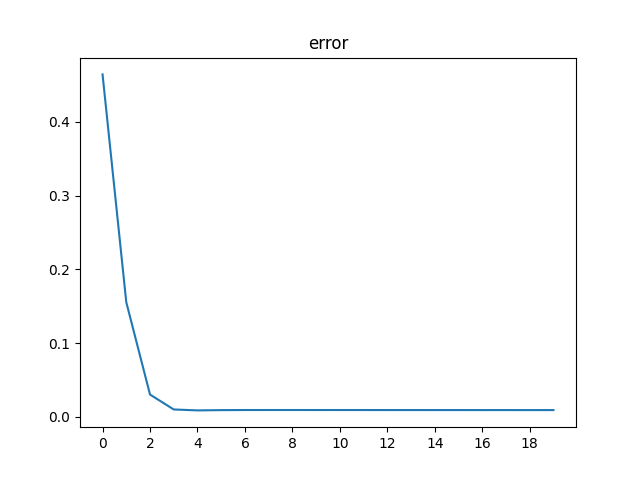}
\caption{Riccati known, $\varpi=1$}
\end{subfigure}
\begin{subfigure}[b]{0.3\linewidth}
\includegraphics[width=\linewidth]{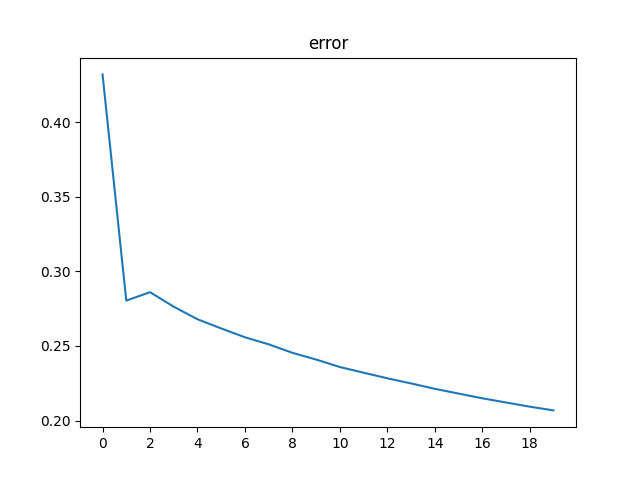}
\caption{Riccati unknown, $\varpi=1$}
\end{subfigure}
\\
\begin{subfigure}[b]{0.3\linewidth}
\includegraphics[width=\linewidth]{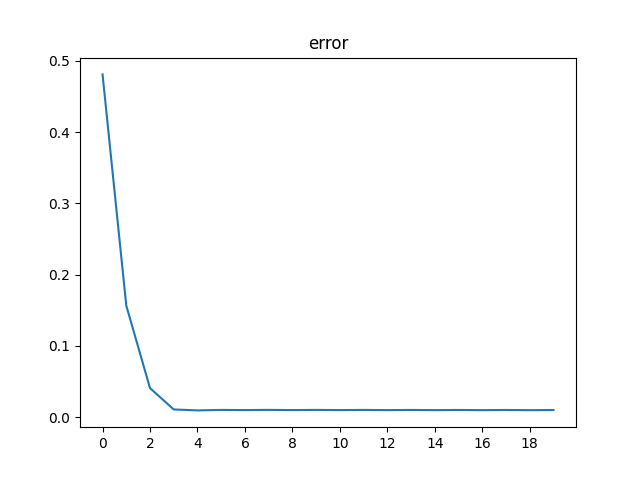}
\caption{Riccati known, $\varpi=1.1$}
\end{subfigure}
\begin{subfigure}[b]{0.3\linewidth}
\includegraphics[width=\linewidth]{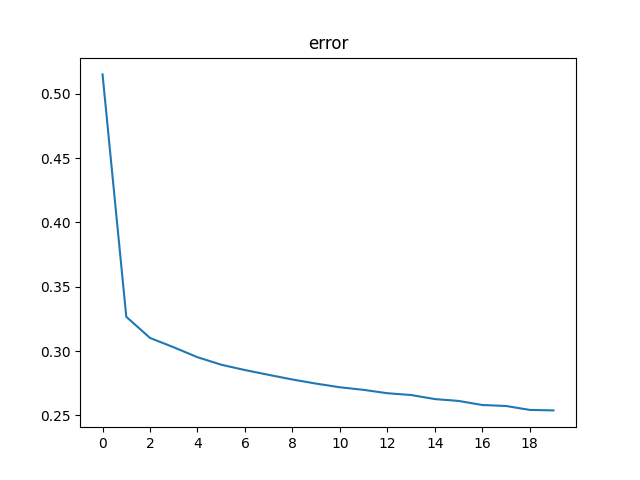}
\caption{Riccati unknown, $\varpi=1.1$}
\end{subfigure}
\caption{Comparison of the error returned by Algorithm 3 (common noise only), depending on whether the solution 
to the Riccati equation is known or not. On the top line, $\varpi=1$. On the bottom line, 
$\varpi=1.1$.}
\label{fig:10}
\end{figure}

We perform the first experiment by running Algorithm 3 (no idiosyncratic noise), but assuming that the solution
$(\eta_{t})_{0 \leq t \leq 1}$ to the Riccati equation 
\eqref{eq:riccati:test}
is known, see 
\eqref{eq:riccati:test:solution}: In Algorithm 3,  $a_{t_{k}}$ is replaced by $\eta_{t_{k}}$. This amounts to say that the coefficients $Q$ and $R$ in 
\eqref{eq:cost:intro} are explicitly known. 
The evolution of the error with the number of iterations is plotted in Figure 
\ref{fig:10}, Plot (A)  when $\varpi=1$ (harmonic fictitious play) and Plot (C) when $\varpi=1.1$ (geometric fictitious play). 
We perform the second experiment by running Algorithm 3, but without assuming any further knowledge of the solution of the Riccati equation 
\eqref{eq:riccati:test}. 
The evolution of the error with the number of iterations is plotted in Figure 
\ref{fig:10}, Plot (B) when $\varpi=1$ and Plot (D) when $\varpi=1.1$.

The experiments are computed with: 
$n=20$ learning iterations, $p=30$ time steps,  
$M=1$, $N =4\times 10^5$, 
$\sigma=0$, $\varepsilon =1$ and $D=4$.
On Plots (A) and (C), the error is less than $0.01$ after iteration 3. 

The main conclusion one may draw from the comparison of these plots is quite clear: when randomness only comes from the 
common noise, it is very difficult for the optimization  method to distinguish between 
$a_{t_{k-1}} x_{t_{k-1}}^{(1,j)}$ and $C_{t_{k-1}}^{(j)}$
in \eqref{eq:dynamics:numerics}. 
In this matter, there is also a difference between Plots (B) and (D), from which we deduce that the harmonic fictitious play behaves better than the geometric one in the absence of independent noise. Comparison of plots (A) and (C), which are very close, suggests that the difference between (B) and (D) mostly comes from the accuracy of the estimate $(a_{t_{k-1}})_{k=1,\cdots,p}$ of the solution to the Riccati equation
\eqref{eq:discretized:Riccati}. Our guess is that the presence of an additional biais in 
\eqref{eq:dynamics:numerics}
(due to the fact that $\varpi>1$) makes the estimation more difficult. On the contrary, 
when $\varpi=1$,
there is no bias and  
the term 
$C_{t_{k-1}}^{(j)} + \varpi  h^{n,(j)}_{t_{k-1}}$ in  \eqref{eq:dynamics:numerics} is very close to $0$. We think that this should help for the identification of $(a_{t_{k-1}})_{k=1,\cdots,p}$. %We observe a similar phenomenon below in presence of an idiosyncratic noise. 

\subsubsection{Error in the optimal path/intercept with both noises}
\label{para:error}
We now proceed with the same analysis but putting the two noises. 
As the total number of simulated paths is $N \times M$, this may be however rather costly. 
We proceed below by freezing the quantity
$N \times M$. 
 
\begin{figure}[h!]
\begin{subfigure}[b]{0.32\linewidth}
\includegraphics[width=\linewidth]{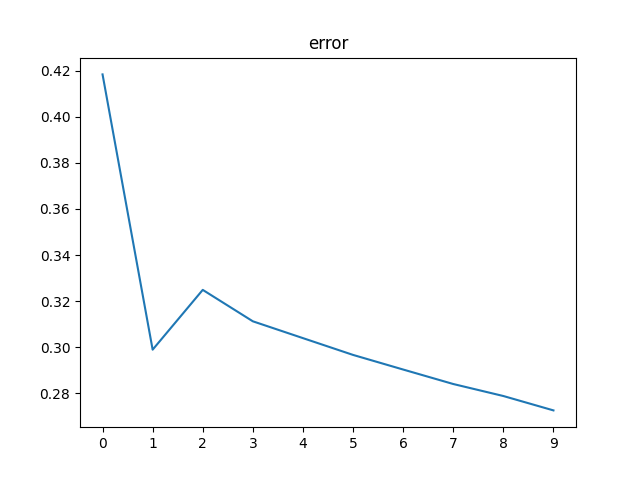}
\caption{$M=4\times 10^4$, $N=1$,  time (cpu): 4867}
\end{subfigure}
\begin{subfigure}[b]{0.32\linewidth}
\includegraphics[width=\linewidth]{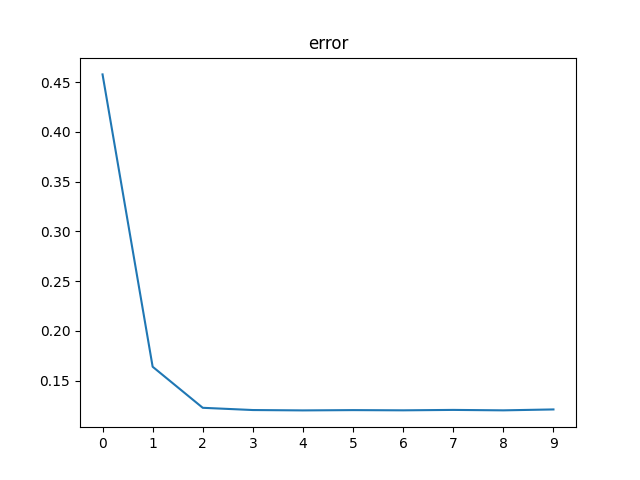}
\caption{$M=2\times 10^4$, $N=20$, time (cpu): 4096} 
\end{subfigure}
\begin{subfigure}[b]{0.32\linewidth}
\includegraphics[width=\linewidth]{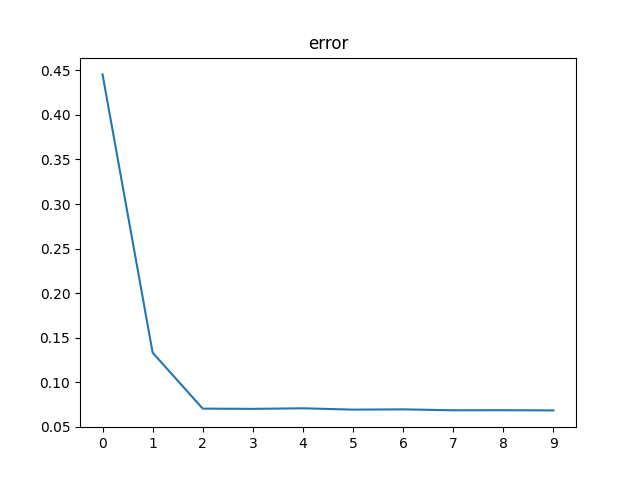}
\caption{$M=10^4$, $N=40$, time (cpu): 8409}
\end{subfigure}

\caption{Comparison of the error returned by Algorithm 2, depending on the number of particles/realizations, $\varpi=1$.}
\label{fig:11}
\end{figure}

The experiments are run with the 
harmonic version of the fictitious play (i.e., 
$\varpi=1$), with: 
$n=10$ learning iterations, $p=30$ time steps,  
$\sigma=1$ and $\varepsilon =1$;
4000 units
in cpu time is around 1h.
The conclusion 
is that idiosyncratic noises 
demonstrate to be useful from the numerical 
point of view, even though the results stated 
in Section 
\ref{Sec:2}
remain the same whatever the value
of $\sigma$. 
A careful inspection  
of the numerical results shows that, in fact,
 idiosyncratic noises 
 provide a better fit 
 of the solution to the Riccati equation, which is 
fully consistent with the conclusion of the previous paragraph. 
In turn, we guess that a model-based method, in which $Q$ and $R$ would be learnt first, and then the solution to the Riccati equation 
would be learnt separately, would make sense in this specific setting. 

When $\varpi=1.1$, we observe a phenomenon similar to the one reported in 
Figure 
\ref{fig:10}: the estimate is less accurate
with the geometric version of the fictitious play.  
Again, we believe that this is due to the learning of the solution to the Riccati equation. In presence of a bias, the letter seems more difficult to catch. To wit, we have plotted in Figure \ref{fig:10:++}, Plot (A),
the error for the 
geometric version of the fictitious play (here, 
$\varpi=1.1$) and for the same parameter as in Plot (C) in Figure \ref{fig:11}: 
$n=10$, $p=30$,  
$\sigma=1$ , $\varepsilon =1$, $M=10^4$ and $N=40$. 
In order to stress that the drawback observed on Plot (A) does not contradict our theoretical results, we have plotted in the same Figure, Plot (B), the results when the solution to the Riccati is known (which is a bit different from what is done in 
Figure 
\ref{fig:10}
because there is here an additional idiosyncratic noise which requires additional Monte-Carlo approximations). It is clear that, in the latter case, the result is better. Possibly, this opens the door for refined procedures in which one first estimates the solution to the Riccati equation.

\begin{figure}[h!]
\begin{subfigure}[b]{0.42\linewidth}
\includegraphics[width=\linewidth]{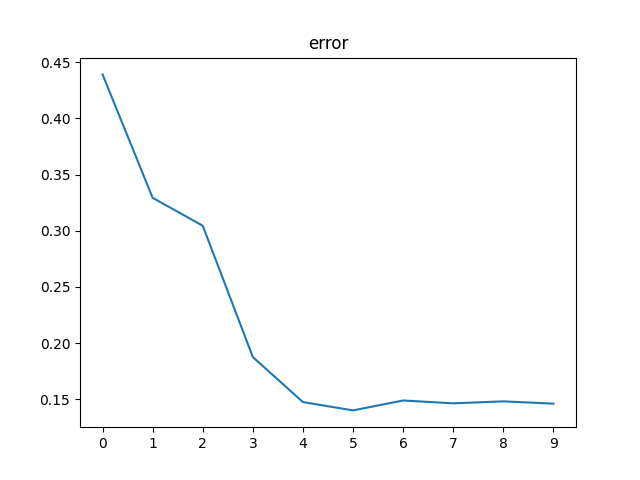}
\caption{$M=10^4$, $N=40$,  Riccati unknown}
\end{subfigure}
\begin{subfigure}[b]{0.42\linewidth}
\includegraphics[width=\linewidth]{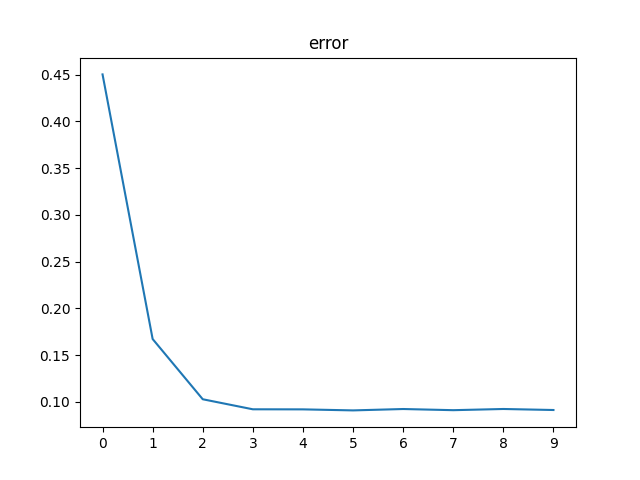}
\caption{$M=10^4$, $N=40$, Riccati known}
\end{subfigure}

\caption{Algorithm 2, $\varpi=1.1$, depending on whether Riccati is known or not.}
\label{fig:10:++}
\end{figure}

\subsubsection{Validation}
The reader could worry about a possible overfitting 
in our numerical experiments. 
Actually, in order to validate our results, 
we can
learn the coefficients ${\boldsymbol a}^n$ and ${\boldsymbol c}^n$ 
in Algorithm 2 for a first series of data 
in 
\eqref{eq:noises:simulation}. In brief, 
this first series of data is used to learn the equilibrium feedback.
Then, 
we can use a second series of data (say, of the same size) in \eqref{eq:noises:simulation} in order to 
compute the error. In clear, 
given this second series of data, 
we can implement 
Algorithm 1 and then compute the resulting 
error 
\eqref{eq:numerical:error}
when 
 $(\overline m_{t_{k}}^{n,(j)})_{j,k}$ and 
$(h_{t_{k}}^{n,(j)})_{j,k}$
therein are obtained from 
the formulas
\eqref{eq:dynamics:numerics}
and
\eqref{eq:approx:hermite} by using the coefficients ${\boldsymbol a}^n$ and ${\boldsymbol c}^n$ 
returned by the first series of data and then by implementing the Euler scheme with respect to the second series of data. 
The resulting plots are very similar to those 
represented in 
Figure \ref{fig:11}  (e.g., when $\varpi=1$). For this reason, we feel useless to insert them here. 
Intuitively, what happens is that there are sufficiently many realizations of the common noise in our experiments to 
guarantee, in the regression 
\eqref{eq:approx:optimisation}, a convenient form of averaging with respect to all these realizations.

\color{black}
\subsection{Experiments in higher dimension}
\label{subse:3:2:bis}
The challenge in higher dimension is twofold. The very main one is to provide an efficient regression 
of ${\boldsymbol h}^{n+1}$ onto $\overline{\boldsymbol m}^{n}$ 
(recall 
\eqref{eq:approx:general}). 
We already reported this difficulty and, as we already announced, we handle it here by using
neural networks. The second issue is directly related with the computational cost. Intuitively, all the tensors that enter the code include an axis corresponding 
to all the possible coordinates of the noises and the states. When the dimension increases, the size of those tensors increase accordingly, which 
impacts the computational effort.
In all the examples below, $\sigma=\varepsilon=1$. 
For simplicity, we just present the result in the case $\varpi=1$ because the difficulties that we report below are the same in the case $\varpi>1$.

\subsubsection{Results with Algorithm 2.}
\label{subsubse:high:dim:1}
In the examples below, we reduce part of the complexity by 
labelling the increments of the independent noises 
$\Delta_{t_k} b^{(i,j)}$
in 
\eqref{eq:noises:simulation}
by the sole 
$i$, which allows us to save memory. 
 We refer to Remark 
\ref{rem:bij} for more details about this. 
We also limit ourselves to batches 
carrying $N=10^3$ realizations of the common noise 
and $M=10^2$ realizations of the independent noise. 
These numbers are a bit less than those used 
in the lower dimensional setting (compare for instance with 
Figure 
\ref{fig:11}). However, differently from what we did in the previous subsection, we now use several batches. In the experiments below, the batches are indexed by 
an index, called \textit{batch}. For a given value $i$ of this index \textit{batch}, we simulate one stack, say $G_{\rm com}[i]$, of $N=10^3$ realizations of the common noise and then a number $m_{\rm ind}$ of sub-batches (or sub-stacks) containing, each, $M=10^2$ realizations
of 
the independent noise. 
Those sub-batches are denoted $G_{\rm ind}[i,j]$, for $j=1,\cdots,m_{\rm ind}$. 
The batches contain independent realizations. In the experiments below, $m_{\rm ind}=10$. 

When we start experiments with a new value $i$ of the index $\textit{batch}$, we perform several runs of ADAM. All these runs use the same batch $G_{\rm com}[i]$ 
of common noises. We then repeat several times the following episode: we do a first run with the first batch 
$G_{\rm ind}[i,1]$
of independent noises, and then a second one with the second batch $G_{\rm ind}[i,2]$, and so on and so forth up until the batch
$G_{\rm ind}[i,m_{\rm ind}]$ of index $m_{\rm ind}$. We repeat those episodes several times: 4 times in the experiments ran below. 
Our choice to proceed in such a way is dictated by the fact that the number $N$ is pretty low, hence the desire to have more batches for the independent noises. 
As for the choice $N \gg M$, it is dictated by our wish to have the lowest possible fluctuations with respect to the common noise, 
as the main object that we want to learn here is ${\mathfrak h}$ in 
\eqref{eq:approx:general}, the input of which is measurable with respect to the common noise only.  
In our experiments, the index \textit{batch} is running over $\{1,\cdots,20\}$. We have chosen to pass once on the realizations of the common noise. 
 
For some choices of the matrix $\Theta$ in  
\eqref{eq:benchmark:d:higher}, our results are bad. In order to appreciate this observation, 
 it is worth recalling that our strategy relies on a change of measure based on 
the Girsanov density
${\mathcal E}({\boldsymbol h}^n)$, see 
\eqref{eq:Girsanov}. In our code, 
the training cost is precisely estimated under this new probability measure, as made clear in 
\eqref{eq:approx:optimisation}. 
Drawing a parallel with preferential sampling in Monte Carlo methods,
one may then guess that the variance 
resulting from the empirical estimate
of the loss in 
\eqref{eq:approx:optimisation} 
may 
dramatically 
 deteriorate as the dimension increases.
 Indeed,  
if ${\boldsymbol h}$ in 
\eqref{eq:Girsanov} has all its coordinates constant equal to $1$, then 
\begin{equation*}
{\mathbb E} 
\Bigl[ {\mathcal E}({\boldsymbol h})^2\Bigr] = \exp \Bigl( \int_0^T \vert h_s \vert^2 ds \Bigr) = \exp(d). 
\end{equation*}
Even more, the fact that the control appearing in 
\eqref{eq:approx:optimisation} contains 
a finite difference of the common noise (see 
\eqref{eq:dynamics:numerics}
) says that the typical size of the variance may be of order 
$p^2 \exp(d)$. 

In order to validate numerically this intuition, we consider the case when  
the matrix $(\Theta_{i,j})_{1 \leq i,j \leq d}$ in 
\eqref{eq:benchmark:d:higher} is given by 
$\Theta_{i,2i  (\rm{mod})d}=-10$, 
$\Theta_{i,3i (\rm{mod})d}=10$
and
$\Theta_{i,j} = 0$ otherwise, where $2i (\rm{mod})d$ is the rest of the Euclidean division of $2i$ by $d$ 
(and similarly for $3i({\rm mod})d$). 
We then represent in Figure 
\ref{fig:high:dim:variance}
the evolution of the variance of the term
\begin{equation}
\label{eq:high:dim:dominant}
{\mathcal E}({\boldsymbol h}) \int_0^T \vert \dot{W}_t^{p,{\boldsymbol h}} \vert^2 dt,
\end{equation}
when 
${\boldsymbol h}$ therein is computed numerically by the FBSDE solver 
described in
Subsection 
\ref{subse:reference}.  We expect this term to give 
the worst contribution to the global variance. As shown by 
the plot in Figure \ref{fig:high:dim:variance}, the log variance becomes 
higher at multiples of 6 (because of the choice $\Theta$)
and the plot even suggests that the variance could blow up exponentially fast along 
the subsequence 6, 12, 18, 24... 
\begin{figure}[h!]
\includegraphics[width=.4\linewidth]{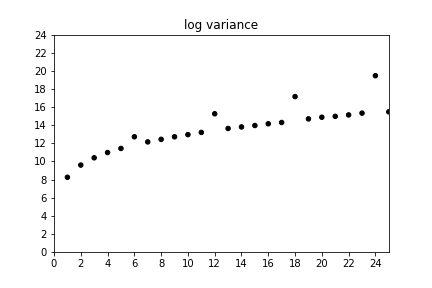}
\caption{Evolution of the log variance of the loss with the dimension for $\Theta_{i,2i  (\rm{mod})d}=-10$, 
$\Theta_{i,3i (\rm{mod})d}=10$ in \eqref{eq:benchmark:d:higher}. Dimension is in $x$-axis. Log variance is in $y$-axis.}
\label{fig:high:dim:variance}
\end{figure}
\vskip 5pt

We thus chose to run Algorithm 2 with $d=12$ with aforementioned values of $M$ and $N$ and with $p=20$ time steps. The results are reported in Figure 
\ref{fig:high:dim:algo2} below. 
\begin{figure}[h!]
\begin{subfigure}[b]{0.4\linewidth}
\includegraphics[width=\linewidth]{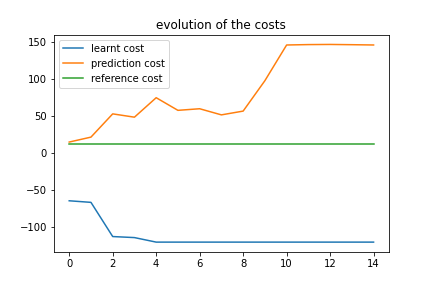}
\caption{Evolution of the costs}
\end{subfigure}
\begin{subfigure}[b]{0.4\linewidth}
\includegraphics[width=\linewidth]{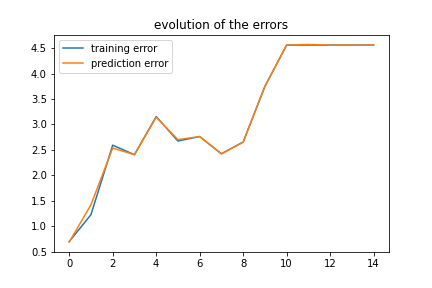}
\caption{Evolution of the errors}
\end{subfigure}
\caption{Results in dimension
$d=12$ for $\Theta_{i,2i  (\rm{mod})d}=-10$, 
$\Theta_{i,3i (\rm{mod})d}=10$ in \eqref{eq:benchmark:d:higher}. Harmonic variant of the fictitious play ($\varpi=1$).}
\label{fig:high:dim:algo2}
\end{figure}
We clearly observe that the learnt cost blows-up. 
Subsequently, the training error does not vanish, 
with the training error being here computed as in 
\eqref{eq:numerical:error}
but solely for the ${\boldsymbol h}$ part, that is
\begin{equation}
\label{eq:numerical:error:2}
\biggl[
\frac1{Np} 
\sum_{j=1}^N 
\sum_{k=0}^{p-1} 
\vert h_{t_{k}}^{(j)} - h_{t_k}^{\star,(j)} \vert^2 \biggr]^{1/2}.
\end{equation}
In the plot, we also include a prediction cost and a prediction error, which are computed on a batch
different from the batches used for training (but of the same size). Given a new batch for the common and independent noises, 
we can indeed use the neural network returned by the training phase to compute, for this new batch, a prediction of the cost (together with a prediction of 
${\boldsymbol h}$), but under the historical probability measure ${\mathbb P}$ (in order to avoid any variance issues). 
The prediction is thus constructed as follows. 
Given the returns $(a_{t_k})_{k=0,\cdots,p-1}$
 and 
 $({\mathfrak h}_{t_k})_{k=0,\cdots,p-1}$ of 
 the neural network for 
 the affine feedback, see
  \eqref{eq:numerical:scheme:alphatk:xtk:ctk}
 and
 \eqref{eq:approx:general}, we implement the 
 following Euler scheme:
\begin{equation*}
\begin{split}
&x_{t_{k}}^{(i,j)} = x_{t_{k-1}}^{(i,j)} + \frac{1}{p}
\alpha_{t_{k}}^{(i,j)} + {\frac{1}{\sqrt{p}}} \Delta_{t_{k}} b^{(i)}
 +
 {\frac{1}{\sqrt{p}}}
 \Delta_{t_{k}} w^{(j)}, 
\quad \ell=1,\cdots,p \ ; 
\quad x_{0}^{(i,j)}=x_{0},
\\
&\alpha_{t_{k}}^{(i,j)} = a_{t_{k-1}} x_{t_{k-1}}^{(i,j)} + {\mathfrak h}_{t_{k-1}}(\overline{m}_{t_{k-1}}^{(j)}), \quad k=1,\cdots,p,
\end{split}
\end{equation*}
with 
$\overline{m}_{t_{k-1}}^{(j)}=M^{-1}
\sum_{i=1}^M
x_{t_{k-1}}^{(i,j)}$. 

The predicted cost is then
\begin{equation*}
\frac1{N}
\sum_{j=1}^N 
\frac1M \sum_{i=1}^M
\biggl[  \frac{1}{2p} \sum_{k=1}^{p} 
\vert \alpha_{t_{k}}^{(i,j)} \vert^2  + \frac12 \Bigl\vert x^{(i,j)}_{1} + g\Bigl(\overline{m}_{1}^{n,(j)}\Bigr)
\Bigr\vert^2
\biggr],
\end{equation*}
and the predicted error is 
\eqref{eq:numerical:error:2}.

While their role is rather limited at this stage, the two prediction cost and error play a more important role in the next paragraph.

%it is worth mentioning that, in dimension $d=13$, it does provide a good estimate.  
%The intuition is that the variance of the dominant term in the learning loss is smaller (see Figure
%\cite{fig:high:dim:variance}). In turn, the learning error is smaller. Since the prediction loss is computed under the historical probability measure, 
%it does not suffer from a large variance. Therefore, the empirical estimate returned by the batch is quite good. 
% 
%

\subsubsection{Variance reduction.}
The issue reported in the previous paragraph calls for the implementation of a variance reduction method.
The method that is addressed in this new paragraph 
aims at reducing the variance associated with the dominant term 
\eqref{eq:high:dim:dominant}. 
The key point is to return back to 
\eqref{eq:effective:cost}, to write the normalization factor $-d \, \varepsilon^2 p/2$ as
 the expectation
 \begin{equation}
  \label{eq:variance:reduction:1}
\varepsilon 
{\mathbb E}^{{\boldsymbol h}_n} \int_0^T \alpha_t \dot{W}_t^{p,{\boldsymbol h}^n} dt 
- 
\frac{\varepsilon^2}2 
{\mathbb E}^{{\boldsymbol h}_n} \int_0^T \bigl\vert \dot{W}_t^{p,{\boldsymbol h}^n}\bigr\vert^2 dt, 
 \end{equation}
which is indeed equal to $-d\, \varepsilon^2 p /2$, and then 
to approximate the above two expectations by two empirical means. 
Intuitively (and this is the computation achieved in \eqref{eq:effective:cost}), the resulting estimator of the cost coincides
with the estimator that would be computed if
we had to solve
(the reader may compare the following cost with 
 \eqref{eq:optimization:exploration}, with $\varpi=1$ for simplicty)
 \begin{equation}
 \label{eq:variance:reduction:2}
\textrm{\rm argmin}_{\boldsymbol \alpha}
	 {\mathbb E}^{{\boldsymbol h}^n}
	\bigl[ {\mathcal R}^p \bigl( {\boldsymbol \alpha} ; \overline{\boldsymbol m}^n;\varepsilon {\boldsymbol W}^p \ \bigr) \bigr],
	\end{equation}
	which is implicitly set over the dynamics 
	\begin{equation}
	 \label{eq:variance:reduction:3}
	\ud X_{t} = \alpha_{t} \ud t + \sigma \ud B_{t} + \varepsilon \ud W_{t}^{p,{\boldsymbol h}^n}, \quad t \in [0,T]. 
\end{equation}
Here, we have two formulations depending on the line that is used in 
 \eqref{eq:optimization:exploration}: $(i)$ We have the formulation 
\eqref{eq:variance:reduction:2}--\eqref{eq:variance:reduction:3},
which is based on the second line in 
 \eqref{eq:optimization:exploration}; $(ii)$ And we have
the formulation based on the top line in 
 \eqref{eq:optimization:exploration} 
 when the corrector $-d \, \varepsilon^2 p/2$ therein is replaced by the empirical mean
deriving from \eqref{eq:variance:reduction:1}. Interestingly, the two formulations do not have the 
same practical interpretation. Indeed, the formulation $(i)$ does not fit 
the principle of Figure 
\ref{fig:2}
in introduction 
 since the common randomization therein does not impact the control directly, but impacts the dynamics. 
 Differently, 
 the formulation $(ii)$ based on the empirical corrector
 associated with  
 \eqref{eq:variance:reduction:1}
preserves the principle 
of Figure \ref{fig:2} since the corrector can be computed separately from the control. Of course, it has a 
practical price: 
implicitly, 
computing 
the
correction empirically 
 makes 
sense only if the dependence of the running cost upon the control is known. In other words,  
part of the model must be known. 

Figure \ref{fig:variance:reduct} below provides 
an estimate of the 
variance of 
the (random) cost in \eqref{eq:variance:reduction:2}
when ${\boldsymbol h}^{n}$ is replaced by 
the numerical solution ${\boldsymbol h}$
of the FBSDE solver 
described in
Subsection  
\ref{subse:reference}
and when $\Theta$ is computed as in Figure
\ref{fig:high:dim:variance}. 
As the reader can notice, the growth is slower (than in Figure
\ref{fig:high:dim:variance}). 
\begin{figure}[h!]
\includegraphics[width=.4\linewidth]{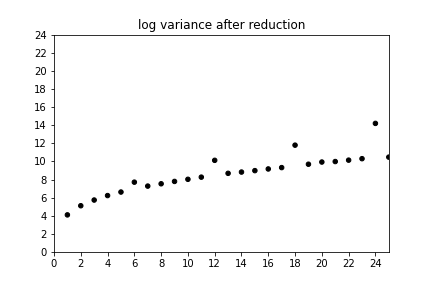}
 \caption{Evolution of the log variance of the corrected optimal cost 
  \eqref{eq:variance:reduction:2}.
 with the dimension for $\Theta_{i,2i  (\rm{mod})d}=-10$, 
$\Theta_{i,3i (\rm{mod})d}=10$ in \eqref{eq:benchmark:d:higher}. Dimension is in $x$-axis. Log variance is in $y$-axis. Harmonic variant of the fictitious play ($\varpi=1$).}
\label{fig:variance:reduct}
\end{figure}

The estimator for 
\eqref{eq:variance:reduction:1}
writes as in \eqref{eq:approx:optimisation}, 
namely
\begin{equation*}
\begin{split}
&\frac{\varepsilon}{N}
\sum_{j=1}^N \biggl(
{\mathcal E}^{n,(j)}
\frac1M \sum_{i=1}^M
\biggl[  \frac{1}{2p} \sum_{k=1}^{p} 
 \alpha_{t_{k}}^{(i,j)} 
 \cdot \Bigl( h^{n,(j)}_{t_{k}}
+p  \Delta_{t_{k+1}} w^{(j)} \Bigr) 
\biggr]\biggr)
\\
&\hspace{15pt} -
\frac{\varepsilon^2}{2N}
\sum_{j=1}^N \biggl(
{\mathcal E}^{n,(j)}
  \frac{1}{2p} \sum_{k=1}^{p} 
  \Bigl\vert h^{n,(j)}_{t_{k}}
+p  \Delta_{t_{k+1}} w^{(j)} \Bigr\vert^2 \biggr).
\end{split}
\end{equation*}

The results are represented in Figure \ref{fig:10:NN}
for the same choice of
$d$ and $\Theta$
as in Figure
\ref{fig:high:dim:algo2}. Obviously, they are much more convincing
than
in Figure \ref{fig:10}. On the left pane (A), the reference cost  is computed by
solving first 
the FBSDE 
described in
Subsection 
\ref{subse:reference} on a series of batches and then by 
averaging the 
corresponding costs over the batches. 
In contrast, the empirical cost is computed by solving the FBSDE
and the associated cost  
on the batch on which the training loss is computed. 
The fact that the reference and empirical costs do not coincide shows that there 
is, in between, some additional Monte-Carlo error that is distinct from the learning procedure.

\begin{figure}[h!]
\begin{subfigure}[b]{0.4\linewidth}
\includegraphics[width=\linewidth]{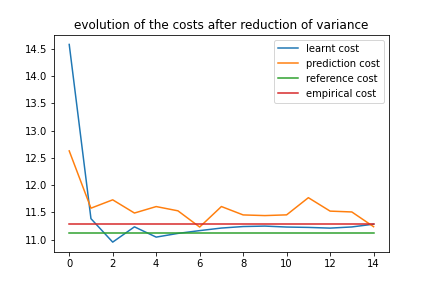}
\caption{Evolution of the losses}
\end{subfigure}
\begin{subfigure}[b]{0.4\linewidth}
\includegraphics[width=\linewidth]{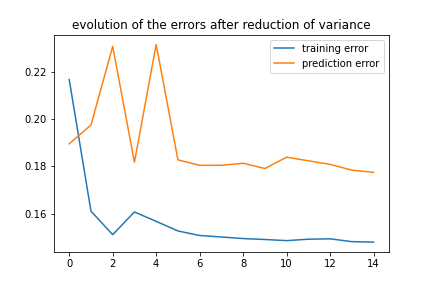}
\caption{Evolution of the errors}
\end{subfigure}
\caption{Results in dimension
$d=12$ for $\Theta_{i,2i  (\rm{mod})d}=-10$, 
$\Theta_{i,3i (\rm{mod})d}=10$ in \eqref{eq:benchmark:d:higher}. Harmonic variant of the fictitious play ($\varpi=1$).}
\label{fig:10:NN}
\end{figure}

The prediction cost and error are computed on a series of 10 batches (of the same size as before). The cost and the error that are plotted are obtained by averaging out on the batches. 
We observe a small bias between the training and prediction errors that would deserve further experiments. Anyway the main message is clear: the plots after variance reduction are much better than before variance reduction. 

We now provide another example, with a new choice of $\Theta$, that exhibits a less trivial evolution of the prediction error. Namely, we choose 
$\Theta=(\theta_{i,j}=(i+j)/(2d))_{i,j}$.
The results in Figure \ref{fig:20} 
may be summarized as follows: the relative error on the cost is less than $0.03$
and the prediction error is less than $0.2$.
As for the latter, it must be recalled that the prediction error 
writes as a $d$-dimensional norm, see 
\eqref{eq:numerical:error:2}. In particular, the mean square error per coordinate is less than $0.2^2/d$. Here, 
$d=12$ and the mean error per coordinate is less than $0.06$. 
We address the same example in Figure 
\ref{fig:21} in dimension $d=20$. The absolute error on the cost is good, except at iterations 11 and 13, but even for these two the relative 
error is less than $0.05$. 
The prediction error is between $0.22$ and $0.25$ after iteration 6, except at iterations 11 and 13, where the prediction error reaches 0.32. 
In any case, the mean error per coordinate is less than $0.07$ after iteration 6.

\begin{figure}[h!]
\begin{subfigure}[b]{0.4\linewidth}
\includegraphics[width=\linewidth]{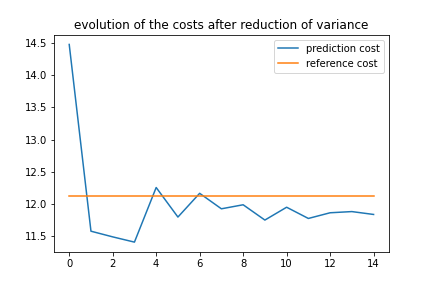}
\caption{Evolution of the losses}
\end{subfigure}
\begin{subfigure}[b]{0.4\linewidth}
\includegraphics[width=\linewidth]{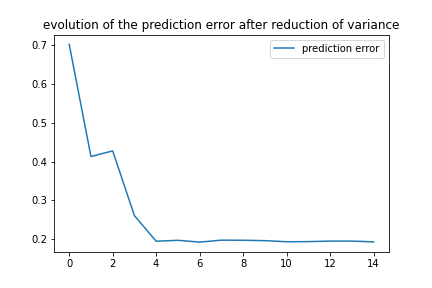}
\caption{Evolution of the errors}
\end{subfigure}
\caption{Results in dimension
$d=12$ for
$\Theta=(\theta_{i,j}=(i+j)/(2d))_{i,j}$ in \eqref{eq:benchmark:d:higher}. The reference loss is around 12.17 and the prediction loss at iteration 14 is 11.84. The prediction error at iteration 14 is around 0.19. Harmonic variant of the fictitious play ($\varpi=1$).}
\label{fig:20}
\end{figure}

\begin{figure}[h!]
\begin{subfigure}[b]{0.4\linewidth}
\includegraphics[width=\linewidth]{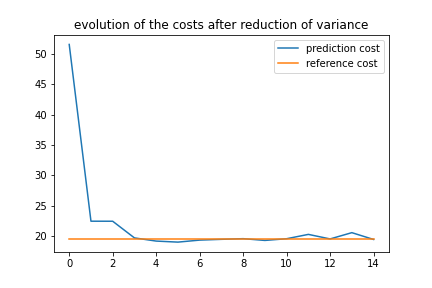}
\caption{Evolution of the losses}
\end{subfigure}
\begin{subfigure}[b]{0.4\linewidth}
\includegraphics[width=\linewidth]{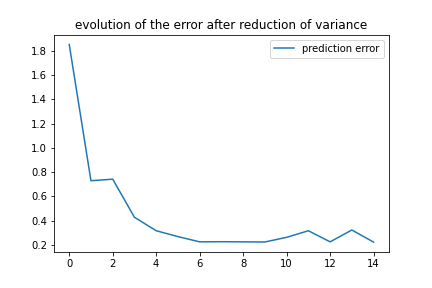}
\caption{Evolution of the errors}
\end{subfigure}
\caption{Results in dimension
$d=20$ for $\Theta=(\theta_{i,j}=(i+j)/(2d))_{i,j}$ in \eqref{eq:benchmark:d:higher}. The reference loss is around 19.47 and the prediction loss at iteration 14 is 19.45. The prediction error at iteration 14 is around 0.22. Harmonic variant of the fictitious play ($\varpi=1$).}
\label{fig:21}
\end{figure}

\color{black}

\subsection{Small viscosity}
\label{subse:3:3}

The next question in our numerical experiments is to address the behavior of the algorithm as the viscosity tends to $0$. As made clear in the analysis performed in Section 
\ref{Sec:2}, the influence of the small viscosity may manifest in an exponential manner
and this was our original motivation to design a geometrically converging scheme. 
%
%
%
%
%in several manners, depending on the error that is addressed: recall indeed that 
%the bound in the statement of Proposition 
%\ref{prop:2} does not depend on $\varepsilon$, whilst the bound provided by 
%the main Theorem 
%\ref{main:thm:2}
%blows up as $\varepsilon$ tends to $0$. 
%

Beyond the theoretical challenge raised by the possible occurrence of singularities in the 
vanishing viscosity limit (an example of which is given by 
Lemma \ref{lem:2}), small viscosity may also create 
additional numerical instability phenomena. 
As a main example, 
high variances may occur in the computations of empirical measures under the tilted probability measure. 
Obviously, this follows from the fact that 
the Girsanov density driving the change of reference measure in the algorithm becomes highly singular 
as the viscosity tends to $0$. We reported a similar drawback in the previous subsection, but 
in the large dimensional framework.
In the first arXiv version \cite{DelarueVasileiadis} of this work, the observation of this phenomenon prompted us to introduce a
method at the intersection between annealed simulating 
and preferential sampling, which we illustrated on a specific example that is recalled below. 
While this method is performing well (we revisit it in the framework of the geometric 
fictitious play in this subsection), it requires to repeat the scheme for 
several values of the viscosity and it is thus rather costly. 

Instead, we here show that the geometric fictitious play can quickly return  a very accurate numerical solution to the vanishing viscosity property 
problem (at least in the example addressed in the first arXiv version \cite{DelarueVasileiadis}). This is a very striking exemplification of our approach. It clearly demonstrates the benefits of choosing 
a higher value of the rate $\varpi$  and therefore of working with the geometric variant (instead of the harmonic variant) of the fictitious play. 
Our numerical experiment based on the aforementioned benchmark model. We choose 
a variant of $g$ in 
\eqref{eq:benchmark:d=1}, with $d=1$:
\begin{equation}
\label{eq:benchmark:d=1:variant}
g(x) = \cos\bigl( \kappa ( x-x_{0}) \bigr) - 2x_{0}, 
\end{equation}
with $\kappa=10$ and where $x_{0}$ is a root of the equation 
\begin{equation*}
\cos \bigl( \kappa x_{0} \bigr) = 2x_{0}. 
\end{equation*}
Numerically, we find that a choice is $x_{0}\approx -0.384$. 
The motivation for such an $x_{0}$ is that $0$ is 
a solution of \eqref{eq:mean_field:d:1:equilibrium}. In other words, 
$0$ is an equilibrium. 
Even more, 
we observe that, if the viscosity is zero, then the iterative sequence defined through the two 
updating rules 
\eqref{eq:backward:step:n}
and
\eqref{eq:forward:step:n}
remains in $({\boldsymbol m}^{n},{\boldsymbol h}^n)=(0,-2x_{0})$
for any $n \geq 1$ if 
${\boldsymbol m}^{0}$ is chosen as $0$. In other words, the standard fictitious (without any exploration) play converges (as predicted by the theory since this 1-dimensional MFG is potential, see  \S \ref{subse:potential}) and chooses the $0$-equilibrium. 
In order to compare with the prediction method exposed in  \S \ref{subse:potential}, we have plotted the corresponding potential
(which is given by a primitive $G$ of $g$). The plots are given in Figure \ref{fig:selection:potential}. We observe that $0$ is just a local minimizer of the potential and that the global minimizer  is around $-0.5$. 
\begin{figure}[h!]
	\begin{subfigure}[b]{0.4\linewidth}
		\includegraphics[width=\linewidth]{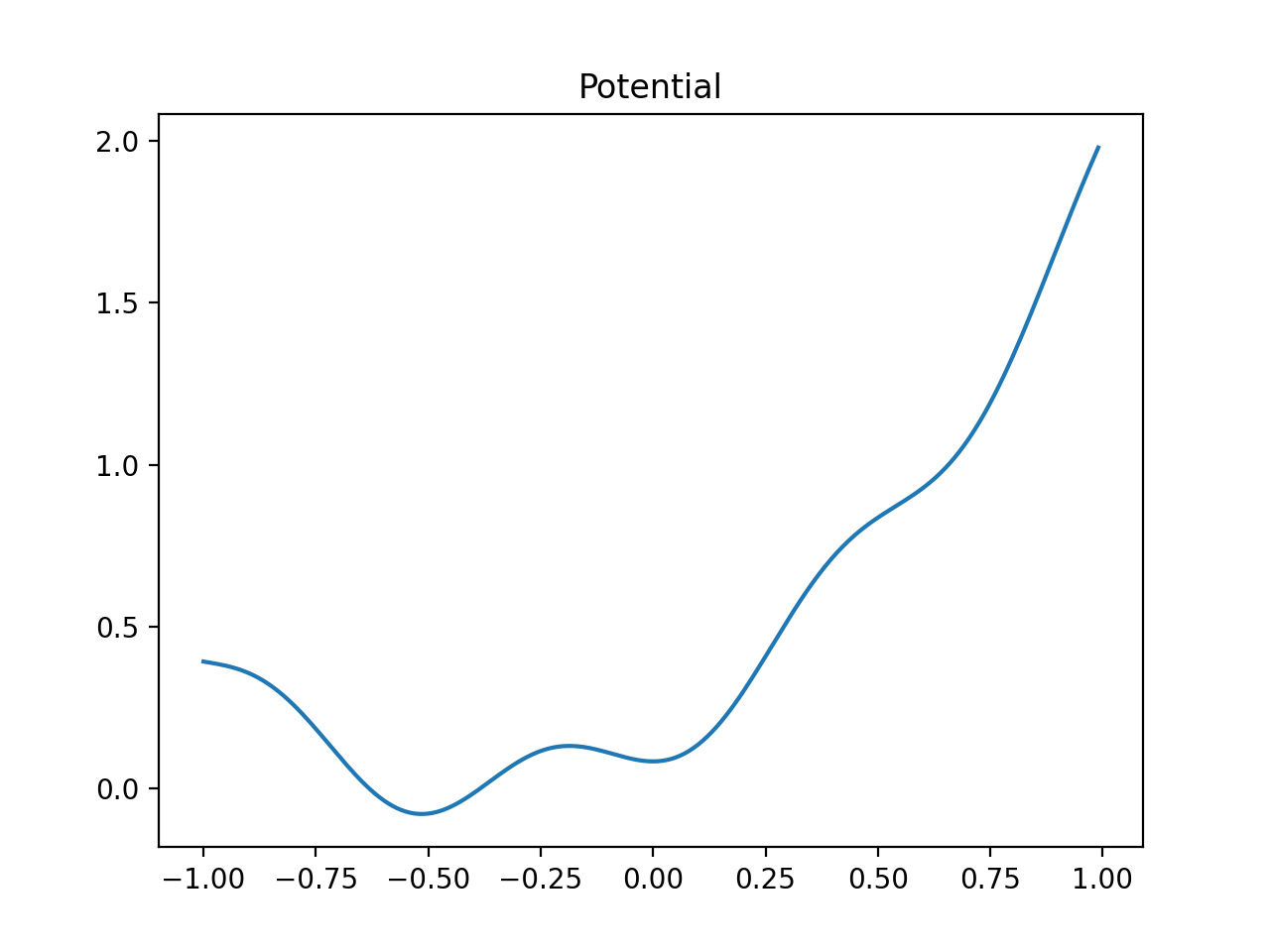}
		%\caption{Evolution of ${\mathscr J}$ with $\kappa$.}
	\end{subfigure}
	\begin{subfigure}[b]{0.4\linewidth}
		\includegraphics[width=\linewidth]{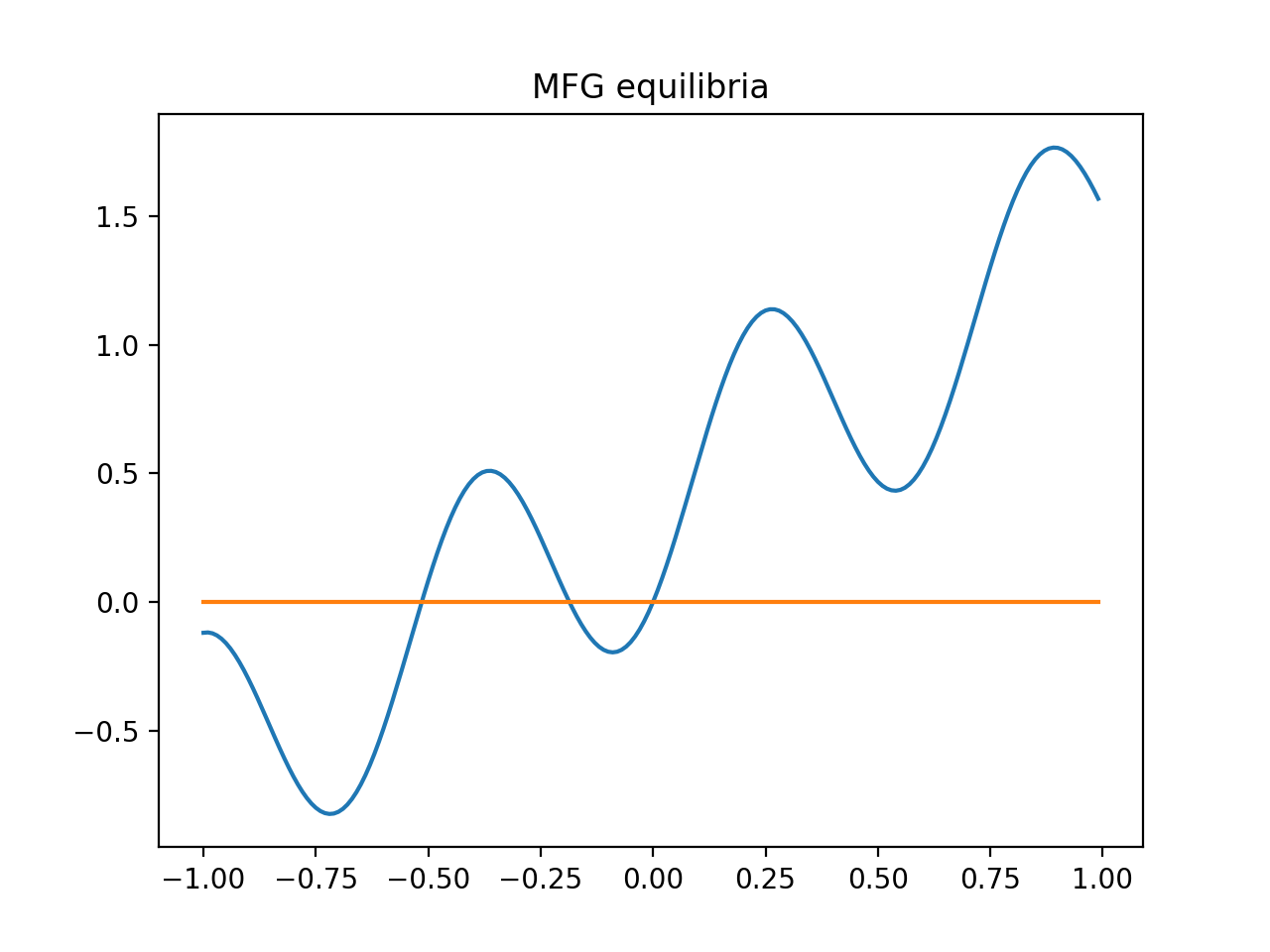}
		%\caption{Evolution of $\frac{\ud}{\ud \beta}{\mathscr J}$ with $\kappa$.}
			\end{subfigure}
		\caption{Equilibrium predicted by the potential rule: Potential $G$ on the left pane; Zeros of the function $g$ in the right pane.}
\label{fig:selection:potential}
\end{figure}

Our geometric fictitious play is able to rule out the $0$-equilibrium and to retain the other one.
In order to check this numerically, we have used three values for $\varpi$: $\varpi=4$, $\varpi=1.5$ and 
$\varpi=1$. As for the other parameters, we have chosen 
  $\sigma=0$ (no idiosyncratic noise, but Riccati is known), 
$M=2 \times 10^4$ (number of Monte-Carlo simulations)
and $p=100$. 
In all the experiments, $\varepsilon$ is set equal to $0.2$. 
The algorithm is initialized from the local minimizer. 
Our regression on the Hermite polynomials goes up to polynomials of degree 6. 
The results are presented in 
Figures 
\ref{fig:selection}, 
\ref{fig:selection:2}  and 
\ref{fig:selection:3}. Therein, we have plotted the histogram, under the tilted probability measure and at the end of each episode, of the terminal conditional mean of the system.
Even though the regime $\varpi=4$ is outside the scope of Theorem
\ref{main:thm:1} (because we assumed $\varpi \in (1,\sqrt{2}]$), the result 
with this choice
is quite impressive: the right equilibrium is selected in five iterations only. 
When $\varpi=1.5$, ten iterations are necessary. 
When $\varpi=1$, the results are satisfactory with around fifteen iterations but it is clear that the histogram features a kind of residual variance, which makes the selection 
less obvious. We believe that these plots are a strong case for supporting our results.

\begin{figure}[h!]
	\begin{subfigure}[b]{0.26\linewidth}
		\includegraphics[width=\linewidth]{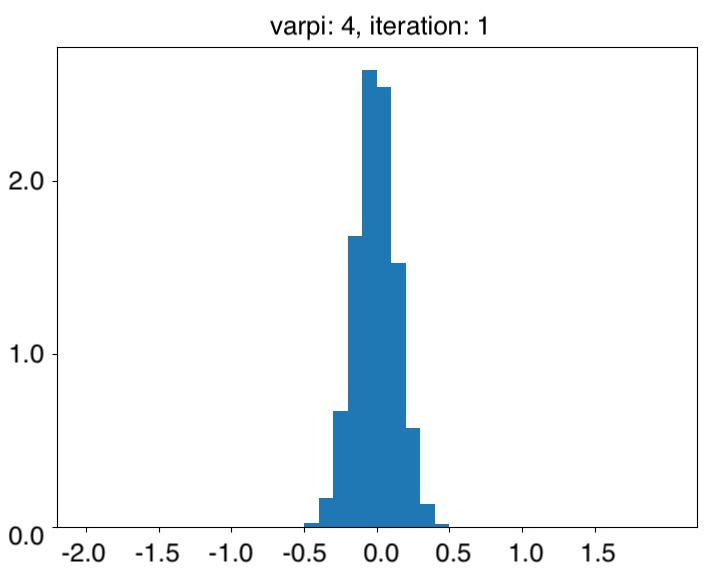}
		%\caption{Evolution of $\frac{\ud}{\ud \beta}{\mathscr J}$ with $\kappa$.}
	\end{subfigure}
	\begin{subfigure}[b]{0.26\linewidth}
		\includegraphics[width=\linewidth]{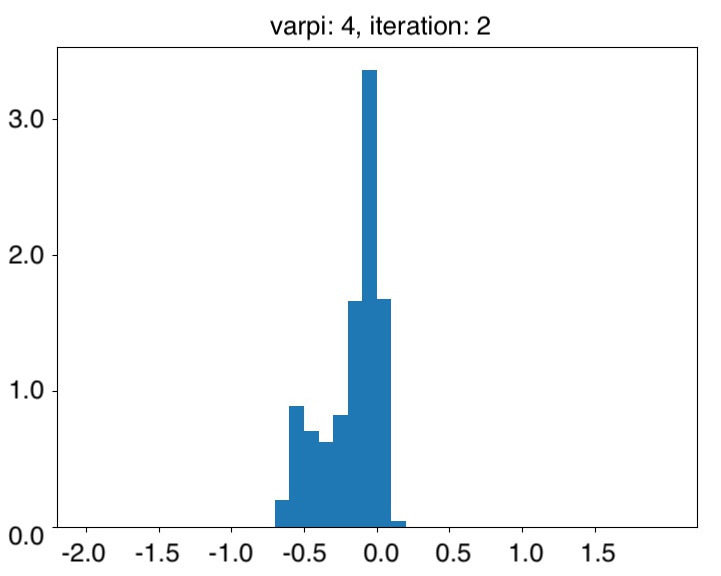}
		%\caption{Evolution of $\frac{\ud}{\ud \beta}{\mathscr J}$ with $\kappa$.}
	\end{subfigure}
	\begin{subfigure}[b]{0.26\linewidth}
		\includegraphics[width=\linewidth]{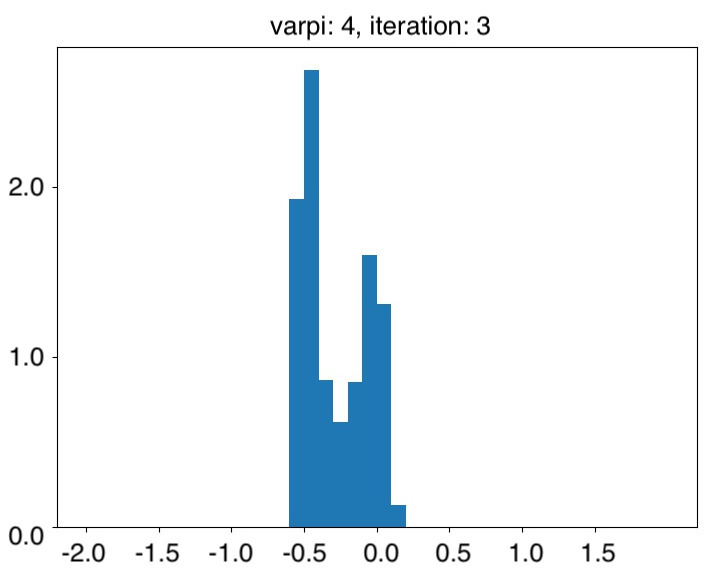}
		%\caption{Evolution of $\frac{\ud}{\ud \beta}{\mathscr J}$ with $\kappa$.}
	\end{subfigure}
	
		\begin{subfigure}[b]{0.26\linewidth}
		\includegraphics[width=\linewidth]{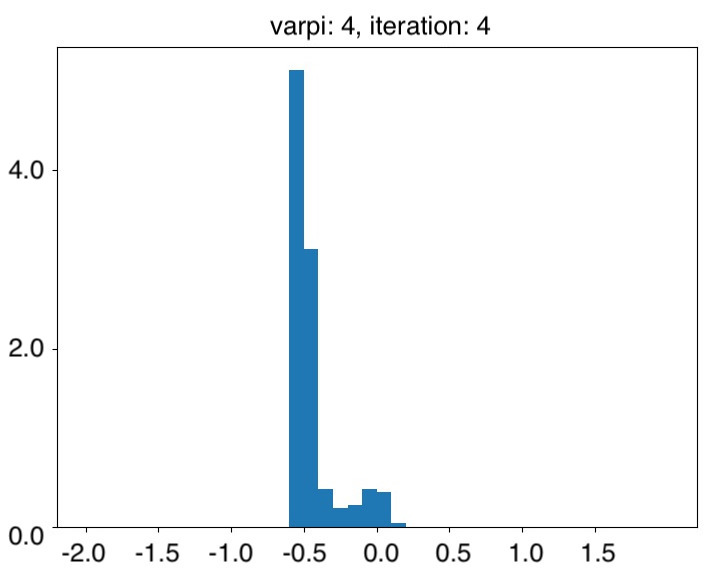}
		%\caption{Evolution of $\frac{\ud}{\ud \beta}{\mathscr J}$ with $\kappa$.}
	\end{subfigure}
		\begin{subfigure}[b]{0.26\linewidth}
		\includegraphics[width=\linewidth]{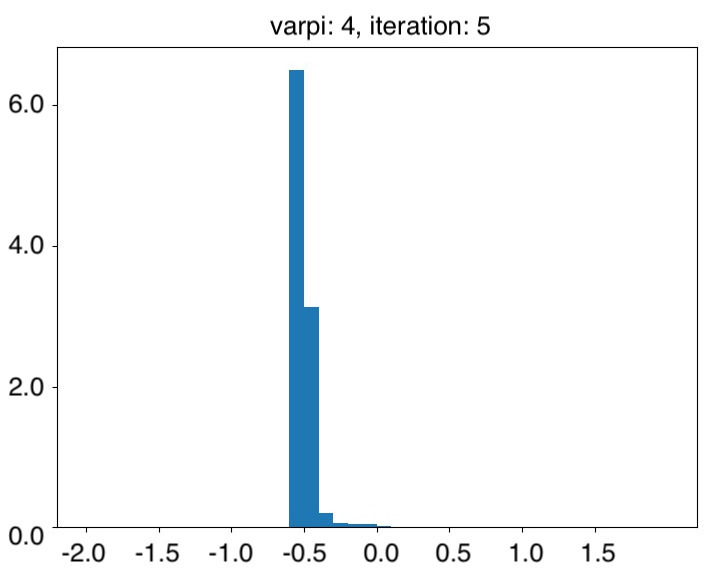}
		%\caption{Evolution of $\frac{\ud}{\ud \beta}{\mathscr J}$ with $\kappa$.}
	\end{subfigure}
		\caption{Selection of a solution by vanishing viscosity: histogram under the tilted measure of the terminal mean for 
	$\varepsilon=0.2$ and
	$\varpi=4$. Equilibrium is clearly selected in 5 iterations. }
%	\begin{subfigure}[b]{0.19\linewidth}
%		\includegraphics[width=\linewidth]{histogram_for_viscosity_0_5.png}
%		%\caption{Evolution of ${\mathscr J}$ with $\kappa$.}
%	\end{subfigure}
%	\begin{subfigure}[b]{0.19\linewidth}
%		\includegraphics[width=\linewidth]{histogram_for_viscosity_0_4.png}
%		%\caption{Evolution of $\frac{\ud}{\ud \beta}{\mathscr J}$ with $\kappa$.}
%	\end{subfigure}
%	\begin{subfigure}[b]{0.19\linewidth}
%		\includegraphics[width=\linewidth]{histogram_for_viscosity_0_3.png}
%		%\caption{Evolution of $\frac{\ud}{\ud \beta}{\mathscr J}$ with $\kappa$.}
%	\end{subfigure}
%	\begin{subfigure}[b]{0.19\linewidth}
%		\includegraphics[width=\linewidth]{histogram_for_viscosity_0_2.png}
%		%\caption{Evolution of $\frac{\ud}{\ud \beta}{\mathscr J}$ with $\kappa$.}
%	\end{subfigure}
%		\begin{subfigure}[b]{0.19\linewidth}
%		\includegraphics[width=\linewidth]{histogram_for_viscosity_0_1.png}
%		%\caption{Evolution of $\frac{\ud}{\ud \beta}{\mathscr J}$ with $\kappa$.}
%	\end{subfigure}
\label{fig:selection}
\end{figure}

\begin{figure}[h!]
	\begin{subfigure}[b]{0.26\linewidth}
		\includegraphics[width=\linewidth]{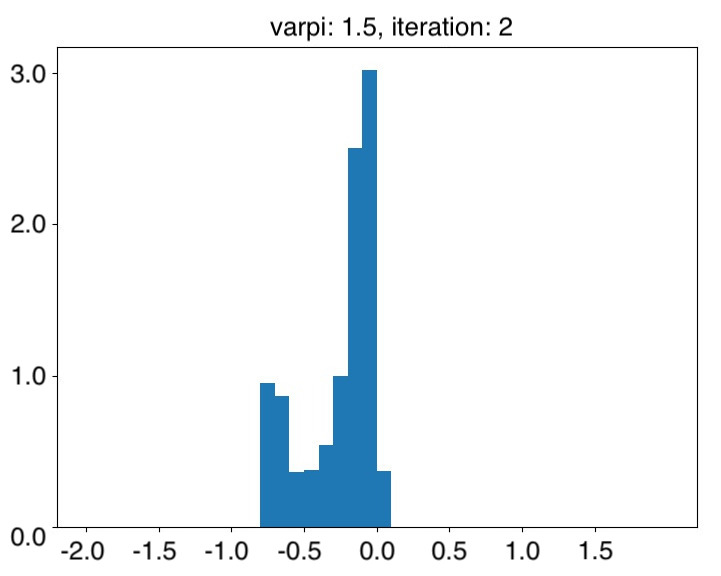}
		%\caption{Evolution of $\frac{\ud}{\ud \beta}{\mathscr J}$ with $\kappa$.}
	\end{subfigure}
	\begin{subfigure}[b]{0.26\linewidth}
		\includegraphics[width=\linewidth]{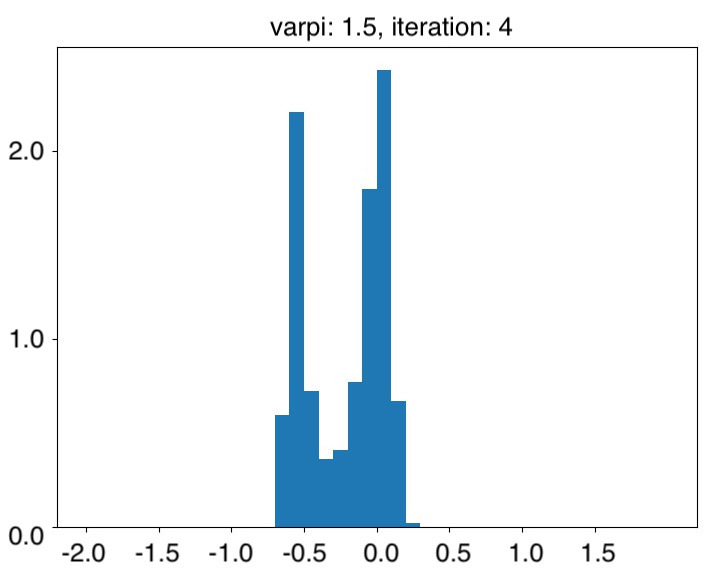}
		%\caption{Evolution of $\frac{\ud}{\ud \beta}{\mathscr J}$ with $\kappa$.}
	\end{subfigure}
	\begin{subfigure}[b]{0.26\linewidth}
		\includegraphics[width=\linewidth]{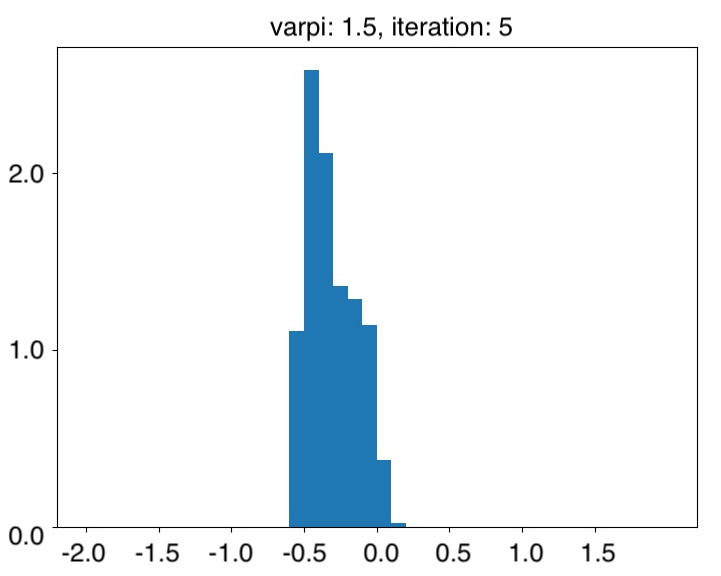}
		%\caption{Evolution of $\frac{\ud}{\ud \beta}{\mathscr J}$ with $\kappa$.}
	\end{subfigure}
	
		\begin{subfigure}[b]{0.26\linewidth}
		\includegraphics[width=\linewidth]{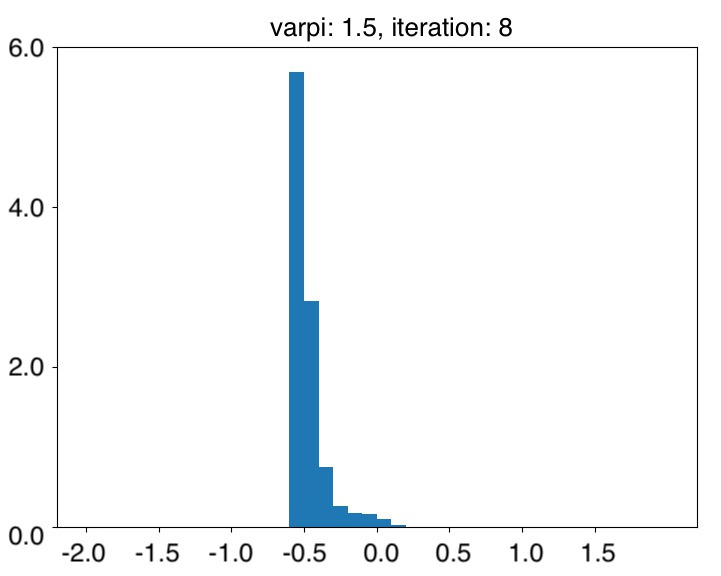}
		%\caption{Evolution of $\frac{\ud}{\ud \beta}{\mathscr J}$ with $\kappa$.}
	\end{subfigure}
		\begin{subfigure}[b]{0.26\linewidth}
		\includegraphics[width=\linewidth]{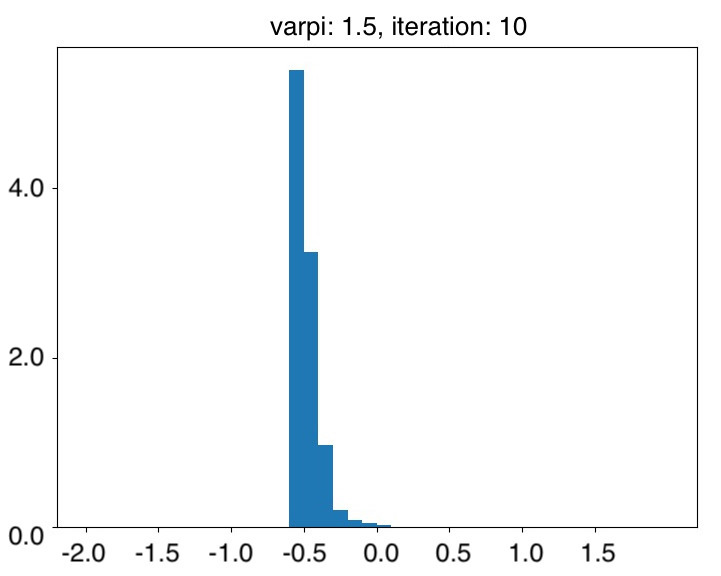}
		%\caption{Evolution of $\frac{\ud}{\ud \beta}{\mathscr J}$ with $\kappa$.}
	\end{subfigure}
%	\begin{subfigure}[b]{0.19\linewidth}
%		\includegraphics[width=\linewidth]{histogram_for_viscosity_0_5.png}
%		%\caption{Evolution of ${\mathscr J}$ with $\kappa$.}
%	\end{subfigure}
%	\begin{subfigure}[b]{0.19\linewidth}
%		\includegraphics[width=\linewidth]{histogram_for_viscosity_0_4.png}
%		%\caption{Evolution of $\frac{\ud}{\ud \beta}{\mathscr J}$ with $\kappa$.}
%	\end{subfigure}
%	\begin{subfigure}[b]{0.19\linewidth}
%		\includegraphics[width=\linewidth]{histogram_for_viscosity_0_3.png}
%		%\caption{Evolution of $\frac{\ud}{\ud \beta}{\mathscr J}$ with $\kappa$.}
%	\end{subfigure}
%	\begin{subfigure}[b]{0.19\linewidth}
%		\includegraphics[width=\linewidth]{histogram_for_viscosity_0_2.png}
%		%\caption{Evolution of $\frac{\ud}{\ud \beta}{\mathscr J}$ with $\kappa$.}
%	\end{subfigure}
%		\begin{subfigure}[b]{0.19\linewidth}
%		\includegraphics[width=\linewidth]{histogram_for_viscosity_0_1.png}
%		%\caption{Evolution of $\frac{\ud}{\ud \beta}{\mathscr J}$ with $\kappa$.}
	\caption{Selection of a solution by vanishing viscosity: histogram under the tilted measure of the terminal mean for 
	$\varepsilon=0.2$ and
	$\varpi=1.5$. Equilibrium is selected in 10 iterations. }
%	\end{subfigure}
\label{fig:selection:2}
\end{figure}

\begin{figure}[h!]
	\begin{subfigure}[b]{0.26\linewidth}
		\includegraphics[width=\linewidth]{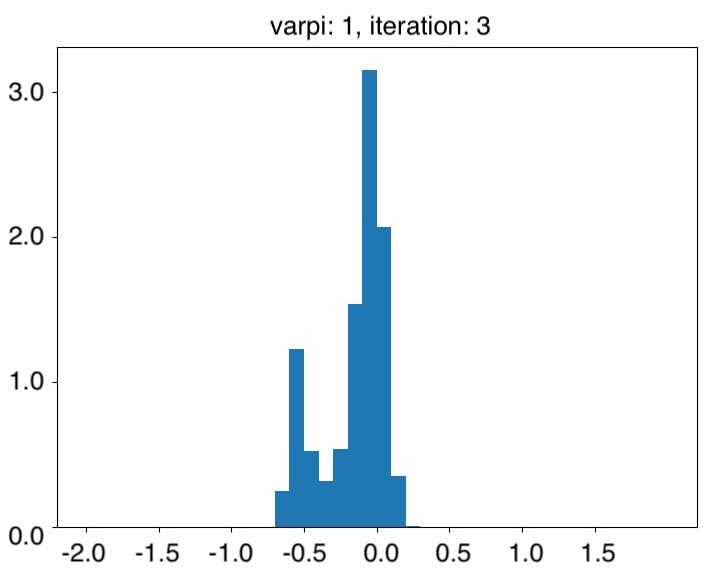}
		%\caption{Evolution of $\frac{\ud}{\ud \beta}{\mathscr J}$ with $\kappa$.}
	\end{subfigure}
	\begin{subfigure}[b]{0.26\linewidth}
		\includegraphics[width=\linewidth]{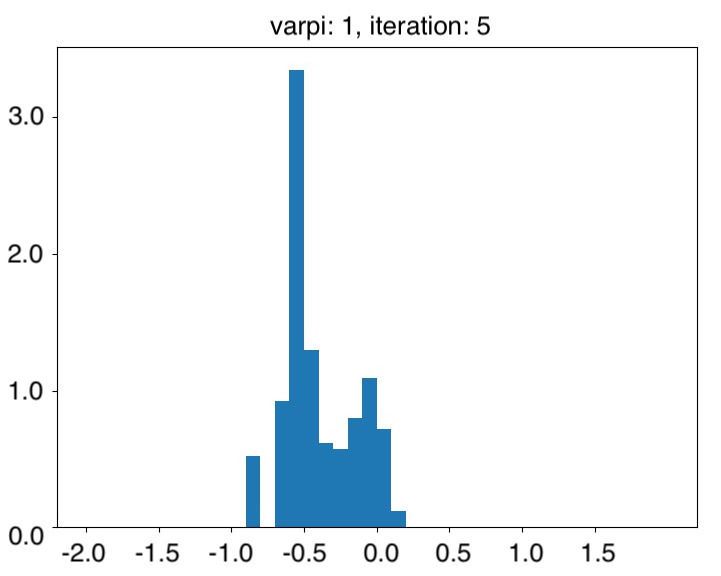}
		%\caption{Evolution of $\frac{\ud}{\ud \beta}{\mathscr J}$ with $\kappa$.}
	\end{subfigure}
	\begin{subfigure}[b]{0.26\linewidth}
		\includegraphics[width=\linewidth]{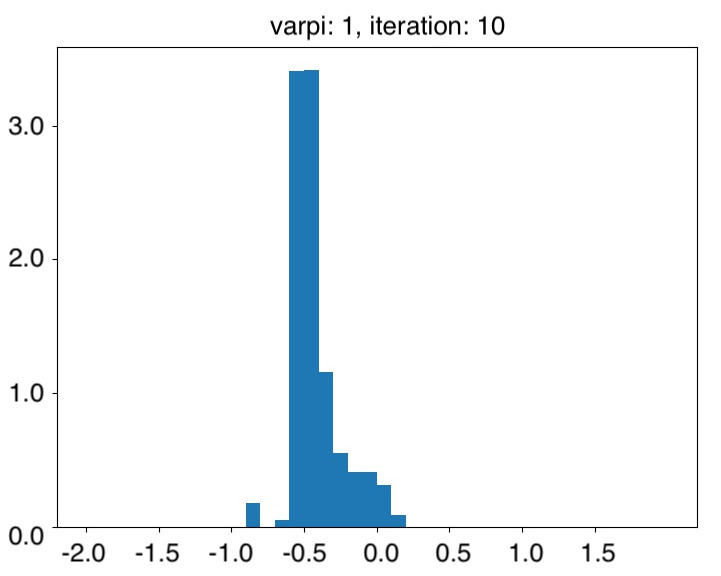}
		%\caption{Evolution of $\frac{\ud}{\ud \beta}{\mathscr J}$ with $\kappa$.}
	\end{subfigure}
	
		\begin{subfigure}[b]{0.26\linewidth}
		\includegraphics[width=\linewidth]{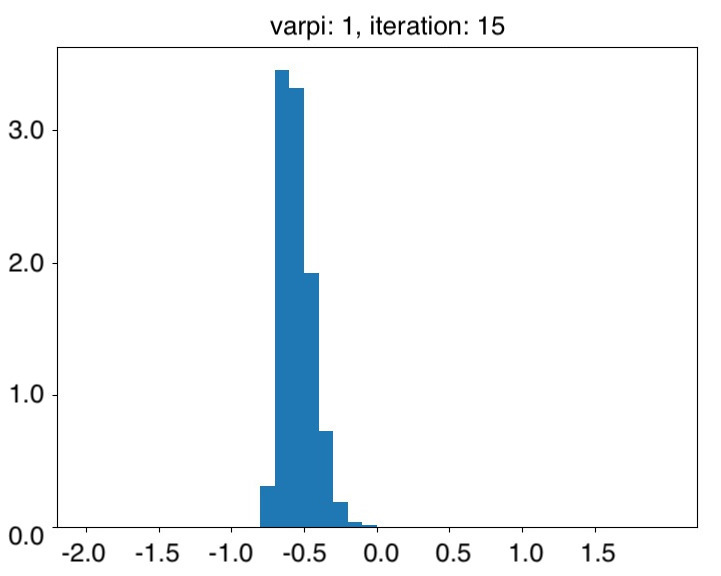}
		%\caption{Evolution of $\frac{\ud}{\ud \beta}{\mathscr J}$ with $\kappa$.}
	\end{subfigure}
		\begin{subfigure}[b]{0.26\linewidth}
		\includegraphics[width=\linewidth]{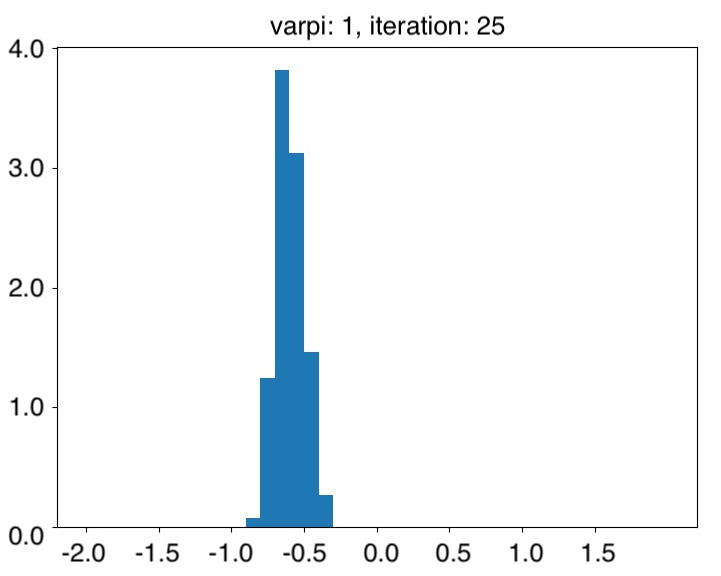}
		%\caption{Evolution of $\frac{\ud}{\ud \beta}{\mathscr J}$ with $\kappa$.}
	\end{subfigure}
	\caption{Selection of a solution by vanishing viscosity: histogram under the tilted measure of the terminal mean for 
	$\varepsilon=0.2$ and
	$\varpi=1$. Equilibrium is selected in more than 10 iterations, with a non-trivial residual variance. }
%	\begin{subfigure}[b]{0.19\linewidth}
%		\includegraphics[width=\linewidth]{histogram_for_viscosity_0_5.png}
%		%\caption{Evolution of ${\mathscr J}$ with $\kappa$.}
%	\end{subfigure}
%	\begin{subfigure}[b]{0.19\linewidth}
%		\includegraphics[width=\linewidth]{histogram_for_viscosity_0_4.png}
%		%\caption{Evolution of $\frac{\ud}{\ud \beta}{\mathscr J}$ with $\kappa$.}
%	\end{subfigure}
%	\begin{subfigure}[b]{0.19\linewidth}
%		\includegraphics[width=\linewidth]{histogram_for_viscosity_0_3.png}
%		%\caption{Evolution of $\frac{\ud}{\ud \beta}{\mathscr J}$ with $\kappa$.}
%	\end{subfigure}
%	\begin{subfigure}[b]{0.19\linewidth}
%		\includegraphics[width=\linewidth]{histogram_for_viscosity_0_2.png}
%		%\caption{Evolution of $\frac{\ud}{\ud \beta}{\mathscr J}$ with $\kappa$.}
%	\end{subfigure}
%		\begin{subfigure}[b]{0.19\linewidth}
%		\includegraphics[width=\linewidth]{histogram_for_viscosity_0_1.png}
%		%\caption{Evolution of $\frac{\ud}{\ud \beta}{\mathscr J}$ with $\kappa$.}
%	\end{subfigure}
\label{fig:selection:3}
\end{figure}

For sure, the reader may wonder about the same plots if we decrease the value of 
$\varepsilon$ (say $\varepsilon=0.1$). 
It is fair to recognize that the results are not so good when $\varepsilon=0.1$. To our mind, this is due to the variance issues that we reported above:
the Girsanov change of reference may induce high variances. Numerical observations seem to indicate that those issues may even get worse when increasing 
the value of 
$\varpi$ (which is not so easy to explain from a theoretical point of view because the integrand $\varpi {\boldsymbol h}^n$ 
in 
\eqref{eq:opti:n+1:n}
is close to the true solution ${\boldsymbol h}$, regardless of the value of $\varpi$). One possible strategy
 to
 reduce the underlying variance 
 would be to implement the same 
 preferential sampling
 argument as in 
 the first arXiv version \cite{DelarueVasileiadis}, but it is fair to say that it is difficult 
 to obtain a plot that is as good as 
 the one obtained in Figure  
 \ref{fig:selection}. 
 For this reason, we feel better 
to address  
 this possible extension in a future work.  

\section*{Acknowledgement}
\noindent The authors are grateful to Mathieu Lauri\`ere for very helpful discussions on the topic.  
They are also grateful to the two anonymous referees whose comments allowed to improve the first version of this work. 
 \bibliographystyle{abbrv}
\bibliography{references2c}

\end{document}